\documentclass[11pt]{amsart}
\usepackage{amsmath}
\usepackage{amsfonts}
\usepackage{amssymb}
\usepackage{amsthm}
\usepackage{enumerate}
\usepackage[totalwidth=450pt, totalheight=650pt, centering]{geometry}

\usepackage{bbm}
\usepackage[mathscr]{eucal}
\usepackage{tikz}
\usepackage{bbm}

\usepackage{graphicx}
\usepackage{verbatim}

\usepackage[T2A,T1]{fontenc}

\newcommand{\Zhe}{\mathrm{Reg}}
\newcommand{\Sha}{\mathrm{Int}}

\newcommand{\iunit}{\imath}
\newcommand{\w}{}
\newcommand{\q}{}
\renewcommand{\zeta}{u}

\newcommand{\uzetah}{\widehat{u}}
\newcommand{\ups}{u}
\newcommand{\upsh}{\widehat u}
\newcommand{\sU}{{\mathscr U}}
\newcommand{\carl}{{\mathit{carl}}}

\newcommand\bbone{{\mathbbm 1}}

\newcommand{\uzeta}{u}
\newcommand{\R}{{\mathbb{R}}}

\newcommand{\Z}{{\mathbb{Z}}}
\newcommand{\N}{{\mathbb{N}}}
\newcommand{\sS}{\mathcal{S}}

\newcommand{\sK}{\mathcal{K}}
\newcommand{\dil}[2]{#1^{\q(#2\w)} }
\newcommand{\hh}{\tilde{h}}
\newcommand{\grad}{\nabla}

\newcommand{\sE}{\mathcal{E}}
\newcommand{\SI}{\mathrm{SI}}

\newcommand{\Qtt}{\widetilde{\mathcal Q}}
\newcommand{\Qt}{\mathcal Q}
\newcommand{\tcQ}{\widetilde{\mathcal Q}}

\newcommand{\Qb}[2]{\overline{Q}_{#1}\q[#2\w]}

\newcommand{\chiz}{\chi_0}
\newcommand{\chizh}{\widehat{\chiz}}
\newcommand{\chil}{\chi_l}
\newcommand{\chilh}{\widehat{\chi_l}}

\newcommand{\CziRd}{C_0^\infty\q(\R^d\w)}
\newcommand{\Czip}[1]{C_0^\infty\q(#1\w)}

\newcommand{\vsig}{\varsigma}
\newcommand{\vsigt}{\tilde{\vsig}}
\newcommand{\vsigj}{\vsig^{(2^j)}}
\newcommand{\vsigjj}{\vsig_j^{(2^j)}}
\newcommand{\vectsig}{{\vec \vsig}}

\newcommand{\phit}{\tilde{\phi}}

\newcommand{\etaph}{\widehat{\eta'}}
\newcommand{\etah}{\widehat{\eta}}

\newcommand{\vsigh}{\widehat{\vsig}}
\newcommand{\chit}{\tilde{\chi}}

\newcommand{\sKN}[2]{\q\| #2\w\|_{\sK_{#1}}}
\newcommand{\sUN}[1]{\q\|#1\w\|_{\sU}}

\newcommand{\sBp}[1]{\mathcal{B}\ci{#1}}
\newcommand{\sBtp}[2]{\mathcal{B}_{#1}\q(#2\w)}
\newcommand{\sBN}[2]{\q\|#2\w\|_{\sBp{#1}} }
\newcommand{\sBzN}[1]{\q\|#1\w\|_{L^1}}
\newcommand{\fBp}[1]{\mathfrak{B}_{#1}}
\newcommand{\fBtp}[2]{\fBp{#1}\q(#2\w)}
\newcommand{\fBN}[2]{{\q\|#2\w\|_{\fBp{#1} }}}

\newcommand{\BesN}[2]{{\|#1\w\|_{B^{\epsilon}_{1,\infty}(#2\w)}}}
\newcommand{\BesNp}[3]{{\q\|#2\w\|_{B^{#1}_{1,\infty}\q(#3\w)}}}

\newcommand{\RPn}{\R\mathrm{P}^n}

\newcommand{\Loloc}{L^1_{\mathrm{loc}}}
\newcommand{\pr}{\mathrm{proj}}

\newcommand{\LpN}[2]{\q\| #2 \w\|_{L^{#1}}}
\newcommand{\LpRdN}[2]{\q\| #2 \w\|_{L^{#1}\q(\R^d\w)}}

\newcommand{\Ltip}[2]{\q<#1, #2\w>_{L^2}}

\newcommand{\CtN}[1]{\q\|#1\w\|_{C^2}}
\newcommand{\CoN}[1]{\q\|#1\w\|_{C^1}}
\newcommand{\CzN}[1]{\q\|#1\w\|_{C^0}}
\newcommand{\Op}{{\mathrm{Op}}}
\newcommand{\LpOpN}[2]{\q\| #2\w\|_{L^{#1}\rightarrow L^{#1}}}
\newcommand{\LtOpN}[1]{\LpOpN{2}{#1}}

\newcommand{\supp}[1]{\mathrm{supp}\q(#1\w)}
 
\newcommand{\Dilj}{{\mathrm {Dil}_{2^j}}}

\newcommand{\dual}{{\mathrm{dual}}}
\newcommand{\av}{{\mathrm{av}}}
\newcommand{\At}{{\mathrm{At}}}
\newcommand{\Carl}{{\mathrm{Carl}}}
\newcommand{\rmr}{{\mathrm {rt}}}
\newcommand{\rml}{{\mathrm {lt}}}
\newcommand{\ann}{{\mathrm{Ann}}}

\newcommand{\ci}[1]{_{{}_{\scriptstyle{#1}}}}

\newcommand{\tr}[1]{{}^{t}\!#1}
\newcommand{\Be}{\begin{equation}}
\newcommand{\Ee}{\end{equation}}

\newcommand{\Bm}{\begin{multline}}
\newcommand{\Em}{\end{multline}}
\newcommand{\Bea}{\begin{eqnarray}}
\newcommand{\Eea}{\end{eqnarray}}
\newcommand{\Beas}{\begin{eqnarray*}}
\newcommand{\Eeas}{\end{eqnarray*}}
\newcommand{\Benu}{\begin{enumerate}}
\newcommand{\Eenu}{\end{enumerate}}
\newcommand{\Bi}{\begin{itemize}}
\newcommand{\Ei}{\end{itemize}}

\def\sB{{\mathscr {B}}}
\def\sK{{\mathscr {K}}}

\def\sC{{\mathscr {C}}}

\def\a{\alpha}

\def\intslash{\rlap{\kern  .32em $\mspace {.5mu}\backslash$ }\int}
\def\qsl{{\rlap{\kern  .32em $\mspace {.5mu}\backslash$ }\int_{Q_x}}}

\def\vth{\vartheta}

\def\vp{\varpi}

\def\emph#1{{\it #1 }}

\def\diam{{\mathrm{diam}}}

\def\Ga{\Gamma}
\def\ga{\gamma}

\def\loc{{\mathrm {loc}}}

\def\comp{{\mathrm {comp}}}

\def\dist{{\mathrm{dist}}}
\def\diag{{\mathrm{diag}}}

\def\inn#1#2{\langle#1,#2\rangle}
\def\biginn#1#2{\big\langle#1,#2\big\rangle}

\def\noi{\noindent}

\def\lc{\lesssim}
\def\gc{\gtrsim}

\def\eps{\varepsilon}
\def\ep{\epsilon}

\def\ka{\kappa}
            
\def\la{\lambda}             \def\La{\Lambda}
\def\sig{\sigma}             
\def\si{\sigma}              
\def\vphi{\varphi}
             
\def\om{\omega}              \def\Om{\Omega}

\def\fB{{\mathfrak {B}}}

\def\fI{{\mathfrak {I}}}

\def\fK{{\mathfrak {K}}}

\def\fM{{\mathfrak {M}}}

\def\fP{{\mathfrak {P}}}
\def\fQ{{\mathfrak {Q}}}

\def\fS{{\mathfrak {S}}}

\def\fz{{\mathfrak {z}}}

\def\bbC{{\mathbb {C}}}

\def\bbN{{\mathbb {N}}}

\def\bbR{{\mathbb {R}}}

\def\bbT{{\mathbb {T}}}

\def\bbZ{{\mathbb {Z}}}

\def\cB{{\mathcal {B}}}
\def\cC{{\mathcal {C}}}
\def\cD{{\mathcal {D}}}

\def\cI{{\mathcal {I}}}
\def\cJ{{\mathcal {J}}}

\def\cL{{\mathcal {L}}}

\def\cQ{{\mathcal {Q}}}
\def\cR{{\mathcal {R}}}
\def\cS{{\mathcal {S}}}
\def\cT{{\mathcal {T}}}

\def\cV{{\mathcal {V}}}

\def\tS{{\widetilde{S}}}

\def\tcQ{{\widetilde{\mathcal {Q}}}}

\def\be#1{\begin{equation}\label{ #1}}
\def\endeq{\end{equation}}
\def\endal{\end{align}}
\def\bas{\begin{align*}}
\def\eas{\end{align*}}
\def\bi{\begin{itemize}}
\def\ei{\end{itemize}}
\def\eps{\varepsilon}

\def\emph#1{{\it #1}}
\def\textbf#1{{\bf #1}}

\def\compl{\complement}
\def\Dil{\mathrm {Dil}}

\newtheorem{thm}{Theorem}[section]
\newtheorem{cor}[thm]{Corollary}
\newtheorem{prop}[thm]{Proposition}
\newtheorem{lemma}[thm]{Lemma}

\newtheorem*{namedtheorem}{\theoremname}
\newcommand{\theoremname}{testing}
\newenvironment{named}[1]{\renewcommand{\theoremname}{#1}
\begin{namedtheorem}}
{\end{namedtheorem}}

\theoremstyle{remark}
\newtheorem{rmk}[thm]{Remark}

\theoremstyle{definition}
\newtheorem{defn}[thm]{Definition}

\theoremstyle{remark}
\newtheorem{example}[thm]{Example}

   \newtheorem*{rem}{Remark}

\numberwithin{equation}{section}
\begin{document}

\title[Multilinear singular integral forms  of Christ-Journ\'e type]{Multilinear singular integral forms \\ of Christ-Journ\'e type}

\author{Andreas Seeger \ \ \ Charles K. Smart \ \ \ Brian Street} 

\address{Andreas Seeger \\ Department of Mathematics \\ University of Wisconsin \\480 Lincoln Drive\\ Madison, WI, 53706, USA} \email{seeger@math.wisc.edu}

\address{Charles K. Smart\\ Department of Mathematics \\ University of Chicago
\\734 S. University Avenue\\
Chicago, IL 60637, USA}
\email{smart@math.uchicago.edu}

\address{Brian Street \\ Department of Mathematics \\ University of Wisconsin \\480 Lincoln Drive\\ Madison, WI, 53706, USA} \email{street@math.wisc.edu}

\begin{abstract} We prove $L^{p_1}(\bbR^d)\times \dots \times L^{p_{n+2}}(\bbR^{d})$
polynomial growth estimates
for the Christ-Journ\'e 
 multilinear singular integral forms and suitable generalizations.
 \end{abstract}


\thanks{A. Seeger was supported in part by NSF grants DMS  1200261 and  DMS 1500162. C. Smart was supported in part by NSF grant  DMS 1461988 and by a Sloan fellowship. B. Street was supported 
in part by NSF grant DMS 1401671.}

\maketitle



\setcounter{tocdepth}{2}
\tableofcontents
\section{Introduction}


\subsection{The  $d$-commutators} 


Let $0<\ep <1$ and let  $\ka\in \cS'(\bbR^d)\cap \Loloc(\bbR^d\setminus \{0\})$  be  a regular Calder\'on-Zygmund convolution kernel on $\bbR^d$, 
satisfying
\begin{subequations}\label{regcz}
the standard size and regularity assumptions,
\begin{gather}\label{regczsize}
|\ka(x)|\le C |x|^{-d},\quad  x\neq 0,
\\
\label{regczdiff}
 |\ka(x+h)-\ka(x)| \le C \frac{|h|^{\ep}}{|x|^{d+\ep}}, \quad x\neq 0, \,\,|h|\le \frac{|x|}{2},
 \end{gather}
 and the $L^2$ boundedness condition
\Be
\label{regczF}\|\widehat \ka\|_\infty\le C<\infty.
\Ee
\end{subequations}
Let  $\|\ka\|_{CZ(\ep)}$ be the smallest constant $C$ for which
the three inequalities \eqref{regcz} hold simultaneously.
For convenience, in order to a priori make  sense of some of the expressions in this introduction  the reader may initially 
assume that $\ka$ is compactly supported in $\bbR^d\setminus\{0\}$.

For  $a\in L^1_\loc(\bbR^d)$  
let $m_{x,y}a$ be the mean of $a$ over the interval connecting $x$ and $y$,  $$m_{x,y}a=\int_0^1 a(sx+(1-s) y) ds.$$  For every $y\in \bbR^d$ this is well defined for almost all $x \in \bbR^d$.
Given $L^\infty$-functions $a_1,\dots, a_n$  on $\bbR^d$ the 
{\it $n$th order $d$-commutator} associated to $a_1,\dots, a_n$, 
is defined by
$$\sC [a_1,\dots, a_n]f(x)
=\int \ka(x-y) \big(\prod_{i=1}^n m_{x,y}a_i\big)
 f(y) dy.
$$
One may consider $\sC$  as an $(n+1)$-linear operator acting on $a_1,\dots, a_n,f$. Pairing with another function and renaming $a_i=f_i$, $i\le n$, 
 $f=f_{n+1}$ one obtains the {\it Christ-Journ\'e multilinear form} defined by
\Be\begin{aligned}\label{CJform} \La_{\rm CJ} (f_1,\dots, f_{n+2})&= 
\iint \ka(x-y) \big(\prod_{i=1}^n m_{x,y}f_i\big) 
f_{n+1}(y) f_{n+2}(x)\, dx\, dy\,.
\end{aligned}
\Ee

In dimension $d=1$ this operator reduces to the Calder\'on commutator. However the emphasis in this paper is on the behavior in dimension $d\ge 2$ where the Schwartz kernels are considerably less regular. 
Christ and Journ\'e \cite{cj} showed that for $a_i$ with $\|a_i\|_\infty \le 1$
the operator 
$\sC[a_1,\dots,a_n]$ is bounded on $L^p$,  $1<p<\infty$,  with operator norm  $O( n^{\alpha})$, for $\alpha>2$. More precisely,
\Be\label{CJbound}
\big|\La_{\rm CJ}(f_1,\dots, f_{n+2})\big|\le C_{p,\epsilon,\alpha} \|K\|_{CZ(\epsilon)}\,n^{\alpha}
\big(\prod_{i=1}^{n}\|f_i\|_\infty\big) \|f_{n+1}\|_p\|f_{n+2}\|_{p'}, \quad \alpha>2.
\Ee 
For related results on Calder\'on commutators for  $d=1$ see the discussion of previous results in \S\ref{background} below.

The form $\La_{\rm CJ} $  is not symmetric in $f_i$, $i=1,\dots, n+2$,
(see the discussion in \S\ref{towards} below)
 and it is natural to ask whether the analogous estimates hold for $f_i \in L^{p_i}$, for other choices of $p_i$. 
The problem has been proposed for example 
in \cite{dgy} and \cite{grafakos-czech}, see also \S\ref{background} for our motivation. 
Homogeneity considerations yield the necessary condition
$\sum_{i=1}^{n+2} p_i^{-1}=1$. In this paper we shall establish the following estimate, as a corollary of  a more general result stated  as Theorem \ref{ThmOpResBoundGen} below.

\medskip

\begin{thm} \label{newCJthm} 
Suppose that $d\ge 1$,  $1<p_i\le\infty$, $i=1,\dots, n+2$, and $\sum_{i=1}^{n+2} p_i^{-1}=1$. 
Let $\ep>0$ and $\min\{p_1,\dots, p_{n+2}\}\ge 1+\delta$. Then  for $\La$ as in \eqref{CJform}
\Be\label{newCJest}
\big|\La_{\rm CJ} (f_1,\dots, f_{n+2})\big|\le 
C(\delta)  \|\ka\|_{CZ(\ep)} 
 n^2 \log^3(2+n)
\prod_{i=1}^{n+2}\|f_i\|_{p_i}\,.
\Ee
\end{thm} 
Our main interest lies in the higher dimensional cases with $d\ge 2$. Polynomial bounds 
are known for $d=1$, although the precise form of Theorem \ref{newCJthm}  may not have been observed before;  see the discussion about previous results in \S \ref{background}.

\subsection{Background and historical remarks} 
\label{background}

\subsubsection*{Motivation}

Our original motivation  for considering estimates \eqref{newCJest} for $p_i\neq \infty$ for $i\le n$  came from  Bressan's problem (\cite{bressan}) on incompressible mixing flows. 
A  version of the approach chosen by Bianchini \cite{bianchini}
leads in higher dimensions 
to the problem of bounding a  
  trilinear singular integral form with even homogeneous kernels $\kappa$. 
One considers a smooth, time-dependent vector field $(x,t)\mapsto \vec b(x,t)$  
which is periodic, i.e. 
$\vec b(x+k,t)=\vec b(x,t)$  for all  $ (x,t)\in \bbR^d\times \bbR$, $ k\in \bbZ^d\,,$ and  divergence-free, $\sum_{i=1}^d \frac{\partial b_i}{\partial x_i}=0.$
Let $\phi $ be the flow generated by $v$, i.e. we have 
$\frac{\partial}{\partial t} \phi_t(x)= v(\phi_t(x),t)$, $\quad \phi_0(x)=x$,
so that for every $t$ the map $\phi_t $ is a diffeomorphism on
$\bbR^d$
satisfying  $\phi(x+k,t)=k+\phi(x,t)$,
for all $x\in \bbR^d$, $k\in \bbZ^d$.

For small $\eps$ consider the truncated Bianchini semi-norm (\cite{bianchini}) defined by
$$
B_\eps[f]= \int_{\eps}^{1/4} \int_{Q} \Big|f(x)- \intslash_{B_r(x)} f(y) dy \Big| \,dx\, \frac{dr}{r}\,.
$$
Let $A$ be a measurable subset of $\bbR^d$ which is invariant under translation by vectors in $\bbZ^d$ (thus  $A+\bbZ^d$  can be identified with a measurable subset of $\bbT^d$). Let $A^\compl=\bbR^d\setminus A$.

A calculation  (\cite{hsss})  shows that
\begin{multline}\label{bianchinidifference}
B_\eps[\bbone_{\phi_T(A)}]-
B_\eps[\bbone_{A}] =\\
V_d^{-1}\int_0^T \int_Q f(x,t) 
\int_{\eps\le |x-y|\le 1/4}\frac{\inn{x-y}{\vec b(x,t)-\vec b(y,t)}}{|x-y|^{d+2}} f(y,t) \,dy\,dx\, dt
\end{multline}
where $Q=[0,1)^d$, $f(y,t)= \frac 12(\bbone_{\phi_t(A)}- \bbone_{\phi_t(A)^\compl})$ and $V_d$ is the volume of the unit ball in $\bbR^d$.

This calculation leads to   an alternative approach to a result by Crippa and DeLellis \cite{Crippa-deLellis}. 
One has  the following estimate 
involving general (a priori) smooth vector fields $x\mapsto v(x)$ on $\bbR^d$  satisfying $\text{div}(v)=0$.
Let 
$Dv$ denote its total derivative. Then for $1< p_1,p_2, p_3\le \infty$,
$\sum_{i=1}^3p_i^{-1}=1$,
\Be\label{bressanintegral}
\Big| \iint_{\eps<|x-y|<N} \frac {\inn{v(x)-v(y)}{x-y}}{|x-y|^{d+2}} f(y) g(x)\,dy\, dx
\Big| \lc \|Dv\|_{p_1} \|f\|_{p_2}\|g\|_{p_3}.
\Ee
Here the implicit constant is independent of $\eps$ and $N$. One can think of \eqref{bressanintegral} as a trilinear form acting on $f$, $g$ and $Dv$; due to the assumption of zero divergence, the entries are not independent and 
one can reduce to the estimation of $d^2-1$ trilinear forms.
In fact, \eqref{bressanintegral}  can be 
 derived from the case $n=1$ of Theorem \ref{newCJthm}, using the choices of
 \Be\label{bressankernels}
 \begin{aligned}\ka_{ij}(x)&= \frac{x_ix_j}{|x|^{d+2}}, \quad i \neq j,\\
 \ka_i(x)&= \frac{x_i^2-x_d^2}{|x|^{d+2}}, \quad 1\le i<d.
 \end{aligned} 
 \Ee
 The case 
with   $f$, $g$ being characteristic functions of sets with finite measure 
and $Dv\in L^{p_1}$ with $p_1$ near $1$ is of particular interest.
Steve Hofmann (personal communication) has suggested that estimates such as \eqref{bressanintegral} can also be obtained from the isotropic version of his off-diagonal $T1$ theorem 
 \cite{hofmann-off-diagonal}.

\subsubsection*{Previous results}
We  list some previous results on the $n+2$-linear form $\La_{\rm CJ}$ in \eqref{CJbound}, including many in dimension $d=1$, covering the classical
Calder\'on commutators.

\smallskip

(i)  The first estimates of the form \eqref{newCJest}, for the case $d=1$ and $n=1$ 
were proved in the seminal paper by A.P. Calder\'on \cite{calderon}.

\smallskip 

(ii) More generally, still in dimension $d=1$, Coifman, McIntosh and Meyer 
\cite{coi-mc-me} proved estimates of the form \eqref{newCJest} for arbitrary $n$, with 
$p_1=\dots=p_n=\infty$ and polynomial bounds $C(n)= O(n^4)$ as $n\to\infty$. This allowed them to establish the $L^2$ boundedness of the Cauchy integral operator on
general Lipschitz curves. See also 
\cite{cdm} for other applications to  related problems of Calder\'on.
Christ and Journ\'e \cite{cj} were able to improve the 
Coifman-McIntosh-Meyer bounds to $C(n)=O(n^{2+\eps})$ 
(and to $O(n^{1+\eps})$ 
for odd kernels $\kappa$).

\smallskip

(iii) Duong, Grafakos and Yan \cite{dgy} developed a rough version of the multisingular integral theory in \cite{grafakos-torres} to cover the estimates
\eqref{newCJest}
 with general exponents
for $d=1$, however  their arguments yield constants $C(n)$ which are  of exponential growth in $n$. 

One should note that \cite{dgy} also  treats  the higher Calder\'on commutators $\sC[f_1,\dots, f_n]$, with target space  $L^p$  where  $p>1/2$. For the bilinear version this had been first done by 
C.P. Calder\'on \cite{cpcalderon}. It would be interesting to obtain appropriate 
similar  results for the 
$d$-commutators.

\smallskip

(iv) Muscalu \cite{muscalu2} recently developed a new approach for proving \eqref{newCJest} in dimension $d=1$, see 
also \cite[Theorem 4.11]{muscalu-schlag2}. An explicit bound for the constant as $A(n, \ell)$ where $\ell$ is the number of indices $j$ such that $p_j\neq \infty$ and, for fixed $\ell$, $n\mapsto A(n,\ell)$ is of polynomial growth. 
However, by using complex interpolation (as in \S 15) to the case when $p_j=\infty$ for all but two $j$, one may remove the dependance of $A$ on $\ell$.  This yields polynomial bounds for all admissible sets of exponents, as in our results.

\smallskip

(v) As mentioned above, crucial  results for $d\ge 2$ were obtained by 
Christ and Journ\'e \cite{cj} who established \eqref{newCJest} for 
$p_1=\dots=p_n=\infty$ and $C(n)=O(n^{2+\eps})$. Several  ideas in our proof   can be traced back to  their work.

\smallskip

(vi) Hofmann \cite{hofmann} obtained estimates \eqref{newCJest} for operators with rougher kernels $\kappa$, and an extension to weighted norm inequalities; however the induction argument in \cite{hofmann} only gives exponential bounds as $n\to\infty$.

\smallskip

(vii)  For the special case that $\kappa$ is an {\it odd and  homogeneous} singular convolution kernel, estimates of the form \eqref{newCJest} for $d\ge 2$ and $n=1$ have been obtained by using the method of rotation. In 
\cite{dgy},
Duong, Grafakos and Yan 
 use uniform results  on the bilinear Hilbert transforms (\cite{grli}, \cite{th})
to obtain such estimates under the additional restriction $\min (p_1, p_2, p_3)>3/2$,
see also the survey  \cite{grafakos-czech}.

\smallskip

We note that one can modify the argument  in \cite{dgy}
to remove this restriction, and also to obtain a version for $n\ge 2$.
Indeed let $\ka_\Om(x)= |x|^{-d}\Omega({x}/{|x|})$ with $\Omega\in L^1(S^{d-1}) $ and $\Omega(\theta)=-\Omega(-\theta)$.
Let $$\sC_\Om[f_1,\dots, f_n]f_{n+1}(x)
= \int \ka_\Omega(x-y) f_{n+1}(y) \prod_{i=1}^n\int_0^1
f_i((1-s_i)x+s_iy) ds_i\, dy; $$
then
\Be\label{rotation}\sC_\Om[f_1,\dots, f_n]f_{n+1}(x)=
\frac 12 \int_{S^{d-1}} \Omega(\theta)\, \cC_\theta[f_1,\dots, f_n,f_{n+1}](x) \, d\theta
\Ee
where 
 $$\cC_\theta [f_1,\dots, f_{n+1}](x)=p.v.
\int_{-\infty}^\infty  
f_{n+1}(x-s\theta) \Big(\prod_{i=1}^n\int_0^1
f_i(x-u s \theta) du \Big)\frac{ds}s
$$
Now if $e_1=(1,0,\dots,0)$ and  $R_\theta$ is a rotation with $R_\theta e_1= \theta$ we have
$$\cC_\theta [f_1,\dots, f_{n+1}](x)= 
\cC_{e_1} [ f_1\!\circ\! R_\theta, \dots,f_{n+1}\!\circ\! R_\theta] (R_\theta^{-1}x)
$$ and thus the operator norms of $\sC_\theta$ are independent of $\theta$.
One notices that
$$\cC_{e_1}[f_1,\dots, f_{n+1}](x_1,x')= 
p.v. \int_{-\infty}^\infty \frac{1}{x_1-y_1} f_{n+1}(y_1, x') \prod_{i=1}^n
\Big(\int_0^1 f_i((1-u)x_1+u y_1,x') du\Big)  dy_1\,,
$$
the Calder\'on commutator acting in the first variable.
The  one-dimensional results for  the commutators  in \cite{calderon}, \cite{dgy}  can now be applied 
to show that for $\sum_{i=1}^{n+2} {p_i}^{-1}=1$, $p_i>1$, 
$$\Big|\int \sC_\Om[f_1,\dots, f_n]f_{n+1}(x)f_{n+2}(x) dx\Big|\lc C(p_1,\dots, p_{n+2})\|\Omega\|_{L^1(S^{d-1})}
\prod_{i=1}^{n+2} \|f_i\|_{L^{p_i}}.$$
Note that the assumption 
$\ka$ odd is crucial in formula  \eqref{rotation} and thus the argument does not seem to be applicable to the $d$-commutators associated with the convolution kernels in 
\eqref{bressankernels}.

(viii) When $n=1$ it is known that the Christ-Journ\'e commutator $\sC[a]$ (with $a\in L^\infty$) is of weak type $(1,1)$. This has been shown 
by Grafakos and Honz\'ik \cite{grafakos-honzik} in two dimensions and by one of the authors \cite{see} in all dimensions. It is an open problem whether the higher degree $d$-commutators $(n\ge 2)$ are of weak type $(1,1)$ in dimension $d\ge 2$. 

\smallskip

\subsection{Towards a more general result}
\label{towards} 
In order to prove Theorem \ref{newCJthm}
 it suffices to  prove estimate
\eqref{newCJest} for the cases where two of the exponents, say $p_i, p_j$, $1\le i<j\le n+2$ belong to $(1,\infty)$ and the other $n$ exponents are 
equal to $\infty$. 
Equivalently, if $\vp$ is a permutation of $\{1,\dots, n+2\}$
and
$$\La_{\mathrm {CJ}}^\vp(f_1,\dots, f_{n+2})= \La_{\rm CJ}(f_{\vp(1)},\dots, f_{\vp(n+2)})$$
one has to show, for $1<p<\infty$,  the inequalities
\Be\label{newCJestperm}\La_{\mathrm{CJ}}^\vp[f_1,\dots, f_{n+2}]\big|\le 
C_{\delta, p} n^2 (\log n)^3
 \|\ka\|_{CZ(\delta)} 
(\prod_{i=1}^{n}\|f_i\|_{\infty}) \|f_{n+1}\|_p \|f_{n+2}\|_{p'}\,,
\Ee
uniformly in $\vp$.

Formally the operator $\La_{\rm CJ}^\vp$ takes the form 
\Be \label{Lapermute}
\La_{\rm CJ}^\vp(f_1,\dots, f_{n+2})=
 \iiint K^\vp(\alpha, x-y)f_{n+2}(x) f_{n+1}(y) \prod_{i=1}^n f_i(x-\alpha_i(x-y))\: d\alpha\: dx\: dy.
 \Ee
 The case  $\varpi=id$ in \eqref{newCJestperm} is covered already by the original  result of Christ and Journ\'e. 
 Thus by the symmetry in $\{1,\dots, n\}$ and essential symmetry in $\{n+1,n+2\}$
 (with a change of variable $\alpha_j\mapsto(1-\alpha_j)$) two cases remain  of particular interest:

\begin{itemize}

\item
If  $\vp^i$ is the permutation that interchanges $i$ and $n+1$ and leaves all 
$k\notin  \{i,n+1\}$ fixed then the kernel $K^{\vp^i}$ is given by
$$
K^{\vp^i}(\alpha, v)=
\begin{cases}
|\alpha_i|^{d-n-1} \ka(\alpha_i v) &\text{ if } \alpha_i\ge 1,\,\,0\le\alpha_j\le \alpha_i, \,j\neq i,
\\
0 &\text{ otherwise. }
\end{cases}
$$

\item
If $1\le i,j\le n$, $i\neq j$  and $\varpi^{ij}$ is the permutation with $\vp^{ij}(i)=n+1$, 
$\vp^{ij}(j)=n+2$ and  $\vp^{ij}(k)=k$ for  
$k\notin  \{i,j, n+1,n+2\}$ then the kernel $K^{\vp^{ij}}$ is given by
\begin{align*}
K^{\vp^{ij}}(\alpha, v)&=|\alpha_i-\alpha_j|^{d-n-1}\kappa((\alpha_i-\alpha_j)(x-y))
\\
&\qquad\text{ either  if }  \alpha_i\le 0, \,\alpha_j\ge 1,\, \alpha_i\le \alpha_k\le \alpha_j  \text{ for }k\neq i,j;
\\&\qquad\text{ or if }
 \alpha_j\le 0, \,\alpha_i\ge 1,\, \alpha_j\le \alpha_k\le \alpha_i  \text{ for }k\neq i,j;
\\
K^{\vp^{ij}}(\alpha, v)&=0 \text{  otherwise. }
\end{align*}
\end{itemize}
Once \eqref{newCJestperm}  is proved for $\varpi=id$, $\varpi=\varpi^i$, $\varpi=\varpi^{ij}$,  the general result follows by complex interpolation for multilinear operators, see \cite[Theorem 4.4.1]{bergh-lofstrom}.

Thus we want  to  study multilinear forms of the type 
\begin{equation}\label{EqnIntroMainOperator}
\La[K](b_1,\ldots, b_{n+2}) = \iiint K(\alpha, x-y)b_{n+2}(x) b_{n+1}(y) 
\prod_{i=1}^n b_i(x-\alpha_i(x-y))\: d\alpha\: dx\: dy,
\end{equation}
where $x\in \R^d$, $\alpha\in \R^n$, and $K(\alpha, x)$ is a Calder\'on-Zygmund kernel in the $x$ variable which depends on a parameter $\alpha\in \R^n$. We will 
impose some regularity conditions on the $\alpha$ variable. The basic example, corresponding to the 
Christ-Journ\'e multilinear forms,  is
$$K(\alpha,x)= 
\bbone\ci{[0,1]^n}\!(\alpha)\ka(x)
$$ where $\ka$ is a regular Calder\'on convolution kernel  satisfying the conditions \eqref{regcz}.

Our goal is to 
\begin{itemize}
\item To introduce a reasonably general class $\sK_\eps$ of kernels  $K(\alpha,x)$, for which linear forms of type  \eqref{EqnIntroMainOperator} are closed under adjoints.  If $\vp$ is a permutation of $\{1,\ldots, n+2\}$, then the multilinear form $\La[K]\q(b_{\vp\q(1\w)}, \ldots, b_{\vp\q(n+2\w) }\w)$ should be  written  as $\La[K^\vp]\q(b_1,\ldots, b_{n+2}\w)$ for a suitable  $K^\vp$, with appropriate bounds on $K^\vp$ in the class $\sK_\eps$.

\smallskip

\item To prove estimates
for this same class of kernels  that cover the estimates for the $d$-commutators in Theorem \ref{newCJthm}.
\end{itemize} 

Roughly the class of admissible kernels  consists of those $K$ for which the norm $\|\cdot\|_{\sK_\eps}$ defined in \eqref{Kepsnorms}, \eqref{EqnResultsKepN} below is finite;  see
 \S\ref{SectionResults} for further discusion of the spaces of distributions on which this definition is made.
The extension to the class $\sK_\eps$ allows us to substantially extend the class of allowable convolution kernels $\ka$ in the definition of the $d$-commutators,
see Example \ref{CJexample} below.

 Let 
$p_1,\ldots, p_{n+2}\in \q(1,\infty\w]$ with $\sum_{j=1}^{n+2} {p_{j}^{-1}}=1$, and 
let  $p_0=\min_{1\leq j\leq n+2} p_j$. 
For  $b_j\in L^{p_j}\q(\R^d\w)$ we shall prove 
the inequality
\begin{equation}\label{EqnIntroMainBound}
\q|\La[K]\q(b_1,\ldots, b_{n+2}\w)\w|\leq C_{p_0, d,\eps} \q\|K\w\|_{\sK_\eps} n^2 \log^3(2+n)\prod_{i=1}^{n+2}\|b_i\|_{p_i}.
\end{equation}
The expression on the left hand side makes a priori sense at least for $K$ supported in a compact subset of $\R^N\times (\R^d\setminus\{0\})$ (and this restriction does not enter in the estimate).
The kernels in  $\sK_\eps$ can be thought of sums of dilates of functions in a
weighted  Besov space; this will be made precise in \S\ref{SectionKernels}.
These weighted Besov spaces are  closely related to  Besov spaces of forms on
$\bbR \mathrm{P}^{n+d}$.
This  motivated some of the considerations in \S\ref{SectionKernels} and \S\ref{SectionAdjoints}.

A key point of the $\sK_\eps$ norms is that they depend  on $n$
in a natural way so that the term $ n^2 \log^3(2+n)$ in 
\eqref{EqnIntroMainBound}
does not become trivial.
We shall derive a stronger version in the next section in Theorem \ref{ThmOpResBound} below in which dependence on the  $\sK_\eps$ 
occurs in a very weak (logarithmic) way. In fact one can define an endpoint space
$\fK_0$ 
which contains the union of the spaces $\sK_\eps$, so that the inequality
\begin{equation}\label{improvedIntroMainBound}
\q|\La[K]\q(b_1,\ldots, b_{n+2}\w)\w|\leq C_{p_0, d,\eps} \q\|K\w\|_{\fK_0} n^2 
\log^3\big(2+n\frac{\|K\|_{\sK_\eps}}{\|K\|_{\fK_0}} \big) 
\prod_{i=1}^{n+2}\|b_i\|_{p_i}.
\end{equation}
  holds.
A crucial point about the classes $\sK_\eps$  is that 
if $K$ belongs to $K_\eps$ then all $K^\vp$ in 
\eqref{Lapermute} belong to some $\sK_{\ep'}$ class with polynomial bound in $n$. 
One can then see that if 
inequality  \eqref{improvedIntroMainBound}
holds for $(p_1,\dots, p_{n+2})=(\infty, \dots, \infty, p_0,p_0')$ then the same is true for the 
kernels $K^\vp$. This invariance under adjoints will be discussed in \S\ref{SectionAdjoints}.

The strategy of proving \eqref{improvedIntroMainBound} for $p_1=\dots=p_n=\infty$ 
then follows Christ and Journ\'e \cite{cj}, with the main inequalities outlined in \S\ref{SectionOutline}.
The subsequent sections contain the details of the proofs.


\section*{Selected Notation}
\label{notation}

\begin{itemize}
\item
We   use the notation $A\lesssim B$ to denote $A\leq C B$, where $C$ is a constant independent of any relevant parameters.
$C$ is allowed to depend on $d$ and $\eps$, but {\it not on $n$.}

\smallskip

\item For two nonnegative numbers $a$, $b$ we occasionally write $a\wedge b=\min\{a,b\}$
and $a\vee b=\max\{a,b\}$

\smallskip

\item
The Euclidean ball in $\bbR^d$ of radius $r$ and with center $x$ is denoted by $B^d(x,r)$.

\smallskip

\item
For a function $g$ on $\bbR^d$ we define dilation operators which leave the $L^1(\R^d)$ norm invariant by
\begin{equation*}
\dil{\\g}{t}\q(x\w):= t^d g\q( tx\w). 
\end{equation*}

\smallskip

\item
For a function $\vsig$ on $\bbR^n\times \bbR^d$ we define dilation operators in the $x$-variable by
\begin{equation*}
\dil{\vsig}{t}\q(\alpha,x\w):= t^d \vsig\q(\alpha ,tx\w). 
\end{equation*}

\smallskip

\item
For a kernel $K$ on  $\bbR^d\times \bbR^d$ we define dilated versions 
by
\begin{equation*}
\mathrm{Dil}_t K (x,y):= t^d K(tx,ty)\,.
\end{equation*}

\smallskip

\item Given  Banach spaces  $E_1, E_2$  we denote by  $\cL(E_1,E_2)$  the Banach space of bounded linear operators from $E_1$ to $E_2$. 

\smallskip

\item We denote by $C^\infty_0(\bbR^d)$
 the space of compactly supported $C^\infty$ functions. 
The subspace 
$C^\infty_{0,0}(\bbR^d)$ consists of all $f\in C^\infty_0(\bbR^d)$ with $\int f(x) dx=0$.

\smallskip

\item Let $\cV$ be an index set, and for each $\nu\in \bbZ$,  let 
$\{\varSigma_N^\nu\}$  be a sequence of operators in $\cL(E_1,E_2)$. 
We say that $\varSigma_N^\nu$  converges in the {\it strong operator topology}
 to $\varSigma^\nu\in \cL(E_1,E_2)$, with {\it equiconvergence with respect to $\cV$}, if for every $f\in E_1$ and every $\eps>0$ there exists a positive integer $N(\eps, f) $ such  that 
 $\|\varSigma_N^\nu f-\varSigma^\nu f\|_{E_2} <\eps$ 
 for all  $N> N(\eps,f)$, $\nu\in \cV$.
 
  Given bounded operators $T_j^\nu\in \cL(E_1,E_2)$, $j\in \bbZ$,   we say that $ \sum_j T_j^\nu$ converges in the strong operator topology,  with  equiconvergence with respect to $\cV$,  if the sequence of partial sums $\varSigma_N =\sum_{j=-N}^N T^\nu_j$ converges 
  in the strong operator topology with equiconvergence with respect to $\cV$.
 
\smallskip

\item Given bounded
$k$-linear operators $L$,  $L_N$, defined on a $k$-tuple
$(A_1,\dots, A_k)$ of normed  spaces with values in a normed space $B$,  we say that $L_N$ converges to $L$ in the strong operator topology (as $N\to\infty$)  if
$\|L_N(a_1,\dots, a_k)- L(a_1,\dots, a_k)\|_B\to 0$ for all $(a_1,\dots, a_k)\in A_1\times\cdots\times A_k$. When $B=\bbC$ or $\bbR$ then there is no difference between strong and weak operator topologies, and we omit the word strong.

\smallskip

\item The spaces $L\sS(\bbR^n\times \bbR^d)$
are defined in \S\ref{SectionResultsKeps}.

\smallskip

\item The operators $P_k$, $Q_k$, $\cQ_k$ and $\overline Q_k[u]$ are introduced in \S\ref{SectionAuxOps} (although $Q_k$ is already used in earlier sections). The class $\sU$ is  defined in Definition \ref{Udefin}.

\smallskip

\item The semi-norms $\|\cdot\|_{\sK_{\eps,i}}$, $i=1,2,3,4, 5$ and the spaces $\sK_\eps$ are defined in \S\ref{SectionResultsKeps}. The related spaces $\fK_\eps$ are defined in \S\ref{SectionResultsDec}.

\smallskip

\item The semi-norms $\|\cdot\|_{\cB_{\eps,i}}$, $i=1,2,3,4,$ and the spaces $\cB_\eps$ are defined in \S\ref{SectionResultsDec}.

\smallskip

\item The Schur classes 
$\Sha^1$, $\Sha^\infty$, $\Sha_\eps^1$, $\Sha_\eps^\infty$ 
and the regularity classes 
$\Zhe_{\eps, \rml}^1$, $\Zhe_{\eps, \rml}^\infty$, $\Zhe_{\eps, \rmr}^1$, 
$\Zhe_{\eps, \rmr}^\infty$ are defined in \S\ref{schurregularity}.

\smallskip

\item The  singular integral classes  $\SI$, $\SI_\eps^1$, $\SI_\eps^\infty$
and annular integrability classes 
$\ann^1$, $\ann^\infty$, $\ann_\av$ 
are defined in \S\ref{SIclasses}.

\smallskip

\item 
The Carleson condition for operators and norm $\|\cdot\|_{\rm Carl}$ is given in Definition \ref{Carlesoncond}. 
The atomic boundedness condition, with norm $\|\cdot\|_{\At}$ is  given in Definition 
\ref{atomicboundedness}.

\smallskip

\item The $\Op_\epsilon$, $\Op_0$ norms are defined in \S\ref{sumsofdilated}.

\smallskip

\item The notion of a Carleson function and the norm $\|\cdot\|_{carl}$ is given in 
definition \ref{carlesonfunction}.

\end{itemize}

\section{Statements of the main results}\label{SectionResults}

	
	
	
\subsection{The classes $\sK_\eps$}
\label{SectionResultsKeps}
We first introduce 
certain classes of tempered distributions on $\bbR^n\times \bbR^d$ which satisfy  integrability properties in the first ($\alpha$-)variable and contain all kernels 
 allowable in 
\eqref{EqnIntroMainOperator}.
For each  $N\in \bbN_0$ consider  the space $M\sS_N'\q(\R^n\times \R^d\w)$
 defined as normed  spaces of tempered  distributions $K$ on $\R^n\times \R^d$ for which there is $C>0$ so that for all
$f\in \sS\q(\R^n\times \R^d\w)$
\begin{equation}\label{EqnKerResMSNorm}
\q|\inn{K}{f}\w| \leq C \sup_{\alpha\in \R^n, x\in \R^d} \sum_{\q|\gamma\w|\leq N} \q(1+\q|x\w|\w)^N \q|\partial_x^\gamma f\q(\alpha, x\w)\w|.
\end{equation}
Here $\inn{K}{ f}$ denotes the pairing between distributions and test functions and the minimal $C$ in \eqref{EqnKerResMSNorm} is the norm in 
$M\sS_N'\q(\R^n\times \R^d\w)$.
The space $MS'(\R^n\times\R^d)$ is the space of tempered distributions 
$K$ on $\bbR^n\times \bbR^d$ for which 
\eqref{EqnKerResMSNorm} holds 
 {\it for some} $N\in \N$.
Note that $M\sS'\q(\R^n\times \R^d\w)$ can be seen as an inductive limit of the normed  spaces $M\sS_N'\q(\R^n\times \R^d\w)$,
and this gives $M\sS'\q(\R^n\times \R^d\w)$ the structure of
a locally convex topological vector space. A net $\{f_\imath\}_{i\in \fI}$  is Cauchy in this topology if there exists  an $N$ so that
all $f_\imath$ belong to $M\sS_N'\q(\R^n\times \R^d\w)$ 
for some fixed $N$ and so that $f_\imath$ is Cauchy in the norm topology of  $M\sS_N'\q(\R^n\times \R^d\w)$.
It is easy to see the normed spaces $M\sS_N'\q(\R^n\times \R^d\w)$ are complete 
and thus  $M\sS'\q(\R^n\times \R^d\w)$ is complete.
Let $M\q(\R^n\w)$ be the space of bounded Borel measures on $\R^n$.  $K\in M\sS'\q(\R^n\times \R^d\w)$ gives rise to a continuous
linear operator $\beta_K:\sS\q(\R^d\w)\rightarrow M\q(\R^n\w)$ defined by
\begin{equation*}
\inn{\beta_K(\phi_2)}{\phi_1}:=\inn{K}{ \phi_1\otimes \phi_2} \text{ for  
$\phi_1\in \sS(\R^n)$, $\phi_2\in \sS(\R^d)$.}
\end{equation*}
Let $L\sS'\q(\R^n\times\R^d\w)$ be the subspace of $M\sS'\q(\R^n\times \R^d\w)$ consisting of those $K$ for which $\beta_K\q(\phi_2\w)\in L^1\q(\R^n\w)$,
for all $\phi_2\in \sS\q(\R^d\w)$.
$L\sS'\q(\R^n\times \R^d\w)$ is a closed subspace of $M\sS'\q(\R^n\times \R^d\w)$ and inherits its complete locally convex topology.

We now define the Banach space $\sK_\eps$ used in \eqref{EqnIntroMainBound}.
For $K\in L\sS'\q(\R^n\times \R^d\w)$ and $\eta\in \sS\q(\R^d\w)$ it makes sense to write $K\q(\alpha,\cdot\w)*\eta$ for the convolution of $K$ and $\eta$ in the $x$-variable. This yields an $L^1$ function
in the $\alpha$ variable, which depends smoothly on $x$.
For $K  \in \Loloc\q(\R^n\times \R^d\w)$, let 
$$\dil{K}{t}\q(\alpha,x\w):=t^d K\q(\alpha,t x\w)$$
and we extend this to $L\sS'\q(\R^n\times \R^d\w)$ by continuity in the usual way.
Fix $\eta\in \sS\q(\R^d\w)$ satisfying
\begin{equation}\label{EqnResultsEta}
\inf_{\theta\in S^{d-1}} \sup_{\tau>0} \q|\etah\q(\tau \theta\w)\w|>0,
\end{equation}
where $\etah$ denotes the Fourier transform of $\eta$.

\begin{defn}
Let $\eta$ be as in  \eqref{EqnResultsEta}, and $0<\eps\le 1$.

\begin{enumerate}[\upshape (i)]
\item
 Define
five semi-norms by
\begin{subequations}\label{Kepsnorms}
\begin{align}
\label{Kepsnormeta1}
\|K\|_{\sK_{\eps,1}^\eta}
&:=\sup_{\substack{1\leq i \leq n \\ t>0 }} \int \q(1+\q|\alpha_i\w|\w)^\eps \LpRdN{2}{\eta * \dil{K}{t}\q(\alpha,\cdot\w)} \: d\alpha ,
\\
\label{Kepsnormeta2}
\|K\|_{\sK^\eta_{\eps,2}}
&:=\sup_{\substack{1\leq i\leq n \\ t>0 \\ 0<h\leq 1 }} h^{-\eps} \int \LpRdN{2}{   \eta*\q[  \dil{K}{t}\q(\alpha+he_i , \cdot\w) - \dil{K}{t}\q(\alpha,\cdot\w) \w]  }\: d\alpha,
\\
\label{Kepsnorm3}
\|K\|_{\sK_{\eps,3}}&:=\sup_{\substack{ 1\leq i \leq n \\ R>0   }} \,\,
 \iint\limits_{R\leq \q|x\w|\leq 2R}  (1+|\alpha_i|)^\eps\q| K\q(\alpha, x\w)  \w|\: dx\: d\alpha,
\\
\label{Kepsnorm4}
\|K\|_{\sK_{\eps,4}}&:=
\sup_{\substack{ 1\leq i \leq n \\ R>0 \\ 0<h\leq 1  }} h^{-\eps}\, \iint\limits_{R\leq \q|x\w|\leq 2R}  \q| K\q(\alpha+he_i, x\w) - K\q(\alpha,x\w) \w|\: dx\: d\alpha,
\\
\label{Kepsnorm5}
\|K\|_{\sK_{\eps,5}}&:=
\sup_{\substack{R>2 \\ y\in \R^d}} R^{\eps}\, \iint\limits_{\q|x\w|\geq R\q|y\w|} \q|K\q(\alpha, x-y\w) - K\q(\alpha,x\w)\w|\: dx\: d\alpha\,.
\end{align}
\end{subequations}

\item
The space $\sK_\eps$ is  the  subspace of $L\sS'\q(\R^n\times \R^d\w)$ consisting of those $K$ for which the norm
\begin{equation}\label{EqnResultsKepN}
\|K\|_{\sK_{\eps}}:= \|K\|_{\sK^\eta_{\eps,1}}
+\|K\|_{\sK^\eta_{\eps,2}}
+\|K\|_{\sK_{\eps,3}}
+\|K\|_{\sK_{\eps,4}}
+\|K\|_{\sK_{\eps,5}}
\end{equation}
is finite.
\end{enumerate}
\end{defn}

The definition of   $\sKN{\eps}{\cdot}$ 
depends on a choice of $\eta\in \sS\q(\R^d\w)$ satisfying \eqref{EqnResultsEta}.  However,
the equivalence class of the norm does not depend on the choice, and the constants in the equivalences of different choices of $\eta$ will not depend on $n$. This is made explicit in Lemma \ref{LemmaResultsEtaDoesntMatter} below.


\begin{example}\label{CJexample}
Let $\epsilon\in (0,1)$ and let $\kappa\in \cS'(\bbR^d)\cap L^1_\loc(\bbR^d\setminus\{0\})$ be a convolution kernel in $\bbR^d$ satisfying 
\Be
\label{kaFT} \|\widehat \ka\|_\infty\le C
\Ee and
\Be
\label{Hoerm}
\sup_{R\ge 2} R^\epsilon \sup_{y\in \bbR^d} \int_{|x|\ge R |y|}|\kappa(x-y)-\kappa(x)| dx \le C.
\Ee
Let $$K(x,\alpha)= \chi_{[0,1]^n}(\alpha) \kappa(x)\,.$$
Then $K\in \sK_{\delta}(\bbR^n\times\bbR^d)$ for $\delta<\epsilon$ and
\Be \label{standardCJex} 
\|K\|_{\sK_\delta}\lc_{\delta,\epsilon} C .
\Ee
The details of \eqref{standardCJex} are left to the reader.
\end{example}

We state a preliminary  version of our boundedness result  (see Theorem \ref{ThmOpResBoundGen} below for a more definitive version).

\begin{thm}\label{MainKepsthm} Let $\eps>0$, $\delta>0$ and $\eta$ as in \eqref{EqnResultsEta}.

(i) There is a constant $C=C(d,\delta,\eps,\eta)$ such that the following statement holds a priori for all kernels in $\sK_\eps$ which also belong to $L^1(\bbR^n\times \bbR^d)$.
The multilinear form 
$$\La[K](b_1,\ldots, b_{n+2})
 = \iiint K(\alpha, x-y)b_{n+2}(x) b_{n+1}(y) 
\prod_{i=1}^n b_i(x-\alpha_i(x-y))\: d\alpha\: dx\: dy,$$
satisfies 
\Be \label{LaKbd}
|\La[K](b_1,\ldots,b_{n+2})|\le C n^2 \log^3\!(1+n) \|K\|_{\sK_\eps}
\prod_{i=1}^{n+2} \|b_i\|_{p_i}\,,
\Ee
for all $b_i\in L^{p_i}(\bbR^d)$, $1+\delta<p_i<\infty$, $\sum_{i=1}^{n+2} p_i^{-1}=1$.

(ii) The multilinear form $(K,b_1,\dots, b_{n+2})\mapsto \La[K](b_1,\ldots, b_{n+2})$
extends to a bounded multilinear form on $\sK_\eps\times L^{p_1}\times\dots\times L^{p_{n+2}}$ satisfying \eqref{LaKbd} for all $K\in \sK_\eps$.
\end{thm}

The proof of Theorem \ref{MainKepsthm} we will
 heavily  rely  on a decomposition theorem for the class $\sK_\eps$, to which we now turn. 
 This decomposition will specify further part (ii) of the theorem, i.e. describe  how to extend  the result from part (i) to all kernels in $\sK_\eps$.
 
\subsection{Decomposition of kernels in $\bold{\sK_\eps}$}
\label{SectionResultsDec}
In the following definition $e_1,\ldots, e_n$ will denote 
 the standard basis of $\R^n$.
\begin{defn}\label{DefnKerBesovSpace}
For $n,d\in \N$ and $0\le \eps\leq 1$ we define four (semi-)norms
\begin{subequations}\label{Bepsnorms}
\begin{align}
\label{Beps1}
\|\vsig\|_{\cB_{\eps,1} }
:= & \max_{1\leq i\leq n} \iint \q(1+\q|\alpha_i\w|\w)^{\eps} \q|\vsig\q(\alpha, v\w)\w|\: d\alpha \: dv,
\\
\label{Beps2}
\|\vsig\|_{\cB_{\eps,2}}
:=&\sup_{\substack{0<h\leq 1 \\ 1\leq i \leq n }} h^{-\eps} \iint \q| \vsig\q(\alpha+he_i,v\w)-\vsig\q(\alpha,v\w) \w|\: d\alpha\: dv,
\\ 
\label{Beps3}
\|\vsig\|_{\cB_{\eps,3} }
:=& \sup_{0<\q|h\w|\leq 1} \q|h\w|^{-\eps} \iint \q|\vsig\q(\alpha, v+h\w)- \vsig\q(\alpha, v\w)\w|\: d\alpha\: dv,
\\
\label{Beps4}
\|\vsig\|_{\cB_{\eps,4}}
:=&\iint \q(1+\q|v\w|\w)^{\eps} \q|\vsig\q(\alpha,v\w)\w|\: d\alpha\: dv.
\end{align}
\end{subequations}
Let 
$\sBtp{\eps}{\R^n\times \R^d}$  be the space of those
$\vsig\in L^1\q(\R^n\times \R^d\w)$ such that the norm 
\begin{equation}\label{Bepsnorm}
\|\vsig\|_{\cB_{\eps} }:= 
\|\vsig\|_{\cB_{\eps,1}} +
\|\vsig\|_{\cB_{\eps,2}} +
\|\vsig\|_{\cB_{\eps,3} }+
\|\vsig\|_{\cB_{\eps,4} }
\end{equation}
is finite.
\end{defn}

For $0<\eps<1$ the space $\sBp{\epsilon}$ is a type of Besov space, hence the notation. See also \S\ref{SectionRoleOfRPnSecond} below.
Recall the notation $\dil{\vsig}{t}\q(\alpha, x\w):= t^d \vsig\q(\alpha, tx\w). $ 
\begin{defn} \label{DefLPdec}
\begin{enumerate}[\upshape (i)]
\item  Let   $\phi\in C_0^\infty(\bbR^d)$ such that $\int \phi(x) dx=1$, let
 $Q_j$ denote 
the operator of convolution with $2^{jd}\phi(2^j\cdot) - 2^{(j-1)d}\phi(2^{j-1}\cdot) $. When acting on  $K\in LS'(\bbR^n\times\bbR^d) $,
we define $Q_j K$ by taking  the convolution in 
$\bbR^d$.

 \item Set \Be\label{vsigK}\vsig_j[K]:=   (Q_j K)^{(2^{-j})}.\Ee

\item  For $K\in LS'(\bbR^n\times\bbR^d)$ let
\Be\label{fK0def}
 \|K\|_{\fK_0}= \sup_{j\in \bbZ}\|  \vsig_j [K] 
\|_{L^1(\bbR^n\times \bbR^d)} \,.
\Ee

\item Let
 $\fK_\eps$ be the space of all  $K\in LS'(\bbR^n\times\bbR^n)$ such 
 that
$$\|K\|_{\fK_\eps}:= \sup_{j\in \bbZ} \|\vsig_j [K] \|_{\cB_\eps(\bbR^n\times \bbR^d)} $$
is finite.
\end{enumerate}
\end{defn}


The relation between the spaces $\sK_\eps$ and $\fK_\eps$ is given in the following theorem.
\begin{thm}\label{ThmResDecompKer}
(i) A distribution $K\in L\sS'\q(\R^n\times \R^d\w)$ belongs to
 $\bigcup_{0<\eps<1} \sK_\eps$ if and only if there exists an $\eps>0$ and a bounded set $\q\{\vsig_j : j\in \Z\w\}\subset \sBtp{\eps}{\R^n\times \R^d}$ satisfying  $$\int \vsig_j\q(\alpha, v\w)\: dv=0$$ for all $j$, $\alpha$ and 
\begin{equation*}
K= \sum_{j\in \Z} \dil{\vsig_j}{2^j},
\end{equation*}
holds with  convergence in the topology on $L\sS'\q(\R^n\times \R^d\w)$ (and thus also in the sense of  distributions). 

(ii) Let $K\in \sK_\eps$. Then for $\delta<\eps$, 
$$\|K\|_{\fK_\delta} \le C_{\delta, \eps,d} \|K\|_{\sK_\eps}\,.$$

(iii) Let $K\in \fK_\ep$. Then for $\delta<\ep/2$
$$\|K\|_{\sK_\delta} \le C_{\delta, \eps,d} \|K\|_{\fK_\ep}\,.$$
\end{thm}

\subsection{Boundedness of multilinear forms}
For any  $\vsig\in \cB_\eps(\R^n\times\bbR^d)$
and for $b_i\in L^{p_i}(\R^d)$ with $\sum_{i=1}^{n+2}p_i^{-1}=1$ 
the multilinear form
$$
\La [\vsig](b_1,\dots, b_{n+2})=
\iiint {\vsig}\q(\alpha, x-y\w)b_{n+2}\q(x\w)b_{n+1}\q(y\w)  \prod_{i=1}^n b_i\q(x-\alpha_1\q(x-y\w))
\: dx\: dy\: d\alpha
$$
is well defined;  more precisely we have
\begin{lemma}\label{LemmaBasicL2BasicLpEstimate}
Let $\vsig\in L^1\q(\R^n\times \R^d\w)$.  Suppose for $1\leq l\leq n+2$, $b_i\in L^{p_i}\q(\R^d\w)$ with $\sum_{i=1}^{n+2} p_i^{-1}=1$.
Then, for all $ j\in \Z$,
\begin{equation*}
\big|\La[\vsigj](b_1,\ldots, b_{n+2})\big|\leq \|\vsig\|_{L^1(\bbR^n\times\bbR^d)} \prod_{i=1}^{n+2} 
\|b_i\|_{p_i}\,.
\end{equation*}
\end{lemma}
\begin{proof}
This follows easily  by H\"older's inequality.
\end{proof}

Theorem  \ref{ThmResDecompKer} suggests to define 
the form  $\La[K]$, for $K\in \sK_\eps$, 
as the  limit of partial sums
\Be  \label{EqnOpResBasicOp}
\sum_{j=-N}^N \La[\vsig_j^{(2^j)}](b_1,\dots, b_{n+2}) \Ee as $N\to \infty$.






Our main boundedness result (a sharper version of 
 Theorem \ref{MainKepsthm}) is
\begin{thm}\label{ThmOpResBoundGen}
 Let $0<\delta<1$, let  $p_1,\ldots, p_{n+2}\in [1+\delta,\infty]$ with $\sum_{l=1}^{n+2} p_l^{-1}=1$.

(i) Let $\fI$ be a finite subset of $\bbZ$ and let $\{\vsig_j: j\in \fI \}$ be a subset of $\cB_{\eps}(\bbR^n\times\R^d)$ so that for every $j\in \fI$, 
$\int \vsig_j(\alpha,x) \,dx=0$ for almost all $\alpha\in \bbR^n$. 
Let $$K_\fI=\sum_{j\in \fI} \vsig_j^{(2^j)}\,.$$ Then
for  $b_l\in L^{p_l}\q(\R^{d}\w)$
we have
\begin{equation*}
\q|\La[K_\fI]\q(b_1,\ldots, b_{n+2}\w)\w|\leq C_{\epsilon,d,\delta} 
n^2 \q\big(\sup_{j\in \Z} \|\vsig_j\|_{L^1(\R^{n+d})} \w\big) \log^3\!\q\Big(2 + n\frac{\sup_{j\in \Z} \sBN{\epsilon}{\vsig_j}} {\sup_{j\in \Z} \sBzN{\vsig_j}}\w\Big) \prod_{l=1}^{n+2} \|b_l\|_{p_l}
\end{equation*}
where the constant $C_{\epsilon,d, \delta} $ 
is independent of $n$ and $\fI$.

(ii) Let $K\in \sK^\eps$ so that $K=\sum_{j\in \bbZ} \vsig_j^{(2^j)}$ in $L\cS'(\bbR^n\times\bbR^n)$ with $\int \vsig_j(\alpha,x) dx=0$ for almost all $\alpha\in \bbR^n$.
Let $\sup_j \|\vsig_j\|_{\cB_\eps}<\infty$,
$b_1\in L^{p_1}$, ..., 
$b_{n+2}\in L^{p_{n+2}}$. Then 
$\sum_{j=-\infty}^\infty  \La[\vsig_j^{(2^j)}] $ converges 
in the  operator topology of $(n+2)$-linear functionals to a limit $\La[K]$ 
satisfying
\begin{equation*}
\q|\La[K]\q(b_1,\ldots, b_{n+2}\w)\w|\leq C_{p_0,\epsilon,d} 
n^2 \q\|K\|_{\fK_0} \log^3\!\q\Big(2 + n\frac{\|K\|_{\fK_\eps}} {\|K\|_{\fK_0}} 
\w\Big) \prod_{l=1}^{n+2} \|b_l\|_{p_l}\,.
\end{equation*}
\end{thm}

We now turn to the multilinear forms defined by adjoint operators.  More generally, given a permutation $\vp$ on $\{1,\dots,n+2\}$ we define the multilinear form $\La^\vp[\vsig]$ by
\Be\label{Lavpdef}
\La^\vp[\vsig](b_1,\dots, b_{n+2})= 
\La[\vsig] (b_{\vp\q(1\w)},\ldots, b_{\vp\q(n+2\w)}) \,.
\Ee

We have the following crucial result which will be proved in \S\ref{SectionAdjoints}. It shows that 
operators of the form \eqref{EqnOpResBasicOp}, and their limits as $N\to \infty$,
 are closed under adjoints.

\begin{thm}\label{ThmOpResAdjoints}
Let $\epsilon>0$.  There exists 
$\epsilon' \ge c(\epsilon) $ (independent of $n$) such that for 
any permutation $\vp$ of $\q\{1,\ldots, n+2\w\}$ there exists a bounded  linear transformation $\ell_\vp: \cB_{\epsilon}(\R^n\times \R^d)\to \cB_{\epsilon'}(\R^n\times \R^d)
$ with
$$(\ell_\vp \vsig)^{(t)}=\ell_\vp(\vsig^{(t)}), \quad t>0,$$ and 
\begin{equation*}
\La^\vp[\vsig]
= \La[\ell_\vp\vsig]\,,
\end{equation*}
 such that 
 $$\|\ell_\vp\vsig\|_{\cB_{\epsilon'}}\lesssim n^{2} \sBN{\epsilon}{\vsig}$$ 
 and
 $$\|\ell_\vp\vsig\|_{L^1} = \|\vsig\|_{L^1}.$$
Furthermore, if $\int \vsig\q(\alpha,v\w)\: dv=0$ a.e. 
then also
$\int \ell_\vp\vsig\q(\alpha,v\w)\: dv=0$ a.e.
\end{thm}


In  light of Theorem \ref{ThmOpResAdjoints},
 the result in Theorem \ref{ThmOpResBoundGen}
 is closed under taking adjoints, and therefore follows from
the following result and complex interpolation (see \S\ref{interpolsect}).
\begin{thm}\label{ThmOpResBound}
Let  $\delta>0$, $b_1,\ldots, b_n\in L^\infty\q(\R^d\w)$, $p\in \q[1+\delta,2\w]$, and let $p'=p/(p-1)$.
 For $b_{n+1}\in L^{p}\q(\R^d\w)$, $b_{n+2}\in L^{p'}\q(\R^d\w)$
we have  
\begin{equation*}
\q|\La[K]\q(b_1,\ldots, b_{n+2}\w)\w|\leq C_{\epsilon,d,\delta}n^2  \q \sup_{j\in \Z} \sBzN{\vsig_j}\w\, \log^3\!\q\big(2 + n\frac{\sup_{j\in \Z} \sBN{\epsilon}{\vsig_j}} {\sup_{j\in \Z} \sBzN{\vsig_j}}\w\big) \q\Big(\prod_{l=1}^{n} \|b_l\|_\infty\Big) \|b_{n+1}\|_p \|b_{n+2}\|_{p'}.
\end{equation*}
\end{thm}

The structure of the proof of Theorem \ref{ThmOpResBound} will be discussed in \S\ref{SectionOutline},  and  the details of the proof will be given in subsequent sections.

\subsection{Remarks on Besov spaces}

\subsubsection{Equivalent norms}\label{RmkResultsConstantsPolyInn}
In Definition \ref{DefnKerBesovSpace} we chose a particular form of the norm $\sBN{\epsilon}{\cdot}$
which is well suited for our goal to prove estimates with polynomial growth in $n$.  There are other equivalent norms which could be used, for  instance, one might replace the expression
$$\sup_{\substack{0<h\leq 1 \\ 1\leq i \leq n }} h^{-\epsilon} \iint \q| \vsig\q(\alpha+he_i,v\w)-\vsig\q(\alpha,v\w) \w|\: d\alpha\: dv$$ with
$$\sup_{\substack{0<\q|h\w|\leq 1}} \q|h\w|^{-\epsilon} \iint \q| \vsig\q(\alpha+h,v\w)-\vsig\q(\alpha,v\w) \w|\: d\alpha\: dv$$
and one ends up with a comparable norm.  These two choices differ by a factor which is polynomial in $n$.
Fortunately, the result in Theorem \ref{ThmOpResBoundGen} only involves $\sBN{\epsilon}{\vsig_j}$ through the expression
$$ \log^3\q(2 + n\frac{\sup_{j\in \Z} \sBN{\epsilon}{\vsig_j}} {\sup_{j\in \Z} \sBzN{\vsig_j}}\w).$$
Thus, if one changes $\sup_{j\in \Z} \sBN{\epsilon}{\vsig_j}$ by a factor which is polynomial in $n$, this only changes the bound in
Theorem \ref{ThmOpResBoundGen} by a constant factor, and therefore does not change the result in Theorem \ref{ThmOpResBoundGen}.
In this way, one can use any one of a variety of equivalent norms when defining $\sBN{\epsilon}{\cdot}$ (as long as one only changes
the norm by a factor which is bounded by a polynomial in $n$) -- we picked out the choice which is
most natural for our purposes.

\subsubsection{The role of projective space}\label{SectionRoleofRPnFirst}

Though it may not be apparent from the above definitions, the space $\RPn$ plays a key role in the intuition behind our main results.  In this section, we
exhibit a special case where the role of $\RPn$ is apparent, and we return to a more general version of these ideas in \S \ref{SectionRoleOfRPnSecond}.

Recall that $\RPn$ is defined as $\R^{n+1}\setminus \q\{0\w\}$ modulo the equivalence relation where we identify $\alpha,\beta\in  \R^{n+1}\setminus \q\{0\w\}$
if there exists $c\in \R\setminus \q\{0\w\}$ with $\alpha=c\beta$.  This sees $\RPn$ has an $n$-dimensional manifold.  Traditionally, there are $n+1$ standard coordinate charts on $\RPn$.
For these, we consider those points in $\alpha=\q(\alpha_1,\ldots, \alpha_{n+1}\w)\in \R^{n+1}\setminus \q\{0\w\}$ with $\alpha_j\ne 0$.  Under the equivalence relation,
$\alpha$ is equivalent to $\alpha_{j}^{-1} \alpha= \q(\alpha_j^{-1} \alpha_1,\ldots, \alpha_j^{-1} \alpha_{j-1}, 1, \alpha_j^{-1} \alpha_{j+1},\ldots, \alpha_{j}^{-1} \alpha_{n+1}\w)$.  This identifies such
points with a copy of $\R^n$ and yields a coordinate chart on $\RPn$--every point in $\RPn$ lies in the image of at least one of these charts.
This sees a copy of $\R^n$ inside of $\RPn$ given by $\q(\alpha_1,\ldots, \alpha_n\w)\mapsto \q(\alpha_1,\ldots, \alpha_{j-1},1,\alpha_{j+1},\ldots, \alpha_n\w)$.

Functions on $\RPn$ can be identified with functions $f:\R^{n+1}\setminus\q\{0\w\}\rightarrow \bbC$ such that $f\q(c\alpha\w)=f\q(\alpha\w)$--i.e., functions which are homogeneous of degree $0$ and are even.  Suppose we are given  $f:\RPn\rightarrow \bbC$.  We obtain a function $f_0:\R^n\rightarrow \bbC$ by viewing $\R^n\hookrightarrow \RPn$
via the map $\q(\alpha_1,\ldots, \alpha_n\w)\mapsto \q(\alpha_1,\ldots, \alpha_n,1\w)$.  Thus, given an even  function $f:\R^{n+1}\setminus\q\{0\w\}\rightarrow \bbC$
which is homogeneous of degree $0$, we obtain a function $f:\RPn\rightarrow \bbC$, and therefore a function $f_0:\R^n\rightarrow \bbC$ (and $f_0$ uniquely determines
$f$ off of a set of lower dimension in $\RPn$).

We consider here the special case when $$K\q(\alpha, v\w)= \gamma\q(\alpha\w) \ka\q(v\w)$$ and $\ka$ is a classical
 Calder\'on-Zygmund
kernel which is {\it homogeneous of degree $-d$} and smooth away from $v=0$.  For $\alpha\in \R^{n}$ and functions $b_1,\ldots, b_{n+2}$, consider
\begin{equation}\label{EqnRoleRPNF0}
\begin{split}
F_0\q(\alpha\w) &=
\iint  \ka\q(x-y\w) 
b_{n+2}\q(x\w) b_{n+1}\q(y\w) \prod_{i=1}^n b_i\q(x-\alpha_i(x-y)\w)
 \:  dx\: dy
\\&= \iint \ka\q(v\w) 
b_{n+2}\q(x\w)
b_{n+1}\q(x-v\w) \prod_{i=1}^n b_i\q(x-\alpha_iv\w)\:  dx\: dv.
\end{split}
\end{equation}
The multilinear form  we wish to study (in this special case) is given by
\begin{equation*}
\int \gamma\q(\alpha\w) F_0\q(\alpha\w) \: d\alpha.
\end{equation*}
One main aspect of our assumptions is that this operator should be of the same form when we permute the roles of $b_{1},\ldots, b_{n+2}$.
Many of these permutations are easy to understand:  permuting the roles of $b_1,\ldots, b_n$ merely permutes the variables $\alpha_1,\ldots, \alpha_n$.
Switching the roles of $b_{n+1}$ and $b_{n+2}$ changes $\alpha$ to $\q(1-\alpha_1,\ldots, 1-\alpha_n\w)$.
Thus, the major difficulty in understanding adjoints can be reduced to understanding the question of switching the roles of $b_j$ ($1\leq j\leq n$) and $b_{n+1}$ (as every permutation of $\q\{1,\ldots, n+2\w\}$ can be generated by the these three types of permutations).

Define a new function $F:\R^{n+1}\setminus\q\{0\w\}\rightarrow \bbC$ by
\begin{equation*}
F\q(\alpha_1,\ldots, \alpha_{n+1}\w) = \iint \ka\q(v\w) b_1\q(x-\alpha_1v\w)\cdots b_n\q(x-\alpha_nv\w) b_{n+1}\q(x-\alpha_{n+1}v\w) b_{n+2}\q(x\w)\:  dx\: dv\,.
\end{equation*}
Because of the homogeneity of $\ka$, we see (for $c\in \R\setminus \q\{0\w\}$), $F\q(c\alpha\w)=F\q(\alpha\w)$.
By the above discussion, $F$ defines a function on $\RPn$, and therefore induces a function $F_0:\R^n\rightarrow \bbC$ as above.
This induced function $F_0$ is exactly the function of the same name from \eqref{EqnRoleRPNF0}.
Thus, we have defined $F_0$ in a way which is symmetric 
in $b_1,\ldots, b_{n+1}$.

$F\q(\alpha\w)$ defines  a function on $\RPn$, and therefore if $\gamma\q(\alpha\w)d\alpha$ were a measure on $\RPn$, it would make sense
to write
\begin{equation*}
\int \gamma\q(\alpha\w) F\q(\alpha\w) \: d\alpha.
\end{equation*}
Indeed, our main assumptions in this special case are equivalent to assuming that $\gamma\q(\alpha\w) d\alpha$ is a density
which lies in the space $\bigcup_{0<\eps<1}B^{\epsilon}_{1,\infty}\q(\RPn\w)$
(where $B^{\epsilon}_{1,\infty}\q(\RPn\w)$ denotes a {\it Besov space of densities} on $\RPn$, see \S \ref{SectionRoleOfRPnSecond} for a proof of this remark).
When we write the expression as
\begin{equation*}
\int \gamma\q(\alpha\w) F_0\q(\alpha\w)\: d\alpha,
\end{equation*}
we are merely choosing the coordinate chart $\R^n\hookrightarrow \RPn$ denoted above.
With this formulation, the operator
\begin{equation*}
\int \gamma\q(\alpha\w) F\q(\alpha\w) \: d\alpha
\end{equation*}
clearly remains of the same form when $b_1,\ldots, b_{n+1}$ are permuted, and from here it is easy to see that the class of operators is ``closed under adjoints.''

\begin{rem}
In our more general setting, $K\q(\alpha, v\w)$ is not homogeneous in the $v$ variable, and therefore we cannot define a function $F$
on $\RPn$ as was done above.  Nevertheless, these ideas play an important  role in our proofs,
see \S\ref{SectionRoleOfRPnSecond} below.
\end{rem}

\section{Kernels} \label{SectionKernels}
	

In this section, we prove  various results announced in Section \ref{SectionResults}. We first show the independence of the space 
$\sK_\eps$ of the particular choice of $\eta$ satisfying \eqref{EqnResultsEta}
and then give the proof of Propositions \ref{PropKerDecompK} and 
\ref{PropKerSumVsig}.

\subsection{Independence of $\eta$} \label{SectionEtaindependence}

The following lemma shows that 
$\sK_\eps$ does not depend on the choice of $\eta\in \sS\q(\R^d\w)$ satisfying \eqref{EqnResultsEta}.

\begin{lemma}\label{LemmaResultsEtaDoesntMatter}
Let $\eta, \eta'\in \cS(\bbR^d)$ and $\eta$ be as in \eqref{EqnResultsEta}.
Let $0<\eps\le 1$.
There exists $C=C\q(\eta,\eta'\w)$ such that for all 
$K\in \sK_\eps$
\begin{equation*}
\|K\|_{\sK_\eps^{ \eta'} } \leq C \|K\|_{\sK_\eps^{ \eta} }
\end{equation*}
The constant $C$ is independent of $n$.
\end{lemma}
\begin{proof}
Let $K\in \sK_\eps$.  Only two of the terms of the definition of $\sKN{\eps}{K}$ depend on the choice of $\eta$.
Thus, the result will follow once we prove the following two estimates.
\begin{equation}\label{EqnKerEtaDoesntMatter1}
\sup_{\substack{1\leq i \leq n \\ t>0 }} \int \q(1+\q|\alpha_i\w|\w)^\eps \LpRdN{2}{\eta'* \dil{K}{t}\q(\alpha,\cdot\w)} \: d\alpha
\leq C 
\sup_{\substack{1\leq i \leq n \\ t>0 }} \int \q(1+\q|\alpha_i\w|\w)^\eps \LpRdN{2}{\eta * \dil{K}{t}\q(\alpha,\cdot\w)} \: d\alpha,
\end{equation}
and
\begin{multline}\label{EqnKerEtaDoesntMatter2}
\sup_{\substack{1\leq i\leq n \\ t>0 \\ 0<h\leq 1 }} h^{-\eps} \int \LpRdN{2}{   \eta'*\q[  \dil{K}{t}\q(\alpha+he_i , \cdot\w) - \dil{K}{t}\q(\alpha,\cdot\w) \w]  }\: d\alpha
\\
\leq C
\sup_{\substack{1\leq i\leq n \\ t>0 \\ 0<h\leq 1 }} h^{-\eps} \int \LpRdN{2}{   \eta*\q[  \dil{K}{t}\q(\alpha+he_i , \cdot\w) - \dil{K}{t}\q(\alpha,\cdot\w) \w]  }\: d\alpha.
\end{multline}
The proofs of these two equations are nearly identical, so we prove only 
\eqref{EqnKerEtaDoesntMatter1}.

\smallskip

Let $\chi\in \CziRd$ be supported in $\q\{\xi:\frac{1}{2}< \q|\xi\w|<2\w\}$
with the property that $\sum_{k\in \Z} \q[\chi\q(2^{-k} \xi\w)\w]^2 =1$, for $\xi\in \bbR^d\setminus \{0\}$.  Since
$\eta'\in \sS\q(\R^d\w)$, we have $\LpN{\infty}{ \chi\q(2^{-k}\cdot\w)\etaph\q(\cdot\w)}\leq C_N 
\min\{2^{-kN}, 1\} $.
By \eqref{EqnResultsEta} and the compactness of $\q\{\xi : \frac{1}{2}\leq \q|\xi\w|\leq 2\w\}$ there is a finite index set
$\Theta$ and real numbers $\tau_\nu>0$ such that
\begin{equation*}
\sum_{\nu\in \Theta} \q|\etah\q(\tau_\nu \xi\w)\w|^2 \geq c>0 \:\text{   for   }\: \frac{1}{2}\leq\q|\xi\w|\leq 2.
\end{equation*}
Let
\begin{equation*}
m_\nu\q(\xi\w) = \frac{\overline{\etah\q(\tau_\nu\xi\w)} \chi\q(\xi\w) }{\sum_{\tilde{\nu}\in \Theta} \q|\etah\q(\tau_{\tilde{\nu}}\xi\w)\w|^2 };
\end{equation*}
then $\LpN{\infty}{m_\nu}\leq C_\nu$ and we have
\begin{equation*}
\etaph\q(\xi\w) = \sum_{k\in \Z} \chi\q(2^{-k}\xi\w) \etah\q(\xi\w) \sum_{\nu\in \Theta} m_\nu\q(2^{-k}\xi\w) \etah\q(2^{-k} \tau_\nu \xi\w).
\end{equation*}
Hence,
\begin{equation*}
\LpRdN{2}{\eta' * \dil{K}{t}\q(\alpha,\cdot\w)} \lesssim \sum_{k\in \Z} \min\{2^{-kN}, 1\} \sum_{\nu\in \Theta} \|m_\nu\|_\infty \LpRdN{2}{\etah\q(2^{-k}\tau_\nu \cdot\w) \widehat{\dil{K}{t}}\q(\alpha,\cdot\w)},
\end{equation*}
where the implicit constant depends on $N$.
Note 
$$\LpRdN{2}{\etah\q(2^{-k} \tau_\nu\cdot\w) \widehat{\dil{K}{t}}\q(\alpha,\cdot\w)} = \q(2^{k}/\tau_\nu\w)^{d/2} \LpRdN{2}{ \eta * \dil{K}{2^{-k}\tau_\nu t}\q(\alpha,\cdot\w)},$$
and so taking $N>d/2$ we obtain
\begin{equation*}
\begin{split}
&\int \q(1+\q|\alpha_i\w|\w)^{\eps} \LpRdN{2}{\eta' * \dil{K}{t}\q(\alpha,\cdot\w)}\: d\alpha
\\&\lesssim  \sum_{k\in \Z} \min\{2^{-k\q(N-d/2\w) },\,2^{kd/2}\w\}\sum_{\nu\in \Theta} C_\nu \int \q(1+\q|\alpha_i\w|\w)^{\eps} \LpRdN{2}{ \eta * \dil{K}{2^{-k} \tau_\nu t}\q(\alpha,\cdot\w)}\: d\alpha
\\&\lesssim \sup_{r>0} \int \q(1+\q|\alpha_i\w|\w)^{\eps} \LpRdN{2}{\eta*\dil{K}{r}\q(\alpha,\cdot\w)}\: d\alpha,
\end{split}
\end{equation*}
which completes the proof of \eqref{EqnKerEtaDoesntMatter1}.
\end{proof}

\subsection{Proof of Theorem \ref{ThmResDecompKer}}
The theorem follows from two propositions. In the first we prove an estimate for the $\vsig_j$ as in 
 \eqref{vsigK}, which arise in the decomposition of $K=\sum_j\vsigjj$.
\begin{prop}\label{PropKerDecompK}
Suppose $\eps\in \q(0,1\w]$, $0<\delta<\eps$.  
For every $K\in \sK_\eps$, let $$\vsig_j=
 (Q_j K)^{(2^{-j})}.$$  
Then 
$\q\{\vsig_j: j\in \Z\w\}$ is a bounded subset of   $\sBtp{\delta}{\R^n\times \R^d}$ satisfying $$\int \vsig_j\q(\alpha, v\w)\: dv=0,$$  for all $j$  and  almost every $\alpha\in \bbR^n$ and 
\begin{equation*}
\sup_{j\in \Z}\|\vsig_j\|_{\cB_\delta}\leq C_{\delta,\eps,d} \sKN{\eps}{K},
\end{equation*}
and such that 
\begin{equation*}
K=\sum_{j\in \Z} \dil{\vsig_j}{2^j},
\end{equation*}
with the sum converging in the sense of the topology on $L\sS'\q(\R^n\times \R^d\w)$.  
\end{prop}
The second proposition provides   $\sK_\delta$-estimates for kernels that are given as sums
$\sum_j \vsigjj$, with uniform $\cB_\eps$-estimates for the $\vsig_j$.
\begin{prop}\label{PropKerSumVsig}
Let $\eps\in \q(0,1\w]$, and $0<\delta<\eps/2$.  Suppose $\q\{\vsig_j : j\in \Z\w\}\subset \sBtp{\eps}{\R^n\times \R^d}$ is a bounded set
satisfying $\int \vsig_j\q(\alpha, v\w)\: dv=0$, for all $j$.  Then the sum
\begin{equation*}
K\q(\alpha, v\w):= \sum_{j\in \Z} \dil{\vsig_j}{2^j}\q(\alpha,v\w)
\end{equation*}
converges in the sense of the topology on $L\sS'\q(\R^n\times \R^d\w)$, and $K\in \sK_\delta$.  Furthermore,
\begin{equation*}
\sKN{\delta}{K}\leq C_{\delta, \eps,d} \sup_{j\in \Z} \sBN{\eps}{\vsig_j}.
\end{equation*}
\end{prop}
The proofs of  these propositions will be given in \S
\ref{pfPropKerDecompK} and \S\ref{pfPropKerSumVsig}

\subsubsection{ Proof of Proposition  \ref{PropKerDecompK}}
\label{pfPropKerDecompK}
We need several lemmata.


\begin{lemma}\label{LemmaKerIntAgainstdeltapower}
Let $\eps>0$.  Then, there exists $\delta=\delta\q(\eps,d\w)>0$ such that for $\vsig\in \sBtp{\eps}{\R^n\times \R^d}$, we have
\begin{equation*}
\iint \q|v\w|^{-\delta} \q|\vsig\q(\alpha,v\w)\w|\: d\alpha\: dv\leq C_{\eps, d} \sBN{\eps}{\vsig}.
\end{equation*}
\end{lemma}
\begin{proof}
Clearly $\iint_{\q|v\w|>1} \q|v\w|^{-\delta} \q|\vsig\q(\alpha, v\w)\w|\: d\alpha\: dv\lesssim \LpN{1}{\vsig}\leq \sBN{\eps}{\vsig}$,
so it suffices to prove \Be\label{vle1integral}
\iint_{\q|v\w|\leq 1} \q|v\w|^{-\delta} \q|\vsig\q(\alpha,v\w)\w|\: d\alpha\: dv\lesssim \sBN{\eps}{\vsig}.
\Ee
By a weak version of the Sobolev embedding theorem (see \cite{stein-si} or \cite{triebel}),  there exists $p=p\q(\eps,d\w)>1$ such that
\begin{equation*}
\int \q\Big( \int \q|\vsig\q(\alpha, v\w)\w|^p \: dv \w\Big)^{\frac{1}{p}} \: d\alpha \lesssim \sBN{\eps}{\vsig}.
\end{equation*}
Let $p'$ be dual to $p$ and let $\delta<1/p'$.  We have, by H\"older's inequality,
and then Minkowski's inequality, 
\begin{align*}
&\iint_{\q|v\w|\leq 1} \q|v\w|^{-\delta} \q|\vsig\q(\alpha,v\w)\w|\: d\alpha\: dv \leq \q\Big(\int_{|v|\le 1}\q|v\w|^{-\delta p'} \: dv\w\Big)^{\frac{1}{p'}}\q\Big( \int \q\Big(\int \q|\vsig\q(\alpha,v\w)\w|\: d\alpha\w\Big)^p \: dv \w\Big)^{\frac{1}{p}} \: d\alpha
\\
&\lc\,\q\Big( \int \q\Big(\int \q|\vsig\q(\alpha,v\w)\w|\: d\alpha\w\Big)^p \: dv \w\Big)^{\frac{1}{p}} \: d\alpha 
\lesssim \sBN{\eps}{\vsig}.
\end{align*}
This shows \eqref{vle1integral} and completes the proof of the lemma.
\end{proof}

\begin{lemma}\label{LemmaKerSumsConverge}
Let $\q\{\vsig_j : j\in \Z\w\}\subset \sBtp{\eps}{\R^n\times \R^d}$ be a bounded set with $\int \vsig_j\q(\alpha, v\w)\: dv=0$,  for all $j\in\bbZ$.  The sum
\begin{equation*}
\sum_{j\in \Z} \dil{\vsig_j}{2^j}\q(\alpha,v\w)
\end{equation*}
converges in the sense of the topology on $L\sS'\q(\R^n\times \R^d\w)$ (and {\it a fortiori} in the sense of tempered distributions).
\end{lemma}
\begin{proof}
Let $f\in \sS\q(\R^n\times \R^d\w)$.  We will show, for some $\delta>0$,
\begin{equation*}
\q\Big|\int \dil{\vsig_j}{2^j}\q(\alpha, v\w) f\q(\alpha, v\w)\: d\alpha\: dv\w\Big| \lesssim 2^{-\q|j\w|\delta} \sup_{\alpha\in \R^n, x\in \R^d} \sum_{\q|\gamma\w|\leq 1} \q(1+\q|x\w|\w) \q|\partial_x^\gamma f\q(\alpha, x\w)\w|,
\end{equation*}
and the result will follow by the completeness of $L\cS'$.

First we consider the case $j\geq 0$.  In this case, we have
\begin{equation*}
\begin{split}
&\Big|\iint \dil{\vsig_j}{2^j}\q(\alpha, v\w) f\q(\alpha, v\w)\: d\alpha\: dv\Big|
= \Big|\iint \dil{\vsig_j}{2^j}\q(\alpha, v\w) \q[ f\q(\alpha, v\w)- f\q(\alpha, 0\w)\w]\: d\alpha\: dv\w\Big|
\\&\lesssim \Big(\sup_{\alpha\in \R^n, x\in \R^d} \sum_{\q|\gamma\w|\leq 1}\q| \partial_x^{\gamma} f\q(\alpha, x\w) \w|\Big) \iint \q|\dil{\vsig_j}{2^j}\q(\alpha, v\w)\w| \q|v\w|^\eps \: dv\: d\alpha
\\&\lesssim 2^{-j\eps}  \Big(\sup_{\alpha\in \R^n, x\in \R^d} \sum_{\q|\gamma\w|\leq 1}\q| \partial_x^{\gamma} f\q(\alpha, x\w) \w|\Big) \sBN{\eps}{\vsig_j},
\end{split}
\end{equation*}
as desired.

We now turn to $j<0$.  Take $\delta>0$ as in Lemma \ref{LemmaKerIntAgainstdeltapower}.  We have
\begin{equation*}
\begin{split}
&\Big|\iint \dil{\vsig_j}{2^j} \q(\alpha, v\w) f\q(\alpha, v\w) \,d\alpha\, dv \Big| 
\leq \Big(\sup_{\alpha\in \R^n, x\in \R^d} \q|x\w|^{\delta} \q|f\q(\alpha, x\w)\w| \Big) \iint \q|\dil{\vsig_j}{2^j}\q(\alpha, v\w)\w|\,\q|v\w|^{-\delta}\: d\alpha\: dv\\
&= \big(\sup_{\alpha\in \R^n, x\in \R^d} \q|x\w|^{\delta} \q|f\q(\alpha, x\w)\w| \big)\,2^{j\delta} \iint \q|\vsig_j\q(\alpha, v\w)\w|\,\q|v\w|^{-\delta}\: d\alpha\: dv\\
&\lesssim \big(\sup_{\alpha\in \R^n, x\in \R^d} \q|x\w|^{\delta} \q|f\q(\alpha, x\w)\w| \big)\,2^{j\delta} \sBN{\eps}{\vsig_j},
\end{split}
\end{equation*}
where in the last line we have used our choice of $\delta$ and Lemma \ref{LemmaKerIntAgainstdeltapower}.  
\end{proof}

Let $\phi\in \Czip{B^d\q(1/2\w)}$ be a radial, non-negative function with $\int \phi=1$.  For $j\in \Z$ let $\dil{\phi}{2^j}\q(v\w) = 2^{jd} \phi\q(2^j v\w)$.
Let $\psi\q(x\w) = \phi\q(x\w) - \frac{1}{2}\phi\q(x/2\w)\in \Czip{B^d\q(1\w)}$.  Let $Q_j f = f*\dil{\psi}{2^j}$.
Note that $f=\sum_{j\in \Z} Q_jf$ for $f\in \cS(\bbR^d)$ with convergence in the sense of tempered distributions.

The heart of the proof of Proposition \ref{PropKerDecompK} is the following lemma.
\begin{lemma}\label{LemmaKerQjKIsRightForm}
Suppose $0<\eps\le 1$, $0<\delta<\eps$ and let $K\in \sK_\eps$.  Let 
\[\vsig(\alpha,v)=Q_0K(\alpha,v).\]
Then, $\vsig\in \sBtp{\delta}{\R^n\times \R^d}$ and
$$\|\vsig\|\ci{\cB_\delta} \le C_{\delta,\eps,d}\|K\|_{\sK_\eps}\,.$$
\end{lemma}

\begin{proof}[Proof of Proposition \ref{PropKerDecompK} given Lemma \ref{LemmaKerQjKIsRightForm}]
Since  $\dil{K}{2^j}$ is of the same form as $K$, the lemma also yields,
with $\vsig_j:=(Q_j K)^{(2^{-j})}$,
$$\sup_{j\in \bbZ}\big \|\vsig_j
\big\|_{\sK_\eps}\,
\le C_{\delta,\eps,d}\|K\|_{\sK_\eps}\,.$$
As $\int \vsig_j(\alpha,x) dx=0$ for all $j$ it follows from standard estimates that $K=\sum_{j\in \Z} \dil{\vsig_j}{2^j}$, in the sense of tempered distributions.
Since we know $\sum_{j\in \Z} \dil{\vsig_j}{2^j}$ converges in the sense of the topology on
$L\sS'\q(\R^n\times \R^d\w)$ it follows that the sum can be taken in that sense as well.
The result now follows from Lemma \ref{LemmaKerQjKIsRightForm}.
\end{proof}

\begin{proof}[Proof of Lemma \ref{LemmaKerQjKIsRightForm}]
Note that, in light of Lemma \ref{LemmaResultsEtaDoesntMatter}, we may replace the test function $\eta$
 with $\psi$
in the definition of $\sKN{\eps}{K}$.

We begin by bounding $\|\vsig\|_{\cB_{\delta,1}}$ as in \eqref{Beps1} and split,
for fixed  $1\leq i\leq n$,
\begin{equation*}
\iint \q(1+\q|\alpha_i\w|\w)^{\delta} \q|\vsig\q(\alpha, x\w)\w|\: dx\: d\alpha = \iint\limits_{\q|x\w|\leq 1} + \iint\limits_{1<\q|x\w|\leq 1+\q|\alpha_i\w|} + \iint\limits_{\q|x\w|>1+\q|\alpha_i\w|}=:(I)+(II)+(III).
\end{equation*}
For $(I)$, we apply the Cauchy-Schwarz inequality to see
\begin{equation*}
\begin{split}
&(I) = \iint_{\q|x\w|\leq 1} \q(1+\q|\alpha_i\w|\w)^{\delta} \q|\vsig\q(\alpha, x\w)\w|\: dx\:d\alpha
 \lesssim \int \q(1+\q|\alpha_i\w|\w)^{\delta} \Big(\int \q|\psi*K\q(\alpha,x\w)\w|^2\: dx \Big)^{\frac{1}{2}}\: d\alpha
\leq \|K\|_{\sK_{\eps,1}^\psi}.
\end{split}
\end{equation*}
For $(II)$, we have
\begin{equation*}
\begin{split}
(II)&=\iint\limits_{1<\q|x\w|\leq 1+\q|\alpha_i\w|} \q(1+\q|\alpha_i\w|\w)^{\delta} \q|\vsig\q(\alpha, x\w)\w|\: dx\: d\alpha
\lesssim \sum_{k\geq 0} \iint\limits_{\substack{1+\q|\alpha_i\w|>2^k \\ 2^k \leq \q|x\w|\leq 2^{k+1}}} \q(1+\q|\alpha_i\w|\w)^{\delta} \q|\psi*K\q(\alpha,x\w)\w|\: dx\: d\alpha
\\&\lesssim \sum_{k\geq 0} 2^{k\q(\delta-\eps\w)} \iint\limits_{2^{k-1}\leq \q|x\w|\leq 2^{k+3}} \q(1+ \q|\alpha_i\w|\w)^{\eps} \q|K\q(\alpha, x\w)\w|\: dx\: d\alpha
\lesssim \|K\|_{\sK_{\eps,3} }
\end{split}
\end{equation*}
For $(III)$, we use that $\int \psi=0$ and $\supp{\psi}\subset B^d\q(0,1\w)$ to see
\begin{equation*}
\begin{split}
(III)&=\iint\limits_{\q|x\w|>1+\q|\alpha_i\w|} \q(1+\q|\alpha_i\w|\w)^{\delta} \q|\vsig\q(\alpha, x\w)\w|\: dx\: d\alpha
\\&
\lesssim \iint_{\q|x\w|>1+\q|\alpha_i\w|} \q(1+\q|\alpha_i\w|\w)^{\delta} \Big|\int \psi\q(y\w) \q[ K\q(\alpha, x-y\w)-K\q(\alpha,x\w) \w]\Big|\: dx\: d\alpha
\\&\lesssim \sum_{k\geq 0} 2^{k\delta} \int \q|\psi\q(y\w)\w|  \iint\limits_{\substack{2^k\leq \q|\alpha_i\w|\leq 2^{k+1} \\ \q|x\w|>2^k}} \q|K\q(\alpha, x-y\w)-K\q(\alpha, x\w)\w|\: dx\: d\alpha\: dy
\\&\lesssim \sum_{k\geq 0} 2^{k\delta} \int_{\q|y\w|\leq 1} \iint_{\q|x\w|>2^k} \q|K\q(\alpha, x-y\w)-K\q(\alpha, x\w)\w|\: dx\: d\alpha\: dy
\\&\lesssim \sum_{k\geq 0} 2^{k\q(\delta-\eps\w)} 
\sKN{\eps,5}{K}
\lesssim \sKN{\eps,5}{K},
\end{split}
\end{equation*}
as desired. Combining the estimates for $(I)$, $(II)$, $(III)$ gives
$$\|\vsig\|_{\cB_{\delta,1}}\lc\|K\|_{\sK_{\eps,1}^\psi} +\|K\|_{\sK_{\eps,3} }+
\sKN{\eps,5}{K}\lc 
\|K\|_{\sK_\eps}\,.$$

We turn to bounding $\|\vsig\|_{\sB_{\delta,2}}$.
Let $1\leq i\leq n$ and $0<h\leq 1$  and split 
\begin{equation*}
\iint \q|\vsig\q(\alpha+he_i, x\w)-\vsig\q(\alpha,x\w)\w|\: dx\: d\alpha = \iint\limits_{\q|x\w|\leq 2} + \iint\limits_{2\leq \q|x\w|\leq 10 h^{-1}} + \iint\limits_{\q|x\w|<10h^{-1}}=:(IV)+(V)+(VI).
\end{equation*}
Our goal is to show $(IV), (V), (VI)\lesssim h^{\delta} \sKN{\eps}{K}$.
We have, by the Cauchy-Schwarz inequality,
\begin{equation*}
\begin{split}
(IV) &= \iint\limits_{\q|x\w|\leq 2} \q|\vsig\q(\alpha+he_i, x\w)-\vsig\q(\alpha,x\w)\w|\: dx\: d\alpha
\\
&\lesssim \int \Big( \int \q|  \psi*\q[ K\q(\alpha+he_i, \cdot\w) - K\q(\alpha,\cdot\w) \w]\q(x\w) \w|^2\: dx \Big)^{\frac{1}{2}}\: d\alpha
\leq h^{\eps} \|K\|_{\sK^\psi_{\eps,2}}\,.
\end{split}
\end{equation*}
For $(V)$, we have
\begin{equation*}
\begin{split}
(V) &= \iint\limits_{2<\q|x\w|\leq 10 h^{-1}} \q|\vsig\q(\alpha + he_i, x\w)-\vsig\q(\alpha, x\w)\w|\: dx\: d\alpha
\\&\lesssim \sum_{1\leq 2^k\leq 10h^{-1}} \iint\limits_{2^{k-1}\leq \q|x\w|\leq 2^{k+2}} |  K\q(\alpha+he_i, x\w)-K\q(\alpha, x\w) |\: dx\:d\alpha
\\&\lesssim \sum_{1\leq 2^k \leq 10 h^{-1}} h^\eps \|K\|_{\sK_{\eps,4}}
\lc h^\eps \log (2+h^{-1}) \|K\|_{\sK_{\eps,4}}\,.
\end{split}
\end{equation*}
For $(VI)$, we use that $\int \psi=0$ and $\supp{\psi}\subset B^d\q(0,1\w)$ and obtain
\begin{equation*}
\begin{split}
(VI) &= \iint\limits_{\q|x\w|\geq 10 h^{-1}} \q|\vsig\q(\alpha+he_i,x\w)-\vsig\q(\alpha, x\w)\w|\: dx\: d\alpha
\leq 2 \iint\limits_{\q|x\w|>10h^{-1}} \q| \psi* K\q(\alpha,x\w)\w|\: dx\: d\alpha
\\&\lesssim \iint\limits_{\q|x\w|>10h^{-1}} \Big|\int \psi\q(y\w) \q[K\q(\alpha,x-y\w)-K\q(\alpha,x\w)\w]\: dy\Big|\: dx\: d\alpha
\\&\lesssim \int \q|\psi\q(y\w)\w| \iint_{\q|x\w|\geq 10 h^{-1}} \q|K\q(\alpha,x-y\w)-K\q(\alpha,x\w)\w|\: dx\: d\alpha\: dy
\\&\lesssim h^{\eps} \|K\|_{\sK_{\eps,5}}.
\end{split}
\end{equation*}
Combining the estimates for $(IV)$, $(V)$, $(VI)$ we get
\[\|\vsig\|_{\cB_{\delta,2}}\lc 
\|K\|_{\sK_{\eps,2}^\psi}+\|K\|_{\sK_{\eps,4}}
+ \|K\|_{\sK_{\eps,5}}
\lc \|K\|_{\sK_{\eps}}.
\]

We now turn to bounding $\|\vsig\|_{\cB_{\delta,3}}$.
Fix $h\in \R^d$ with $0<\q|h\w|\leq 1$.  Using that $\int \psi=0$, we have
\begin{equation*}
\begin{split}
&\iint|\vsig\q(\alpha, x+h\w)- \vsig\q(\alpha, x\w)|\: dx\: d\alpha
\\&\leq \iint\limits_{\q|x\w|\leq 10} \Big| \int_0^1 \q\inn{h}{\grad_x \psi * K\q(\alpha, x+sh\w) \w} ds \Big|\: dx\: d\alpha
\\&\quad+ \sum_{8\leq 2^k \leq 10\q|h\w|^{-1}}\:\: \iint\limits_{2^k\leq \q|x\w|\leq 2^{k+1}} \Big| \int_0^1 \q\inn{h}{\grad_x \psi*K\q(\alpha, x+sh\w)\w}\: ds \Big|\: dx\: d\alpha
\\&\quad+ 2\iint_{\q|x\w|\geq 9\q|h\w|^{-1}}  \Big| \int \psi\q(y\w) \q[  K\q(\alpha, x-y\w)-K\q(\alpha, x\w)\w]\: dy \Big|\: dx\: d\alpha
\\&=: (VII) + (VIII) + 2 (IX).
\end{split}
\end{equation*}
We need to  show $(VII), (VIII), (IX)\lesssim \q|h\w|^\delta \sKN{\eps}{K}$.

We begin with $(VII)$ and use the Cauchy-Schwarz inequality to see
\begin{equation*}
\begin{split}
(VII)&= \iint\limits_{\q|x\w|\leq 10} \Big| \int_0^1 \inn{ h }{ \grad_x \psi *K\q(\alpha, x+sh\w)}\: ds \Big|\: dx\: d\alpha
\\&\leq \q|h\w|\iint_{\q|x\w|\leq 11} \q| \grad \psi * K\q(\alpha, x\w)\w| \:dx\:d\alpha
\\&\lesssim \q|h\w|\int\Big( \int \q| \grad \psi  *K\q(\alpha, x\w) \w|^2\: dx\Big)^{\frac{1}{2}}\: d\alpha
\\&\lesssim \q|h\w| \|K\|\ci{\sK^{\nabla\psi}_{\eps,1}}\, .
\end{split}
\end{equation*}
For $(VIII)$ we have
\begin{equation*}
\begin{split}
(VIII) &=  \sum_{8\leq 2^k \leq 10\q|h\w|^{-1}}\:\: \iint\limits_{2^k\leq \q|x\w|\leq 2^{k+1}} \Big| \int_0^1 \inn{h}{\grad_x \psi*K\q(\alpha, x+sh\w)}\: ds \Big|\: dx\: d\alpha
\\&\leq \q|h\w| \sum_{8\leq 2^k \leq 10\q|h\w|^{-1}}\:\: \iint\limits_{2^{k-1}\leq \q|x\w|\leq 2^{k+2}} \q| \grad \psi * K\q(\alpha, x\w)\w|\: dx\: d\alpha
\\&\lesssim \q|h\w| \sum_{8\leq 2^k \leq 10\q|h\w|^{-1}}\:\: \iint\limits_{2^{k-2}\leq \q|x\w|\leq 2^{k+3}} \q|  K\q(\alpha, x\w)\w|\: dx\: d\alpha
\\&\lesssim \q|h\w|\sum_{8\leq 2^k \leq 10\q|h\w|^{-1}}\|K\|_{\sK_{0,3}}
\lesssim \q|h\w|\log\q(2+\q|h\w|^{-1}\w)\|K\|_{\sK_{0,3}}\,.
\end{split}
\end{equation*}
For  $(IX)$ we use $\supp{\psi}\subset B^d\q(0,1\w)$ and estimate
\begin{equation*}
\begin{split}
(IX)&=\iint_{\q|x\w|\geq 9\q|h\w|^{-1}}  \Big| \int \psi\q(y\w) \q[  K\q(\alpha, x-y\w)-K\q(\alpha, x\w)\w]\: dy \Big|\: dx\: d\alpha
\\&\leq \int \q|\psi\q(y\w)\w| \iint_{\q|x\w|\geq 9 \q|h\w|^{-1}} \q|K\q(\alpha, x-y\w)-K\q(\alpha, x\w)\w|\: dx\: d\alpha\: dy
\\&\lesssim h^{\eps} \sKN{\eps,5}{K},
\end{split}
\end{equation*}
as desired.  Summarizing,
$$\|\vsig\|_{\cB_{\delta,3}}\lc
 \|K\|\ci{\sK^{\nabla\psi}_{\eps,1}}+
\|K\|_{\sK_{0,3}}+
\sKN{\eps,5}{K}\lc 
\sKN{\eps}{K}$$
where in the last inequality we have used Lemma \ref{LemmaResultsEtaDoesntMatter}.

Finally  we estimate $\|\vsig\|_{\cB_{\eps,4}}$ and split 
\begin{equation*}
\iint \q(1+\q|x\w|\w)^{\delta} \q|\vsig\q(\alpha,x\w)\w|\: d\alpha\: dx = \iint\limits_{\q|x\w|\leq 10} + \iint\limits_{\q|x\w|>10} =:(X)+(XI).
\end{equation*}
We have, by the Cauchy-Schwarz inequality,
\begin{equation*}
(X) = \iint_{\q|x\w|\leq 10} \q(1+\q|x\w|\w)^{\delta} \q|\vsig\q(\alpha, x\w)\w|\: dx\: d\alpha \lesssim \int\Big(\int \q| \psi*K\q(\alpha, x\w) \w|^2\: dx\Big)^{\frac{1}{2}}\: d\alpha\lesssim \|K\|_{\sK_{\eps,1}^\psi}.
\end{equation*}
Using that $\int \psi=0$ and $\supp{\psi}\subset B^d\q(0,1\w)$, we have
\begin{equation*}
\begin{split}
(XI) &= \iint_{\q|x\w|>10} \q(1+\q|x\w|\w)^{\delta} \q|\vsig\q(\alpha,x\w)\w|\: dx\: d\alpha
\\&\lesssim \sum_{k\geq 3} 2^{k\delta} \iint\limits_{2^k\leq \q|x\w|\leq 2^{k+1}}
\Big| \int \psi\q(y\w)\q[ K\q(\alpha, x-y\w)-K\q(\alpha,x\w) \w]\: dy \Big|\: dx\: d\alpha
\\&\lesssim \sum_{k\geq 3} 2^{k\delta} \int \q|\psi\q(y\w)\w| \iint_{\q|x\w|\geq 2^k} \q|K\q(\alpha, x-y\w)-K\q(\alpha,x\w)\w|\: dx\:d\alpha\: dy
\\&
\lesssim \sum_{k\geq 3} 2^{k\q(\delta-\eps\w)}\|K\|_{\sK_{\eps,5}}
\lc \|K\|_{\sK_{\eps,5}}.
\end{split}
\end{equation*}
Hence
$$\|\vsig\|_{\cB_{\eps,4}}\lc \|K\|_{\sK_{\eps,1}^\psi}+
 \|K\|_{\sK_{\eps,5}}\lc\|K\|_{\sK_{\eps}}.$$
This completes the proof.
\end{proof}

\subsubsection{Proof of Proposition \ref{PropKerSumVsig}}
\label{pfPropKerSumVsig}


We  begin with a preparatory lemma.
Let $\Phi\in \sS\q(\R^d\w)$ satisfy $\int \Phi(x)dx=1$,
and let $\Psi\q(x\w) = \Phi\q(x\w) - \tfrac{1}{2}\Phi\q(\tfrac x2\w)$.  Define $Q_j f = f* \dil{\Psi}{2^j}$.

\begin{lemma}
Let $\eps>0$ and $\vsig\in \sBtp{\eps}{\R^n\times \R^d}$.  
Then, for $l>0$,
\begin{equation}\label{EqnKer7a}
\iint \q|Q_l \vsig\q(\alpha,x\w)\w|\: dx\: d\alpha + 2^{-l} \iint \q|\grad_x Q_l \vsig\q(\alpha,x\w)\w|\: dx\:d\alpha\lesssim 2^{-l\eps}\sBN{\eps}{\vsig},
\end{equation}
\begin{equation}\label{EqnKer7b}
\iint\limits_{\q|x\w|\geq R} \q|Q_l \vsig\q(\alpha, x\w)\w|\: dx\: d\alpha + 2^{-l} \iint\limits_{\q|x\w|\geq R} \q|\grad_x Q_l \vsig\q(\alpha, x\w)\w|\: dx\: d\alpha\lesssim R^{-\eps} \sBN{\eps}{\vsig},
\end{equation}
and for $\q|h\w|\leq 1$,
\begin{equation}\label{EqnKer7c}
\iint\limits_{\q|x\w|\geq R} \big|Q_l \vsig\q(\alpha, x+h\w)- Q_l \vsig\q(\alpha,x\w)\big|\: dx\: d\alpha \lesssim \min\{ 2^l |h|,1\} \min\{2^{-l\eps} ,R^{-\eps}\} \sBN{\eps}{\vsig}.
\end{equation}

Let $0<\delta<\eps$.  Then for $R\geq 0$, $i=1,\ldots n$,
\begin{equation}\label{EqnKer7d}
\iint_{\q|x\w|\geq R} \q(1+\q|\alpha_i\w|\w)^\delta \q|Q_l\vsig\q(\alpha,x\w)\w|\: dx\: d\alpha \lesssim \min\{ 2^{-l(\eps-\delta)},  R^{-(\eps-\delta)}\} 
\sBN{\eps}{\vsig},
\end{equation}
and for all $0<\q|\tau\w|\leq 1$, $j=1,\ldots, n$,
\begin{equation}\label{EqnKer7e}
\q|\tau\w|^{-\delta} \iint_{\q|x\w|\geq R} \big| Q_l\vsig\q(\alpha+\tau e_j, x\w) - Q_l\vsig\q(\alpha,x\w)\big|\:dx\:d\alpha
 \lesssim 
 \min\{ 2^{-l(\eps-\delta)},  R^{-(\eps-\delta)}\} 
\sBN{\eps}{\vsig}.
\end{equation}
\end{lemma}
\begin{proof}
First observe that \eqref{EqnKer7a} 
is an immediate consequence of the definitions. Next, for the proof of 
\eqref{EqnKer7b}  we may assume $R\ge 1$.
Also, observe, for every $N\in \N$,
\begin{equation*}
\begin{split}
&\iint\limits_{\q|x\w|\geq R} \q|Q_l \vsig\q(\alpha,x\w)\w|\: dx\: d\alpha +2^{-l} \iint\limits_{\q|x\w|\geq R} \q|\grad_x Q_l \vsig\q(\alpha,x\w)\w|\: dx\: d\alpha
\\&\leq C_N \iiint\limits_{\q|x\w|\geq R} \frac{2^{ld}}{\q(1+2^l\q|y\w|\w)^N} \q|\vsig\q(\alpha,x-y\w)\w|\: dx\: d\alpha \:dy
\\& = C_N \iiint\limits_{\substack{\q|x\w|\geq R \\ \q|y\w|\leq R/2}} +C_N \iiint\limits_{\substack{\q|x\w|\geq R \\ \q|y\w|> R/2}} =:C_N \big((I)+(II)\big).
\end{split}
\end{equation*}
For $(I)$ we have
\begin{equation*}
(I)\lesssim R^{-\eps}  \iiint\limits_{\substack{\q|x\w|\geq R \\ \q|y\w|\leq R/2}}\frac{2^{ld}}{\q(1+2^l\q|y\w|\w)^N} \q(1+\q|x-y\w|\w)^{\eps} \q|\vsig\q(\alpha,x-y\w)\w|\: dx\: d\alpha \:dy\lesssim R^{-\eps} \sBN{\eps,4}{\vsig}.
\end{equation*}
For $(II)$, taking $N\geq d+1$, we have
\begin{equation*}
(II)\lesssim \LpN{1}{\vsig} \int_{\q|y\w|> R/2} \frac{2^{ld}}{\q(1+2^l\q|y\w|\w)^N} \: dy\lesssim \q(2^l R\w)^{-1} \LpN{1}{\vsig}
\leq R^{-\eps} \sBN{0,4}{\vsig},
\end{equation*}
and \eqref{EqnKer7b} follows.  \eqref{EqnKer7c} follows by combining \eqref{EqnKer7a} and \eqref{EqnKer7b}.

We now turn to \eqref{EqnKer7d} and we separate the proof into two cases, $R\leq 2^l$ and $R\ge 2^l$. 
For  $R\leq 2^l$ we have, by \eqref{EqnKer7a},
\begin{equation*}
\iint\limits_{\q|x\w|\geq R} \q(1+\q|\alpha_i\w|\w)^\delta \q|Q_l\vsig\q(\alpha,x\w)\w|\: dx\: d\alpha\leq \iint\limits_{\q|\alpha_i\w|\leq 2^{l}} + \iint\limits_{\q|\alpha_i\w|>2^l} =:(III)+(IV).
\end{equation*}
For $(III)$, we apply \eqref{EqnKer7a} to see
\begin{equation*}
(III)\lesssim 2^{l\delta} \iint \q|Q_l \vsig\q(\alpha,x\w)\w|\: dx\:d\alpha\lesssim 2^{-l\q(\eps-\delta\w)} \sBN{\eps}{\vsig}.
\end{equation*}
Also, we have
\begin{align*}
(IV)&\lesssim 2^{-l\q(\eps-\delta\w)} \int \q(1+\q|\alpha_i\w|\w)^{\eps} \int \q|Q_l \vsig\q(\alpha,x\w)\w|\: dx\: d\alpha
\\&\lesssim 2^{-l\q(\eps-\delta\w)} \iint \q(1+\q|\alpha_i\w|\w)^{\eps}  \q| \vsig\q(\alpha,x\w)\w|\: dx\: d\alpha\lesssim 2^{-l\q(\eps-\delta\w)} \sBN{\eps}{\vsig}.
\end{align*}
In the second case,  $R\geq 2^l$, we have
\begin{equation*}
\iint\limits_{\q|x\w|\geq R} \q(1+\q|\alpha_i\w|\w)^\delta \q|Q_l\vsig\q(\alpha,x\w)\w|\: dx\: d\alpha \leq \iint\limits_{\substack{\q|\alpha_i\w|\leq R\\ \q|x\w|\geq R}} + \iint\limits_{\q|\alpha_i\w|>R} =:(V)+(VI).
\end{equation*}
Using \eqref{EqnKer7b},
\begin{equation*}
(V) \lesssim R^\delta \iint_{\q|x\w|\geq R} \q|Q_l\vsig\q(\alpha,x\w)\w|\: dx\:d\alpha \lesssim R^{\delta-\eps} \sBN{\eps}{\vsig}.
\end{equation*}
And,
\begin{align*}
(VI)& \lesssim R^{\delta-\eps} \int_{\q|\alpha_i\w|>R} \q(1+\q|\alpha_i\w|\w)^{\eps}\int \q|Q_l\vsig\q(\alpha,x\w)\w|\: dx\:d\alpha
\\&\lesssim R^{\delta-\eps} \iint \q(1+\q|\alpha_i\w|\w)^\eps \q|\vsig\q(\alpha,x\w)\w|\: dx\: d\alpha\lesssim R^{\delta-\eps} \sBN{\eps}{\vsig},
\end{align*}
which completes the proof of \eqref{EqnKer7d}.

Finally, we turn to \eqref{EqnKer7e}.  This we separate into four cases.  In the first case,  $R\le 2^l$, $\tau\geq 2^{-l}$, we have
\begin{equation*}
\q|\tau\w|^{-\delta} \iint_{\q|x\w|\geq R} \q|Q_l \vsig\q(\alpha+\tau e_j,x \w) -Q_l\vsig\q(\alpha,x\w)\w|\: dx\: d\alpha \lesssim 2^{l\delta} \iint \q|Q_l\vsig\q(\alpha,x\w)\w|\: dx\: d\alpha\lesssim 2^{-l\q(\eps-\delta\w)}\sBN{\eps}{\vsig}.
\end{equation*}
In the second case, $R\le 2^l$, $\q|\tau\w|\leq 2^{-l}$, we have
\begin{equation*}
\begin{split}
&\q|\tau\w|^{-\delta} \iint_{\q|x\w|\geq R} \q|Q_l \vsig\q(\alpha+\tau e_j,x \w) -Q_l\vsig\q(\alpha,x\w)\w|\: dx\: d\alpha 
\\&\lesssim 2^{-l\q(\eps-\delta\w)} \q|\tau\w|^{-\eps} \iint \q|\vsig\q(\alpha+\tau e_j, x\w)-\vsig\q(\alpha,x\w)\w|\: dx\:d\alpha \lesssim 2^{-l\q(\eps-\delta\w)} \sBN{\eps}{\vsig}.
\end{split}
\end{equation*}
In the third case, $R\geq 2^l$,  $\q|\tau\w|\geq R^{-1}$, we have
\begin{equation*}
\q|\tau\w|^{-\delta} \iint_{\q|x\w|\geq R} \q|Q_l \vsig\q(\alpha+\tau e_j,x \w) -Q_l\vsig\q(\alpha,x\w)\w|\: dx\: d\alpha  \lesssim R^{\delta} \iint_{\q|x\w|\geq R} \q|Q_l\vsig\q(\alpha,x\w)\w|\: dx\: d\alpha\lesssim R^{\delta-\eps} \sBN{\eps}{\vsig},
\end{equation*}
where in the last inequality we have used \eqref{EqnKer7b}.  In the last case,
 $R\geq 2^{l}$, $\q|\tau\w|\leq R^{-1}$,
\begin{equation*}
\begin{split}
&\q|\tau\w|^{-\delta} \iint_{\q|x\w|\geq R} \q|Q_l \vsig\q(\alpha+\tau e_j,x \w) -Q_l\vsig\q(\alpha,x\w)\w|\: dx\: d\alpha
\\& \lesssim R^{\delta-\eps} \q|\tau\w|^{-\eps} \iint \q|\vsig\q(\alpha+\tau e_j,x \w) -\vsig\q(\alpha,x\w)\w|\: dx\: d\alpha \lesssim R^{\delta-\eps} \sBN{\eps}{\vsig},
\end{split}
\end{equation*}
as desired.  This completes the proof.
\end{proof}

\begin{proof}[Proof of Proposition \ref{PropKerSumVsig}, conclusion.]
Let $\vsig_j$ be as in the statement of the proposition.  By Lemma \ref{LemmaKerSumsConverge} we already know the sum $\sum_{j\in \Z} \dil{\vsig_j}{2^j}$ converges in the topology
on $L\sS'\q(\R^n\times \R^d\w)$. 
Our goal is to show convergence of the sum
$\sKN{\delta}{\sum_{j\in \Z} \dil{\vsig_j}{2^j}}$ in $\sK_\delta$ for $0<\delta<\eps/2$.  Fix $j_1,j_2\in \Z$, $j_1<j_2$.  Define $K= \sum_{j_1\leq j\leq j_2} \dil{\vsig_j}{2^j}$.
We will show $\sKN{\delta}{K}\lesssim \sup_{j}\sBN{\eps}{\vsig_j}$, with the implicit constant independent of $j_1,j_2$.  The result then follows y a limiting argument. 
In what follows, summations  in $j$ are taken over the range  ${j_1\leq j\leq j_2}$.  We assume, without loss of generality, $$\sup_{j} \sBN{\eps}{\vsig_j}=1.$$

Let $\chiz\in \sS\q(\R^d\w)$ be so that $\chizh\q(\xi\w)=1$ for $\q|\xi\w|\leq 1$ and $\chizh$ is supported in $\q\{\xi : \q|\xi\w|\leq 2\w\}$.
For $l\geq 1$ let $\chil = \dil{\chiz}{2^l}-\dil{\chi_0}{2^{l-1}}$, so that $\sup_{l\in \bbZ} \chilh\q(\xi\w)=1$ for $\xi\ne 0$.
We write
\begin{equation*}
K=\sum_{j} \dil{\vsig_j}{2^j} = \sum_{l\geq 0} \sum_j \dil{\vsig_{j,l}}{2^j},
\end{equation*}
where  $$\vsig_{j,l}\q(\alpha,\cdot\w)= \chil* \vsig_j\q(\alpha, \cdot\w)$$ and  the 
 convolution is in $\R^d$.
Let $$K_l =\sum_j \dil{\vsig_{j,l}}{2^j}.$$ The proof will be complete once we have shown
$\sKN{\delta}{K_l}\lesssim 2^{-l\q(\eps-2\delta\w)}$.

Our  first goal is to 
show, for $1\leq i\leq n$, $t\in \R$,
\begin{equation}\label{EqnKerSumFirst}
\int \q(1+\q|\alpha_i\w|\w)^{\delta} \LpRdN{2}{\eta*\dil{K_l}{t}}\: d\alpha \lesssim \q(1+l\w) 2^{-l\q(\eps-\delta\w)}
\end{equation}
which gives
 $\|K_l\|_{\sK^\eta_{\delta,1}}\lc \q(1+l\w) 2^{-l\q(\eps-\delta\w)}.$  
To prove \eqref{EqnKerSumFirst}, we will show
\begin{equation}\label{EqnKerSum2}
\int \q(1+\q|\alpha_i\w|\w)^\delta \big\|\eta*\dil{\vsig_{j,l}}{2^j t}\q(\alpha,\cdot\w)\big\|_2 \: d\alpha
\lesssim
\begin{cases}
\q(2^{l\q(\eps-\delta\w)} 2^{j} t\w)^{-\frac{\eps-\delta}{1+\eps-\delta } } &\text{if }2^jt\geq 2^l,\\
2^{-l\q(\eps-\delta\w)} &\text{if } 2^{-2l}\leq 2^{j} t\leq 2^l,\\
\q(2^{l+j} t\w)^{d/2} &\text{if }2^j t\leq 2^{-2l}.
\end{cases}
\end{equation}
Summing \eqref{EqnKerSum2} in $j$ yields \eqref{EqnKerSumFirst}, so we focus on \eqref{EqnKerSum2}.

First we consider the case when $2^{j}t \geq 2^l$.
Letting $r\in \q[1,2^j t\w]$ be chosen later, we use that $\int \vsig_{j,l}\q(\alpha,x\w)\: dx=0$ to see
\begin{equation*}
\begin{split}
&\int \q(1+\q|\alpha_i\w|\w)^{\delta} \big\|\eta * \dil{\vsig_{j,l}\q(\alpha,\cdot\w)}{2^jt}\big\|_2\: d\alpha
\\&\lesssim \int \q(1+\q|\alpha_i\w|\w)^\delta \Big( \int\Big|\int \q[\eta\q(x-y\w)-\eta\q(x\w)\w] \q(2^j t\w)^d \vsig_{j,l}\q(\alpha,2^j t y\w)\: dy \Big|^2\: dx \Big)^{\frac{1}{2}}\: d\alpha
\\&\lesssim \int\q(1+\q|\alpha_i\w|\w)^{\delta} \int \q|\vsig_{j,l}\q(\alpha,v\w)\w| \LpRdN{2}{\eta\q(\cdot-\tfrac{v}{2^jt}\w) - \eta\q(\cdot\w) }\: dv\: d\alpha
\\&\lesssim \iint \q(1+\q|\alpha_i\w|\w)^{\delta}  \q|\vsig_{j,l}\q(\alpha,v\w)\w| \min\{ \tfrac{|v|}{2^jt},1\w\} \: dv\: d\alpha
\\&= \iint_{\q|v\w|\leq r} + \iint_{\q|v\w|>r} =:(I)+(II).
\end{split}
\end{equation*}
We have, using \eqref{EqnKer7d} with $R=0$,
\begin{equation*}
(I) \lesssim \frac{r}{2^jt}  \iint \q(1+\q|\alpha_i\w|\w)^{\delta}  \q|\vsig_{j,l}
\q(\alpha,v\w)\w|  \: dv\: d\alpha\lesssim \frac{r}{2^jt} 2^{-l\q(\eps-\delta\w)}.
\end{equation*}
Using \eqref{EqnKer7d} with $R=r$, 
\begin{equation*}
(II) \lesssim \iint_{\q|x\w|\geq r} \q(1+\q|\alpha_i\w|\w)^{\delta}  \q|\vsig_{j,l}\q(\alpha,v\w)\w|  \: dv\: d\alpha \lesssim r^{-\q(\eps-\delta\w)}.
\end{equation*}
We choose  $r$ so that $r^{1+\eps-\delta}=2^{l\q(\eps-\delta\w)} 2^j t$; this yields  \eqref{EqnKerSum2}
in the case $2^{j}t\geq 2^l$ under consideration.

For $2^{-2l}\leq 2^j t\leq 2^l$ we use the trivial $L^1\to L^2$ 
bound for convolution with $\eta$ and a change of variables,
combined with \eqref{EqnKer7d} (with $R=0$) to see
\begin{equation*}
\int \q(1+\q|\alpha_i\w|\w)^{\delta} \|\eta*\dil{\vsig_{j,l}}{2^jt}\q(\alpha,\cdot\w)\|_2\: d\alpha\lesssim \int \q(1+\q|\alpha_i\w|\w)^\delta \|\vsig_{j,l}\q(\alpha,\cdot\w)\|_1 \: d\alpha\lesssim 2^{-l\q(\eps-\delta\w)},
\end{equation*}
as desired.

Now assume $2^j t\leq 2^{-l}$. 
 Let $\ups\in \sS\q(\R^d\w)$ be such that $\upsh\q(\xi\w)=1$ for $\q|\xi\w|\leq 2$, so that $\upsh\q(2^{-l}\cdot\w)=1$ on the support
of $\vsigh_{j,l}$.  We then have, using $\|\upsh\q(2^{-j-l} t^{-1}\cdot\w) \etah\q(\cdot\w)\|_2\lesssim \q(2^{j+l} t\w)^{d/2}$,
\begin{equation*}
\begin{split}
&\int\q(1+\q|\alpha_i\w|\w)^{\delta} \| \eta*\dil{\vsig_{j,l}}{2^j t} \q(\alpha,\cdot\w)  \|_2\: d\alpha 
\lesssim \int \q(1+\q|\alpha_i\w|\w)^\delta \|\eta* \dil{\ups}{2^{j+l} t}\|_2 \|\vsig_{j,l}\q(\alpha,\cdot\w)\|_1\: d\alpha
\\&\lesssim \int\q(1+\q|\alpha_i\w|\w)^\delta \|\upsh\q(2^{-j-l} t^{-1}\cdot\w) \etah\q(\cdot\w)\|_2 \|\vsig_{j,l}\q(\alpha,\cdot\w)\|_1\: d\alpha
\lesssim \q(2^{j+l} t\w)^{d/2}.
\end{split}
\end{equation*}
This completes the proof of \eqref{EqnKerSum2} and therefore of \eqref{EqnKerSumFirst}.

A simple modification of the above proof, using \eqref{EqnKer7e} in place of \eqref{EqnKer7d}, gives for $\q|\tau\w|\leq 1$,
\begin{equation*}
\int \big\| \eta*\q[ \dil{\vsig_{j,l} }{2^{j}t}\q(\alpha+\tau e_j,\cdot \w) - \dil{\vsig_{j,l}}{2^j t}\q(\alpha, \cdot\w) \w] \big\|_2\: d\alpha
\lesssim \q|\tau\w|^\delta \cdot
\begin{cases}
\q(2^{j} t\w)^{-\q(\eps-\delta\w)} &\text{if } 2^j R \geq 2^l,
\\ 2^{-l\q(\eps-\delta\w)} &\text{if } 2^{-2l}\leq 2^j R\leq 2^l,
\\ \q(2^{l+j} R\w)^d &\text{if } 2^j R\leq 2^{-2l}.
\end{cases}
\end{equation*}
Summing in $j$ shows that for $0<h\leq 1$,
\begin{equation*}
h^{-\eps} \int \big\|\eta*\q[ \dil{K_l}{t}\q(\alpha+he_i,\cdot\w)- \dil{K_l}{t}\q(\alpha,\cdot\w) \w] \big\|_2\: d\alpha \lesssim \q(1+l\w)2^{-l\q(\eps-\delta\w)}
\end{equation*}
and hence $\|K_l\|_{\sK_{\delta,2}^\eta} \lesssim \q(1+l\w)2^{-l\q(\eps-\delta\w)}.$

Next we wish to show $\|K_l\|_{\sK_{\delta,3}} \lesssim \q(1+l\w)2^{-l\q(\eps-\delta\w)}$, that is, for $1\leq i\leq n$, $R>0$,
\begin{equation}\label{EqnKerSum3}
\iint\limits_{R\leq \q|x\w|\leq 2R} \q(1+\q|\alpha_i\w|\w)^{\delta} \q|K_l\q(\alpha,x\w)\w|\: dx\:d\alpha\lesssim \q(1+l\w) 2^{-\q(\eps-\delta\w)l}.
\end{equation}
To prove \eqref{EqnKerSum3} we will show
\begin{equation}\label{EqnKerSum4}
\iint\limits_{R\leq \q|x\w|\leq 2R} \q(1+\q|\alpha_i\w|\w)^\delta \q|\dil{\vsig_{j,l}}{2^j} \q(\alpha,x\w)\w|\: dx\:d\alpha
\lesssim
\begin{cases}
\q(2^j R\w)^{-\q(\eps-\delta\w)} &\text{if }2^j R\geq 2^l, \\
2^{-l\q(\eps-\delta\w)} &\text{if }2^{-2l}\leq 2^j R\leq 2^l,\\
\q(2^{l+j} R\w)^d &\text{if }2^j R \leq 2^{-2l}.
\end{cases}
\end{equation}
Summing \eqref{EqnKerSum4} in $j$ yields \eqref{EqnKerSum3}.
Now, applying \eqref{EqnKer7d},
\begin{equation*}
\begin{split}
&\iint\limits_{R\leq \q|x\w|\leq 2R} \q(1+\q|\alpha_i\w|\w)^\delta \q|\dil{\vsig_{j,l}}{2^j} \q(\alpha,x\w)\w|\: dx\:d\alpha
\leq \iint\limits_{2^j R\leq \q|x\w|} \q(1+\q|\alpha_i\w|\w)^\delta \q|\vsig_{j,l} \q(\alpha,x\w)\w|\: dx\:d\alpha
\\
&\lesssim
\begin{cases}
\q(2^{j}R\w)^{-\q(\eps-\delta\w)} &\text{if }2^j R\geq 2^l,\\
2^{-l\q(\eps-\delta\w)} &\text{if }2^j R\leq 2^l.
\end{cases}
\end{split}
\end{equation*}
Thus, to complete the proof of \eqref{EqnKerSum4} we need only consider the case when $2^j R\leq 2^{-2l}$.
We have
\begin{equation*}
\begin{split}
&\iint\limits_{R\leq \q|x\w|\leq 2R} \q(1+\q|\alpha_i\w|\w)^\delta \q|\dil{\vsig_{j,l}}{2^j} \q(\alpha,x\w)\w|\: dx\:d\alpha
= \iint\limits_{2^j R\leq \q|x\w|\leq 2^{j+1}R} \q(1+\q|\alpha_i\w|\w)^\delta \q|\vsig_{j,l} \q(\alpha,x\w)\w|\: dx\:d\alpha
\\&\lesssim \q(2^j R\w)^d \int \q(1+\q|\alpha_i\w|\w)^\delta \LpRdN{\infty}{\vsig_{j,l}\q(\alpha,\cdot\w)}\: d\alpha
\lesssim \q(2^j R\w)^d 2^{ld} \int \q(1+\q|\alpha_i\w|\w)^\delta \LpRdN{1}{\vsig_j \q(\alpha,\cdot\w)}\: d\alpha
\\&\lesssim \q(2^{j+l} R\w)^d,
\end{split}
\end{equation*}
competing the proof of \eqref{EqnKerSum4} and therefore of \eqref{EqnKerSum3}.

A simple modification of the above yields, for $0<\q|\tau\w|\leq 1$,
\begin{equation*}
\iint\limits_{R\leq \q|x\w|\leq 2R} \q|\dil{\vsig_{j,l}}{2^j} \q(\alpha+\tau e_i, x\w)- \dil{\vsig_{j,l}}{2^j}\q(\alpha,x\w)\w|\: dx\: d\alpha
\lesssim
\q|\tau\w|^\delta \cdot\begin{cases}
\q(2^j R\w)^{-\q(\eps-\delta\w)}&\text{if }2^j R\geq 2^l,
\\ 2^{-l\q(\eps-\delta\w)} &\text{if }2^{-2l}\leq 2^j R\leq 2^l,
\\ \q(2^{l+j} R\w)^d&\text{if }2^jR\leq 2^{-2l}.
\end{cases}
\end{equation*}
Summing in $j$ yields, for $0<h\leq 1$, $R>0$,
\begin{equation*}
h^{-\delta} \iint\limits_{R\leq \q|x\w|\leq 2R} \q|K_l\q(\alpha+he_i, x\w)- K_l\q(\alpha, x\w)\w|\: dx\: d\alpha \lesssim  \q(1+l\w) 2^{-\q(\eps-\delta\w)l}
\end{equation*}
and hence $\|K\|_{\sK_{\eps,4}}\lesssim  \q(1+l\w) 2^{-\q(\eps-\delta\w)l}$.

Finally, we wish to show, for $R\geq 2$, $y\in \R^d$,
\begin{equation}\label{EqnKerSum5}
R^{\delta} \iint\limits_{\q|x\w|\geq R\q|y\w|} \q|K\q(\alpha, x-y\w)- K\q(\alpha,x\w)\w|\: dx\: d\alpha\lesssim 2^{-l\q(\eps-2\delta\w)}.
\end{equation}
First, estimate 
\begin{equation*}
\begin{split}
&R^\delta \iint\limits_{\q|x\w|\geq R\q|y\w|} \q|\dil{\vsig_{j,l}}{2^j}\q(\alpha, x-y\w) - \dil{\vsig_{j,l}}{2^j}\q(\alpha, x\w)\w|\: dx\: d\alpha = R^\delta \iint\limits_{\q|x\w|>2^j \q|y\w| R} \q|\vsig_{j,l} \q(\alpha, x-2^jy\w) - \vsig_{j,l}\q(\alpha,x\w)\w|\: dx\:d\alpha
\\&\lesssim R^\delta \min\{1,  2^{j+l}\q|y\w|\} \min \{2^{-l\eps} ,
(2^{j} \q|y\w| R\w)^{-\eps}\}=:\sE\q(j,l,R\w).
\end{split}
\end{equation*}
Here we applied \eqref{EqnKer7c} with $2^{j}\q|y\w|$ in place of $\q|h\w|$ and $2^j\q|y\w| R$ in place of $R$.
Note the left hand side of \eqref{EqnKerSum5} is bounded by $\sum_j \sE\q(j,l,R\w)$.

In the case  $R\geq 2^{2l}$, we estimate
\begin{multline*}
\sum_{j} \sE\q(j,l,R\w)\lesssim
\\
\sum_{2^j|y|\ge 2^{-l}} R^{\delta-\eps} (2^j|y|)^{-\eps}
+
\sum_{2^l/R\le 2^j|y|\le 2^{-l} }
2^l\q(2^j\q|y\w|\w)^{1-\eps} R^{\delta-\eps} 
+\sum_{2^j\q|y\w|\leq 2^l/R}R^{\delta}\q(2^j\q|y\w|\w)2^{l\q(1-\eps\w)} .
\end{multline*}
The first two sums are $O(R^{\delta-\eps}2^{l\eps})$, and the third sum is
$O(R^{\delta-1} 2^{\q(2-\eps\w)l})$; here we used $R\ge 2^{2l}$.


In the case $R\leq 2^{2l}$ we have
\begin{equation*}
\sum_{j} \sE\q(j,l,R\w)
 \lesssim \sum_{2^j\q|y\w|\geq 2^l/R}R^{\delta-\eps}\q(2^j\q|y\w|\w)^{-\eps}
+ 
\sum_{2^{-l}\leq 2^j\q|y\w|\leq 2^l/R}R^\delta 2^{-l\eps}
+
\sum_{2^j\q|y\w|\leq 2^{-l}}R^\delta 2^j\q|y\w| 2^{l\q(1-\eps\w)}.
\end{equation*}
The first sum is $O(R^\delta 2^{-l\eps})$,
the second sum is $O(R^\delta 2^{-l\eps} \log(1+ 2^{2l}/R))$, and since 
$R\le 2^{2l}$ the third sum
is $O(R^\delta 2^{-l\eps})$.
In both cases we obtain $\sum_{j} \sE\q(j,l,R\w)\lc 2^{-l(\eps-2\delta)}$.
This completes the proof of \eqref{EqnKerSum5}.
Combining all of the above inequalities completes the proof of the proposition. 
\end{proof}

\section{Adjoints}\label{SectionAdjoints}

	
	
	
	



This section is devoted to studying the space $\sBp{\eps}$; in particular will give the proof of Theorem \ref{ThmOpResAdjoints}. It will be advantageous to work with a variant of this class, for functions on $\bbR^N$, with $N=n+d$.

\begin{defn}
Fix $\eps>0$ and $N\in \N$.  We define a Banach space $\fBtp{\eps}{\R^N}$ to be the space
of measurable functions $\gamma:\R^N\rightarrow \bbC$ such that the 
norm 
\begin{equation*}
\fBN{\eps}{\gamma}:= \max_{1\leq i\leq N} \int \q(1+\q|s_i\w|\w)^{\eps}\q|\gamma\q(s\w)\w|\: ds + \sup_{\substack{0<h\leq 1 \\ 1\leq i \leq N }} h^{-\eps} \int \q|\gamma\q(s+he_i\w)-\gamma\q(s\w)\w|\: ds,
\end{equation*}
is finite. Here $e_1,\ldots, e_N$ denotes the standard basis of $\R^N$.
\end{defn}

\begin{rmk}\label{RmkAdjfBandsBSame}
The spaces $\fBtp{\eps}{\R^{n+d}}$ and $\sBtp{\eps}{\R^n\times \R^d}$ coincide;
indeed, for $\vsig\in \sBtp{\eps}{\R^n\times \R^d}$,
we have the equivalence
\begin{equation*}
\fBN{\eps}{\vsig}\approx \sBN{\eps}{\vsig},
\end{equation*}
with implicit constants depending only on $d$.  In this section we find it more useful to use the space $\fBp{\eps}$ as it treats the $\alpha$ and $x$ variables
symmetrically.
\end{rmk}

The following two propositions involve operations on functions in $\fB_\eps$ 
involving inversions and multiplicative shears. They are the  main technical results needed for the proof of Theorem  \ref{ThmOpResAdjoints}.

\begin{prop}\label{PropAdjInvts1}
Let $\eps>0$ and $\delta<\eps/3$.  Let $\gamma\in \fBtp{\eps}{\R^N}$ and
\begin{equation*}
J_1\gamma\q(s_1,\ldots, s_N\w):=s_1^{-2} \gamma\q(s_1^{-1},s_2,\ldots, s_N\w),
\end{equation*} $\gamma\in \fBtp{\eps}{\R^N}$. Then
$J_1\ga\in \fBtp{\delta}{\R^N}$ 
and $$\fBN{\delta}{J_1\ga}\lesssim \fBN{\eps}{\gamma}.$$
\end{prop}

\begin{prop}\label{PropAdjMultis1}
Let $\eps>0$ and $\delta<\eps/3$. Let $\gamma\in \fBtp{\eps}{\R^N}$, $n\in \q\{1,\ldots, N\w\}$   and set 
\begin{equation*}
M\ga \q(s_1,\ldots, s_N\w):=s_1^{n-1} \gamma\q(s_1, s_1 s_2,\ldots,s_1 s_n, s_{n+1},s_{n+2},\ldots, s_N\w).
\end{equation*}
Then $M\ga\in \fBtp{\eps'}{\R^N}$ and 
$$\fBN{\delta}{M\ga}\lesssim n \fBN{\eps}{\gamma}.$$
\end{prop}

For later use in \S\ref{SectionRoleOfRPnSecond} we state these results in a different form:
\begin{cor}\label{ThmAdjTechResult}
Let $1\leq n\leq N$.  For $\gamma\in \fBtp{\eps}{\R^N}$ define two functions
\begin{equation*}\Gamma_1\q(s_1,\ldots, s_N\w):= s_1^{-n-1} \gamma\q(s_1^{-1},s_1^{-1}s_2,\ldots, s_1^{-1} s_n, s_{n+1},\ldots, s_N\w),\end{equation*}\begin{equation*}\Gamma_2\q(s_1,\ldots, s_N\w) := s_1^{-\q(n-1\w)} \gamma\q(s_1, s_1^{-1} s_2,\ldots, s_1^{-1} s_n, s_{n+1},\ldots, s_N\w).\end{equation*} There exists $\eps'=\epsilon'\q(\eps\w)>0$ (depending neither on 
 $N$ nor $n$) such that
\begin{equation*}
\fBN{\eps'}{\Gamma_1}+\fBN{\eps'}{\Gamma_2}\leq C_{\eps,\eps'} n \fBN{\eps}{\gamma}.
\end{equation*}
\end{cor}
\begin{proof}
Notice that $\Gamma_1= J_1M\gamma$, $\Gamma_2= J_1MJ_1\gamma$ where $J_1$ and $M$ are as in the propositions above.
\end{proof} 

\subsection{Proof of Theorem \ref{ThmOpResAdjoints}}
We assume Proposition \ref{PropAdjInvts1} and Proposition
\ref{PropAdjMultis1}
and deduce Theorem \ref{ThmOpResAdjoints}.
If $\vsig\in L^1(\bbR^n\times \bbR^d)$ 
and $\vp$ is a permutation of $\q\{1,\ldots, n+2\w\}$, we shall show
\begin{equation*}
\La[\vsig]({b_{\vp\q(1\w)},\ldots, b_{\vp\q(n+2\w)}} )= \La[\ell_\vp \vsig]
({b_1,\ldots, b_{n+2}}),
\end{equation*}
such that
$\LpN{1}{\ell_\vp \vsig}=\LpN{1}{\vsig}$ and such that there exists $\eps' >c\eps$, with $c$ independent of $\vp$, and 
$$\sBN{\eps'}{\ell_\vp\vsig}\lesssim n^{2} \sBN{\eps}{\vsig}$$
for $\vsig\in \cB_\eps$.


Every permutation of $\q\{1,\ldots, n+2\w\}$ is a composition of at most four permutations of the following three forms, with the permutation in (iii) occuring at most twice.

\begin{enumerate}[(i)]
\item A permutation of $\q\{1,\ldots, n\w\}$, leaving $n+1$ and $n+2$ fixed.
\item The permutation which switches $n+1$ and $n+2$, leaving all other elements fixed.
\item The permutation which switches $n+1$ and $1$, leaving all other elements fixed.
\end{enumerate}

{\it Case (i) }If $\vp$ is a permutation of $\q\{1,\ldots, n\w\}$, leaving $n+1$ and $n+2$ fixed, then it is immediate to verify
\Be\label{firstnvar}\ell_\vp\vsig\q(\alpha, v\w)=\vsig\q(\alpha_{\vp^{-1}\q(1\w)},\ldots, \alpha_{\vp^{-1}\q(n\w)}, v \w),\Ee
and thus $\sBN{\eps}{\ell_\vp\vsig}=\sBN{\eps}{\vsig}$ and $\sBzN{\ell_\vp\vsig}=\sBzN{\vsig}$.

{\it Case (ii).}
If $\vp$ is the permutation which switches $n+1$ and $n+2$, leaving all other elements fixed, then it is immediate to verify that
\Be\label{transpostriv}\ell_\vp\vsig\q(\alpha, v\w)=\vsig\q(1-\alpha_1,\ldots, 1-\alpha_n,v\w).\Ee  We have
$\sBN{\eps}{\vsig_{\vp}}\approx \sBN{\eps}{\vsig}$ and $\sBzN{\vsig_{\vp}}=\sBzN{\vsig}$.

In both of the above cases, if $\int \vsig\q(\alpha, v\w)\: dv=0$ $\forall \alpha$ then $\int\vsig_\vp\q(\alpha, v\w)\: dv=0$ $\forall \alpha$.

{\it Case (iii).} 
We compute
\begin{align*}
&\La[\vsig](b_{n+1}, b_2,\dots, b_n, b_1, b_{n+2})\\
&=
\iiint \vsig(\alpha,v) b_{n+1}(x-\alpha_1v) \big(\prod_{i=2}^n b_i(x-\alpha_i v)\big)
b_1(x-v) b_{n+2}(x) \,dv\,dx\, d\alpha
\\
&=\iiint|\alpha_1|^{-d}\vsig (\alpha, \alpha_1^{-1}w)
 b_{n+1}(x-w) \big(\prod_{i=2}^n b_i(x-\alpha_i \alpha^{-1}_1 w)\big)
b_1(x-\alpha_1^{-1}w) b_{n+2}(x) \,dx\,dw \,d\alpha
\\&=
\iiint
\beta^{d-n-1} \vsig(\beta_1^{-1}, \beta_1^{-1}\beta_2,\dots, \beta_1^{-1}\beta_n, \beta_1w)
\prod_{i=1}^n b_i(x-\beta_i v) b_{n+1}(x-w) b_{n+2}(x) dx\, dw\, d\beta
\end{align*}
where we have first changed variables $v=\alpha_1^{-1}u$, then interchanged the order of integration, and changed variables $\alpha_1=\beta_1^{-1}$, 
$\alpha_i= \beta_i\beta_1^{-1}$ for $i=2,\dots,n$. Hence if $\vp$ is the transposition interchanging $1$ and $n+1$ and leaving $2,\dots, n, n+2$ fixed then $\La^\vp[\vsig]=\La[\ell_\vp\vsig]$ with
\Be \label{elltranspos}
\ell_\vp\vsig (\alpha_1,\dots, \alpha_n,v)=
 \vsig(\alpha_1^{-1}, \alpha_1^{-1}\alpha_2,\dots, \alpha_1^{-1}\alpha_n, \alpha_1 v)\,.
 \Ee
Now if we define the inversion $J$, with respect  to the $\alpha_1$ variable, and multiplicative shears $M_{n-1}, \widetilde M_d$ by
\begin{align*}
Jg(\alpha_1,\dots,\alpha_n,v)&= \alpha_1^{-2} g(\alpha_1^{-1}, \alpha_2,\dots,\alpha_n,v)
\\
M_{n-1}g(\alpha_1,\dots,\alpha_n,v)
&= \alpha_1^{n-1}g(\alpha_1, \alpha_1\alpha_2,\dots,\alpha_1\alpha_n, v)
\\
\widetilde M_d g(\alpha_1,\dots,\alpha_n,v)&=
\alpha_1^d g(\alpha_1,\dots,\alpha_n, \alpha_1v)
\end{align*} 
then it is straightforward to check that the linear transformation $\ell_\vp$ in 
\eqref{elltranspos} can be factorized as 
\Be\label{factorization}\ell_\vp= J\circ \widetilde M_{d}\circ J\circ M_{n-1}\circ J
\,.\Ee
By Remark \ref{RmkAdjfBandsBSame} the $\cB_\eps(\bbR^n\times \bbR^d) $ and the $\fB_{\eps}(\bbR^{n+d})$ norms are equivalent with equivalence constants not depending on $n$.
By Proposition \ref{PropAdjInvts1}  we have 
$\|J g\|_{\cB_\eps'}\lc \|g\|_{\cB_\eps},$ and 
by Proposition \ref{PropAdjMultis1}  we have 
$\|M_{n-1} g\|_{\cB_\eps'} \lc n \|g\|_{\cB_\eps},$ and 
$\|M_{d} g\|_{\cB_\eps'} \lc  \|g\|_{\cB_\eps},$
for $\eps'<\eps/3$. Hence $\|\ell_\vp \vsig\|_{\cB_\delta}\lc n \|\vsig\|_{\cB_\eps}$,
at least when $\delta<3^{-5}\eps$.

Finally if $\vp$ is a general permutation than we can split 
$\vp=\vp_1\circ\vp_2\circ\vp_3\circ \vp_4$, each $\vp_i$ of the form in (i), (ii) or (iii), with at most two of the form in (iii). Hence we get 
$\La^\vp[\vsig]= \La[\ell_\vp\vsig]$ where $\|\ell_\vp\vsig\|_{\cB_{\delta}} \lc n^2
\|\vsig\|_{\cB_\eps}$, at least for $\delta<3^{-10} \eps$. 
We remark that if we avoid the factorization \eqref{factorization} and use 
the formula for $\ell_\vp$ directly we should get a better range for $\delta$ but this  will be irrelevant for our final boundedness results on  the forms $\La^\vp$. \qed

\subsection{Proof of  Propositions   \ref{PropAdjInvts1} and \ref{PropAdjMultis1}}
We first prove several preliminary lemmata, then give the proof of Proposition 
\ref{PropAdjInvts1} in \S\ref{Pf42} and 
the proof of Proposition  \ref{PropAdjMultis1} in \S\ref{Pf43}.
\subsubsection{ Preparatory Results}\label{prepresults}

We first  recall a standard fact about  Besov spaces
$B^\eps_{1,q}(\bbR)$; $1\le q\le \infty$. If  $0<\eps<1$ then  there  the characterizations
\begin{subequations}\label{besovchar}
\Be\label{besovcharq} \|f\|_{B^\eps_{1,q}}\approx \|f\|_1 +   
\Big (\int_0^1  \|f(\cdot+h)-f\|_1^q\frac{dh}{h^{1+\eps q}}\Big)^{1/q}, \quad 1\le q<\infty,
\Ee
and
\Be\label{besovcharinfty} \|f\|_{B^\eps_{1,\infty}}\approx \|f\|_1 +   
\sup_{0<h<1}  h^{-\eps} \|f(\cdot+h)-f\|_1\,.
\Ee
\end{subequations}
Moreover there are  the continous embeddings
\Be\label{besovemb} 
B^\eps_{1,q_1}\subset B^\eps_{1,q_2},
\quad q_1<q_2.
\Ee
For \eqref{besovchar} and \eqref{besovemb} we refer to \cite[\S V.5]{stein-si} or \cite{triebel}. As a corollary we get
\begin{lemma}\label{StandardBesov}
Let $0<\delta< \eps<1$.  Then for functions in $L^1(\bbR)$ 
then there are constants $c,C>0$ depending only on $\eps, \delta$ such that 
$$\begin{aligned} 
&c \|f\|_1 + c\int_{0<h<1} h^{-\delta}\|f(\cdot+h)-f\|_{1} \frac{dh}{h}
\\ \le &
 \|f\|_1 + \sup_{0<h<1} h^{-\eps}\|f(\cdot+h)-f\|_1
\\\le& C
\|f\|_1 + C\int_{0<h<1} h^{-\eps}\|f(\cdot+h)-f\|_1\frac{dh}{h}\,.
\end{aligned}
$$
 \end{lemma}

We let $e_i$, $i=1,\dots,N$, denote the standard basis vectors in $\bbR^N$ and let $e_i^\perp $ to be the orthogonal complement. For $g\in L^1(\bbR^N)$  and $w\in e_i^\perp$ define
\Be \label{piwg}
\pi_i^w g(s) = g(se_i+w);
\Ee this is defined as an $L^1(\bbR) $ function for almost every $w\in e_i^\perp$, and by Fubini $w\mapsto \int_\bbR|\pi_i^w g(s) |ds$ belongs to $L^1(e_i^\perp)$. Moreover if $g\in \fB_\eps(\bbR^N)$ for some 
$\eps>0$ then for almost every $w\in e_i^\perp$ the function
$h\mapsto \int_\bbR|\pi_i^w g(s+h)-\pi_i^w g(s)| ds$ is continuous.

\begin{lemma} \label{slicinglemma}
Let $0\le \delta<1$.  Then the following statements hold.

(i) $$
\|g\|_{\fB_\delta(\bbR^N)} \le \max_{i=1,\dots,n} \int_{e_i^\perp} \big\|\pi_i^w g\big\|_{\fB_\delta(\bbR)} dw\,.
$$

(ii) If $0<\delta<\eps\le 1$ then there exists $C=C(\eps,\delta)>0$ (not depending on $N$) such that for all $f\in \fB_\eps(\bbR^N)$
$$\max_{i=1,\dots,N} \int_{e_i^\perp} \big\|\pi_i^w g\big\|_{\fB_\delta(\bbR)} dw
\le C \|g\|_{\fB_\eps(\bbR^N)}\,.$$
\end{lemma}
\begin{proof} (i) follows immediately from the definitions of $\fB_\delta(\bbR)$ and $\fB_\delta(\bbR^N)$. 
For (ii) fix $i\in \{1,\dots, N\}$ and split 
$ \int_{e_i^\perp} \big\|\pi_i^w g\big\|_{\fB_\delta(\bbR)} dw=I+II$
where
\begin{align*}
I&=
\int_{e_i^\perp} \int (1+|s|)^\delta |g(se_i+w)| ds \,dw
\\
II&=\int_{e_i^\perp}\sup_{0\le h\le 1} |h|^{-\delta} \int |g((s+h)e_i+w) -g(se_i+w)| ds \,dw\,.
\end{align*}
It is immediate that $I\le \|g\|_{\fB_{\delta}(\bbR^N)}\le \|g\|_{\fB_{\eps}(\bbR^N)}$.
For the second term we use  Lemma
\ref{StandardBesov} to estimate
\begin{align*}
II &\le C_\delta 
\int_{e_i^\perp}\int_{0\le h\le 1} |h|^{-\delta} \int |g((s+h)e_i+w) -g(se_i+w)| ds\,\frac{dh}{h} \,dw
\\
&= C_\delta \int_{0}^1 h^{\eps-\delta} h^{-\eps} \int_{\bbR^N} 
 |g(x+ he_i) -g(x)| dx\,\frac{dh}{h} 
 \\
 &\le  C_\delta (\eps-\delta)^{-1} \sup_{0<h<1} |h|^{-\eps} \|g(\cdot+he_i)-g\|_{L^1(\bbR^N)} 
 \end{align*} and hence $II\le C(\eps,\delta)  \|g\|_{\fB_\eps(\bbR^N)}$.
 \end{proof}
  
  \begin{lemma}\label{largeslemma}
  Let $R\ge 1$ and let $\Omega^i_R=\{x\in \bbR^N: |x_i|\ge R\}$. Then
$$\int_{\Omega_R^i}|g(x)| dx
 \le R^{-\eps} \|g\|_{\fB_\eps(\bbR^N)}.\,$$
 
\end{lemma}
\begin{proof} This is immediate from
\[\int_{\Omega_R^i}|g(x)|\, dx \le R^{-\eps}\int (1+|x_i|)^\eps |g(x) |\,dx\,.\qedhere\]
\end{proof}

  The following lemma is a counterpart to Lemma \ref{largeslemma} which is used when integrating over sets whose projection to a coordinate axis has small 
  measure.
  It can be seen as a standard  application of a Sobolev embedding theorem for functions on the real line. For measurable $J\subset \bbR$ we denote by $|J|$ the Lebesgue measure.
  
  \begin{lemma}\label{Sobemb}
   Let $0<\eps\le 1$ and  $f\in \fB_\eps(\bbR^N)$, and let $0<\eps'<\eps$. 
   Let $E\subset \bbR^N $ and let $$\pr_i(E)=\{s\in  \bbR: se_i+w\in E \text{ for some }w\in e_i^\perp\}.$$
   Then
$$   \int_E |f(x)| dx\le C_{\eps,\eps'} |\pr_i(E)|^{\eps'} \|f\|\ci{\fB_\eps(\bbR^N)}\,.$$
Moreover for $i=1,\dots, N$, 
$\delta<\eps$,
$$\int_{e_i^\perp}\int_{|x_i|\le 1} |x_i|^{-\delta} |f(x)| dx \le C(\eps,\delta) \|f\|_{\fB_\eps(\bbR^N)} .$$
\end{lemma}
\begin{proof}  For $k\ge 0$ let  $E_k=\{x\in \bbR^N:2^{-k-1}\le |x_i|\le 2^{-k}\}$. The second inequality is a consequence of the first applied to the sets $E_k$.

To prove the first statement pick $p= (1-\eps')^{-1}>1$ so that $\eps'=1-p^{-1}$. By H\"older's inequality, 
$$\int_E |f(x)| dx\le  |\pr_i(E)|^{\eps'}  \int_{e_i^\perp} \Big(\int |f(se_i+w)|^p ds\Big)^{1/p}dw\,.$$
Let $\pi_i^wf(s)
=f(se_i+w)$.
Let $\phi\in \cS(\bbR)$, $\int\phi(s) ds=1$  such that the Fourier transform $\widehat\phi$ is supported in $\{|\xi|\le 1\}$. Let $\psi_k= 2^k\phi(2^k\cdot)-2^{k-1}\phi(2^{k-1}\cdot)$. Choose  $\phit\in \cS(\bbR)$ whose   Fourier transform is equal to $1$ on $\{|\xi|\le 2\}$ and let $\phit_k=2^k\phit(2^k\cdot)$.  
Then
$$\pi_i^wf
= \phit*\phi* \pi_i^wf+\sum_{k=1}^\infty \phit_k*\psi_k*\pi_i^wf$$ and thus, by Young's inequality, 
\begin{align*}
\|\pi_i^wf\|_{L^p(\bbR)} &\le \|\phit\|_{L^p(\bbR)} \|\phi*\pi_i^wf\|_{L^p(\bbR)}+ \sum_{k=1}^\infty \|\phit_k\|_{L^p(\bbR)} \|\psi_k*\pi_i^wf\|_{L^1(\bbR)}
\\&\lc \|\phi*\pi_i^wf\|_{L^p(\bbR)}  +\sum_{k=1}^\infty 2^{k(1-1/p)}  \|\psi_k*\pi_i^wf\|_{L^1(\bbR)}.
\end{align*}
Since $\int\psi_k(s) ds=0$ we have 
\begin{align*}
\big| \psi_k*\pi_i^wf(s)  \big|&= \Big|\int \psi_k(h) \big[
\pi_i^wf(s-h)-\pi_i^wf(s)
 \big] dh \Big|
\\ &\lc \int \frac{2^k}{(1+2^k|h|)^3}
\big|\pi_i^wf(s-h)-\pi_i^wf(s)\big| dh\,.
\end{align*}
Using this in the above expression we get  after integration in $w$
\begin{align*}
&\int_{e_i^\perp} \Big(\int |f(se_i+w)|^p ds\Big)^{1/p}dw \,\\&\lc 
\|f\|_1 + 
\sum_{k=1}^\infty 2^{k(1-\frac 1p)}
\int \frac{2^k}{(1+2^k|h|)^3}
\big|\pi_i^wf(s-h)-\pi_i^wf(s)\big| ds\, dw\, dh 
\\
&\lc 
\|f\|_1 + 
\sum_{k=1}^\infty 
\int_{|h|\le 1}  \frac{2^{k(2-\frac 1p)}|h|^{\eps}}{(1+2^k|h|)^3} dh\,
\sup_{|u|\le 1}
\frac{\|f(\cdot+u e_i)-f(\cdot)\|_1}{|u|^{\eps}}
+ \sum_{k=1}^\infty 
\int_{|h|\ge 1 } 
 \frac{2^{k(2-\frac 1p)}}{(1+2^k|h|)^3} dh\,
\|f\|_1.
\end{align*}
The last term is $\lc \sum_{k=1}^\infty2^{-k(1+1/p)} \|f\|_1\lc \|f\|_1$.
The middle term is $\lc \sum_{k=1}^\infty 2^{k(-\eps+1-1/p)} \|f\|_{\fB_\eps}$ and since $1-1/p=\eps'<\eps$ we  obtain the required bound.
\end{proof}

\subsubsection{Proof of Proposition \ref{PropAdjInvts1}}\label{Pf42}

 The main  lemma needed in the proof is an estimate for functions on the real line.
\begin{lemma}\label{inversionlemma}
For $g\in \fB_\eps(\bbR)$ let $Jg(s)=s^{-2} g(s^{-1})$. Then for $\delta<\eps/3$
$$\|Jg\|_{\fB_{\delta}(\bbR)} \le C(\eps, \delta) \|g\|_{\fB_\eps(\bbR)}.$$\end{lemma}

\begin{proof}
First observe that for $\eps'<\eps$
$$\int(1+|\sigma|)^{\eps'}|Jg(\sigma)|d\sigma=\int (1+|s|^{-1})^{\eps'} |g(s)| ds
\lc \|g\|_{\fB_\eps(\bbR)},$$
by Lemma \ref{Sobemb}. 
Thus, in light of Lemma \ref{StandardBesov}  it remains to prove that for $\rho\le 1/2$,
\Be\label{besovJg}
\int_\rho^{2\rho} \int |Jg(\sigma+h)-Jg(\sigma)| d\sigma \frac {dh}{h} \lc \rho^{\delta'} \|g\|_{\fB_\eps}, 
\Ee
for any $\eps'<\delta'<\eps/3$.
Choose any $\beta\in (\delta'/\eps, 1/3)$.
We have by changes of variables 
$$
\int_\rho^{2\rho} \int_{|\sigma|\le \rho^\beta} |Jg(\sigma+h)|+|Jg(\sigma)| d\sigma \frac {dh}{h} \lc \int_{|\sigma|\le 3\rho^\beta} |Jg(\sigma)| d\sigma
\le  \int_{|s|\ge \rho^{-\beta}/3 }|g(s)|ds\le \rho^{\beta\eps}\|g\|_{\fB_\eps}
$$
by Lemma \ref{largeslemma}. Also
$$
\int_\rho^{2\rho} \int_{|\sigma|\ge \rho^{-\beta}} |Jg(\sigma+h)|+|Jg(\sigma)| d\sigma \frac {dh}{h} \lc \int_{|\sigma|\le \rho^\beta /2} |Jg(\sigma)| d\sigma
\le  \int_{|s|\le 2\rho^{\beta} }|g(s)|ds\le \rho^{\beta\eps}\|g\|_{\fB_\eps},
$$
by Lemma \ref{Sobemb}. It remains to consider
\begin{align*}
\int_\rho^{2\rho} \int_{\rho^\beta}^{\rho^{-\beta}}|Jg(\sigma+h)-Jg(\sigma)| d\sigma \frac {dh}{h} &=
\int_\rho^{2\rho} \int_{\rho^\beta}^{\rho^{-\beta}}\big|\tfrac{s^{-2}}{(s^{-1}+h)^2} g\big(\tfrac{1}{s^{-1}+h}\big) -g(s) \big| ds\,\frac{dh}{h}
\\&=
\int_\rho^{2\rho} \int_{\rho^\beta}^{\rho^{-\beta}}
\big|\tfrac{1}{(1+hs)^2} g\big(\tfrac{s}{1+hs}\big) -g(s) \big| ds\,\frac{dh}{h};
\end{align*}
here we have performed the change of variable $s=\sigma^{-1}$.
We now interchange the order of integration and then change variables
$u= \frac{s}{1+hs}-s= -\frac{s^2h}{1+hs}$.
Observe that $du/dh= s^2(1+hs)^{-2}$ and thus
$\frac{|du|}{|u|} = |1+hs|^{-1} \frac{|dh|}{|h|}.$  Therefore for $|h|\approx\rho$ and $\rho^\beta<|s|\le \rho^{-\beta}$ we can replace $|dh|/|h| $ by $|du|/|u|$.
Also observe that $h=-u (su+s^2)^{-1}$ and $1+hs= s(u+s)^{-1}$.
Thus the last displayed expression can be written as
\[ \int_{\rho^\beta\le |s|\le \rho^{-\beta}} 
\int_{-\frac{2\rho s^2}{1+2\rho s}}^{-\frac{\rho s^2}{1+\rho s}}\big| 
\big (\tfrac {u+s}{s}\big)^2 g(s+u) -g(s)\big| \frac{du}{|u|} ds \le (I)+(II)
\]
where
\begin{align*}
(I)&:=\iint\limits_{\substack{\rho^\beta\le |s|\le \rho^{-\beta}\\
|u|\approx \rho s^2}} 
\big|
\tfrac {(u+s)^2}{s^2}-1\big| | g(s+u)| \frac{du}{|u|}\, ds
\\ (II)&:=\iint\limits_{\substack{\rho^\beta\le |s|\le \rho^{-\beta}\\
|u|\approx \rho s^2}} 
 \big|g(s+u) -g(s)\big| \frac{du}{|u|} \,ds\,.
\end{align*}
First estimate  
\begin{align*}(I) &\lc 
\int_{\rho^\beta\le |s|\le \rho^{-\beta}}\int_{|u|\approx \rho s^2}
|g(u+s)|  \frac{u^2+2|us|}{s^2}  \frac{du}{|u|} \,ds
\\&\lc 
  \int_0^{C\rho^{1-2\beta} }\int_{c\rho^\beta}^{C\rho^{-\beta} }(\rho+|s|^{-1})|g(s)|ds\,
 du \\&\lc \rho^{1-2\beta}\|g\|_1 +\sum_{\substack{k\ge 0 \\2^{-k}\ge c\rho^\beta}}
 2^k \int_{2^{-k}\le |s|\le 2^{1-k}} |g(s)| ds
 \end{align*} 
 and, since by Lemma \ref{Sobemb}  $\int_{|s|\le 2^{-k}} |g(s)|ds \lc 2^{-k\eps''}\|g\|_{\fB_\eps}$ for $\eps''<\eps$, we get  
 \begin{align*} 
 (I)& \lc \rho^{1-2\beta}  \Big(\|g\|_1+ \sum_{\substack{k\ge 0 \\2^{-k}\ge c\rho^\beta}}2^{k(1-\eps'')} \|g\|_{\fB_\eps}\Big)
 \lc \rho^{1-3\beta+\beta\eps''}\|g\|_{\fB_\eps}.
  \end{align*}
Finally,
\begin{align*}  (II)&\le \sum_{k: 2^{-k}\le C\rho^{1-2\beta}}
\int\limits_{2^{-k}\le |u|\le 2^{1-k}}\|g(\cdot +u)-g\|_1\frac{du}{|u|}
\\&\le \sum_{k: 2^{-k}\le C\rho^{1-2\beta}}2^{-k\eps} \|g\|_{B^\eps_{1,\infty}} \lc
\rho^{(1-2\beta)\eps}  \|g\|_{\fB_\eps} .
\end{align*}
Now collect the estimates and keep in mind that $\beta<1/3$ is chosen close to $1/3$. We may choose $\eps''$ above so that $3\delta'<\eps''<\eps$.  Then the asserted estimate \eqref{besovJg} follows, and the lemma is proved.
\end{proof}

\begin{proof}[Proof of Proposition \ref{PropAdjInvts1}, concluded]
Let $\pi_i^w g(s)=g(se_i+w)$ be as in \eqref{piwg}. 
We have
$$\|J_1\gamma\|_{\fB_\delta} \le \max_{1\le i\le N} \int_{e_i^\perp}
 \|\pi_i^w(J_1 g)
 \|_{\fB_\delta(\bbR)} dw.$$
By Lemma \ref{slicinglemma}
and a change of variable $w_1\mapsto w_1^{-1}$ we obtain for $2\le i\le n$, $\delta_1>\delta$,
$$
\int_{e_i^\perp} \| \pi_i^w(J_1 g)\|_{\fB_\delta(\bbR)} dw
=\int_{e_i^\perp} \|\pi_i^w g\|_{\fB_\delta(\bbR)} dw\lc
\|g\|_{\fB_{\delta_1}(\bbR^N)}.
$$
Let $3\delta<\tilde \eps<\eps$. For the main term with $i=1$ we use
 Lemma \ref{inversionlemma} and then Lemma  \ref{slicinglemma} to get
$$
\int_{e_1^\perp} \|\pi_i^w (J_1 g)\|_{\fB_\delta(\bbR)} dw=
\int_{e_1^\perp} \|J_1(\pi_i^w g)\|_{\fB_\delta(\bbR)} dw
\lc \int_{e_1^\perp} \|\pi_i^w g\|_{\fB_{\tilde\eps}(\bbR)} dw\lc
\|g\|_{\fB_{\eps}(\bbR^N)}.
$$
This concludes the proof of the proposition.
\end{proof}

\subsubsection{Proof of Proposition \ref{PropAdjMultis1}}\label{Pf43}
We now turn to Proposition \ref{PropAdjMultis1}.  
Fix $\eps>0$, $n\in \q\{1,\ldots, N\w\}$, $\gamma\in \fBtp{\eps}{\R^N}$ and recall  the definition $$M\gamma(s)= s_1^{n-1} \gamma\q(s_1, s_1s_2,\ldots, s_1s_n, s_{n+1},\ldots, s_N\w).$$ 
We separate the proof into three lemmata. The most straightforward one is
\begin{lemma}
Let $0<\eps<1$. For $\delta<\eps/2$, $i=1,\dots, N$,
$$\int \q(1+\q|s_i\w|\w)^{\delta} \q|M\ga\q(s\w)\w|\: ds\lesssim \fBN{\eps}{\gamma}.$$
\end{lemma}
\begin{proof}
Let $\eps'>0$ be a number, to be chosen later.
If $i=1$ or $n+1\leq i\leq N$, we have, by a change of variable,
\begin{equation*}
\int \q(1+\q|\sigma_i\w|\w)^{\eps'} \q|M\ga\q(\sigma\w)\w|\: d\sigma = \int \q(1+\q|s_i\w|\w)^{\eps'} \q|\gamma\q(s\w)\w|\: ds \lesssim \fBN{\eps}{\gamma},\quad \eps'\le \eps.
\end{equation*}

Let $2\leq i\leq n$.  We have by a change of variable 
\begin{equation*}
\int \q(1+\q|\sigma_i\w|\w)^{\eps'}\q|M\ga\q(\sigma\w)\w|\: d\sigma= \int \q(1+\q|\frac{s_i}{s_1}\w|\w)^{\eps'} \q|\gamma\q(s\w)\w|\: ds.
\end{equation*}
Let
$\Omega_1=\{s: |s_1|\ge 3\}$,
$\Omega_2=\{s: |s_1|\le 3, |s_i|\ge |s_1|^{-1} \}$,
$\Omega_3=\{s: |s_1|\le 3, |s_i|\le |s_1|^{-1} \}$, and bound the integrals over the three regions separately.
First, for $\eps'\leq \eps$,
\begin{equation*}
\int_{\Omega_1} \q(1+ \q|\frac{s_i}{s_1}\w|\w)^{\eps'} \q|\gamma\q(s\w)\w|\: ds \lesssim \int \q(1+ \q|s_i\w|\w)^{\eps'} \q|\gamma\q(s\w)\w|\: ds\leq \fBN{\eps}{\gamma}, 
\end{equation*}
Next, for $\eps'\le\eps/2$,
$$\int_{\Omega_2} \q(1+ \q|\frac{s_i}{s_1}\w|\w)^{\eps'} \q|\gamma\q(s\w)\w|\: ds \lesssim \int \q(1+ \q|s_i\w|\w)^{2\eps'} \q|\gamma\q(s\w)\w|\: ds\leq \fBN{\eps}{\gamma}.
$$
Finally, for the third term we use Lemma \ref{Sobemb} to estimate,
for $\eps'< \eps/2$,
$$\int_{\Omega_3} \q(1+ \q|\frac{s_i}{s_1}\w|\w)^{\eps'} \q|\gamma\q(s\w)\w|\: ds \lesssim_{\eps'}\int_{|s_1|\le 3} \q(1+ \q|s_1\w|^{-2\eps'}) \q|\gamma\q(s\w)\w|\: ds\leq \fBN{\eps}{\gamma}.  
$$
The asserted estimate follows.
\end{proof}

\begin{lemma}\label{none1differenceslem}
(i) For $n+1\leq i\leq N$, $\eps>0$
\begin{equation*}
\sup_{0<h\leq 1} h^{-\eps} \|M\ga(\cdot+he_i)-M\ga\|_1  \le \fBN{\eps}{\gamma}.
\end{equation*}

(ii) For $2\le i\le n$, $\delta<\eps/2$ 
\begin{equation*}
\sup_{0<h\leq 1} h^{-\delta} \|M\ga(\cdot+he_i)-M\ga\|_1  \lc \fBN{\eps}{\gamma}.
\end{equation*}
\end{lemma}

\begin{proof}  In the case  $n+1\leq i\leq N$ a change of variables shows, 
\begin{equation*}
 \int_{\bbR^N}\q|M\gamma\q(\sigma+he_i\w)-M\gamma\q(\sigma\w)\w|\: d\sigma =  \int_{\bbR^N} \q|\gamma\q(s+he_i\w)-\gamma\q(s\w)\w|\: ds, 
\end{equation*}
and the result follows.

Now consider the case $2\le i\le n$. By Lemma \ref{StandardBesov}
it suffices to show that for $\rho\le 1$
\begin{equation}\label{integralcond}
\int_{\rho}^{2\rho} \int_{\bbR^N}
\q|M\gamma\q(\sigma+he_i\w)-M\gamma\q(\sigma\w)\w|\: d\sigma\:\frac{dh}{h}\lesssim \rho^{\eps'}  \fBN{\eps}{\gamma}, \quad \eps'\le \eps/2.
\end{equation}
Our assumptions are symmetric in $s_2,\ldots, s_n$, and thus  it suffices to prove 
\eqref{integralcond}  for $i=2$.
The result is trivial for $10^{-2} \le \rho \le 1$, so we may assume  $\rho\le 10^{-2}$. 
In the inner integral we change variables, setting 
$(s_1,\dots, s_N)= (\si_1, \sigma_1\sigma_2,\dots, \sigma_1,\sigma_n,\si_{n+1},\dots, \sigma_N)$ and the left hand side of 
\eqref{integralcond} becomes
\begin{equation*}
\begin{split}
&\int_{\rho}^{2\rho} \int_{\bbR^N} 
  \q| \gamma\q(s_1, s_2+s_1h,s_1 s_3,\ldots,s_n, s_{n+1},\ldots, s_N\w) - \gamma\q(s\w) \w|\: ds\: \frac{dh}{h}\\
&=\iint\limits_{\substack{\rho\le h\le 2\rho  \\ \q|s_1\w|\geq \rho^{-\beta}}} 
+\iint\limits_{\substack{\rho\le h\le 2\rho  \\ \q|s_1\w|\leq \rho^{-\beta}}} 
=:(I)+(II)
\end{split}
\end{equation*}
where $\beta\in (0,1)$ is to be determined. We have the following estimate for the first term:
\begin{equation*}
\begin{split}
&(I)\leq 2\iint\limits_{\substack{\rho\le h\le 2\rho  \\ \q|s_1\w|\geq \rho^{-\beta}}} 
\q|\gamma\q(s\w)\w|\: ds\:\frac{dh}{h}
\lesssim \int\limits_{\q|s_1\w|\geq \rho^{-\beta}} \q|\gamma\q(s\w)\w|\: ds
\lesssim \rho^{\beta\eps} \int\q(1+\q|s_1\w|\w)^{\eps} \q|\gamma\q(s\w)\w|\: ds
\lc \rho^{\beta\eps} \fBN{\eps}{\gamma}.
\end{split}
\end{equation*}
For the term  $(II)$ we interchange the order of integration and put for fixed
$s_1$,   $\hh=s_1 h$ so that $d\hh/\hh = dh/h$.  Also, on the domain of integration of $(II)$,
we have $\q|\hh\w|\leq 2\rho^{1-\beta}$.   Thus we may estimate
\begin{equation*}
\begin{split}
&(II)\leq \int_{\q|\hh\w|\leq 2\rho^{1-\beta} }\|\gamma\q(\cdot+\hh e_2\w)-\gamma\q(\cdot\w) \|_1 \q|\hh\w|^{-1} d\hh
\leq \fBN{\eps}{\gamma}\int_0^{2\rho^{1-\beta}} \hh^{\eps-1}\: d\hh
\lesssim \rho^{\eps(1-\beta)} \fBN{\eps}{\gamma}.
\end{split}
\end{equation*}
If we choose $\beta=1/2$ then \eqref{integralcond} follows from the estimates for $(I)$ and $(II)$.
\end{proof}

\begin{rem} One can replace the application of Lemma \ref{StandardBesov} by a more careful argument to show that \eqref{integralcond} implies that the statement (ii) in the lemma holds even for the endpoint  $\delta=\eps/2$. However this is not important for the purposes of this paper.
\end{rem}

The  main technical estimate in the proof of Proposition \ref{PropAdjMultis1}
is an  analogue of  
Lemma \ref{none1differenceslem} for regularity in the first variable, given 
as Lemma \ref{e1differencelemma} below. We first give an auxiliary estimate for functions of two variables.

\begin{lemma} \label{auxiltwo}
Let $\beta<1/2$, $\eps'<\eps$. For $g\in \fB_\eps(\bbR^2)$, 
 and $0<\rho\le 1$,
 \begin{multline*}
 \iiint\limits_{\substack{ \rho^{\beta} \le |s_1|\le \rho^{-\beta} \\ \rho\le h\le 2\rho}}\Big|  \big(1+\tfrac{h}{s_1}\big) g(s_1+h, (1+\tfrac{h}{s_1})s_2)-
 g(s_1+h, s_2)\Big|\, ds_1 ds_2 \frac{dh}{h} 
 \\
 \le C(\beta, \eps') \big( \rho^{\eps'\beta}+\rho^{1-2\beta}\big)\|g\|_{\fB_\eps(\bbR^2)}\,.
 \end{multline*}
\end{lemma}

\begin{proof} We may assume that $\rho\le 10^{-2/\beta}$, since otherwise the bound is trivial. We wish to discard the contributions of the integral where $|s_2|\le \rho^\beta$ or $|s_2|\ge \rho^{-\beta}$. We estimate the  left hand side by $A+ I_1+I_2 +II_1+ II_2$ 
where
\begin{align*}
A&\,=\,  \iiint\limits_{\substack{ \rho^{\beta} \le |s_1|, s_2\le \rho^{-\beta} \\ \rho\le h\le 2\rho}}\Big|  \big(1+\tfrac{h}{s_1}\big) g(s_1+h, (1+\tfrac{h}{s_1})s_2)-
 g(s_1+h, s_2)\Big| \,ds_1 ds_2 \frac{dh}{h} \,,
 \\
 I_1+II_1&\,=\,
 \iiint\limits_{\substack{ \rho^{\beta} \le |s_1|\le \rho^{-\beta} \\ 
 |s_2| \le \rho^{\beta}  \\ 
 \rho\le h\le 2\rho}}
 \,+\,
  \iiint\limits_{\substack{ \rho^{\beta} \le |s_1|\le \rho^{-\beta} \\ 
 |s_2| \ge \rho^{-\beta}  \\ 
 \rho\le h\le 2\rho}}
  \Big|  \big(1+\tfrac{h}{s_1}\big) g(s_1+h, (1+\tfrac{h}{s_1})s_2)\Big| \,ds_1 ds_2 \frac{dh}{h} \,,
\\
I_2+II_2&\,=\,
\iiint\limits_{\substack{ \rho^{\beta} \le |s_1|\le \rho^{-\beta} \\ 
 |s_2| \le \rho^{\beta}  \\ 
 \rho\le h\le 2\rho}} +
 \iiint\limits_{\substack{ \rho^{\beta} \le |s_1|\le \rho^{-\beta} \\ 
 |s_2| \ge \rho^{-\beta}  \\ 
 \rho\le h\le 2\rho}} 
 |  g(s_1+h, s_2)| \,ds_1 ds_2 \frac{dh}{h} \,.
  \end{align*}
To bound $I_1$ we change (for fixed $h$, $s_1$) variables  as $\sigma_2=(1+h/s_1) s_2$ and observe that $(1+h/s_1)\approx 1$. Thus the $\sigma_2$ integration is extended over $\sigma_2\lc \rho^\beta$, and we may apply 
Lemma \ref{Sobemb}. A similar argument applies to $I_2$, and we get
$$I_1+I_2 \lc \rho^{\beta\eps'} \|g\|_{\fB_\eps(\bbR^2)}.$$

The same argument applies to the terms $II_1$, $II_2$, 
with the $\sigma_2$ integration now extended over $|\sigma_2|\ge \rho^{-\beta}-2\rho \ge c \rho^{-\beta}$ for $c>0$. Now we apply Lemma  
\ref{largeslemma} instead and the result is 
$$II_1+II_2 \lc \rho^{\beta\eps} \|g\|_{\fB_\eps(\bbR^2)}.$$

We now consider the term $A$ and  estimate 
$A\le III +IV$ where
\begin{align*}
III&\,=\,\iiint\limits_{\substack{ \rho^{\beta} \le |s_1|, |s_2|\le \rho^{-\beta} \\ \rho\le h\le 2\rho}}
 \big|1+\tfrac{h}{s_1}\big|
 \big |g(s_1+h, (1+\tfrac{h}{s_1})s_2)-g(s_1+h, s_2)\big|\,ds_1 ds_2 \frac{dh}{h} \,,
 \\
 IV&\,=\,
 \iiint\limits_{\substack{ \rho^{\beta} \le |s_1|, |s_2|\le \rho^{-\beta} \\ \rho\le h\le 2\rho}}\tfrac{|h|}{|s_1|}|g(s_1+h, s_2)| \,ds_1 ds_2 \frac{dh}{h} \,.
 \end{align*} 
  Since $h\approx \rho$ and $|s_1|\gc \rho^{\beta}$ in the domain of integration we immediately get
 $$IV \lc \rho^{1-\beta} \|g\|_{L^1(\bbR^2)}\,.$$
  In the estimation of $III$ we may ignore the factor $1+h/s_1$ which is $O(1)$.
 We make the change of variable $\sigma_1=s_1+h$ which does not substantially change the domain of integration since $\tfrac 12\rho^\beta\le |\sigma_1|\le 2\rho^{-\beta}$ for the ranges of $\rho$ we consider here.
 We see that
 $$III\lc
 \iiint\limits_{\substack{ \frac 12\rho^{\beta} \le |\si_1|, |s_2|\le 2\rho^{-\beta} \\ \rho\le h\le 2\rho}} \big|
 g(\si_1, (1+\tfrac{h}{\si_1-h})s_2) -g(\si_1, s_2)\big|\, d\si_1 ds_2 \frac{dh}{h} 
 $$
 We now interchange the order of integration, and then, for fixed $\si_1, s_2$ change variables  $u =u(h)= \tfrac{hs_2}{\si_1-h}$. Then observe that
 $$\frac{\partial u}{\partial h}= \frac{\sigma_1s_2}{(\sigma_1-h)^2}, \quad
 \frac{du}{u}=\frac{\sigma_1}{\sigma_1-h}\,\frac{dh}{h};$$ 
 moreover the range of $|u|$ is contained in  $[\frac 14 \rho^{1+2\beta}, 4\rho^{1-2\beta}]$. Since $|du|/|u|\approx |dh|/|h|$ we get 
 the estimate
 \begin{align*} III&\lc \sum_{2^{-k-1} \le 4\rho^{1-2\beta}}\int_{2^{-k-1}}^{2^{-k}}
 \iint |
 g(\si_1, s_2+u) -g(\si_1, s_2)\big|\,d\si_1 ds_2 \frac{du}{|u|} 
 \\&\lc \sum_{2^{-k-1} \le 4\rho^{1-2\beta}} 2^{-k\eps} \|g\|_{\fB_\eps(\bbR^2)} 
 \lc \rho^{1-2\beta}\|g\|_{\fB_\eps(\bbR^2)}\,. 
 \end{align*}
 We collect the estimates and obtain the desired bound.\end{proof}

\begin{lemma}\label{e1differencelemma}
For $0<\eps\le 1$, $\delta<\eps/3$,
\begin{equation*}
\sup_{0<h\le 1} h^{-\delta} \|M\ga(\cdot+he_1)-M\ga\|_1 \lc
n \fBN{\eps}{\gamma}.
\end{equation*}
\end{lemma}
\begin{proof} Let $\tilde \eps<\eps$, $\delta_1>\delta$ be such that
$\delta<\delta_1<\tilde \eps/3$. By Lemma \ref{StandardBesov} it suffices to show for $\rho\le 1$ the inequality
\begin{equation}\label{rhomodification}
\int_\rho^{2\rho}
\|M\ga(\cdot+he_1)-M\ga\|_1 \frac{dh}{h}\lc\rho^{\delta_1}
n \fBN{\eps}{\gamma}.
\end{equation}
We let $\beta<1/2$ to be chosen later; a suitable  choice will be $\beta \in (\delta_1/\tilde \eps, 1/3)$.
We may assume $\rho\le 10^{-2/\beta}$ since otherwise the result is obvious.
We first discard the contributions of the integral for $|s_1|\le \rho^\beta$ or $|s_1|\ge \rho^{-\beta}$.
We estimate
$$
\int_\rho^{2\rho}
\|M\ga(\cdot+he_1)-M\ga\|_1 \frac{dh}{h}
\lc\rho^{\delta_1}\le (A)+ (I_1)+(I_2)+ (II_1)+(II_2)
$$
where
\begin{align*}
(A)&= 
\int_\rho^{2\rho}\int_{s:\rho^{\beta}\le |s_1|\le \rho^{-\beta}}
|M\ga(s+he_1)-M\ga(s)| \,ds  \frac{dh}{h}\,,
\\
(I_1)+(I_2)&= \int_\rho^{2\rho}\int_{s:|s_1|\le \rho^{\beta}}
|M\ga(s+he_1)|\, ds  \frac{dh}{h}+
 \int_\rho^{2\rho}\int_{s:|s_1|\le \rho^{\beta}}
|M\ga(s)| \,ds  \frac{dh}{h}\,,
\\
(II_1)+(II_2)&= \int_\rho^{2\rho}\int_{s:|s_1|\ge \rho^{-\beta}}
|M\ga(s+he_1)| \,ds  \frac{dh}{h}
+ \int_\rho^{2\rho}\int_{s:|s_1|\ge \rho^{-\beta}}
|M\ga(s)| \,ds  \frac{dh}{h}\,.
\end{align*}
We make a change of variable 
$\sigma=(s_1+h, (s_1+h)s_2, \dots, (s_1+h)s_n, s_{n+1}, \dots, s_N)$ and estimate
$$(I_1)\le 
\int_\rho^{2\rho}\int_{\si:|\si_1|\le \rho^{\beta}+2\rho}
|\ga(\sigma)| \,d\si   \frac{dh}{h}\lc \rho^{\beta\eps}\|\ga\|_{\fB_\eps(\bbR^N)}.
$$
where we have used Lemma \ref{Sobemb}.
Similarly 
$$(II_1)\le  \int_\rho^{2\rho}
\int_{\si:|\si_1|\ge \rho^{-\beta}-2\rho}
|\ga(\sigma)|\, d\si   \frac{dh}{h}\lc \rho^{\beta\eps}\|\ga\|_{\fB_\eps(\bbR^N)}.
$$
by Lemma \ref{largeslemma} and the estimate $2\rho \le \frac 12\rho^{-\beta}$ which holds in the range of $\rho$ under consideration.
The bound $(I_2)+ (II_2)\lc \rho^{\beta\eps}\|\ga\|_{\fB_\eps(\bbR^N)}$ follows in the same way. 

It thus remains to estimate $(A)$.
We change variables 
 and write 
\begin{align*}(A)&=
\int_\rho^{2\rho}\int_{s:\rho^{\beta}\le |s_1|\le \rho^{-\beta}}
\big|(s_1+h)^{n-1}\gamma(s_1+h, (s_1+h)s_2, \dots, (s_1+h)s_n, s_{n+1},\dots, s_N)
\\&\quad\quad\quad\quad\quad\quad\quad- s_1^{n-1} \gamma\q(s_1, s_1s_2,\ldots, s_1 s_n, s_{n+1},\ldots, s_N\w)\big|\: ds\: \frac{dh}{h}
\\
&=\int_\rho^{2\rho}\int_{s:\rho^{\beta}\le |s_1|\le \rho^{-\beta}}
\big|\q\big(1+\tfrac{h}{s_1}\big)^{n-1}
\gamma\q(s_1+h, (1+\tfrac{h}{s_1})s_2,\ldots, (1+\tfrac{h}{s_1})s_n, s_{n+1},\ldots, s_N\w) 
\\&\quad\quad\quad\quad\quad\quad\quad -  \gamma\q(s_1, s_2,\ldots, s_n, s_{n+1},\ldots, s_N\w)\big|\: ds\: \frac{dh}{h}.
\end{align*}
We split the integrand as a sum of $n$ differences $\Delta_k(s,h)$, $k=0,\dots, n-1$, 
where
$$ \Delta_0(s,h)= \ga(s+he_1)-\gamma(s)$$ and, for $k=1,\dots, n-1$,
\begin{align*}
\Delta_k(s,h)=&  
\big(1+\tfrac{h}{s_1}\big)^{k}
\gamma\q(s_1+h, (1+\tfrac{h}{s_1})s_2,\ldots, (1+\tfrac{h}{s_1})s_k, (1+\tfrac{h}{s_1})s_{k+1},s_{k+2}\ldots, s_N\w) \\& -
\big(1+\tfrac{h}{s_1}\big)^{k-1} 
\gamma\q(s_1+h, (1+\tfrac{h}{s_1})s_2,\ldots, (1+\tfrac{h}{s_1})s_k, s_{k+1}, \ldots, s_N\w) .
\end{align*}
Then $(A)\le \sum_{k=0}^{n-1} (A_k)$ where
$$(A_k)=\int_\rho^{2\rho}\int_{s:\rho^{\beta}\le |s_1|\le \rho^{-\beta}}|\Delta_k(s,h) |\, ds \frac{dh}{h}.
$$
It is immediate that
$$(A_0) \lc \rho^\eps \|\gamma\|_{\fB_\eps}\,.$$
For the estimation of $(A_k)$ we make  a change of variable 
in  the $s_i$ variables where $2\le i\le k$;  
this replaces $(1+h/s_1)s_i$ by $s_i$ 
(i.e. there is no change of variable if $k=1$).
This gives, for $1\le k\le n-1$,
\begin{multline*}
(A_k)
=\int_\rho^{2\rho}\int_{s:\rho^{\beta}\le |s_1|\le \rho^{-\beta}}\Big|
\big(1+\tfrac{h}{s_1}\big)
\gamma\q(s_1+h, s_2,\ldots, s_k, (1+\tfrac{h}{s_1})s_{k+1},s_{k+2}\ldots, s_N\w) \\
-\gamma\q(s_1+h, s_2,\ldots, s_k, s_{k+1}, \ldots, s_N\w) \Big |\,
ds \frac{dh}{h}\,.
\end{multline*}
By symmetry considerations we may assume $k=1$.
We may now freeze the $s_3,\dots, s_N$-variables,  apply the auxiliary Lemma 
\ref{auxiltwo} for functions of $(s_1,s_2)$ and obtain for $\eps'<\tilde\eps$
$$
(A_k) \lc 
\big(\rho^{\eps'\beta}+\rho^{1-2\beta}\big)
\int \cdots\int\big\| g(\cdot, \cdot, s_3,\dots, s_N)\big\|_{\fB_{\tilde\eps}(\bbR^2)} ds_3\cdots ds_N.
$$
Since $\tilde\eps<\eps$ this  also implies, by Lemma \ref{slicinglemma},
$$
(A_k) \lc 
\big(\rho^{\eps'\beta}+\rho^{1-2\beta}\big)
\|g\|_{\fB_{\eps}(\bbR^N)}.
$$
We collect estimates we see that the quantity on the left hand side 
of \eqref{rhomodification}  is estimated by
$$C(\beta,\eps',\eps) n\big ( \rho^{\beta\eps'}+\rho^{1-2\beta}\big)\|f\|_{\fB_\eps(\bbR^N)}\,$$
and with the correct choice of $\eps'\in (3\delta,\eps)$ and then $\beta\in
(\delta/\eps', 1/3)$ we see that \eqref{rhomodification} is established.
\end{proof}

\subsection{A decomposition lemma}

Later in the paper, we will need a decomposition result for $\sBtp{\eps}{\R^n\times \R^d}$, which we present here.
\begin{lemma}\label{ThmDecompVsig}
Fix $0<\eps<1$ and $0<\delta<\eps/2$. 
If  $\vsig\in \sBtp{\eps}{\R^n\times \R^d}$.  Then there are 
$\vsig_m \in \sBtp{\delta}{\R^n\times \R^d}$, $m\in \bbN$, 
with $\supp{\vsig_m}\subseteq \q\{\q(\alpha,v\w) : \q|v\w|\leq 1/4\w\}$ and
\begin{equation*}
\vsig=\sum_{m\ge 0}
 \dil{\vsig_m}{2^{-m}},
\end{equation*}
such that $$
\|\vsig_m\|_{\cB_\delta}\lesssim 2^{-m(\eps-2\delta)} \sBN{\eps}{\vsig}.$$
\end{lemma}
\begin{proof}
Let $\eta_0\in C^\infty_0$ be supported in $\{|x|\le 1/4\}$ such that 
 with $0\leq \eta_0\leq 1$ and $\eta_0(x)=1$ for $|x|\le 1/8$.
 Set $\eta_1\q(v\w) = \eta_0\q(v\w)-\eta_0\q(2v\w)$, so that $0\leq |\eta_1|\leq 1$, $\supp{\eta_1}\subseteq \q\{\frac{1}{16}\leq \q|v\w|\leq \frac{1}{4}\w\}$
and $1=\eta_0\q(v\w) +\sum_{m\geq 1} \eta_1\q(2^{-m} v\w)$.
For $m\in \N$, define
\begin{equation*}
\vsig_m(v)=
\begin{cases}
\eta_0\q(v\w) \vsig\q(\alpha,v\w) &\text{if }m=0,
\\ \eta_1\q(v\w) 2^{md}\vsig\q(\alpha,2^{m}v\w) &\text{if }m\geq 1.
\end{cases}
\end{equation*}
Then
${\vsig_m}(x)=0$ for $|x|\ge 1/4$ and
$\vsig=\sum_{m\ge 0}\dil{\vsig_m}{2^{-m}}$. Clearly 
 $\sBN{\eps}{\vsig_0}\lesssim \sBN{\eps}{\vsig}.$  
 It remains to bound $\|\vsig_m\|_{\cB_\delta}$ for $m\ge 1$.
 
 We  show
\begin{align} \label{weightestpieces}
 &\iint (1+|\alpha_i|)^\delta|\vsig_m(\alpha,v)|\,d\alpha\, dv
 + \iint (1+|v|)^\delta|\vsig_m(\alpha,v)|\,d\alpha\, dv
 \lc 2^{-m(\eps-\delta)}\|\vsig\|_{\cB_\eps},
 \\
 \label{alpharegdelta}
 &\sup_{|h|\le 1}|h|^{-\delta} \iint\q|\vsig_{m}\q(\alpha+he_i,v\w)-\vsig_{m}\q(\alpha,v\w)\w|\: d\alpha \: dv 
 \lc 2^{-m(\eps-\delta)}\|\vsig\|_{\cB_\eps}\,.
 \end{align}
 We change variables and see that the left hand side  of \eqref{weightestpieces}
 is bounded by
 \[
  \iint (1+|\alpha_i|)^\delta|\vsig(\alpha,v)| |\eta_1(2^{-m} v)| \,
  d\alpha\, dv
 + \iint (1+|2^{-m}v|)^\delta|\vsig(\alpha,v)| |\eta_1(2^{-m} v)| \,d\alpha\, dv\,.
 \]
  We estimate
 \begin{align*}
 &\iint\limits_{\substack{|\alpha_i|\le 2^m\\|v|\approx 2^m}}
 (1+|\alpha|)^\delta|\vsig(\alpha,v)|d\alpha dv\lc 
 2^{-m(\eps-\delta)}\iint (1+|v|)^\eps|\vsig(\alpha,v)| d\alpha dv
  \lc 2^{-m(\eps-\delta)}\|\vsig\|_{\cB_\eps},
 \\
 &\iint\limits_{\substack{|\alpha_i|\ge 2^m\\|v|\approx 2^m}}
 (1+|\alpha|)^\delta|\vsig(\alpha,v)|d\alpha dv\lc 
 2^{-m(\eps-\delta)}\iint (1+|\alpha_i|)^\eps|\vsig(\alpha,v)| d\alpha dv
  \lc 2^{-m(\eps-\delta)}\|\vsig\|_{\cB_\eps},
 \\
 &\iint\limits_{\substack{|v|\approx 2^m}}
 (1+2^{-m}|v|)^\delta|\vsig(\alpha,v)|d\alpha dv\lc 
 2^{-m(\eps-\delta)}
 \iint (1+|v|)^\eps|\vsig(\alpha,v)| d\alpha dv
 \lc 2^{-m(\eps-\delta)}\|\vsig\|_{\cB_\eps} ,
 \end{align*}
 and \eqref{weightestpieces} follows.
 
 Next, we consider, for $|h|\le 1$, the expression 
  \[
  \iint\q|\vsig_{m}\q(\alpha+he_i,v\w)-\vsig_{m}\q(\alpha,v\w)\w|\: d\alpha \: dv  
 \lesssim \iint |\eta_1(2^{-m}v)| \q|\vsig\q(\alpha+ he_i, v\w)- \vsig\q(\alpha,v\w)\w|\: d\alpha\: dv
 \] and distinguish the cases $2^m|h|\le1$ and $2^m|h|\ge1$.
 If $2^{m}\ge |h|^{-1}$ then we estimate
 $$
\iint_{|v|\approx 2^m} \q|\vsig\q(\alpha+ he_i, v\w)- \vsig\q(\alpha,v\w)\w|\: d\alpha\: dv\lc  2^{-m\eps} \int(1+|v|)^\eps |\vsig(\alpha,v)| \,d\alpha dv \lc |h|^\delta 2^{-m(\eps-\delta)} \|f\|_{\cB_\eps}
$$
and if 
 $2^{m}\le |h|^{-1}$,
 $$
\iint_{|v|\approx 2^m} \q|\vsig\q(\alpha+ he_i, v\w)- \vsig\q(\alpha,v\w)\w|\: d\alpha\: dv\lc  |h|^\eps\|\vsig\|_{\cB_\eps} 
\lc |h|^\delta 2^{-m(\eps-\delta)} \|f\|_{\cB_\eps}\,.
$$
Now \eqref{alpharegdelta} follows. Note that so far we have only used $\delta<\eps$.

For our last estimate we need $\delta<\eps/2$, and we need to  show
\Be \label{vregpieces}
\iint|\vsig_m(\alpha,v+h)-\vsig_m(\alpha, v)|d\alpha dv \lc |h|^\delta 2^{-m(\eps-2\delta)}\|\vsig\|_{\cB_\eps}.
\Ee
The left hand side is estimated by  $(I)+(II)$  where
\begin{align*}
(I)&=\iint \q|\eta\q(v+h\w)- \eta\q(v\w)\w|2^{md} 
\q|\vsig\q(\alpha, 2^m\q(v+h\w)\w)\w|\: d\alpha \: dv\,,
\\
(II)&= \iint  \q|\eta\q(v\w)\w| 2^{md}\q| \vsig\q(\alpha, 2^{m}\q(v+h\w)\w)-\vsig\q(\alpha, 2^m v\w) \w|\:d\alpha\: dv\,.
\end{align*}
Note that $\q|\eta\q(v+h\w)-\eta\q(v\w)\w|\lesssim \chi_{\q\{ \frac{1}{32}\leq \q|v\w|\leq \frac{1}{2} \w\}} \q|h\w|$ and so the first term is estimated as
\begin{equation*}
\begin{split}
(I)&\lesssim
 \q|h\w| \iint\limits_{\frac{1}{64}\leq \q|v\w|\leq 1}  2^{md} \q|\vsig\q(\alpha, 2^m v\w)\w|\: d\alpha\: dv=\q|h\w| \iint\limits_{2^{m-8}\leq \q|v\w|\leq 2^m} \q|\vsig\q(\alpha,v\w)\w|\: d\alpha\: dv
\\&
\lesssim 2^{-m\eps}\q|h\w| \iint \q(1+\q|v\w|\w)^{\eps} \q|\vsig\q(\alpha,v\w)\w|\: d\alpha\: dv \lesssim \q|h\w|2^{-m\eps} 
 \sBN{\eps}{\vsig}
\end{split}
\end{equation*}
which is a  better bound  than the one in \eqref{vregpieces}. More substantial is the estimate for $(II)$. Here we first consider the case 
$|h|\ge 2^{-2m}$
and bound
\begin{equation*}
\begin{split}
(II) &\lesssim \iint\limits_{2^{-4}\leq \q|v\w|\leq 2^{-2}}  2^{md} \q|\vsig\q(\alpha, 2^{m}\q(v+h\w)\w)- \vsig\q(\alpha,2^m v\w)\w|\: d\alpha \: dv
\leq 2\iint\limits_{2^{-8}\leq \q|v\w|\leq 2^{-1}}  2^{md} \q|\vsig\q(\alpha, 2^{m}v\w)\w|\: d\alpha\: dv
\\&\lesssim \iint\limits_{2^{m-8}\leq \q|v\w|\leq 2^{m-1}}  \q|\vsig\q(\alpha,v\w)\w|\: d\alpha\: dv
\lesssim 2^{-m\eps} \iint \q(1+\q|v\w|\w)^{\eps} \q|\vsig\q(\alpha,v\w)\w|\: d\alpha\: dv
\\&\lesssim 2^{-m\eps}  \sBN{\eps}{\vsig} \lc 
|h|^\delta  2^{-m(\eps-2\delta)} \| \vsig\|_{\cB_\eps}.
\end{split}
\end{equation*}
Finally for the case 
$|h|\le 2^{-2m}$ we get
\begin{equation*}
\begin{split}
&(II)\lesssim  \iint \q|\vsig\q(\alpha, v+ 2^m h\w)- \vsig\q(\alpha,v\w)\w|\: d\alpha\: dv \lesssim  \q(2^{m} \q|h\w|\w)^{\eps} \sBN{\eps}{\vsig} \lesssim 
|h|^\delta  2^{-m(\eps-2\delta)}
\|\vsig\|_{\cB_\eps}.
\end{split}
\end{equation*}
This yields  \eqref{vregpieces} and the proof is complete.
\end{proof}

\subsection{Invariance properties}\label{invariance}
We state certain identities concerning  the behavior of our multilinear forms with respect to 
scalings and translations.  These will be used repeatedly. The straightforward proofs are omitted.

\begin{lemma}\label{scalinglemma}
Let $\vsig\in L^1(\bbR^n\times \bbR^d)$, and $\vsigj(\alpha,\cdot)=2^{jd} \vsig(\alpha, 2^j\cdot)$. Let $b_i\in L^{p_i}(\bbR^d)$, for $i=1,\dots, n+2$. Then
\begin{enumerate}[(i)]

\item Let $\tau_h f= f(\cdot -h)$. Then
\[
\La[\vsig](\tau_h b_1,\dots, \tau_h b_{n+2}) = 
\La[\vsig]( b_1, \dots, b_{n+2})\,.
\]

\item 
\[
\La[\vsigj](b_1,\dots, b_{n+2}) = 2^{-jd}
\La[\vsig]( b_1(2^{-j}\cdot), \dots, b_{n+2}(2^{-j}\cdot))\,.
\]

\item \label{kjdef}
\[
\La[\vsigj](b_1,\dots, b_{n+2}) = \int  b_{n+2}(x)  \int 2^{jd} k_j(2^j x,2^jy) b_{n+1}(y) \, dy\, dx\]
where
\[ k_j(x,y)= \int \vsig(\alpha, x-y) \prod_{i=1}^n b_i(2^{-j}(x-\alpha_i(x-y))) \,d\alpha.\]

\item
If $g_i= 2^{-jd/p_i}b_{i}(2^{-j}\cdot)$  then $\|g_i\|_{p_i}=\|b_i\|_{p_i}$, and 
\[
\La[\vsigj](b_1,\dots, b_{n+2}) =
\La[\vsig](g_1,\dots, g_{n+2}) \, \text{ if
 }\sum_{i=1}^{n+2} p_i^{-1}=1\,.
\]

\item Let $\kappa_1, \dots, \kappa_{n+2}$ be bounded Borel measures 
and $\kappa_i^{(t)}= t^d\kappa(t\cdot)$.
Set  $\tilde b_i (x)= b_i(2^{-j} x)$.
Then 
\[\La[\vsigj](\kappa_1*b_1,\dots,\kappa_{n+2}* b_{n+2}) =2^{-jd}
\La[\vsig](\kappa_1^{(2^{-j})}*\tilde b_1,\dots,\kappa_{n+2}^{(2^{-j})}*\tilde  b_{n+2}) .
\]

\item
\[
\La[\vsigj](\kappa_1*b_1,\dots, \kappa_{n+2}*b_{n+2}) = \int  2^{jd} \widetilde k_j(2^j x,2^jy) b_{n+1}(y) \, b_{n+2} (x)\,dx\]
where
\[ \widetilde k_j(x,y)= \iint \vsig(\alpha, w-z) \prod_{i=1}^n \kappa_i^{(2^{-j})}* [b_i(2^{-j}\cdot)]
(w-\alpha_i(w-z))
d\ka_{n+2}((x- w) d\ka_{n+1}(z-y) \,.\]

\end{enumerate}
\end{lemma}

\subsection{The role of projective space, revisited}\label{SectionRoleOfRPnSecond}

A particular special case of Theorems 
\ref{ThmOpResAdjoints} and  \ref{ThmOpResBoundGen}
involve the case
when $$K\q(\alpha,v\w)= \gamma_0\q(\alpha\w) K_0\q(v\w),$$ $K_0$ is a classical Calder\'on-Zygmund convolution kernel
which is {\it homogeneous} of degree $-d$, smooth away from $0$, and $\gamma_0\in \fBtp{\epsilon}{\R^n}$ for some $\epsilon>0$.
We saw in Section \ref{SectionRoleofRPnFirst} that such operators would be closed under adjoints
provided we could see the space of $\gamma_0$ as a space of densities on $\RPn$ in an appropriate way.
Indeed, this is the case, and this section is devoted to discussing that fact.  These results are not used
in the sequel, and are intended as motivation for our main results.

For a measurable function $f:\R^n\rightarrow \bbC$, and $0<\epsilon<1$, we set
\begin{equation*}
\BesN{f}{\R^n}:=\LpN{1}{f} +\max_{i=1,\dots,n}  \sup_{0<h_i\leq 1} |h_i|^{-\epsilon}\int\q|f\q(s+h_ie_i\w)-f\q(s\w)\w|\: ds,
\end{equation*}
where $e_1,\ldots, e_n$ is the standard basis for $\R^n$.

Let $M$ be a compact manifold of dimension $n$, without boundary.  Let $\mu$ be a measure on $M$.
Take a finite open cover $V_1,\ldots, V_L$ of $M$ such that each $V_j$ is diffeomorphic to 
$B^n\q(1\w)$--the open ball of radius $1$ in $\R^n$.
Let $\Phi_j:B^n\q(1\w)\rightarrow V_j$ be this diffeomorphism
 and let $\phi_1,\ldots, \phi_L$
be a $C^\infty$ partition of unity subordinate to this cover. 
On each neighborhood $V_j$, let $\Phi_j^{\#} \mu$ denote the pull back
of $\mu$ via $\Phi_j$.  We suppose $\Phi_j^{\#}\mu$ is absolutely
continuous with respect to Lebesgue measure on $B^n\q(1\w)$ and
we write $d\Phi_j^{\#}\mu=:\gamma_j\q(x\w) \: dx$ where $dx$ denotes Lebesgue measure.

\begin{rmk}
$\gamma_j$ is called a density, because of the way it transforms under diffeomorphisms. 
\end{rmk}

\begin{defn}
For $0<\epsilon<1$ we define $B^{\epsilon}_{1,\infty}\q(M\w)$ to be the space of those
measures $\mu$ such that the following norm is finite:
\begin{equation*}
\BesN{\mu}{M}:=\sum_{j=1}^L \BesN{\phi_j\circ\Phi_j\q(\cdot\w) \gamma_j\q(\cdot\w) }{\R^n}.
\end{equation*}
\end{defn}

\begin{rmk}
The norm $\BesN{\cdot}{M}$ depends on various choices we made:  the finite open cover, the diffeomorphisms $\Phi_j$, and the partition
of unity $\phi_j$.  However, the {\it equivalence class} of the norm $\BesN{\cdot}{M}$ does not depend on any of these choices,
and therefore the Banach space $B^{\epsilon}_{1,\infty}\q(M\w)$ does not depend on any of these choices.
\end{rmk}

We now turn to the case $M=\RPn$.  Given a measure $\mu\in B^{\epsilon}_{1,\infty}\q(\RPn\w)$, we consider the map
taking $\R^n\hookrightarrow \RPn$ induced by the map $\R^n\hookrightarrow \R^{n+1}$ given
by $\q(x_1,\ldots, x_n\w)\mapsto\q(x_1,\ldots, x_n,1\w)$.  Pulling $\mu$ back via this map,
we obtain a measure on $\R^n$--since $\mu\in B^{\epsilon}_{1,\infty}\q(\RPn\w)$ this pulled back
measure is absolutely continuous with respect to Lebesgue measure and we write this pulled back measure
as $\gamma_0\q(x\w)\: dx$.  This induces a map taking measures in $B^{\epsilon}_{1,\infty}\q(\RPn\w)$
to functions  $\R^n$ given by $\mu\mapsto \gamma_0$.

\begin{thm}\label{ThmRPnBesovisBesov}
The map $\mu\mapsto \gamma_0$ is a bijection $\bigcup_{0<\epsilon<1} B^{\epsilon}_{1,\infty}\q(\RPn\w)\rightarrow \bigcup_{0<\epsilon<1} \fBtp{\epsilon}{\R^n}$
in the following sense:

(i) $\forall \epsilon\in (0,1)$, $\exists \epsilon'\in (0,\epsilon]$, and $C=C\q(\epsilon, n\w)<\infty$ such that $\forall \mu\in B^{\epsilon}_{1,\infty}\q(\RPn\w)$,
$\gamma_0\in \fBtp{\epsilon'}{\R^n}$ and $\fBN{\epsilon'}{\gamma_0}\leq C \BesN{\mu}{\RPn}$.

(ii) $\forall \epsilon\in (0,1)$, $\exists \epsilon'\in (0,\eps]$, $\forall \gamma_0\in \fBtp{\epsilon}{\R^n}$, there exists a unique $\mu\in B^{\epsilon'}_{1,\infty}\q(\RPn\w)$
with $\mu\mapsto \gamma_0$ under this map.  Furthermore, $\exists C=C\q(\epsilon,n\w)$ such that $\BesNp{\epsilon'}{\mu}{\RPn}\leq C \fBN{\epsilon}{\gamma_0}$.
\end{thm}
\begin{proof}
Fix $\epsilon\in (0,1)$ and let $\mu\in B^{\epsilon}_{1,\infty}\q(\RPn\w)$.
We define an open cover of $\RPn$.  For $j=1,\ldots, n+1$, let $V_j$ denote those points
$\q\{\q(x_1,\ldots, x_{j-1}, 1,x_{j},\ldots, x_n\w) : x\in \R^n, \q|x\w|<2\w\}$, written in homogenous coordinates on
$\RPn$.  $V_j$ is an open subset of $\RPn$ which is diffeomorphic to $B^n\q(2\w)$, and $\cup_{j=1}^{n+1} V_j=\RPn$.

Let $\phi_j$, $1\leq j\leq n+1$ be a smooth partition of unity subordinate to the cover $V_1,\ldots, V_{n+1}$.
$\mu=\sum_{j=1}^{n} \phi_j \mu$.  By the assumption that $\mu\in B^{\epsilon}_{1,\infty}\q(\RPn\w)$,
it follows that $\phi_j\mu = \gamma_j\q(x\w)\: dx$, when written in the standard coordinates on $V_j$,
and $\BesN{\gamma_j}{\R^n}\lesssim \BesN{\mu}{\RPn}$.  Since $\gamma_j$ has compact support,
we have $\fBN{\epsilon}{\gamma_j}\lesssim \BesN{\gamma_j}{\R^n}\lesssim \BesN{\mu}{\RPn}$.  Finally,
\begin{equation*}
\gamma_0\q(x\w)\: dx = \gamma_{n+1}\q(x\w)\: dx + \sum_{j=1}^n 
 x_j^{-n-1} \gamma_j\q(x_j^{-1}x_1, x_{j}^{-1} x_2, \ldots, x_{j}^{-1} x_{j-1}, x_j^{-1} x_{j+1},\ldots, x_j^{-1} x_n, x_j^{-1}\w)
\:dx\,.
\end{equation*}
It follows from 
Corollary  \ref{ThmAdjTechResult}, applied to each term of the sum,
that $\fBN{\epsilon'}{\gamma_0}\leq C_n\BesN{\mu}{\RPn}$, and part (i) is proved. 

Because $\gamma_0$ uniquely determines $\mu$ except at those point which cannot be written in homogeneous coordinates
as $\q(x_1,\ldots, x_n,1\w)$, it follows that there is at most one $\mu\in \cup_{\epsilon>0} B^{\epsilon}_{1,\infty}\q(\RPn\w)$ which maps
to a given $\gamma_0$ (because such a $\mu$ is absolutely continuous with respect to Lebesgue measure in every coordinate chart,
and gives such points measure $0$).  Hence, given $\gamma_0\in \fBtp{\epsilon}{\R^n}$ there is at most one
$\mu$ such that $\mu\mapsto \gamma$.  We wish to construct such a $\mu$.

Let $\phi_j$ be the coordinate charts from above.  Given $\gamma_0\in \fBtp{\epsilon}{\R^n}$ define $\gamma_{n+1}\q(x\w)\: dx:= \phi_{n+1}\q(x\w)\gamma_0\q(x\w)\: dx$ and for $1\leq j\leq n$,
\begin{equation*}
\gamma_j\q(x\w)\: dx :=\phi_j\q(x\w)  x_n^{-n-1} \gamma_0\q(x_n^{-1}x_1,\ldots, x_n^{-1} x_{j-1}, x_n^{-1}, x_n^{-1}x_j,\ldots, x_n^{-1} x_{n-1}\w)\:dx.
\end{equation*}
Define $d\mu_j:= \gamma_j\q(x\w) \: dx$ on $V_j$.  By Corollary \ref{ThmAdjTechResult}, 
there exists $\epsilon'>0$
with $\fBN{\epsilon'}{\gamma_j}\leq C \fBN{\epsilon}{\gamma}$.
We set $\mu=\sum_{j=1}^{n+1} \mu_j$.  We have $\BesNp{\epsilon'}{\mu}{\RPn}\leq C' \fBN{\epsilon}{\gamma_0}$ and $\mu\mapsto \gamma_0$, as desired.
\end{proof}

\begin{rmk}
In this section we were not explicit about how each constant depends on $n$.  The above can be set up in such a way that
all constants are polynomial in $n$, which is natural for our purposes--see 
\S\ref{RmkResultsConstantsPolyInn}.
In fact, it would be hard to avoid this polynomial dependance on $n$, since there are naturally $n+1$ coordinate charts
in the definition of $\RPn$.
\end{rmk}

\begin{rmk}
Corollary  \ref{ThmAdjTechResult} 
implies that the space $\bigcup_{\epsilon>0} \fBtp{\epsilon}{\R^n}$ (when thought of as densities on $\RPn$)
is closed under the action of a particular diffeomorphism of $\RPn$.  Namely, if $\gamma\in \bigcup_{\epsilon>0} \fBtp{\epsilon}{\R^n}$, then
\begin{equation*}
s_1^{-n-1} \gamma\q(s_1^{-1}, s_1^{-1}s_2,\ldots, s_1^{-1} s_n\w)\in \bigcup_{\epsilon>0} \fBtp{\epsilon}{\R^n}.
\end{equation*}
Theorem \ref{ThmRPnBesovisBesov} tells us that more is true:  $\bigcup_{\epsilon>0} \fBtp{\epsilon}{\R^n}$
is closed under the action of {\it any} smooth diffeomorphism of $\RPn$ (as $\bigcup_{\epsilon>0} B^{\epsilon}_{1,\infty}\q(\RPn\w)$
clearly is).
It is not hard to see that, when taking adjoints of our multilinear operator in the special case when $K\q(\alpha,x\w) = \gamma_0\q(\alpha\w) K_0\q(x\w)$
where $K_0$ is a homogenous Calder\'on-Zygmund kernel, each permutation of $b_1,\ldots, b_{n+2}$ corresponds to the action
of a diffeomorphism of $\RPn$ on $\gamma_0$.  In fact, each permutation corresponds to an action of an element
of $\mathrm{GL}\q(n+1,\R\w)$ on $\RPn$ (where the action of $\mathrm{GL}\q(n+1,\R\w)$ on $\RPn$ is defined in the usual way).
\end{rmk}

\section{Outline of the proof of boundedness}\label{SectionOutline}

In this section, we begin the proof of Theorem \ref{ThmOpResBound} on the
 boundedness of our multilinear forms.
Let $\phi$ be an even $C^\infty_0$ function supported in $\{|x|<1\}$
such that $\phi\geq 0$ and $\int \phi =1$.  For $j\in \Z$ define
$\dil{\phi}{2^j}\q(x\w):= 2^{jd} \phi\q(2^j x\w)$ and define the operator $P_j f = f*\dil{\phi}{2^j}$.  Furthermore, we choose $\phi$ to be even
so that $P_j^{*}=P_j={}^t\!P_j$ (here $P_j^{*}$ is the adjoint of $P_j$ and ${}^t\!P_j$ is the transpose).
There are two key facts to note about $P_j$. First, for all $f\in \cS(\bbR^d)$,   \Be\label{Pjapprox} 
\lim_{j\rightarrow +\infty} P_jf=f, \quad
\lim_{j\rightarrow -\infty} P_jf=0, 
\Ee
with convergence in $\cS'$.
Secondly, by the nonnegativity of $\phi$  the operator norm on $L^\infty$ is bounded by $1$:
\Be \label{contraction}\|P_j\|_{L^\infty\to L^\infty}=1\,.\Ee

In Theorem \ref{ThmOpResBound} 
we are given a {\it bounded}  family in $\cB_\eps$,
\Be\label{fSdef} \vectsig=\{ \vsig_j: \,j\in \bbZ\}.\Ee
For (parts of the) proof of Theorem \ref{ThmOpResBound}  we shall also need to assume the cancellation condition
\Be\label{cancvsigj}
\int \vsig_j\q(\alpha,v\w)\: dv=0
\Ee for all $j\in \bbZ$.  Of particular interest are  the choices 
in Proposition \ref{PropKerDecompK}, namely
$\vsig_j= (Q_j K)^{(2^{-j})}$, given $K\in \sK_\a$ for some $\a>\eps$.
Theorem \ref{ThmOpResBound} 
concerns the  sum
\begin{equation}\label{EqnOutlineMainSum}
\La\q(b_1,\ldots, b_{n+2}\w) =\lim_{N\to\infty}\sum_{j=-N}^N\La[\vsigjj]
(b_1,\ldots, b_{n+2}),
\end{equation}
where $b_1,\ldots, b_{n}\in L^\infty\q(\R^d\w)$, $b_{n+1}\in L^{p}\q(\R^d\w)$, and $b_{n+2}\in L^{p'}\q(\R^d\w)$, with
$p\in\q(1,2\w]$ and $p'\in \q[2,\infty\w)$ is the dual exponent to $p$.  
We have not yet shown that this sum converges in any reasonable sense though
it is easy to see that it converges if all $b_j$ belong to $\CziRd$. 
One first establishes estimates for the  partial sums
$\sum_{j=-N}^N\La[\vsigjj]
(b_1,\ldots, b_{n+2})$ which are independent of $N$. Thus, in order to state a priori results 
one should  first assume that all but finitely many of the $\vsig_j$ are zero.
In the general case we shall establish 
convergence in the operator topology of multilinear functionals (or in slightly stronger convergence modes). 
Throughout we take $n\geq 1$, as the result for $n=0$ is classical. Our estimates will involve quantities depending on the family $\vectsig$. It will be convenient to use the following notation.
Let
\Be\label{Gammaeps}
\Gamma_\eps\equiv \Gamma_\eps[\vectsig]:= \frac{\sup_j \|\vsig_j\|\ci{\cB_\eps}}
{\sup_j \|\vsig_j\|\ci{L^1} },
\Ee
and for 
 $n\ge 1$, $\nu\ge 0$ set
\Be
\label{fM-quantity}
\fM^{n,\eps}_\nu\equiv \fM^{n,\eps}_\nu[\vectsig]
:= 
\sup_{j} \|\vsig_j\|\ci{L^1}
 \log^\nu( 1+n\, \Gamma_\eps(\vectsig))\,.
\Ee


We split the  sum  \eqref{EqnOutlineMainSum} into various terms which we study separately.  For $1\leq l_1<l_2\leq n+2$, we define
\begin{equation} \label{Lal1l2}
\begin{split}
&\La^{1}_{l_1,l_2}\q(b_1,\ldots, b_{n+2}\w) 
\\&:= \sum_{j\in \Z} \La[\vsigjj](b_1,\ldots, b_{l_1-1}, \q(I-P_j\w) b_{l_1}, P_j b_{l_1+1},\ldots, P_j b_{l_2-1}, \q(I-P_j\w) b_{l_2}, P_j b_{l_2+1},\ldots, P_j b_{n+2}).
\end{split}
\end{equation}
For $1\leq l\leq n+2$, we define
\begin{equation}\label{La2l}
\La^{2}_{l}\q(b_1,\ldots, b_{n+2}\w) :=\sum_{j\in \Z} \La[\vsigjj](P_j b_1,\ldots, P_j b_{l-1}, \q(I-P_j\w) b_l, P_j b_{l+1}, \ldots, P_j b_{n+2}).
\end{equation}
Finally, we define
\begin{equation} 
\La^{3}\q(b_1,\ldots, b_{n+2}\w) := \sum_{j\in \Z} \La[\vsigjj](P_j b_1,\ldots, P_{j} b_{n+2}).
\end{equation}
One verifies (by induction on $n$) that 
\begin{multline} \label{Lasums}
\La\q(b_1,\ldots, b_{n+2}\w) \\= \sum_{1\leq l_1<l_2\leq n+2} \La^{1}_{l_1,l_2} \q(b_1,\ldots, b_{n+2}\w)+ \sum_{1\leq l\leq n+2} \La^2_l\q(b_1,\ldots, b_{n+2}\w) + \La^3\q(b_1,\ldots, b_{n+2}\w).
\end{multline}
For $b_1,\ldots, b_{n}\in L^\infty\q(\R^d\w)$ fixed, we can identify 
the multilinear form $\La$  with an operator $T\equiv T\q[b_1,\ldots, b_n\w] $  
defined by
\begin{equation} \label{operatorassoc}
\int g\q(x\w) \,T\q[b_1,\ldots, b_n\w] 
f(x) \: dx := \La\q(b_1,\ldots, b_n, f,g\w).
\end{equation}
In this way we associate operators 
$T^1_{l_1,l_2}$, $T^2_{l}$ and $T^3$ to the forms
$\La^{1}_{l_1,l_2}$, $\La^2_l$ and $\La^3$. 
We shall see that the sums defining these operators converge in the strong 
operator topology as operators  $L^p\rightarrow L^p$ (for fixed $b_1,\ldots, b_n\in L^\infty\q(\R^d\w))$, see \S\ref{notation} for the definitions.

 
\subsubsection*{The main estimates}
We separate the proof of Theorem \ref{ThmOpResBound} into the following five parts.

\begin{thm} \label{main-parts}
Let  $p\in \q(1,2\w]$ and $p'\in \q[2,\infty\w)$ with $\frac{1}{p}+\frac{1}{p'}=1$.

(a) Suppose that $\vsig_j=0$ for all but finitely many $j$. Then 
 
\begin{enumerate}[(I)]
\item\label{ItemBoundRoughest}
$$\big|\La^1_{n+1,n+2}\q(b_1,\ldots, b_{n+2})\big|\lesssim
\fM^{n,\eps}_2[\vectsig]
\big(\prod_{i=1}^n \LpN{\infty} {b_i}\w\big)\LpN{p}{b_{n+1}} \LpN{p'}{b_{n+2}}.$$

\item\label{ItemBoundT1} 
For $ 1\leq l_1\leq n$, $l_2\in \q\{n+1,n+2\w\}$, 
$$\q|\La^1_{l_1,l_2}\q(b_1,\ldots, b_{n+2}\w)\w|\lesssim 
\fM^{n,\eps}_{5/2}[\vectsig]\, \big(\prod_{i=1}^n \|b_i\|_\infty \big)\|b_{n+1}\|_p \|b_{n+2}\|_{p'}.$$

\item\label{ItemBoundTwoPts} For $ 1\leq l_1<l_2\leq n$,
$$\q|\La^1_{l_1,l_2}\q(b_1,\ldots, b_{n+2}\w)\w|\lesssim 
\fM^{n,\eps}_{3}[\vectsig]\, \big(\prod_{i=1}^n \|b_i\|_\infty \big)\|b_{n+1}\|_p \|b_{n+2}\|_{p'}.$$

\item\label{ItemBound1IminusPt}  Under the additional cancellation condition 
\eqref{cancvsigj} we have, for $ 1\leq l\leq n+2$,
$$\q|\La^2_{l}\q(b_1,\ldots, b_{n+2}\w)\w|\lesssim n \,
\fM^{n,\eps}_{3}[\vectsig]\, \big(\prod_{i=1}^n \|b_i\|_\infty \big)\|b_{n+1}\|_p \|b_{n+2}\|_{p'}.
$$

\item\label{itemBoundAllPts}
Suppose that \eqref{cancvsigj} holds. Then
$$\q|\La^3\q(b_1,\ldots, b_{n+2}\w)\w|\lesssim n^2\,
\fM^{n,\eps}_{3}[\vectsig]\, \big(\prod_{i=1}^n \|b_i\|_\infty \big)\|b_{n+1}\|_p \|b_{n+2}\|_{p'}.$$
\end{enumerate}

In the above inequalities the implicit constants depend only on $p\in \q(1,2\w]$, $d\in \N$, and $\epsilon>0$.

(b) For general families 
$\vectsig=\{\vsig_j:j\in \bbZ\}$, bounded in $\cB_\eps$,  the sums defining the above five functionals converge in the  operator topology of multilinear functionals and the limits satisfy the above estimates.

(c) The sums defining the operators 
$T^1_{l_1,l_2}[b_1,\dots,b_n]$, $T^2_{l}[b_1,\dots, b_n]$ and $T^3[b_1,\dots,b_n]$ associated to the forms
$\La^{1}_{l_1,l_2}$, $\La^2_l$ and $\La^3$ via  \eqref{operatorassoc}  converge in the strong 
operator topology as operators from $L^p\rightarrow L^p$.
\end{thm}

Summing up the estimates for the five parts yields Theorem \ref{ThmOpResBound}.

\section{Some auxiliary operators}\label{SectionAuxOps}
In this section, we introduce some auxiliary  operators which play a role
in the proof of Theorems \ref{ThmOpResBound}, \ref{main-parts}.
Recall that in \S\ref{SectionOutline} we introduced the operator $P_j$, which was defined as
$P_j f = f*\dil{\phi}{2^j}$, where $\phi\in \Czip{B^d\q(1\w)}$ was a fixed even function with $\int \phi=1$, $\phi\geq 0$, and
$\dil{\phi}{2^j}\q(x\w) = 2^{jd} \phi\q(2^j x\w)$.

Define $\psi\q(x\w):=\phi\q(x\w)-2^{-d}\phi\q(x/2\w)\in \Czip{B^d\q(0,2\w)}$, and let $Q_k f = f*\dil{\psi}{2^k}$ so that
\Be\label{Qkdefinition} Q_k=P_k-P_{k-1}\,.\Ee
Note that, in the sense of distributions, we have the following identities
\begin{equation}\label{EqnAuxOpId}
I=\sum_{j\in \Z} Q_j, \quad P_j = \sum_{k\leq j} Q_k, \quad I-P_j = \sum_{k>j} Q_k
\end{equation}
with convergence in the strong operator topology (as operators $L^p\to L^p$, $1<p<\infty$).

\begin{rmk}\label{RmkAuxOpSublety}
There is one subtlety that we must consider.  While $\lim_{j\rightarrow -\infty} P_j f=0$ for $f\in \CziRd$ (or even $f\in L^p$, $p\ne \infty$)
is it not the case that $\lim_{j\rightarrow -\infty} P_j f=0$ for $f\in L^\infty$.  Indeed, this is not true for a constant function.
Thus, the first two identities in \eqref{EqnAuxOpId} 
do not hold when thought of as operators on $L^\infty$.  However, the third identity
does hold (with the limit taken almost everywhere), which we shall use.
\end{rmk}

Let $\chi_0\in \cS(\bbR^d)$ so that  $\chi_0(\xi)=1$ for $|\xi|<1/2$ and $\chi_0$ is supported in $\{|\xi|<1\}$. 
For $j\ge 1 $ let $\eta_j$ be defined via
\Be \label{etakdef}
\widehat{\eta_j} (\xi)=\chi_0(2^{-j}\xi)-\chi_0(2^{1-j}\xi)
\Ee
so that  $\widehat {\eta_j}$ is supported in the annulus
$\{ \xi: 2^{j-2}\le |\xi|\le 2^j\}$  and $\sum_{j\in \bbZ}\widehat \eta_j(\xi)=1$ for $\xi\neq 0$.
Let $\widetilde \eta_0$ be a Schwartz function so that its Fourier transform vanishes in a neighborhood of the origin and is compactly be supported, and equal to $1$ on the support of $\widehat \eta_0$. Let $\widetilde \eta_j= \widetilde \eta_0^{(2^{j})}$.
Note that $\eta_j$, $\widetilde \eta_j$ belong to $ \sS_0\q(\R^d\w)$ -- the space of Schwartz functions, all of whose moments vanish.
Define 
\Be \label{cQdef}
\cQ_j f= f*\eta_j, \quad
\widetilde \cQ_j f= f*\widetilde \eta_j.
\Ee
and note that \Be\label{reproducing}\cQ_j=\cQ_j\tcQ_j = \tcQ_j \cQ_j\Ee
 and $I=\sum_{j\in \Z} \cQ_j = \sum_{j\in \Z} \cQ_j \tcQ_j= \sum_{j\in \Z} \tcQ_j\cQ_j $, where this identity holds in the weak (distributional sense) and also 
 in the strong operator topology, as operators on $L^p$,  if $1<p<\infty$.
We also have the following well known estimates for the associated Littlewood-Paley square functions:
for $1<p<\infty$, $f\in L^p\q(\R^d\w)$,
\begin{equation}\label{EqnAuxLittlewoodPaley}\|f\|_p\approx \Big\|\Big( \sum_{j\in \Z} |\cQ_j f|^2 \Big)^{\frac{1}{2} }\Big\|_p 
\approx\Big\|\Big(\sum_{j\in \Z} |\tcQ_j f|^2 \Big)^{\frac{1}{2} }\Big\|_p
\end{equation}
with implicit constants depending only on $p$ and $d$.  The same estimates hold with $\cQ_k$ and $\tcQ_k$ replaced by their adjoints.

We introduce a class of operators generalizing $Q_j$, $\Qt_j$, and $\Qtt_j$.
\begin{defn}\label{Udefin}
$\sU$ is defined to be the space of those functions $\uzeta\in C^1\q(\R^d\w)$ such that the norm
\begin{equation*}
\|\uzeta\|_{\sU}:=\sup_{x\in \R^d} \q(1+\q|x\w|^{d+\frac 12}\w) \q(\q|\uzeta\q(x\w)\w|+\q|\nabla \uzeta\q(x\w)\w|\w)
\end{equation*}
 is finite
and such that $$\int \uzeta\q(x\w)\: dx=0.$$
\end{defn}

\begin{defn}
For $\uzeta\in \sU$ and $j\in \Z$, define $\Qb{j}{\uzeta}f:=f*\dil{\uzeta}{2^j}$.
\end{defn}

\begin{rmk}
Note that $\psi,\eta_0,\widetilde \eta_0\in \sU$ and $Q_j = \Qb{j}{\psi}$, $\cQ_j= \Qb{j}{\eta_0}$, and $\tcQ_j=\Qb{j}{\widetilde\eta_0}$.
\end{rmk}

The class $\sU$ comes up through the following proposition (which is very close to a similar one in \cite{cj}).
\begin{prop}\label{PropAuxQb}
If $\q\{f_j\w\}_{j\in \Z}\subset L^2\q(\R^d\w)$, then
\begin{equation*}
\Big\|\sum_{j\in \Z} Q_j f_j\Big\|_2 \lesssim \sup_{\substack{\uzeta\in \sU\\ \sUN{\uzeta}=1}} \Big( \sum_{j\in \Z} \|\Qb{j}{\uzeta} f_j \|_2^2\Big)^{\frac{1}{2}},
\end{equation*}
in the sense that  $\sum_j Q_jf_j$ converges unconditionally in the $L^2$ norm if the right hand side is finite.
\end{prop}

\subsection{Proof of Proposition \ref{PropAuxQb}} 
We need several lemmata.

\begin{lemma}\label{LemmaAuxBound1}
For $\ell\leq 0$, $\phi\in \Czip{B^d\q(2\w)}$, $\uzeta\in \sS\q(\R^d\w)$ if we define
$\gamma_{-\ell}:=\phi*\dil{\uzeta}{2^{-\ell}}$, we have $\gamma_{-\ell}\in \sU$ and $\sUN{\gamma_{-\ell}}\lesssim 2^{\ell/2}.$
\end{lemma}
\begin{proof}
It is clear that $\gamma_{-\ell}\in C^\infty\q(\R^d\w)$, so it suffices to prove the bound on $\sUN{\gamma_{-\ell}}$.
Because, for $\nu=1,\ldots,d$,  $\partial_{x_\nu} \gamma_{-\ell}$ is of the same form as $\gamma_{-\ell}$,
it suffices to show $\q|\gamma_{-\ell}\q(x\w)\w|\lesssim 2^{\ell/2} \q(1+\q|x\w|^{d+1/2}\w)^{-1}$.
This is evident for  $\q|x\w|\leq 4$, since 
$|\gamma_{-\ell}| \le \|\phi\|_\infty\|\uzeta\|_1
\lc 1$.

Since $\phi\q(x-y\w)$ is supported on $\q|x-y\w|\leq 2$, we have for
$\q|x\w|\geq 4$ and any $m$,
\begin{equation*}
\q|\gamma_{-\ell}\q(x\w)\w|\lesssim \int_{\q|x-y\w|\leq 2} 2^{-\ell d} \q(1+2^{-\ell}\q|y\w|\w)^{-m}\: dy\approx \int_{\q|x-y\w|\leq 2} 2^{-\ell d} \q(1+2^{-\ell}\q|x\w|\w)^{-m}\: dy\lesssim 2^{-\ell d} \q(1+2^{-\ell} \q|x\w|\w)^{-m}.
\end{equation*}
Taking $m=d+1/2$, we have 
\begin{equation*}
\q|\gamma_{-\ell}\q(x\w)\w| \lesssim 2^{-\ell d} \q(1+2^{-\ell }\q|x\w|\w)^{-d-1/2}
\lesssim 2^{\ell/2}\q(1+\q|x\w|^{d+1/2}\w)^{-1}, \quad\q|x\w|\geq 4,
\end{equation*}
as desired.
\end{proof}

\begin{lemma}\label{LemmaAuxBound2}
Suppose $\uzeta_1\in \sS\q(\R^d\w)$, $\uzeta_2\in \sS_0\q(\R^d\w)$.  For $j\geq 0$, let $\uzeta_j := \uzeta_1 *\dil{\uzeta_2}{2^j}$.  Then, for $m=0,1,2,\dots$,
\begin{equation*}
\sum_{\q|\alpha\w|\leq m} \q|\partial_x^{\alpha} \uzeta_j\q(x\w)\w|\lesssim 2^{-jm} \q(1+\q|x\w|\w)^{-m}.
\end{equation*}
\end{lemma}
\begin{proof}
The goal is to show, for every $m$, $\q\{2^{jm} \uzeta_j  : j\geq 0\w\}\subset \sS\q(\R^d\w)$ is a bounded set.  To do this, we show $\q\{2^{jm} \uzetah_j : j\geq 0\w\}\subset \sS\q(\R^d\w)$ is a bounded set.
We have, for every $\alpha$,
\begin{equation*}
\begin{split}
&\Big|\partial_\xi^{\alpha} \uzetah_j\q(\xi\w)\Big| =\q|\sum_{\beta+\gamma=\alpha} C_{\beta,\gamma} \partial_\xi^\beta \uzetah_1\q(\xi\w) \partial_\xi^{\beta} \uzetah_2\q(2^{-j} \xi\w)\w|
\lesssim  \sum_{\beta+\gamma=\alpha}2^{-j\q|\gamma\w|} \q|\partial_\xi^{\beta} \uzetah_1\q(\xi\w) \q(\partial_\xi^{\gamma}\uzetah_2\w)\q(2^{-j}\xi\w)\w|
\\
&\lesssim \sum_{\beta+\gamma=\alpha} 2^{-j\q|\gamma\w|} \q(1+\q|\xi\w|\w)^{-2m} \q|2^{-j}\xi\w|^{m} \q(1+\q|2^j\xi\w|\w)^{-2m}
\lesssim 2^{-mj} \q(1+\q|\xi\w|\w)^{-m}. 
\end{split}
\end{equation*}
The result follows.
\end{proof}

\begin{lemma}\label{LemmaAuxDecompPsi}
There exists functions $\vphi_1,\ldots, \vphi_d\in \Czip{B^d\q(2\w)}$ such that $\psi=\sum_{\nu=1}^d \partial_{x_\nu} \vphi_\nu$.
\end{lemma}
\begin{proof}
Indeed, write
\begin{equation*}
\psi\q(x\w) = \phi\q(x\w) - 2^{-d} \phi(2^{-1} x) = \sum_{\nu=1}^d \psi_\nu\q(x\w),
\end{equation*}
where $\psi_\nu(x)$ is given by  $$ 2^{-\q(\nu-1\w)} \phi\q(x_1/2,x_2/2,\ldots, x_{\nu-1}/2, x_\nu,x_{\nu+1},\ldots, x_d\w) - 2^{-\nu} \phi\q(x_1/2,x_2/2,\ldots, x_{\nu}/2, x_{\nu+1},\ldots, x_d\w).$$
Letting $\vphi_\nu\q(x\w) = \int_{-\infty}^{x_\nu} \psi_\nu\q(x_1,\ldots, x_{\nu-1}, y_\nu, x_{\nu+1},\ldots, x_d\w)\: dy_\nu$, the result follows.
\end{proof}

\begin{lemma}\label{LemmaAuxBound3}
For $j,k\in \Z$, $\Qtt_{j+k} Q_j = 2^{-\q|k\w|/2} \Qb{j}{\uzeta_{k}}$, where $\uzeta_k\in \sU$ and $\sUN{\uzeta_k}\lesssim 1$.
\end{lemma}
\begin{proof}
By scale invariance, it suffices to consider the case  $j=0$; then  $\uzeta_k = \psi * \dil{\widetilde \eta_0}{2^k}$.
When $k\leq 0$, we use Lemma \ref{LemmaAuxDecompPsi} to see 
\begin{equation*}
\uzeta_k\q(x\w) = \sum_{\nu=1}^d \int \q(\partial_{x_\nu} \vphi_\nu\w) \q(y\w) \dil{\widetilde \eta_0}{2^k} \q(x-y\w)\: dy = -2^{k} \sum_{\nu=1}^d \int \varphi_\nu\q(y\w) \dil{\q(\partial_{x_\nu} \widetilde \eta_0\w)}{2^k}\q(x-y\w)\: dy.
\end{equation*}
From here, the desired estimate follows from Lemma \ref{LemmaAuxBound1}.
For $k\geq 0$, the result follows immediately from Lemma \ref{LemmaAuxBound2}.
\end{proof}

\begin{proof}[Proof of Proposition \ref{PropAuxQb}, conclusion]
Let $\q\{f_j : j\in \Z\w\}\subset L^2\q(\R^d\w)$ and let $g\in L^2\q(\R^d\w)$ with $\LpN{2}{g}=1$.  Let $\inn{\cdot}{\cdot}$ denote the inner product in $L^2$.  We have, letting $\uzeta_k$ be as in Lemma \ref{LemmaAuxBound3},
\begin{equation*}
\begin{split}
&\Big|\biginn{g}{\sum_{j=J_1}^{J_2}Q_j f_j }\Big| = \Big|\biginn{g}{\sum_{j=J_1}^{J_2}\sum_{k\in \Z} \Qt_{j+k} \Qtt_{j+k} Q_j f_j} \Big|
\leq \sum_{k\in \Z}\sum_{j=J_1}^{J_2} \Big|\biginn{\Qt_{j+k}^{*} g} {\Qtt_{j+k} Q_j f_j}\Big|
\\&\lesssim \sum_{k\in \Z} \Big(\sum_{j=J_1}^{J_2} \big\|\Qt_{j+k}^{*} g \big\|_2^2 \Big)^{\frac{1}{2}}
\Big( \sum_{j=J_1}^{J_2}\big\| \Qtt_{j+k} Q_j f_j \big\|_2^2 \Big)^{\frac{1}{2}}
\lesssim \sum_{k\in \Z}2^{-\q|k\w|/2} \Big( \sum_{j=J_1}^{J_2} \big\|\Qb{j}{\uzeta_k}f_j \big\|_2^2\Big)^{\frac 12}.
\end{split}
\end{equation*}
The result follows easily.
\end{proof}

\subsection{A decomposition result for functions in $\sU$}   The proof of the following result  follows closely a similar result in
\cite{NRS}.

\begin{prop}\label{PropDecomZeta}
Let $\uzeta\in \sU$.  Then there exists $\uzeta_j\in C_0^1\q(B^d\q(\frac{1}{4}\w)\w)$ with $\CzN{\uzeta_j}\lesssim \sUN{\uzeta}$,
$\int \uzeta_j=0$, and 
\begin{equation*}
\uzeta=\sum_{j\leq 0} 2^{j/2} \dil{\uzeta_j}{2^j}.
\end{equation*}
\end{prop}
\begin{proof}
Let $\chi_0\in C^\infty_0$, supported in $\{|x|\le 1/4\}$. with
$0\leq \chi_0\leq 1$ and $\chi_0(x)=1$ for $|x|\le 1/8$.
For $j\ge 1$ define $\chi_j(x)=\chi_0(2^{-j}x)-\chi_0(2^{1-j}x)$ so that
 that for $j\ge 1$, 
 $\supp{\chi_j}\subseteq \q\{2^{j-4}\leq \q|x\w|\leq 2^{j-2}\w\}$,
and
\begin{equation*}
1=\sum_{j=0}^\infty \chi_j(x).
\end{equation*}
Observe that
\begin{equation}\label{EqnAuxDecompZeta3}
\int \chi_j\q(x\w) \: dx = (2^{jd}-1)\int \chi_0(x) dx \gtrsim 2^{jd}.
\end{equation}

Also let 
$$\chit_j\q(x\w) = \frac{\chi_j\q(x\w)}{\int \chi_j\q(y\w)\: dy}.$$
Set $a_j =\int \uzeta\q(x\w) \chi_j\q(x\w)\: dx$ and $A_j = \sum_{k\geq j} a_k = -\sum_{0\leq k<j} a_k$ (where the second equality follows from the fact that $\sum a_j = \int \uzeta = 0$).

Note $\q|a_0\w|\lesssim 1$, and for $j\geq 1$,
\begin{equation}\label{EqnAuxDecompZeta1}
\q|a_j\w| \leq \int \q|\uzeta\q(x\w) \w| \q|\chi_j\q(x\w)\w|\: dx \leq \int_{2^{j-4}\leq \q|x\w|\leq 2^{j-2}} \q(1+\q|x\w|^{d+1/2}\w)^{-1}\: dx \sUN{\uzeta}\lesssim 2^{-j/2} \sUN{\uzeta}.
\end{equation}
Thus, 
\begin{equation}\label{EqnAuxDecompZeta2}
\q|A_j\w| \leq \sum_{k\geq j} \q|a_k\w|\lesssim 2^{-j/2} \sUN{\uzeta}.
\end{equation}

Notice, $A_0=0$.
We have,
\begin{equation*}
\begin{split}
&\uzeta\q(x\w) = \sum_{j\geq 0} \uzeta\q(x\w) \chi_j\q(x\w) = \sum_{j\geq 0} \q(\uzeta\q(x\w) \chi_j\q(x\w) - a_j \chit_j\q(x\w)\w) + \sum_{j\geq 0} \q(A_j -A_{j+1}\w) \chit_j\q(x\w)
\\&= \sum_{j\geq 0} \q(\uzeta\q(x\w) \chi_j\q(x\w) - a_j \chit_j\q(x\w) \w) + \sum_{j\geq 1} A_j \q(\chit_j\q(x\w) - \chit_{j-1}\q(x\w)\w) =:\sum_{j\geq 0} B_j\q(x\w),
\end{split}
\end{equation*}
where $B_j\q(x\w) = \uzeta\q(x\w) \chi_j\q(x\w) - a_j \chit_j\q(x\w) +\q(A_j \q(\chit_j\q(x\w) - \chit_{j-1}\q(x\w)\w) \w) \epsilon_j$ and
$\ep_j=1$ if $j\ge 1$, $\ep_0=0$.
Here we have used $A_0=0$ and $\lim_{j\rightarrow \infty} A_j=0$.
Clearly $\int B_j =0$, and $\supp{B_j}\subseteq \q\{ \q|x\w|\leq 2^{j-2} \w\}$.
We have
\begin{equation*}
\q|B_j\q(x\w)\w|\leq \q|\uzeta\q(x\w) \chi_j\q(x\w)\w| + \q|a_j\w| \q|\chit_j\q(x\w)\w| + \q|A_j\w| \q( \q|\chit_j\q(x\w)\w| +\q|\chit_{j-1}\q(x\w)\w| \w)\epsilon_j.
\end{equation*}
\eqref{EqnAuxDecompZeta3} shows $\q|\chit_j\q(x\w)\w|\lesssim 2^{-jd}$.  The support of $\chi_j$ shows
$\q|\uzeta\q(x\w) \chi_j\q(x\w)\w|\lesssim 2^{-j\q(d+\frac{1}{2}\w)}\sUN{\uzeta}$.  Combining this with \eqref{EqnAuxDecompZeta1}
and \eqref{EqnAuxDecompZeta2}
shows
$\q|B_j\q(x\w)\w|\lesssim 2^{-j\q(d+\frac{1}{2}\w)} \sUN{\vsig}$.
Setting, for $j\geq 0$, $\uzeta_{-j}\q(x\w) = 2^{jd} 2^{j/2} B_j\q(2^j x\w)$, the result follows easily. 
\end{proof}

\section{Basic $L^2$ estimates}\label{Sectionltwo}


\subsection{An $L^2$ estimate for rough kernels}
An essential part to many of our estimates is the following $L^2$ estimate.

\begin{thm}\label{ThmBasicL2Resultfirst}
Let $u$ be a continuous function supported in $\{y\in \bbR^d: |y|\le 1/4\}$ such that 
$ \|u\|_\infty\le 1$ and 
\[\int u(y) dy=0
.\] Let $\fQ_k$ be the operator of convolution with $u^{(2^k)}$.
Let $0<\eps<1$, $\vsig\in \sBtp{\eps}{\R^n\times \R^d}$ and assume that
$\supp {\vsig}\subset\{(\alpha,v):|v|\le 1/4\}$. 
Then 
for all $k\in \bbN$, for $b_{n+1}, b_{n+2}\in L^2(\bbR^d)$, $b_i\in L^\infty(\bbR^d)$, 
$i=1,\dots, n,$  
\begin{equation*}
|\La[\vsig]( b_1,\ldots,  \fQ_k b_{n+1}, b_{n+2}  )\w|
\lesssim  
2^{-k\eps/(3d+3)}
n\|\vsig\|_{B_\eps} \|b_{n+1}\|_2 \|b_{n+2}\|_2 \prod_{i=1}^n \|b_j\|_{\infty}.
\end{equation*}
\end{thm}

In \S\ref{L2generalizations}  below 
 we shall prove a similar theorem  without the support assumptions on  $\vsig$ and $u$.
 In what follows we give the proof of Theorem \ref{ThmBasicL2Resultfirst}.

\subsubsection{Applying the Leibniz rule}
We have 
\begin{equation}\label{LaQcanc}
\La[\vsig](b_1,\ldots, \fQ_k b_{n+1},b_{n+2}) 
=\iint F_k[\vsig](x,y) \:b_{n+1}(y) b_{n+2}(x)\: dx\: dy,
\end{equation}
where, using the cancellation of $u$ we have 
\begin{align*} 
&F_k[\vsig](x,y)=\iint 
\vsig(\alpha, x-z)
\prod_{i=1}^n b_i(x-\alpha_i(x-z))\,   u^{(2^k)}(z-y)\:dz 
\: d\alpha
\\ 
&=\iint 
\Big[\vsig(\alpha, x-z)
\prod_{i=1}^n b_i(x-\alpha_i(x-z))
-
\vsig(\alpha, x-y)
\prod_{i=1}^n b_i(x-\alpha_i(x-y)) \Big]
\,  u^{(2^k)}(z-y)\:dz 
\: d\alpha\,.
\end{align*}
We let $T_k[\vsig]$ denote the operator with Schwartz kernel 
$F_k[\vsig]$.

For further decomposition we use  a Leibniz  rule for differences 
\begin{multline*}\prod_{j=0}^n A_j-\prod_{j=0}^n B_j =
\\
(A_0-B_0) \big(\prod_{j=1}^n{A_j}\big)+ \sum_{i=1}^{n-1} \Big(\big(\prod_{j=0}^{i-1} B_j\big)(A_i-B_i)
\big(\prod_{j=i+1}^n A_j\big)\Big) + \big(\prod_{j=0}^{n-1}B_j\big)  (A_n-B_n).
\end{multline*}
Thus
$$F_k[\vsig]=\sum_{i=0}^n F_{k,i}[\vsig]$$ where 
\begin{align*}
F_{k,0}[\vsig](x,y)=&
\iint 
\big[\vsig(\alpha, x-z)- \vsig(\alpha, x-y)\big]
\prod_{j=1}^n b_j(x-\alpha_j(x-z))
\,   u^{(2^k)}(z-y)\:dz 
\: d\alpha\,,
\\
F_{k,i}[\vsig](x,y)=&
\iint 
\vsig(\alpha, x-y)\big]
\prod_{j=1}^{i-1} b_i(x-\alpha_i(x-y))\,\times \\&\quad \Big( b_i(x-\alpha_i(x-z)-b_i(x-\alpha_i(x-y))\Big) \prod_{j=i+1}^n b_j(x-\alpha_j(x-z))
\,   u^{(2^k)}(z-y)\:dz 
\: d\alpha\,,
\end{align*}
with the convention that the products $\prod_{j=1}^0$ and $\prod_{j=n+1}^n$ stand for the number $1$.
We thus have to estimate the $L^2\to L^2$ operator norms for the operators 
$T_{k,i}[\vsig]$
with Schwartz kernels $F_{k,i}[\vsig]$.
For $i=0$ we may use the standard Schur test and the condition $\vsig \in \cB_\eps$
\begin{subequations}\label{Fk0Schur}
\begin{align}
&\sup_x\int|F_{k,0}[\vsig](x,y)|\, dy\\ \notag &\le \sup_x
\prod_{j=1}^n\|b_j\|_\infty 
\int_{|h|\le 2^{-k}}|u^{(2^k)}(h)|\int
|\vsig(\alpha, x-y-h)- \vsig(\alpha, x-y)| \,dy \,d\alpha\, dh 
\\ \notag &\lc 
\prod_{j=1}^n\|b_j\|_\infty
 \sup_{|h|\le 2^{-k}}
\int \|\vsig(\alpha, \cdot-h)- \vsig(\alpha, \cdot)\| d\alpha  \lc
2^{-k\eps} \prod_{j=1}^n\|b_j\|_\infty \|\vsig\|_{\cB_\eps}
\end{align}
and similarly
\Be
\sup_y\int|F_{k,0}[\vsig](x,y)| dx\lc 2^{-k\eps} \prod_{j=1}^n\|b_j\|_\infty\|\vsig\|_{\cB_\eps}.
\Ee
\end{subequations}
Hence 
\Be\label{Tk0est}
\|T_{k,0}[\vsig]\|_{L^2\to L^2}\lc 2^{-k\eps} \prod_{j=1}^n\|b_j\|_\infty \|\vsig\|_{\cB_\eps}.
\Ee

We shall now turn to the operators $T_{k,i}[\vsig]$, $i=1,\dots, n$. We start with a trivial bound.
\begin{lemma}\label{trivialLpbound} 
For $1\le p\le \infty$ 
$$\|T_{k,i}[\vsig]\|_{L^p\to L^p} \lc \|\vsig\|_{L^1(\bbR^n\times\bbR^d)} 
\prod_{j=1}^n \|b_i\|_{\infty}.$$\end{lemma}
\begin{proof} This follows immediately from Schur's test since
\[\sup_x\int |F_{k,i}[\vsig](x,y) |\,dy+\sup_y\int| F_{k,i}[\vsig](x,y) |\,dx \lc
\|\vsig\|_{L^1(\bbR^n\times\bbR^d)} \prod_{j=1}^n \|b_i\|_{\infty}. \qedhere\]
\end{proof}

We begin  with a regularization of $\vsig$, in the $x$ and the $\alpha_i$ variables, depending on a parameter $R$
 to be chosen later. Here  $1\ll R \ll 2^k$
(we shall see that  $R= 2^{k/(3d+3)}$ will  be a good choice).

Let $\phi\in C^\infty(\bbR^d)$ supported in $\{x:|x|\le 1/2\}$ so that $\int \phi(x) dx=1$. 
Let $\varphi\in C^\infty(\bbR)$  be supported in $\{u:|u|\le 1/2\}$ so that $\int \varphi(u) du=1$.  
Define
$$
\vsig_R^i(\alpha,v) 
=\iint \chi_{[-R,R]}(\alpha-se_i) \vsig(\alpha-se_i, v-z) R\varphi(Rs) R^d\phi(Rz) \,dz \,ds.
$$
\begin{lemma}\label{vsigRdiffbd}
For $i=1,\dots, n$,
\begin{enumerate}[(i)] 
\item\label{vsigRdiffbd(i)}
$$\|\vsig-\vsig_R^i\|_{L^1(\bbR^n\times\bbR^d)}\lc R^{-\eps}$$
\item\label{vsigRdiffbd(ii)}
$$\|T_{k,i}[\vsig-\vsig_R^i]\|_{L^2\to L^2}
\lc R^{-\eps} \|\vsig\|_{\cB_\eps}\,.
$$
\end{enumerate}
\end{lemma}
\begin{proof}
We expand $\vsig-\vsig_R^i=  I+II+III$ where
\begin{align*}
I(\a,v)&=\int \big[
 \vsig(\alpha, v) \,-\, \vsig(\alpha, v-z)\big] R^d\phi(Rz) \,dz ,
\\
II(\a,v)&=
\iint \big[\vsig(\alpha, v-z)  \,-\,
 \vsig(\alpha-se_i, v-z)\big]  R\varphi(Rs) R^d\phi(Rz)\, dz \,ds,
\\
III(\a,v)&= 
\iint \chi_{[-R,R]^\complement}(\alpha-se_i) \vsig(\alpha-se_i, v-z) R\varphi(Rs) R^d\phi(Rz)\, dz \,ds\,.
\end{align*}
Then
\[
\|I\|_{L^1(\bbR^n\times\bbR^d)}\lc \int R^d|\phi(Rz)| \iint\big|\vsig(\alpha, v) \,-\, \vsig(\alpha, v-z)\big| \, d\alpha\, dv\,|R^d\phi(Rz)| \,dz  \lc R^{-\eps} \|\vsig\|_{\cB_{\eps,3}}.
\]
For the second term,
\[
\|II\|_{L^1(\bbR^n\times\bbR^d)}\lc \int R|\varphi(Rs) |\iint
\big|\vsig(\alpha, v) \,-\, \vsig(\alpha-se_i, v)\big|\, d\alpha\,  dv\,  ds 
 \lc R^{-\eps} \|\vsig\|_{\cB_{\eps,2}}.
\]
Finally
\[
\|III\|_{L^1(\bbR^n\times\bbR^d)}\lc\int \int_{[-R,R]^\complement}|\vsig(\alpha, v)| \,d\alpha \, dv \lc 
R^{-\eps} \|\vsig\|_{\cB_{\eps,1}}\] and part 
\eqref{vsigRdiffbd(i)} follows. The second part follows from Lemma \ref{trivialLpbound} applied to $\vsig-\vsig_R^i$,  and the first part.
\end{proof}

For the more regular term $\vsig_R^i$ we shall need the inequalities
\begin{lemma} \label{estregvsigR} Let $0<\eps<1$, $d\ge 2$.
Then
\begin{enumerate}[(i)]
\item
\[
\int \Big(\int\big| \vsig^i_R(\alpha,v)\big |^2dv\Big)^{\frac 12} d\a\lc R^{\frac d2 -\eps}
 \|\vsig\|_{\cB_\eps}.
\]
\item 
Let $\theta\in S^{d-1}$ and let $\theta^\perp$ the orthogonal complement of 
$\bbR\theta$. Then
\[
\int\sup_\theta \Big(\int_{\theta^\perp}  \sup_{s\in \bbR}\big|
\vsig^i_R(\alpha,v^\perp+s\theta) |^2 dv_{\theta^\perp}\Big)^{\frac 12} d\a\lc R^{\frac{d+1}{2}-\eps} \|\vsig\|_{\cB_\eps}
\]
and
\[
\int\sup_\theta \Big(\int_{\theta^\perp}  \sup_{s\in \bbR}\big|\partial_{\alpha_i}
\vsig^i_R(\alpha,v^\perp+s\theta) |^2 dv_{\theta^\perp}\Big)^{\frac 12} d\a\lc R^{\frac{d+3}{2}-\eps} \|\vsig\|_{\cB_\eps}\,.\]
\end{enumerate}
\end{lemma}
\begin{proof}

Let $\beta_0\in \cS(\bbR^d)$ so that  $\widehat \beta_0(\xi)=1$ for $|\xi|<1/2$ and $\widehat \beta_0$ is supported in $\{|\xi|<1\}$. Let $\beta_1=\beta_0^{(2)}-\beta_0$ and $\beta_k=\beta_1^{(2^{k-1})}$ so that $\widehat \beta_k$ has support in an annulus 
$\{|\xi|\approx 2^k\}$, and $f=\sum_{k=0}^\infty \beta_k* f$ in the sense of distributions. Let $\widetilde \beta_0\in \cS(\bbR^d)$  be such that its Fourier transform equals $1$ on the support of $\widehat \beta_0$.
Let $\widetilde \beta_1$ be a Schwartz function so that its Fourier transform vanishes in a neighborhood of the origin and is compactly be supported, and equal to $1$ on the support of $\widehat \beta_1$. Let $\widetilde \beta_k= \widetilde \beta_1^{(2^{k-1})}$.

Let $$
\widetilde 
\vsig_R^i(\alpha,v) 
=\iint \chi_{[-R,R]}(\alpha-se_i) \vsig(\alpha-se_i, v) R\varphi(Rs)  \,ds
$$
so that $\widetilde 
\vsig_R^i(\alpha,\cdot)*\phi_R
=\vsig_R^i$ (the definition of $\varphi$ was given right before the statement of  Lemma \ref{vsigRdiffbd}). 
Then
$$ \vsig^i_R(\alpha,\cdot) =\sum_{l=0}^\infty \beta_l*\widetilde 
\vsig_R^i(\alpha,\cdot)*\phi_R*\widetilde\beta_l .$$
By Young's inequality 
$$\|\vsig^i_R(\alpha,v)*\widetilde \beta_l*\beta_l\|_2 
\le \|\widetilde \vsig^i_R(\alpha,\cdot)*\beta_l\|_1 \|\widetilde \beta_l*\Phi_R\|_2 
$$
and it is easy to see that 
\[\|\widetilde \beta_l*\Phi_R\|_2  \le C_M 2^{ld/2} \min \{1, (R2^{-l})^M\} \,.
\] 
Thus 
\begin{align*}
&\int \Big(\int\big| \vsig^i_R(\alpha,v)
\big |^2dv\Big)^{\frac 12} d\a
\\&\quad\lc \sum_{l=0}^\infty  2^{ld/2} \min \{1, (R2^{-l})^M\} 
\iint \Big|  \int \beta_l(v-w)  \vsig^i_R(\alpha,w) dw\Big | \,dv \,d\a
\\&\quad\lc \sum_{l=0}^\infty  2^{ld/2} \min \{1, (R2^{-l})^M\} 2^{-l\eps}
\|\widetilde \vsig^i_R\|_{\cB_\eps}  \lc R^{\frac d2-\eps} \|\vsig\|_{\cB_\eps}.
\end{align*}

The first  inequality in (ii) is proved similarly, except that we first use the  one-dimensional version of Young's inequality  in the $\theta$-direction. Since the  Fourier transform of $\beta_l$ is supported on a set of diameter $O(2^l)$ 
we have, for fixed  $\theta$ and almost every $\alpha$,
\begin{equation*}\Big(\int_{\theta^\perp}  \sup_{s\in \bbR}\big|
\beta_l*\vsig^i_R(\alpha,v^\perp+s\theta) |^2 dv^\perp\Big)^{\frac 12} 
\lc
2^{l/2}
\Big(\int_{\theta^\perp}\int_{-\infty}^\infty\big|
\beta_l*\vsig^i_R(\alpha,v^\perp+s\theta) |^2 ds\,dv^\perp\Big)^{\frac 12} .
\end{equation*}
Notice that the double integral on the right hand side is just the $L^2(\bbR^d)$ norm of
$\vsig_R^i(\alpha,\cdot)$ and thus  does not depend on $\theta$. Take the sup over $\theta$,  then integrate in $\a$, and sum in $l$. Arguing as above we obtain: \begin{align*}&\int \sup_\theta\Big(\int_{\theta^\perp}  \sup_{s\in \bbR}\big|
\beta_l*\vsig^i_R(\alpha,v^\perp+s\theta) |^2 dv^\perp\Big)^{\frac 12} d\a
\\&\lc \sum_{l\ge 0} 
  2^{l/2} \int \Big(
|\beta_l* \vsig^i_R(\alpha,v) |^2dv \Big)^{1/2} d\a 
\\&\lc \sum_{l\ge 0} 
  2^{l(d+1)/2} \min \{1, (R2^{-l})^M\} \iint 
|\beta_l*\widetilde \vsig^i_R(\alpha,v) | \,dv \,d\a 
\\
&\lc \sum_{l\ge 0} 
  2^{l(d+1)/2} \min \{1, (R2^{-l})^M\} 2^{-l\eps} \|\widetilde \vsig_R^i\|_{\cB_\eps}\lc 
  R^{\frac{d+1}{2}-\eps}  \| \vsig\|_{\cB_\eps}.
\end{align*}

The second inequality in (ii) is proved in the same way. The differentiation in $\alpha_i$ hitting the mollifier $R\varphi(R\cdot)$ produces an additional factor of $R$.
\end{proof}

By the support assumptions on $\vsig$ and $u$, we have
$$\supp {F_{k,i}[\vsig_R^i]}\subseteq\q\{\q(x,y\w): \q|x-y\w|<1\w\}.$$ 
We shall use the following lemma to obtain the bound $C(R) 2^{-k\eps'}$  of the $L^2$ operator norms.
\begin{lemma}\label{LemmaBasicL2HowToEst}
Suppose $V\q(x,y\w)\in L^1_{\mathrm{loc}}\q(\R^d\times\R^d\w)$ is supported in the strip $\q\{\q(x,y\w) : \q|x-y\w|\leq 1\w\}$ and let $\cV$ be the operator with 
 Schwartz kernel $V$. Then,
\begin{equation*}
\LtOpN{\cV}^2\lesssim \sup_z \iint_{\substack{\q|x-z\w|<1 \\ \q|y-z\w|<1}} \q|V\q(x,y\w)\w|^2\: dx\: dy.
\end{equation*}
\end{lemma}
\begin{proof} Let $A$ denote the quantity on the right hand side.
For $\fz\in \bbZ^d$ let $q_\fz$ be the cube $\fz+[0,1]^d$ and $f_\fz=\chi\ci{q_\fz}$.
Then $f= \sum_\fz f_\fz$ and for each $\fz$, $Vf_{\fz} $ is supported in the union 
$q^*_\fz$ of cubes which have a common side with $q_\fz$. By H\"older's inequality  it is immediate that $$\|\cV f_\fz\|_2\lc  \Big(\iint_{q^*_\fz\times q_\fz} |V(x,y)|^2 dx\,dy\Big)^{1/2} \|f_\fz\|_2 \l \le C(d) A \|f_\fz\|_2,$$  and then
\[\|\cV f\|_2=
\Big\|\sum_\fz \cV f_\fz\Big\|_2 \le 3^{d/2}\Big(\sum_\fz 
\big\|\cV f_\fz\big\|_2^2\Big)^{1/2} 
\le C'(d) A 
\Big(\sum_\fz \big\|f_\fz\big\|_2^2\Big)^{1/2}  \le C'(d) A \|f\|_2\,.\qedhere\]
\end{proof}

In light of Lemma \ref{LemmaBasicL2HowToEst} the following proposition gives a basic $L^2$ bound for the operators $T_{k,i}[\vsig_R]$.
\begin{prop} \label{kernelFkiprop}
For $k\ge 0$ 
\begin{equation}\label{EqnBasicL2FinalToShow}
\Big(\sup_{y_0} \iint_{\substack{\q|x-y_0\w|<1 \\ \q|y-y_0\w|<1 }} \q|F_{k,i}[\vsig_R^i]\q(x,y\w)\w|^2\: dx\: dy\Big)^{\frac{1}{2}}\lesssim  2^{-k/3} R^{d+1} \prod_{i=1}^n\|b_i\|_\infty.
\end{equation}
\end{prop}

\subsubsection{Proof of Proposition \ref{kernelFkiprop} }
Note that the class of operators is invariant under translations. 
That is,  if $\tau_a f:=f(x-a)$, then the kernel 
of $\tau_a T_{k,i}[\vsig^i_R]\tau_{-a}$, i.e. $F_{k,i}[\vsig_R^i](x-a,y-a)$, is 
 of the same form of $F_{k,i}$, with the functions $b_j$ replaced by $\tau_a b_j$.
Therefore we may take $y_0=0$ in  Proposition \ref{kernelFkiprop}. We may also 
assume 
\Be\label{normalizationfj}\|b_j\|_\infty \le 1, \quad 1\le j\le n\,.
\Ee
As in \S \ref{SectionAdjoints} we decompose  $\alpha$  as $\alpha=\alpha_i e_i+ \alpha_i^\perp$ where $\alpha_i^\perp=(..., \alpha_{i-1}, \alpha_{i+1},...)\in \R^{n-1}$.
We bound, using the Cauchy-Schwarz inequality in the $z$-variable, and then Minkowski's inequality in the $\alpha_i^\perp $ variables, as well as 
\eqref{normalizationfj} for $j\neq i$,
\begin{align*}
&\iint\limits_{\substack{\q|x\w|<1 \\ \q|y\w|<1 }} \q|F_{k,i}[\vsig_R^i]\q(x,y\w)\w|^2\: dx\: dy\Big)^{\frac{1}{2}}
\\
&\lc 
\int\Big(\iiint\limits_{\substack{|x|,|y|\le 1\\ |y-z|\le 2^{-k}}}
2^{kd}\Big| \int\vsig_R^i(\alpha,x-y) \big[
b_i(x-\alpha_i(x-z))-b_i(x-\alpha_i(x-y))\big] d\alpha_i \Big]^2 dz\,dx\,dy\Big)^{1/2}
d\alpha_i^\perp
\\
&\lc 
\int\Big(2^{kd}
\iiint\limits_{\substack{|x|,|v|,|w|\le 2\\|v-w|\le 2^{-k}}}
\Big| \int \vsig_R^i(\alpha,v) \big[
b_i(x-\alpha_i v)-b_i(x-\alpha_iw)\big]  d\alpha_i\Big|^2 dv\, dw\, dx\Big)^{1/2}d\a_i^\perp
\end{align*}
where for the last integral we have  changed variables to  $v=x-z$, $w=x-y$.
The  proof of Proposition \ref{kernelFkiprop}  will be complete after the following lemma is proved.

\begin{lemma}\label{LemmaBasicL2TheMainTechLemma}
Let $\vsig_R^i$ be as in Proposition \ref{kernelFkiprop}. Then for 
$g\in L^\infty\q(\R^d\w)$ 
and $k>0$,
\begin{equation*}
\Big(2^{kd} \iiint\limits_{\substack{\q|x\w|< 2 \\ \q|v\w|,\q|w|<2\\ \q|v-w|<2^{-k}  } } \Big| \int\vsig_R^i\q(\alpha,v\w) \big(g\q(x-\alpha_i v) - g(x-\a_i w) \big)d\alpha \Big|^2  \: dx\, dv\, dw\Big)^{\frac{1}{2}}   \lesssim 
 R^{d+1-\eps}2^{-k\eps/3}\|\vsig\|_{\cB_\eps}\|g\|_\infty.
\end{equation*}
\end{lemma}
\begin{proof} We may and shall assume $\LpN{\infty}{g}=1$. Let $g_R(x)= g(x) $ if $|x|\le2 R+2$ and $g_R(x)=0$ if $|x|>2R+2$.
We first observe that since   $\vsig^i_R(\alpha,v)=0$ for $|\alpha_i|\ge R+1$
we may replace $g$ by $g_R$ in the above expression. Note that 
\Be\label{L2byLinfty}
\|g_R\|_2\lc R^{d/2}.
\Ee
 We interchange the $(v,w)$- and $x$-integrations, then  apply Plancherel's theorem, and interchange integrals again to get
\begin{align*}
&\int\Big(2^{kd} \iiint\limits_{\substack{\q|x\w|< 2 \\ \q|v\w|,\q|w|<2\\ \q|v-w|<2^{-k}  } } \Big| \int\vsig_R^i\q(\alpha,v\w) \big(g_R\q(x-\alpha_i v) - g_R(x-\a_i w)\big) d\alpha \w\Big|^2  \: dx\, dv\, dw\Big)^{\frac{1}{2}}   d\alpha_i^\perp
\\&= \int 
\Big( \int |\widehat g_R(\xi)|^2 2^{kd} 
\iint\limits_{\substack{ \q|v\w|,\q|w|<2\\ \q|v-w|<2^{-k}  } } \Big| \int\vsig_R^i\q(\alpha,v\w) \big( e^{2\pi \iunit  \alpha_i \inn{v}{\xi}} -e^{2\pi \iunit  \alpha_i \inn{w}{\xi}} \big)
d\a_i\Big|^2  \: dv\, dw\, d\xi\Big)^{\frac{1}{2}}   d\a_i^\perp.
\end{align*}
For a constant $U\ge 1$ (to be determined)  we split the $\xi$-integration into the parts for $|\xi|\le U$ and $|\xi|\ge 1$.

For $|\xi|\le U $ we bound $|e^{2\pi \iunit  \alpha_i \inn{v}{\xi}} -e^{2\pi \iunit  \alpha_i \inn{w}{\xi}} |\lc RU 2^{-k}$ since $|\alpha_i|\le (R+1)$ and $|v-w|\le 2^{-k}$.
Hence we obtain
\begin{align}
&\int 
\Big( \int_{|\xi|\le U} |\widehat g_R(\xi)|^2 2^{kd} 
\iint\limits_{\substack{ \q|v\w|,\q|w|<2\\ \q|v-w|<2^{-k}  } } \Big| \int\vsig_R^i\q(\alpha,v\w) \big( e^{2\pi \iunit  \alpha_i \inn{v}{\xi}} -e^{2\pi \iunit  \alpha_i \inn{w}{\xi}} \big)
d\a_i\Big|^2  \: dv\, dw\, d\xi\Big)^{\frac{1}{2}}   d\a_i^\perp
\label{smallxiest}
\\
\notag
&\lc RU 2^{-k}  \|\widehat g_R\|_2 \int \Big(\int |\vsig_R^i(\a,v)|^2 dv\Big)^{1/2}d\a
\\
&\lc  R^{\frac {d+2}2-\eps} U 2^{-k} 
 \|g_R\|_2 \|\vsig\|_{\cB_\eps}
 \notag
\end{align}
where in the last inequality we have used part (i) of Lemma \ref{estregvsigR}.

Next we consider the part when $\q|\xi\w|>U$.  Using the symmetry in $v,w$ we may estimate

\begin{equation*}
\begin{split}
&\int 
\Big( \int_{|\xi|\ge U} |\widehat g_R(\xi)|^2 2^{kd} 
\iint\limits_{\substack{ \q|v\w|,\q|w|<2\\ \q|v-w|<2^{-k}  } } \Big| \int\vsig_R^i\q(\alpha,v\w) \big( e^{2\pi \iunit  
\inn{v}{\xi}\alpha_i } -e^{2\pi \iunit  \inn{w}{\xi}\alpha_i } \big)
d\a_i\Big|^2  \: dv\, dw\, d\xi\Big)^{\frac{1}{2}}   d\a_i^\perp
\\
&\le 2 
\int 
\Big( \int_{|\xi|\ge U} |\widehat g_R(\xi)|^2 2^{kd} 
\iint\limits_{\substack{ \q|v\w|,\q|w|<2\\ \q|v-w|<2^{-k}  } } \Big| \int\vsig_R^i\q(\alpha,v\w) e^{2\pi \iunit   \inn{v}{\xi}\alpha_i } 
d\a_i\Big|^2  \: dv\, dw\, d\xi\Big)^{\frac{1}{2}}   d\a_i^\perp
\\
&\lc \int 
\Big( \int_{|\xi|\ge U} |\widehat g_R(\xi)|^2 
\int \Big| \int\vsig_R^i\q(\alpha,v\w) e^{2\pi \iunit   \inn{v}{\xi}\alpha_i } 
d\a_i\Big|^2  \: dv\,
 d\xi\Big)^{\frac{1}{2}}   d\a_i^\perp\,.
\end{split}
\end{equation*}

For fixed $\xi=|\xi|\theta$ ($\theta\in S^{d-1}$) we  separate the $v$-integral into two parts. Let $0<b<1$ (which will be optimally chosen later).
For fixed $\theta=\xi/|\xi|$, $\alpha_i^\perp$ we have
$v=\pi_{\theta^\perp}v+s \theta$ where $\pi_{\theta^\perp}v$ is the projection of $v$ to the orthogonal complement of $\bbR\theta$ and $s= \inn{\theta}{v}$.
We split
\begin{align*}
\int \Big| \int\vsig_R^i\q(\alpha,v\w) e^{2\pi \iunit  \inn{v}{\xi}\alpha_i } 
d\a_i\Big|^2  \: dv\,
&=
\int \int \Big| \int\vsig_R^i\q(\alpha,\pi_{\theta^\perp}v+s\theta)\w) e^{2\pi \iunit   s|\xi|\alpha_i }d\alpha_i
\Big|^2 ds dv_{\theta^\perp} 
\\
&=: I_b(\a_i^\perp, |\xi|\theta) + II_b(\a_i^\perp, |\xi|\theta)
\end{align*}
where
\begin{align*}
I_b(\a_i^\perp, |\xi|\theta)&:=
\int \int_{[-b,b]} \Big| \int\vsig_R^i\q(\alpha,\pi_{\theta^\perp}v+s\theta)\w) e^{2\pi \iunit   s|\xi|\alpha_i }d\alpha_i
\Big|^2 ds dv_{\theta^\perp} 
\\
II_b(\a_i^\perp, |\xi|\theta)&:=
\int \int_{[-b,b]^\complement} \Big| \int\vsig_R^i\q(\alpha,\pi_{\theta^\perp}v+s\theta)\w) e^{2\pi \iunit   s|\xi|\alpha_i }d\alpha_i
\Big|^2 ds dv_{\theta^\perp} 
\end{align*}
so that
\begin{align*}
&\int 
\Big( \int_{|\xi|\ge U} |\widehat g_R(\xi)|^2 
\int \Big| \int\vsig_R^i\q(\alpha,v\w) e^{2\pi \iunit   \inn{v}{\xi}\alpha_i } 
d\a_i\Big|^2  \: dv\,
 d\xi\Big)^{\frac{1}{2}}   d\a_i^\perp
\\& \lc 
 \int 
\Big( \int_{|\xi|\ge U} |\widehat g_R(\xi))|^2 
[I_b (\a_i^\perp, \xi)+ 
II_b(\a_i^\perp,\xi)] d\xi\Big)^{\frac{1}{2}}   d\a_i^\perp\,.
 \end{align*}
 The expression $I_b$ is estimated
as
$$
I_b(\a_i^\perp, |\xi|\theta)| \le 2b \int \sup_{|s|\le b} \Big[\int|\vsig_R^i\q(\alpha,\pi_{\theta^\perp}v+s\theta)\w) | d\alpha_i
\Big]^2  dv_{\theta^\perp} 
$$
and we get  using part (ii) of Lemma \ref{estregvsigR}
\begin{align}
&\int 
\Big( \int_{|\xi|\ge U} |\widehat g_R(\xi))|^2 
 I_b (\a_i^\perp, \xi)  d\xi\Big)^{\frac{1}{2}}   d\a_i^\perp
 \notag
\\
&\lc  b^{1/2} \|g_R\|_2 
\int \Big(\sup_\theta \int \sup_s \Big[\int|\vsig_R^i\q(\alpha,\pi_{\theta^\perp}v+s\theta)\w) | d\alpha_i
\Big]^2  dv_{\theta^\perp} \Big)^{1/2} d\alpha_i^\perp
\notag
\\&\lc 
b^{1/2} R^{\frac{d+1}{2}-\eps} \|\vsig\|_{\cB_\eps}
\|g_R\|_2\,.
\label{small-s-est}
\end{align}

To estimate $II_b(\a_\perp,\xi)$   we observe that
the function $\alpha_i\mapsto \vsig_R^i(\alpha, v)$ is smooth and compactly supported.
We use integration by parts
to write
$$
\int\vsig_R^i\q(\alpha,\pi_{\theta^\perp}v+s\theta)\w)
 e^{2\pi \iunit   s|\xi|\alpha_i }d\alpha_i = -\int \partial_{\alpha_i}
 \vsig_R^i\q(\alpha,\pi_{\theta^\perp}v+s\theta)\w)
(2\pi \iunit  |\xi|)^{-1} s^{-1} e^{2\pi \iunit s\alpha_i} d\alpha_i
 $$
 and thus for $|\xi|\ge U$
\begin{align*} II_b(\a_i^\perp, |\xi|\theta)|& \le 
\int_{b}^\infty |\xi|^{-2}|s|^{-2} ds  
\int \sup_t \Big[\int|\partial_{\alpha_i}\vsig_R^i\q(\alpha,\pi_{\theta^\perp}v+t\theta)\w) | d\alpha_i
\Big]^2  dv_{\theta^\perp} 
\\
&\lc U^{-2} b^{-1} 
\int \sup_t \Big[\int|\partial_{\alpha_i}\vsig_R^i\q(\alpha,\pi_{\theta^\perp}v+t\theta)\w) | d\alpha_i
\Big]^2  dv_{\theta^\perp} \,.
\end{align*}
Hence, by the second inequality in part (ii) of Lemma 
\ref{estregvsigR},
\begin{align}
&\int 
\Big( \int_{|\xi|\ge U} |\widehat g_R(\xi))|^2 
II_b (\a_i^\perp, \xi)  d\xi\Big)^{\frac{1}{2}}   d\a_i^\perp
 \notag
\\
&\lc U^{-1}b^{-1/2}  \|g_R\|_2 
\int \Big(\sup_\theta \int \sup_t \Big[\int|\partial_{\alpha_i }\vsig_R^i\q(\alpha,\pi_{\theta^\perp}v+t\theta) | d\alpha_i
\Big]^2  dv_{\theta^\perp} \Big)^{1/2} d\alpha_i^\perp
\notag
\\&\lc 
U^{-1} b^{-1/2} R^{\frac{d+3}{2}-\eps} \|\vsig\|_{\cB_\eps}
\|g_R\|_2.
\label{large-s-est}
\end{align}

We combine \eqref{smallxiest},
\eqref{small-s-est}, \eqref{large-s-est}
to deduce
\Be\label{combiningthree}
\begin{aligned}
&\int\Big(2^{kd} \iiint\limits_{\substack{\q|x\w|< 2 \\ \q|v\w|,\q|w|<2\\ \q|v-w|<2^{-k}  } } \Big| \int\vsig_R^i\q(\alpha,v\w) \big(g_R\q(x-\alpha_i v) - g_R(x-\a_i w) \big)d\alpha \Big|^2  \: dx\, dv\, dw\Big)^{\frac{1}{2}}   d\alpha_i^\perp
\\&\lc \big (    R^{\frac{d+2}{2}-\eps} U 2^{-k}\,+  \, R^{\frac{d+1}{2}-\eps}  b^{1/2} 
\,+\,
 R^{\frac{d+3}{2}-\eps}U^{-1} b^{-1/2}
\big)\|\vsig\|_{\cB_\eps} \|g_R\|_2\,.
\end{aligned}
\Ee
We choose $b,U$ so that the three terms are comparable,
i.e. $b=RU^{-1}$, $U= 2^{2k/3}$. The result is that  the left hand side of 
\eqref{combiningthree} is bounded by a constant times
$$ 
R^{\frac{d+2}{2}-\eps} 2^{-k/3} \|\vsig\|_{\cB_\eps} \|g_R\|_2
\lc
R^{d+1-\eps}
2^{-k/3} \|\vsig\|_{\cB_\eps},
$$ by \eqref{L2byLinfty}, and the proof is complete.
\end{proof}

\subsubsection{Proof of Theorem \ref{ThmBasicL2Resultfirst}}
By \eqref{Tk0est}, $$\|T_{k,0}[\vsig]\|_{L^2\to L^2} \lc 2^{-k\eps} \|\vsig\|_{\cB_\eps} \prod_{l=1}^n\|b_l\|_\infty.$$
By Lemma \ref{vsigRdiffbd} and Proposition \ref{kernelFkiprop} we have
for $i=1,\dots, n$,
\begin{align*}
\|T_{k,i}[\vsig]\|_{L^2\to L^2}  &\le
\|T_{k,i}[\vsig-\vsig^i_R]\|_{L^2\to L^2} +
\|T_{k,i}[\vsig^i_R]\|_{L^2\to L^2}
\\&\lc R^{-\eps}( 1+ 2^{-k/3} R^{d+1}) 
\|\vsig\|_{\cB_\eps} \prod_{l=1}^n \|b_l\|_\infty\,.
\end{align*} 
Choosing  $R=2^{k/(3d+3)}$ yields the bound
$$
\sum_{i=0}^n\|T_{k,i}[\vsig]\|_{L^2\to L^2}  \lc (n+1) 2^{-k\eps /(3d+3)} 
\|\vsig\|_{\cB_\eps} \prod_{l=1}^n \|b_l\|_\infty
$$
and thus the  estimates for the multilinear forms claimed in Theorem \ref{ThmBasicL2Resultfirst}.\qed

\subsection{Generalizations of Theorem
\ref{ThmBasicL2Resultfirst} }
 \label{L2generalizations}

We shall now drop the support assumptions on $x\mapsto \vsig
(\alpha,x)$ and on $u$ in Theorem \ref{ThmBasicL2Resultfirst}.
Moreover, we extend to $L^p$ estimates and replace $\vsig$ by  the scaled versions $\vsig^{(2^j)}$ (with the scaling  in the $x$ variables).

\begin{thm}\label{ThmBasicL2Result}
There exists $c>0$, independent of $n$ and $\eps$, so that the following statement holds for all $1\le p\le\infty$. 
For all 
$\vsig\in \cB_\eps(\R^n\times \R^d)$, for all  $j,k\in \Z$, $1\leq l_1\ne l_2\leq n+2$, $b_{l_1}\in L^2(\bbR^d)$, $ b_{l_2}\in L^{2}\q(\R^d\w)$, $b_l\in L^\infty\q(\R^d\w)$ for $l\ne l_1,l_2$,
and $\uzeta\in \sU$,
\begin{multline*}
\q|\La[\vsigj] (b_1,\dots,b_{l_2-1}, \Qb{k}{\uzeta} b_{l_2}, b_{l_2+1},\ldots, b_{n+2}  )\w|\\
\lesssim \min\{
n 2^{c\eps(j-k)} \sBN{\eps}{\vsig} ,\, \LpN{1}{\vsig} \} \|u\|\ci\sU\Big(\prod_{l\ne l_1,l_2} \LpN{\infty}{b_{l}}\w\Big) \|b_{l_1}\|_2\|b_{l_2}\|_{2}\,.
\end{multline*}
\end{thm}



\begin{proof}
 In light of Theorem \ref{ThmOpResAdjoints}, Theorem \ref{ThmBasicL2Result} follows immediately from Lemma \ref{LemmaBasicL2BasicLpEstimate} and the estimate
(for some $c'>0$, independent of $n$) 
\Be\label{BasicL2for:n+1}
\La[\vsigj](b_1,\ldots,b_n,  \Qb{k}{\uzeta}b_{n+1}, b_{n+2}) 
\lesssim 
\sBN{\eps}{\vsig} n 2^{-c'\eps(k-j)} \sUN{\uzeta}\Big(\prod_{l=1}^n \|b_l\|_\infty\Big) \|b_{n+1}\|_2 
\|b_{n+2}\|_2\,.
\Ee

By  scaling  (Lemma \ref{scalinglemma}) it suffices to prove \eqref{BasicL2for:n+1} for $j=0$.
Theorem \ref{ThmBasicL2Resultfirst}
 covers the case of $\vsig$ supported in 
$\bbR^n \times\{|x|\le 1/4\}$. To cover the general case
we apply Proposition \ref{PropDecomZeta} to write
$u=\sum_{l\ge 0}2^{-l/2} u_l^{(2^{-l})}$
where  $u_l$ is continuous and supported in $\{|x|\le 1/4\}$,  $\int u_l=0$, and $\CzN{\uzeta_l}\lesssim \sUN{\uzeta}$.
We apply Theorem \ref{ThmDecompVsig} to write $\vsig= \sum_{m\geq 0} 
2^{-mc_1\eps} \dil{\vsig_m}{2^m}$
 for some $c_1>0$,
where $\vsig_m\in \cB_{c_1\eps}$, $\|\vsig_m\|_{\cB_{c_1\eps}}\lesssim \sBN{\eps}{\vsig}$, and $\supp{\vsig_m}\subset \q\{\q(\alpha,v\w): \q|v\w|\leq \frac{1}{4}\w\}$.
We then have 
\begin{align*}
&\big|\La[\vsig]\big(b_1,\dots, b_n, \Qb{k}{u}b_{n+1}, b_{n+2}\big)\big|
\\&\le 
\sum_{l\ge 0} \sum_{m\ge 0} 2^{-l/2}2^{-m c_1\eps}
\big|\La[\vsig_m^{(2^{-m})}]\big(b_1,\dots, b_n, \Qb{k}{u_l^{(2^{-l})}}b_{n+1}, b_{n+2}\big)\big|
\\&= 
\sum_{l\ge 0} \sum_{m\ge 0} 2^{-l/2}2^{-m c_1\eps}
\big|\La[\vsig_m]\big(g_1,\dots, g_n, 
\Qb{k-l+m}{u_l}g_{n+1}, g_{n+2}\big)\big|
\end{align*}
where $g_l= b_l(2^m \cdot)$, $l=1,\dots,n$, 
$g_{n+1}= 2^{md/2}b_{n+1}(2^m\cdot),$
$g_{n+2}= 2^{md/2}b_{n+2}(2^m\cdot)$ (see Lemma \ref{scalinglemma}).
By Theorem 
 \ref{ThmBasicL2Resultfirst} we have, for some $c_2>0$
\begin{multline*}
 \big|\La[\vsig_m]\big(g_1,\dots, g_n, 
\Qb{k-l+m}{u_l}(g_{n+1}), g_{n+2}\big)\big|
\\
 \lc \min\{1, n 2^{-(k-l+m)c_2\eps} \} 
 \|u\|\ci{\sU}\Big(\prod_{i=1}^n\|g_i\|_\infty\Big)
 \|g_{n+1}\|_2\|g_{n+2}\|_2\,.
 \end{multline*}
  Now  $\sum_{l\ge 0} \sum_{m\ge 0} 2^{-l/2}2^{-m c_1\eps}
 \min\{1, n 2^{-(k-l+m)c_2\eps} \}  \lc n 2^{-k c_3\eps}$
 for some $c_3$ with $0<c_3<\min \{1/2, c_2\}$ and 
 \eqref{BasicL2for:n+1} for $j=0$  follows easily.
 \end{proof}

\section{Some results from Calder\'on-Zygmund theory}

In this section, we present  some essentially well known results from the Calder\'on-Zygmund theory
which do not seem to be stated in the literature in the precise form we need them.
We begin by recalling some  classical results (see \cite{stein-ha}).

Consider kernels $K\in \cD'(\bbR^d\times\bbR^d)$ such that $K$ is locally integrable on $(\bbR^d\times \bbR^d)\setminus \Delta$; here $\Delta=\diag(\bbR^d\times \bbR^d)=\{(x,x):x\in \bbR^d\}$. Let $T_K: C^\infty_0(\bbR^d) \to \cD'(\bbR^d)$ be the operator with Schwartz kernel $K$. Then  the expression
$$\inn{T_K f}{g} =\iint K(x,y) f(y) g(x) \,dy\, dx$$
makes sense for bounded functions $f$, $g$ with compact and disjoint supports.
For such kernels $K$ we define the singular integral semi-norms
\begin{align}
\label{Hoer1}
\SI^1[K]&:= \sup_{y,y'}\int_{|x-y|\ge 2|y-y'|}|K(x,y)-K(x,y')|\, dx, \\
\label{Hoerinfty}
\SI^\infty[K]&:=\sup_{x,x'}\int_{|y-x|\ge 2|x-x'|}|K(x,y)-K(x',y)| \, dy .
\end{align}

Let $1<q<\infty$. It is a standard and classical theorem (see \cite{stein-ha}) that if $T_K$ extends as a bounded operator on $L^q(\bbR^d)$ and $\SI^1[K]<\infty$ then $T_K$ extends as an operator of weak type $(1,1)$, as an operator mapping the Hardy space $H^1(\bbR^d)$ to $L^1(\bbR^d)$  and as a bounded operator on $L^p$, $1\le p<2$, and one has the following estimates for the operator norms (or quasi-norms).
\Be\label{hardy} \|T_K\|_{H^1\to L^1}+\|T_K\|_{L^1\to L^{1,\infty}} \lc \|T_K\|_{L^q\to L^q}+\SI^1[K].\Ee
We note that in order to prove the $H^1\to L^1$ result, it suffices to check 
$\|T_Ka\|_1\le \|T_K\|_{L^q\to L^q}+\SI^1[K]$ for $q$-atoms, see \cite{MSV}.
Let $L^\infty_0$ be the subspace of $L^\infty$ consisting of functions with compact support (in the sense of distributions). 
Then we also have for $q\ge1$
\Be\label{bmothm} \|T_K\|_{L^\infty_0\to BMO} \lc \|T_K\|_{L^q\to L^q}+\SI^\infty[K].\Ee
Furthermore (taking $q=2$), by  interpolation 
\Be\label{marc}
\|T_K\|_{L^p\to L^p} \le C_{p,d} (\|T_K\|_{L^2\to L^2} + \|T_K\|_{L^2\to L^2}^{2-\frac 2p }
(\SI^1[K])^{\frac 2p-1}), \quad 1<p<2,
\Ee
and
\Be\label{bmointerpol}
\|T_K\|_{L^p\to L^p} \le C_{p,d} (\|T_K\|_{L^2\to L^2} + \|T_K\|_{L^2\to L^2}^{\frac 2p }(\SI^\infty[K])^{1-\frac 2p}), \quad 2<p<\infty.
\Ee

We will apply these results to  singular integral kernels  given by
\Be\label {Ksum}
K= \sum_j \Dilj \tau_j \equiv \sum_j 2^{jd} \tau_j(2^j\cdot, 2^j\cdot)
\Ee
where $\tau_j $ satisfy suitable uniform Schur and regularity conditions.

\subsection{Classes of kernels}

\subsubsection{Schur Norms and Regularity Conditions}\label{schurregularity}

In what follows we consider complex-valued locally integrable functions $(x,y)\mapsto k(x,y)$ on $\bbR^d\times \bbR^d$.

We formulate  conditions related to the usual Schur test, involving integrability conditions in the $x$ and $y$ variables.
We let  $\Sha^1$ 
be the class of kernels $k\in L^1_\loc(\bbR^d\times\bbR^d)$  for which 
\Be\label{schur1}
\Sha^1[k]=\sup_{y\in \bbR^d}\int |k(x,y)| \,dx\Ee
is finite. Here and in what follows  $\sup_y$ is used synonymously with essential supremum (or $L^\infty$-norm).
We let  $\Sha^\infty$ 
be the class of kernels $k\in L^1_\loc(\bbR^d\times\bbR^d)$  for which 
\Be\label{schurinfty}
\Sha^\infty[k]
=\sup_{x\in \bbR^d}\int |k(x,y)| \, dy
\Ee
is finite. Here the supremum is interpreted as essential supremum (i.e. the $L^\infty$ norm with respect to $y$). The notation is motivated by the fact that for 
$k\in \Sha^1$ the integral operator with kernel $k$ is bounded on 
$L^\infty(\bbR^d)$, with operator norm $\Sha^1[k]$,
and
for 
$k\in \Sha^\infty $ this operator is bounded on $L^\infty(\bbR^d)$, with operator norm $\Sha^\infty[k]$.

Next we need stronger conditions, which add some weights in terms of the  distance of $(x,y)$ to the diagonal $\Delta$.
Define 
\begin{align}
\label{schureps1}
\Sha_{\eps}^1[k]&:=
\sup_{y\in \bbR^d}\int (1+|x-y|)^\eps |k(x,y)| \, dx,
\\
\label{schurepsinfty}
\Sha_{\eps}^\infty[k]&:=
\sup_{x\in \bbR^d}\int (1+|x-y|)^\eps |k(x,y)| \, dy.
\end{align} 
 Let 
$$k^\dual(x,y)=k(y,x)$$
and note that
$\Sha_\eps^\infty[k]=\Sha_\eps^1[k^\dual].$

\medskip

In Calder\'on-Zygmund theory we also need some variants involving regularity, in either the left ($x$-) or right ($y$-)variable.
We define
 \begin{align}
\Zhe_{\eps,\rml}^1[k]
\label{reg1epsleft}
&:=  \sup_{0<|h|\leq 1}\sup_y |h|^{-\epsilon}\int|k(x+h,y)-k(x,y)|\: dx,
\\
\label{reg1epsright}
\Zhe_{\eps,\rmr}^1[k]
&:=  \sup_{0<|h|\leq 1}\sup_y |h|^{-\epsilon}\int|k(x,y+h)-k(x,y)|\: dx,
\end{align}
and
\begin{align}
\label{reginftyepsleft}
\Zhe_{\eps,\rml}^\infty[k]
&:=  \sup_{0<|h|\leq 1}\sup_x |h|^{-\epsilon}\int|k(x+h,y)-k(x,y)|\: dy,
\\
\label{reginftyepsright}
\Zhe_{\eps,\rmr}^\infty[k]
&:=  \sup_{0<|h|\leq 1}\sup_x |h|^{-\epsilon}\int|k(x,y+h)-k(x,y)|\: dy,
\end{align}
so that
$\Zhe_{\eps,\rml}^\infty[k]=
\Zhe_{\eps,\rmr}^1[k^\dual]$ and 
$\Zhe_{\eps, \rml}^\infty[k]=
\Zhe_{\eps,\rmr}^1[k^\dual]$.

\subsubsection{Singular Integral Kernels}\label{SIclasses}
We now consider distributions $K\in \cD'((\bbR^d\times \bbR^d)\setminus \Delta)$ which are locally integrable in $(\bbR^d\times\bbR^d)\setminus \Delta$.
We define variants of \eqref{Hoer1}, \eqref{Hoerinfty} 
 with more decay away from the diagonal (here $\eps\ge 0$)
\begin{align}
\label{Hoer1eps}
\SI_\eps^1[K]&:=\sup_{y,y'}
\sup_{R\ge 2} R^\eps
\int_{|x-y|\ge R|y-y'|}|K(x,y)-K(x,y')| \,dx \,,
\\ \label{Hoerinftyeps}
\SI_\eps^\infty[K]&:=\sup_{x,x'}\sup_{R\ge 2} R^\eps
\int_{|y-x|\ge R|x-x'|}|K(x,y)-K(x',y)|\, dy \,.
\end{align}
Note that for $\eps=0$ we recover the norms defined in \eqref{Hoer1}, \eqref{Hoerinfty}.
\begin{rem}
We shall also use the   alternative notation
$\|K\|_{\SI^1_\eps}=\SI^1_\eps[K]$ etc.
We will say 
$K\in \SI^1_\eps$ if $\SI^1_\eps[K]<\infty$ etc.
\end{rem}

We say that 
$K\in L^1_\loc((\bbR^d\times\bbR^d)\setminus\Delta)$ 
satisfies one of the {\it uniform annular integrability conditions} $\ann^1$, $\ann^\infty$
if 
the respective expressions
\begin{align}\label{annular1}
\ann^1[K]&:=  \sup_{R>0}\sup_y
 \int_{x: R\le |x-y|\le 2R}|K(x,y)| \, dx,
 \\
 \label{annularinfty}
 \ann^\infty[K] &:=
 \sup_{R>0} \sup_x \int_{y:R\le |x-y|\le 2R}|K(x,y)| \, dy
\end{align}
are finite.

We say that $K$ satisfies the 
{\it averaged  annular integrability condition} $\ann_{\av}$ if
\Be\label{aveannular}
\ann_{\av}[K]=
\sup_{a\in \bbR^d}\sup_{R>0} R^{-d}\iint\limits_{\substack {|x-a|\le R\\ R\le|x-y|\le 2R}}
|K(x,y)| \,dy \,dx 
\Ee
is finite.

The last notion will be used in \S\ref{Journesect} below.

\begin{lemma} Let $K\in L^1_\loc((\bbR^d\times\bbR^d)\setminus\Delta\})$. Then
$$\ann_\av[K] \approx \ann_\av[K^\dual].$$
Moreover, $$\ann_\av[K]\lc \min \{\ann^1[K], \ann^\infty[K]\}\,.$$
\end{lemma}
\begin{proof} Immediate from the definitions.\end{proof}

\begin{lemma} \label{SIepsvs0}
Let $K\in L^1_\loc ((\bbR^d\times \bbR^d )\setminus \Delta)$.
Suppose that for some $\eps>0$,
$$\SI^1_\eps[K] \le B, \quad \ann[K]\le A.$$ Then
$$
\SI^1_0[K] \lc A \log (2+\eps^{-1} B/A).
$$ 
\end{lemma}
\begin{proof} 
Fix $y\neq y'$ and split
$$ \int_{|x-y|\ge 2|y-y'|}|K(x,y)-K(x,y')| \,dx  =I+II$$
where
\begin{align*}
I&= \int_{2|y-y'| \le |x-y|\le R|y-y'|}|K(x,y)-K(x,y')|\, dx \,,
\\
II&= \int_{ |x-y|\ge R|y-y'|}|K(x,y)-K(x,y')| \, dx \,.
\end{align*}
Then if we apply condition $\ann_1$ with $O(\log R)$ annuli  to estimate
$$I\lc A \log R;$$ 
moreover we have
$$II\lc B R^{-\eps} .$$
If we choose $R=2+(B/A)^{1/\eps}$ the assertion follows.
\end{proof}

\subsubsection{Integral conditions for singular integrals}
We formulate a  proposition which is used to verify
the condition $\SI^1_\eps$ , $\SI^\infty_\eps$  
for kernels of the form \eqref{Ksum}.

\begin{prop} \label{ShaZheimplSI}
Suppose that $\tau_j \in \Sha^1_\eps\cap \Zhe^1_{\eps, R}$ and 
\begin{gather*}
 \sup_j \Sha_0^1[\tau_j]\le A,
 \\
  \sup_j \Sha^1_\eps[\tau_j] +
   \sup_j \Zhe^1_{\eps,\rmr}[\tau_j]\le B.
\end{gather*}
Then the sum \eqref{Ksum} converges in the sense of $L^1_\loc ((\Bbb R^d\times \Bbb R^d)\setminus\Delta) $ and the limit $K$ satisfies
\Be \label{SIepshalf} \SI^1_{\eps/2}[K] \lc  B.\Ee
 Moreover,
\Be \label{SI-log}
\SI^1_0[K] \lc A \log (2+B/A)\,.
\Ee
\end{prop}
\begin{proof}
We fix $y,y'$ and $R\ge 0$ and consider
\begin{align*}I_j^R(y,y')&=\int_{x: |x-y|\ge R|y-y'|}|\Dilj \tau_j(x,y)- \Dilj \tau_j(x,y')|\,dx
\\
&=\int_{x: |x-2^jy|\ge R|2^jy-2^jy'|}|\tau_j(x,2^j y)-  \tau_j(x,2^jy')|\,dx.
\end{align*}
Clearly $I_j^R(y,y')\le 2A$. 
We now  give two estimates, the first valid when
$2^j|y-y'|\ge 1/R$, the second valid when  $2^j|y-y'|\le 1$;
thus both estimates will be valid when $1/R\le 2^j|y-y'|\le 1$.

For $2^j|y-y'|\ge 1/R$ we have 
\begin{align*}
&\int_{x: |x-2^jy|\ge R|2^jy-2^jy'|}
|\tau_j(x,2^j y)|dx
\le 
\int
|\tau_j(x,2^j y)| 
\frac{(1+|x-2^jy|)^\eps}{(R2^j|y-y'|)^\eps}
dx
\\
&\le (2^j|y-y'|R)^{-\eps} 
 \Sha^1_\eps[\tau_j] \le B (2^j|y-y'|R)^{-\eps} .
\end{align*}
Also note that if $|x-2^jy|\ge R|2^jy-2^jy'|$ then also 
$|x-2^jy'|\ge (R-1)|2^jy-2^jy'|$.  Thus the last  argument also gives (for $R\ge 2$)
\[\int_{x: |x-2^jy|\ge R|2^jy-2^jy'|}|\tau_j(x,2^j y')|dx\le 
B (2^j|y-y'|(R-1))^{-\eps} \]
and hence 
$$ I_j^R(y,y')\lc B (2^j|y-y'|)^{-\eps} R^{-\eps} \text { if } 2^j|y-y'|\ge 1/R\,.$$ 

For $2^j|y-y'|\le 1$  we  obtain 
$$
I_j^R(y,y')
\le \int |\tau_j(x,2^j y)-  \tau_j(x,2^jy')|\,dx \le \Zhe_\eps^1[\tau_j]
(2^j|y-y'|)^\eps 
\le B (2^j|y-y'|)^\eps .
$$
Hence 
$$\sum_{j\in \bbZ} I_j^R(y,y') \lc  
\sum_{j: 2^j|y-y'|\le R^{-1/2}}B (2^j|y-y'|)^\eps 
+
\sum_{j: 2^j|y-y'|> R^{-1/2}}B (R2^j|y-y'|)^{-\eps} \lc BR^{-\eps/2}
$$
and \eqref{SIepshalf} follows.
The same argument gives
$$\sum_{j\in \bbZ} I_j^R(y,y') \lc  
\sum_{j\in \bbZ}\min \{A, B (2^j|y-y'|)^{\eps}, B (2^j|y-y'|)^{\eps} \lc
A(\log(2+B/A))
$$
which yields \eqref{SI-log}.

\end{proof}


The following proposition is useful for verifying membership in the classes $\ann^1$, $\ann^\infty$ for kernels 
of the form \eqref{Ksum}.

\begin{prop}\label{ShaZheimplAnn}
 Suppose that $\tau_j \in \Sha^1_\eps\cap \Zhe^1_{\eps, \rml}$ such that 
\begin{gather*}
 \sup_j \Sha^1_0[\tau_j]\le A\, ,
 \\
  \sup_j \Sha^1_\eps[\tau_j] + \sup_j \Zhe^1_{\eps,\rml}[\tau_j]\le B\,.
\end{gather*}
Then the sum $K=\sum_j\Dilj \tau_j$ converges in the sense of $L^1_\loc ((\bbR^d\times \bbR^d)\setminus\Delta) $ and 
$$\ann^1[K] \lc A  \log(2+B/A)\,.$$
\end{prop}
This follows from the following lemma regarding functions in $L^1(\bbR^d)$.

\begin{lemma}
\label{lemmaann}
Let $0<\eps<1$,  $g_j\in L^1(\bbR^d)$ such that
\[
\int|g_j(x)| \,dx \le A,
\]
\[ \int|g_j(x)|(1+|x|)^\eps \,dx \le B_1,
\]
and 
\[\sup_{|h|<1}|h|^{-\eps} \int|g_j(x+h)-g_j(x)|\, dx \le B_2\,.\]
Then for every compact set $K\subset \bbR^d\setminus \{0\}$, the series 
$G(x)=\sum_{j\in \bbZ} 2^{jd} g_j(2^j x)$ converges in $L^1(K)$, so that $G\in L^1_{\loc}(\bbR^d\setminus\{0\})$. Moreover, if 
$K_R= \{x:R\le |x|\le 2R\}$,
\[ \sup_{R>0}\int_{K_R} |G(x)| dx \lc A \log(1+ \frac{B_1+B_2}{A})\,.
\]
\end{lemma}
\begin{proof}
It suffices to consider the case $K=K_R$. Let $G_j=2^{jd}g_j(2^j\cdot)$
then $$\|G_j\|_{L^1(K_R)}=\|g_j\|_{L_1(K_{2^jR})}\le A.$$
First assume that $2^j R\ge 1$.  In this case 
$$\|g_j\|_{L_1(K_{2^jR})}\lc  (2^j R)^{-\eps}B_1.$$
For $2^j R\le 1$ we have by H\"older's inequality
$$\|g_j\|_{L_1(K_{2^jR})}\lc (2^jR)^{d/p'} \|g_j\|_p,$$
and by Sobolev imbedding it follows 
$ \|g_j\|_p\lc B_1$ provided that $d/p' <\eps$. 
Hence we obtain for  $0<\eps'<\eps$ we get 
\[\|G\|_{L^1(K_R)}\lc \sum_{j\in \bbZ} \min\{ A, B_1 (2^jR)^{-\eps}, B_2 (2^j R)^{\eps'}\} \lc A \log \big(1+ \frac{B_1+B_2}{A}\big)\,.\qedhere \]
\end{proof}

\begin{proof}[Proof of Proposition \ref{ShaZheimplAnn}]
Apply Lemma \ref{lemmaann} to the  functions
$v\mapsto K(y+v,y)$.
\end{proof}

\subsubsection{Kernels with cancellation}
We state a standard estimates involving the Schur test for compositions with operators exhibiting some cancellation; this will be  used when proving $L^2$ estimates in \S\ref{Sectionpartt1}.

\begin{lemma}\label{LemmaBoundT1RonS}
Fix $0<\eps\le1$.  Let $\ell\in \Z$ with $\ell\leq 0$.  Suppose $\rho,\, \sigma_\ell:\R^d\times \R^d\rightarrow \bbC$ are measurable functions satisfying
\begin{subequations}
\begin{align}
\label{EqnBoundR}
&\Sha^1[\rho]\le A_1, \quad \,\Sha_\eps^\infty[\rho] \le A_{\eps,\infty},
\\ \label{EqnBoundSell}
&\Sha^1[\sigma_\ell] \le B_1,\, \quad\Sha^\infty[\sigma_\ell] \le B_\infty,
\end{align}
and 
\Be\label{EqnBoundnablaxSell}
\Sha^\infty [\nabla_x \sigma_\ell ] \le 2^{-\ell}\widetilde B_\infty.
\Ee
\end{subequations}
Assume 
\Be \label{rhocancel}\int \rho(x,y)\: dy =0 \, \text{ for almost every $x\in\bbR^d$}.\Ee
Let $R$, $S_\ell$ be the integral operators with Schwartz kernels $\rho(x,y)$, $\sigma_\ell(x,y)$.
Then \begin{equation*}
\LpOpN{2}{R S_\ell}\lesssim 2^{-\ell\eps/2} \sqrt{A_1 A_{\eps,\infty} B_1(B_\infty+\widetilde B_\infty)}.
\end{equation*}
\end{lemma}
\begin{proof}
 Let $k_{\ell} $ be the Schwartz kernel of $RS_\ell$. Then, by the cancellation assumption,
 \[k_{\ell}(x,y)
 =\int \rho(x,z) \big(\sigma_\ell(z,y) -\sigma_\ell(x,y)\big) \,dz\]

 Clearly for a.e. $y\in\bbR^d$
\Be\notag
\int |k_{\ell}(x,y)| dx \le \int |\sigma_\ell(z,y)|\int|\rho(x,z)| dx\,dz \lc B_1A_1.
\Ee
Moreover,
\[
\int |k_{\ell}(x,y)| dy \le (I_x)+(II_x)
\]
where
\begin{align*}
(I_x)&:= \int_{|x-z|\le 2^{\ell}}| \rho(x,z)|\int\big|\sigma_\ell(z,y) -\sigma_\ell(x,y)\big| \,dy\, dz,
\\
(II_x)&:= \int_{|x-z|\ge 2^{\ell}}| \rho(x,z)|\int \big(|\sigma_\ell(z,y)|+|\sigma_\ell(x,y)| \big) \,dy\, dz.
\end{align*}
Now by assumption, for fixed $x,z$
\[
\int |\sigma_\ell(z,y)| \:dy\,+\, \int | \sigma_\ell(x,y)|  dy \lc B_\infty\]
and
\begin{align*}\int\big|\sigma_\ell(z,y) -\sigma_\ell(x,y)\big| \,dy&=\int\Big|
\int_0^1 \inn{z-x}{\nabla_x \sigma_\ell((1-s)x+s z),y} \, ds\Big|\: dy
\\
&\le |x-z| \int_0^1 \int 
|\nabla_x \sigma_\ell((1-s)x+s z),y)| \:dy\, d\tau \lc \widetilde B_\infty 2^{-\ell}|x-z|.
\end{align*}
For  $(I_x)$ we then get
\[
(I_x) \le \widetilde B_\infty \int_{|z-x|\le 2^{\ell}} |\rho(x,z)|2^{-\ell}|x-z| \,dz
\]
and  estimate (using  $\eps\le 1$)
\begin{align*} 
&\int_{|z-x|\le 2^{\ell}} |\rho(x,z)|2^{-\ell}|x-z| \,dz
\le \int_{|z-x|\le 2^{\ell}} |\rho(x,z)|[2^{-\ell}|x-z|]^\eps \,dz
\\ &\lc 2^{-\ell\eps }\int |\rho(x,z) (1+|x-z|)^\eps dz \lc 2^{-\ell\eps} A_{\eps,\infty}.
\end{align*}
Hence  $(I_x)\lc 2^{-\ell\eps} \widetilde B_\infty A_{\eps,\infty}$.
 For $(II_x)$ we have 
$$(II_x) \le  B_\infty  \int_{|z-x|\ge 2^{\ell}} |\rho(x,z)| \,dz 
\lc B_\infty 2^{-\ell\eps} \int_{|z-x|\ge 2^{\ell}} |\rho(x,z)| (1+|x-z|)^\eps \,dz $$
and thus $(II_x)\lc  2^{-\ell\eps}  B_\infty A_{\eps,\infty}$.
Finally,
we obtain by  Schur's test
$$\|RS_\ell\|_{L^2\to L^2} \le \sqrt{\Sha_1[k_\ell]} \sqrt{\Sha_\infty[k_\ell]}\lc
 \sqrt{A_1B_1} \sqrt{(B_\infty+\widetilde B_\infty)A_{\eps,\infty} 2^{-\ell\eps}}.$$ The assertion is proved.
  \end{proof}


\subsubsection{On operator topologies}
We finish this section by stating  a version of the uniform boundedness principle which is used for the partial sums of operators defined by kernels of the form \eqref{Ksum}.

\begin{lemma} \label{UBPlemma}
Let $X$, $Y$ be Banach spaces and let $\Sigma_N:X\to Y$ be bounded operators. Assume that $\Sigma_N$ converges in the weak operator topology, i.e. there is a linear operator $\Sigma:X\to Y$ so that for every $f\in X$ and every linear functional  $g\in Y'$, 
$$\lim_{N\to\infty} \inn {\Sigma_Nf}{g}=\inn{\Sigma f}{g}.$$
Then $\Sigma:X\to Y$ is  bounded, and there exists  $B<\infty$ so that
$$\|\Sigma\|_{X\to Y}\le \sup_N \|\Sigma_N\|_{X\to Y}\le B.$$
\end{lemma} 
\begin{proof}
We have $\sup_N\|\inn{\Sigma_N f}{g}|\le C_{f,g}<\infty $ for every $f,\in X$, $g\in Y'$.
By the uniform boundedness principle this implies
$\sup_N \|\Sigma_N f\|_Y \le C_f<\infty$ for all $f\in X$.
By the uniform boundedness principle again there is $A<\infty$ so that
$A:=\sup_N \|\Sigma_N \|_{X\to Y}  <\infty$.
Thus $C_{f,g} \le A \|f\|_X\|g\|_{Y'}$.  Passing to the limit we see
$|\inn{\Sigma f}{g}|\le A\|f\|_X \|g\|_{Y'}$ which implies $\|\Sigma \|_{X\to Y}\le A$.
\end{proof}

Given a formal series $\sum_{j \in \bbZ} T_j$ of bounded operators we say that
$\sum_{j\in \bbZ}  T_j$ converges in the weak operator topology as operators $X\to Y$
if the partial sums $\Sigma_N=\sum_{j=-N}^N T_j$ satisfy the assumptions in Lemma
\ref{UBPlemma}.

\begin{lemma} \label{densenesslemma}
Let $X$, $Y$ be Banach   spaces, let $W$ be a linear subspace of $X$ which is dense in $X$.
 Let  $\Sigma_N:X\to Y$ be bounded operators. Assume that 
 $$\sup_N\|\Sigma_N\|_{X\to Y} \le A$$ and that
 for every $f\in W$, and every $g\in Y'$ 
$$\lim_{N\to\infty} \inn {\Sigma_Nf}{g}=\inn{\Sigma f}{g}$$
where $\Sigma:W\to Y$ is a linear operator. Then $\Sigma_N$ converges to $\Sigma$ in the weak 
 operator topology (as operators $X\to Y$) and we have $\|\Sigma\|_{X\to Y} \le A$.
\end{lemma} 
\begin{proof} The assumptions imply that $\|\Sigma f\|_Y \le \|f\|_X$ for all $f\in W$, and $\Sigma$ 
extends uniquely to a bounded operator $X\to Y$ with operator norm at most $A$.
Moreover, using $\|\Sigma_N\|_{X\to Y} \le A$ it follows easily that $\Sigma_N\to\Sigma$ in the weak operator topology.
\end{proof}

\subsubsection{Consequences for sums of dilated  kernels}
We now formulate some consequences of the propositions above and the boundedness result \eqref{marc}.

%

\begin{prop}\label{CorWeakTypeInterp} 
Let $\tau_j\in \Sha^1_\eps\cap \Zhe^1_{\eps, \rmr}$, so that
$$\Sha^1_0[\tau_j]\lc A , \quad \Sha^1_\eps[\tau_j] +\Zhe^1_{\eps, \rmr}[\tau_j]\le 
B.$$
Let  $T_j$ denote the integral operator with kernel $\Dilj \tau_j$.
 
 (i) Suppose that 
$T=\sum_{j\in \Z} T_j$ 
converges in the weak operator topology as operators $L^2\rightarrow L^2$.
Then, for $1<p\le 2$, $T$ extends to an operator bounded on $L^p$
such that
\begin{equation*}
\|T\|_{L^p\rightarrow L^p}\leq C_{d,p,\epsilon} 
\big(  \|T\|_{L^2\rightarrow L^2} 
+ \|T\|_{L^2\rightarrow L^2} ^{2-\frac 2p}\big(A  \log (2+B/A)
\big)^{\frac 2p-1}\big).
\end{equation*}

Moreover $T$ extends to an operator bounded from $H^1$ to $L^1$ and 
\begin{equation*}
\|T\|_{H^1\rightarrow L^1}\leq C_{d,\epsilon} 
\big(  \|T\|_{L^2\rightarrow L^2} 
+ A  \log (2+B/A)\big).
\end{equation*}

(ii) Suppose that 
$T=\sum_{j\in \Z} T_j$ 
converges in the strong operator topology, as operators $L^2\rightarrow L^2$.
Then  the sum also converges 
in the strong operator topology as operators $L^p\rightarrow L^p$, $1<p<2$
and
in the strong operator topology as operators $H^1\rightarrow L^1$.
\end{prop}
\begin{proof}
By Proposition 
\ref{ShaZheimplSI}
 we have 
for $K$ as in \eqref{Ksum}
$\SI_{0}^1[K] \le \log (2+B/A)$ and the assertion (i)  follows from \eqref{marc} and \eqref{hardy}.

For (ii) we examine the proof of $H^1\to L^1$ boundedness.
Let $a$ be a $2$-atom supported in a cube $Q$ with center $y_Q$, i.e. we have $\|a\|_2\le |Q|^{-1/2}$, $\int a(x) dx=0$.
Let $Q^*$ be the double cube with the same center. 
By assumption 
$\sum_{j=-N}^N T_j a$ converges in $L^2(Q^*)$  and by H\"older's inequality in $L^1(Q^*)$.
Also, by the argument 
 in the proof of 
Proposition \ref{ShaZheimplSI},
\begin{align*}\|T_j a\|_{L^1(\bbR^d\setminus Q^*)}&\lc \int|a(y)| \int_{\bbR^d\setminus Q^*}
|\Dilj \tau_j(x,y)-\Dilj \tau_j(x,y_Q)| \,dx\, dy
\\ &\lc B \min \{ (2^j\diam (Q))^{\eps},  (2^j\diam (Q))^{-\eps}  \}
\end{align*} 
and clearly $\sum_{j=-N}^N T_j a$ converges in $L^1(\bbR^d\setminus Q^*)$ as well.

Let $f\in H^1$;  we need to establish convergence of $\sum_j T_jf$ in $L^1$.
By the atomic decomposition $f=\sum_{\nu=1}^\infty c_\nu a_\nu$ where $a_\nu$ are $2$-atoms  and $\sum_\nu |c_\nu|\lc \|f\|_{H^1}$. Given $\eps>0$ take $M$ so that
$\sum_{\nu=M}^\infty |c_\nu|\le \eps$.
Then there is $C$ independent of $M$, $\eps$ so that for all $N$ we have
$$\Big\|\sum_{j=-N}^N T\big( \sum_{\nu=M}^\infty c_\nu a_\nu\big)\Big\|_1<C\eps.$$ 
It is now straightforward to combine the arguments and deduce the convergence of
$\sum_j T_j f$ in $L^1$. 

In order to prove convergence in the strong operator topology as operators $L^p\to L^p$, $1<p<2$,
we apply the interpolation inequality $\|h\|_p\le 
\|h\|_1^{\frac 2p-1}\|h\|_2^{2- \frac 2p}$ to 
$h=\sum_{j\in \cJ} T_j g$ where $g\in H^1\cap L^2$. This yields that 
 $\sum_j T_j g$ converges in $L^p$. Since $H^1\cap L^2$ is dense 
and since the operator norms $\sum_{j\in \cJ} T_j$ are bounded uniformly in $\cJ$, it is now straightforward to show convergence of $\sum_j T_jf$ for every $f\in L^p$.
\end{proof}

In our applications we work with the following
setting.
Let $\phi\in C_0^\infty(B^d(1))$ have $\int \phi=1$ and define $P_j f = f*\dil{\phi}{2^j}$.
Set $\psi(x)= \phi(x)-2^{-d}\phi(2^{-1}x)$, and set $Q_j f = f*\dil{\psi}{2^j}$.
We have $I=\sum_{j\in \Z} Q_j$, $P_j=\sum_{k\leq j} Q_k$ and $I-P_j= \sum_{k>j} Q_k$ in the sense of distributions.


\begin{cor} 
\label{ThmWeakTypeWithP}
Let $s_j:\bbR^d\times\bbR^d\to \bbC$ be a sequence of locally integrable kernels  and assume that
$$\sup_j \Sha^1_0[s_j] \le A, \quad
\sup_j \Sha^1_\eps[s_j] \le B\,.
$$
Let $S_j$ be the integral operator with integral kernel $\Dilj s_j$.
Suppose the sum $S=\sum_{j\in \Z} S_jP_j$ converges in the weak operator topology as operators $L^2\rightarrow L^2$.  Then, for $1<p\le 2$, 
$S:L^p\rightarrow L^p$ is bounded  and
\begin{equation*}
\| S\|_{L^p\rightarrow L^p} \leq C_{d,p,\epsilon} \big( \|S\|_{L^2\rightarrow L^2} + 
\|S\|_{L^2\to L^2}^{2-\frac 2p} (A\log( 2+B/A))^{\frac 2p-1}\big).
\end{equation*}
\end{cor}
\begin{proof}
The kernel of $S_jP_j$ is equal to $\Dilj \tau_j$ where
$$\tau_j(x,y)= \int s_j(x,z) \phi(z-y) \,dz\,.$$
Clearly $\Sha_\eps^1[\tau_j]\lc \Sha_\eps^1[s_j]$ for $\eps\ge 0$ and in view of the regularity and support of $\phi$ we also have
$$\Zhe^1_{\delta,\rmr}[\tau_j]\lc \Sha^1_0[s_j]$$ for $\delta\le 1$. The assertion now follows from Corollary \ref{CorWeakTypeInterp}.
\end{proof}

\begin{cor}
\label{PropWeakTypeWithQ}
Let $s_j$, $S_j$ be  as in Corollary \ref{ThmWeakTypeWithP}
For $k\in \N$ define $S^k := \sum_{j\in \bbZ} S_j Q_{j+k}$.  Suppose that
this sum converges in the weak operator topology as
operators $L^2\rightarrow L^2$, and suppose that
for some $\eps'>0$
\begin{equation*}
D_{\eps'}:=\sup_{k>0} 2^{k\epsilon'} \| S^k\|_{L^2\rightarrow L^2}<\infty.
\end{equation*}
Also define
$D_0:=\sup_{k>0} \| S^k\|_{L^2\rightarrow L^2}.$
Then, for $1<p\le 2$, 
\[\| S^k\|_{L^p\rightarrow L^p} \leq C_{p,d,\epsilon} \Big(
\min\{  2^{-k\epsilon'} D_{\epsilon'} , D_0  \}
+ 
\big(\min\{  2^{-k\epsilon'} D_{\epsilon'} , D_0  \}\big)^{2-\frac 2p}
\big(A 
\log (2^k +B/A) \big)^{\frac 2p-1}\Big).
\]
\end{cor}

\begin{proof} By definition $\|S^k\|_{L^2\to L^2}\le
\min\{  2^{-k\epsilon'} D_{\epsilon'} , D_0  \}$.
The integral kernel of $S_jQ_{j+k} $ is given by $\Dilj \tau_{j,k}$ where
$$\tau_{j,k}(x,y)= \int s_j(x,z) 2^{kd}\psi(2^k(z-y)) \, dz\,.$$
We have  $\Sha_\eps^1[\tau_{j,k}]\lc \Sha_\eps^1[s_j]$ for $\eps\ge 0$ and 
now $$\Zhe^1_{\delta,\rmr}[\tau_{j,k}]\lc 2^k \Sha^1_0[s_j] \lc 2^kA$$ for $\delta\le 1$. The assertion follows from Corollary \ref{CorWeakTypeInterp}.
\end{proof}

\begin{cor}\label{CorWeakTypeWithQ}
Let
$s_j$, $S_j$, $S^k$ be as in
Corollary \ref{PropWeakTypeWithQ}.
Define $\tS:=\sum_{j\in \Z} S_j(I-P_j) = \sum_{k>0} S^k$.  For $1<p\leq 2$ ,
\[\|\tS\|_{L^p\rightarrow L^p} \le C_{p,d,\epsilon,\epsilon'} 
\Big( D_0 \log \big(2+ \frac{D_{\epsilon'}}{D_0}\big)  +D_0^{2-\frac 2p} A^{\frac 2p-1} 
\log \big(2+\frac{D_{\epsilon'}}{D_0}\big) \log^{\frac 2p-1} \!\big( 2+ \frac{D_{\epsilon'} }{D_0}+\frac BA\big)  \Big).
\end{equation*}
\end{cor}
\begin{proof}  
By Corollary  \ref{PropWeakTypeWithQ}, we have
\begin{equation*}
 \| \tS\|_{L^p\rightarrow L^p} \lesssim \sum_{k>0}   \min\{2^{-k\epsilon'} D_{\epsilon'} , D_0\}  + \sum_{k>0}
 \big( \min\{2^{-k\epsilon'} D_{\epsilon'} , D_0\}\big)^{2-\frac 2p}   \big( A \log
 ( 2^k +B/A) )^{\frac 2p-1} .
\end{equation*}
Clearly, $\sum_{k>0} \min\{2^{-k\epsilon'} D_{\epsilon'} ,D_0 \}\lesssim D_0 \log(2+D_{\eps'}/D_0)$. 
Also, the second sum equals
\begin{equation*}
D_0^{2-\frac 2p} A^{\frac 2p-1} \sum_{k>0} \min\big 
\{2^{-k\epsilon'} \frac{D_{\epsilon'}}{D_0}, 1\big\}^{2-\frac 2p} \big(\log(2^k +\frac BA)\big)^{\frac 2p-1}.
\end{equation*}
To conclude apply the following  Lemma
\ref{LemmaBasicSum} with $\beta= -1+2/p$.
\end{proof}

\begin{lemma}\label{LemmaBasicSum}
Fix $\epsilon>0$,  $\alpha>0$, $\beta\ge 0$. Let $U,V\geq 1$, then
\begin{equation*}
\sum_{k\geq 0}(\min\{2^{-k\epsilon} U, 1\})^\alpha \log^{\beta}\!(2^k+ V) 
\leq C_{\epsilon, \alpha,\beta} \log(1+U) \log^{\beta} (1+U+V).
\end{equation*}
\end{lemma}
\begin{proof}
Let $J_k(U,V)=  (\min\{2^{-k\epsilon} U, 1\})^\alpha 
\log^{\beta}\!(2^k+ V) $.

We first consider the terms with 
$2^{-k\eps/2}U\le 1$.
Observe 
$$\sum_{\substack {2^{-k\eps/2}U
 \le 1\\ 2^k\le V}}
J_k(U,V) 
\lc \log^\beta\!(1+V) \sum_{2^{k\eps/2} \le U} (U 2^{-k\eps})^\alpha \lc
 \log^\beta\!(1+V)
$$
and 
$$
\sum_{\substack {2^{-k\eps/2}U \le 1\\ 2^k> V}}
J_k(U,V) \lc \sum_{2^{-k\eps/2} U\ge 1} (U 2^{-k\eps})^\alpha k^\beta
\lc \sum_{k: 2^{-k\eps/2}U \le 1} ( U 2^{-k\eps/2})^{\alpha} \,\lc 1\,.
$$
The main contribution comes from the terms with $2^{-k\eps/2}U\ge 1$; here we use
$$\sum_{\substack {2^{-k\eps/2}U
 \ge 1\\ 2^k\le V}}
J_k(U,V) \lc \log^\beta\!(1+V) \sum_{2^{k\eps/2} \le U} 1 \lc \log(1+U)
\log^\beta\!(1+V)
$$
and
$$\sum_{\substack {k: 2^{-k\eps/2} U\ge 1
 \\ 2^k\ge  V}}
J_k(U,V) \lc \sum_{k:2^{k\eps/2}\le U} k^{\beta} \lc \log^{\beta+1}(1+U)\,.
$$
Clearly, all four terms are $\lc  \log(1+U)\log^\beta(1+U+V)$ and the asserted bound follows.
\end{proof}

\subsection{On a result of Journ\'e}\label{Journesect}
For a cube $Q$ let $Q^*$ be the double cube with same center.

\begin{defn}\label{Carlesoncond}
Let $T: C^\infty_0(\bbR^d)\to \cD'(\bbR^d)$ be an operator with Schwartz kernel $K$. We say that $T$ satisfies a Carleson condition if there is a constant $C$ so that  for all cubes $Q$ and for all
bounded functions $f$ supported in $Q$, $Tf\in L^1(Q^*)$ and the inequality
$$\int_{Q^*} |Tf(x)|dx \le C|Q| \|f\|_\infty$$ is satisfied.
We denote by $\|T\|_{\Carl} $ the best constant in the displayed inequality.
\end{defn}
\medskip

Journ\'e \cite{journe}
considered a class of operators associated with  regular singular integral  kernels
satisfying, say, $|K(x,y)|\lc |x-y|^{-d}$, $|\nabla_x K(x,y)|+|\nabla_y K(x,y)|
\lc |x-y|^{-d-1}$ and showed
that the following conditions are equivalent.

\begin{itemize}
 \item $T$ satisfies a Carleson condition.

\item $T$ maps $H^1$ to $L^1$.

\item  $T$ maps $L^\infty_0$ to $BMO$.
\end{itemize}
He then used  an interpolation theorem to show that each condition is equivalent with
\begin{itemize}
\item   $T$ maps $L^2$ to $L^2$.
\end{itemize}

We now give versions of   Journ\'e's  theorem  for larger classes of  kernels which arise  in our main result.


\begin{defn}\label{atomicboundedness}
(i) A integrable function is called an $\infty$-atom associated to a cube $Q$  if $a$  is supported on $Q$, and satisfies 
$\|a\|_\infty\le |Q|^{-1}$ and $\int a(x) dx=0$.

(ii) A linear operator defined on compactly supported functions with integral zero satisfies 
 the {\it atomic boundedness condition} if 
 $$\|T\|_{\At} :=\sup \|Ta\|_1 <\infty$$ where the sup is taken over all $\infty$-atoms.
  \end{defn}

\begin{rmk}One can also make a definition of a class $\At(q)$ where one works 
 with $q$-atoms satisfying $\supp a\subset Q$, $\|a\|_q\le |Q|^{-1+1/q}$
and $\int a(x) dx=0$. Define $\|T\|_{\At(q)}=\sup \|Ta\|_1$ where the supremum is taken over all $q$-atoms. 
For  the case $1<q<\infty$ one has $T\in \At(q)$ if and only if $T$ extends to a bounded operator $H^1\to L^1$, and 
$\|T\|_{\At(q)}\approx \|T\|_{H^1\to L^1}$. This is a special case of a result by Meda, Sj\"ogren and Vallarino \cite{MSV}. 
The  equivalence may fail for the case $q=\infty$, as was shown by  Bownik \cite{bownik}.
We remark that for special classes of Calder\'on-Zygmund operators the equivalence holds true even for $q=\infty$ (see \cite[\S7.2]{MC}, and the proof of Theorem 
\ref{Journe-equiv} below). For most situations in harmonic analysis 
the use of $\infty$-atoms (instead of  $q$-atoms) does not yield a significant advantage, but in our application it will be crucial  to work with $\infty$-atoms.
\end{rmk}

In the  following three propositions  
$T: C^\infty_0(\bbR^d)\to \cD'(\bbR^d)$  will denote a linear operator with Schwartz kernel $K\in \cD'(\bbR^d\times \bbR^d)\cap L^1_{\loc}((\bbR^d\times \bbR^d)setminus\Delta)$. The proofs 
use the arguments of  Journ\'e \cite[\S 4.2]{journe}.
\begin{prop}\label{Jpropi}
Suppose that $T$ satisfies the atomic boundedness condition and the averaged annular integrability condition. 
Then 
$$\|T\|_{\Carl} \lc  \|T\|_{\At}+ \ann_\av[K]
\,.$$
\end{prop}
\begin{prop}\label{Jpropii}
Suppose that $\SI^\infty[K]<\infty$ , $\ann_\av[K]<\infty$
 and
that $T$ satisfies a Carleson condition.  Then $T$ extends to a bounded operator from 
$L^\infty_0$ to $BMO$ satisfying 
$$\|T\|_{L^\infty_0\to BMO} \lc \|T\|_{\Carl} + \SI^\infty[K].
$$
\end{prop}
\begin{prop}\label{Jpropiii}
Suppose that  
$\SI^1[K]<\infty$ 
and that 
$T$ extends to  a bounded operator $T: L_0^\infty\to BMO$. Then
$T$ satisfies the atomic boundedness condition and 
$$\|T\|_{\At} \lc \|T\|_{L_0^\infty\to BMO} + \SI^1[K].
$$
\end{prop}
For the convenience of the reader we give the proof of the three propositions.
In what follows $Q$ will  denote a cube, $x_Q$ its center, and as above
$Q^*$ will be the double cube with same center.

\begin{proof} [Proof of Proposition \ref{Jpropi}]
Let $f$ be a bounded function supported in a cube $Q$.  We need to establish the estimate
\Be\label{J(i)} \int_{Q^*}|Tf|dx\lc C|Q|\|f\|_\infty \big(\|T\|_{\At}+\ann_\av[K]\big).\Ee

Let $Q_1$ be a  cube with the same sidelength of $Q^*$ and of distance $\diam ( Q^*)$ to $Q^*$.
Let $f_1$ be a function supported in $Q\cup Q_1$ so that
$f_1(y)=f(y)$ for $y\in Q$, $\|f_1\|_\infty \le \|f\|_\infty$  and $\int f_1(y) dy=0$.
Then, if
$$a(x)={|Q|^{-1}\|f\|_\infty^{-1}}f_1(x)$$
then there is $C_d>0$ so that $C_d^{-1} a$ is an $\infty$-atom.
Set $f_2=f-f_1$ so that $f_2$ is supported in  $Q_1$ and split
$$\int_{Q^*}|Tf|dx \lc \int_{Q^*}|Tf_1|dx
+\int_{Q^*}|Tf_2|dx. $$
We estimate  
\Be\label{Tf1onQ*}
\int_{Q^*}|Tf_1|dx
\lc |Q|\,\|T\|_{\At}\|f\|_\infty.
\Ee
Since $\dist (Q^*,Q_1) \approx \diam(Q_1)\approx \diam(Q^*)\approx \diam(Q)$ we may use the averaged annular integrability condition and estimate
\[
\frac{1}{|Q|}\int_{Q_1} \int_{Q^*} 
|K(x,y)| dy\, dx  \lc \ann_\av[K].
\]
This yields
\Be\label{Tf2onQ*}
\int_{Q^*}|Tf_2|dx \lc \int_{Q^*} \int_{Q_1} |K(x,y)||f_2(y)| dx\, dy 
\lc \|f_2\|_\infty |Q| \ann_\av[K].
\Ee
Since $\|f_2\|_\infty \le2\|f\|_\infty$, 
\eqref{J(i)} 
 follows from
\eqref{Tf1onQ*} and \eqref{Tf2onQ*}.
\end{proof}

\begin{proof} [Proof of Proposition \ref{Jpropii}]
Let $g\in L^\infty_0$ and let $Q$ be any  cube with center $x_Q$.
We have to verify
\Be \label{J(ii)}\inf_C \intslash_Q |Tg(x)-C|dx \le  \|T\|_{\Carl} +\SI^\infty[K]\Ee
where the slashed integral denotes the average over $Q$.

Let
$g_1= g\bbone_{Q^*}$,
$g_2= g\bbone_{\bbR^d\setminus Q^*}$, so that $g=g_1+g_2.$
Since $g$ has compact support it is immediate by the assumed finiteness of $\ann_\av[K]$ 
that $Tg_2(w)$ is finite for almost every $w$ in $$ B_Q:=\{w: |w-x_Q|\le (2d)^{-1}\diam(Q)\}.$$ 
Now 
\[
\inf_C \intslash_Q |Tg(x)-C|dx \lc
\intslash_{B_Q}\Big[
\intslash_Q |Tg_1(x)| dx
+
\intslash_Q |Tg_2(x)-Tg_2(w)| dx\, \Big] dw.
\]
From the Carleson condition we get
$$
\intslash_Q |Tg_1(x)| dx \le 4^d \|T\|_{\Carl} \|g_1\|_\infty
\lc \|T\|_{\Carl} \|g\|_\infty\,.
$$
Moreover, 
\begin{align*}\intslash_{B_Q}
\intslash_Q |Tg_2(x)-Tg_2(w)| dx\,dw&\le
\|g_2\|_\infty
\sup_{w\in B_Q} \intslash_Q
\int_{\bbR^d\setminus Q^*} |K(x,y)-K(w,y)| dy \, dx
\\&\lc \SI^\infty[K] \|g\|_\infty\,.
\end{align*}
and 
\eqref{J(ii)} follows.
\end{proof}

\begin{proof} [Proof of Proposition \ref{Jpropiii}]
Let $a$ be an $\infty$-atom, associated with the cube $Q$.
We need to verify
\Be\label{atomicestimate} 
\|Ta\|_1 \lc  \|T\|_{L_0^\infty\to BMO} +
\SI^\infty[K]\,.
\Ee
First estimate $Ta$ in the complement of  $Q^*$, using the cancellation of $a$:
\begin{align*}
&\int_{\bbR^d\setminus Q^*} |Ta(x)|\, dx\lc
\int_{\bbR^d\setminus Q^*} 
\Big|\int_Q [K(x,y)-K(x,x_Q)] a(y) dy\Big|\, dx
\\
&\le \int_Q|a(y)|\int\limits _{ |x-x_Q|\ge 2|y-x_Q|}
|K(x,y)-K(x,x_Q)| \, dx\, dy
\\
&\le \SI^1[K] \|a\|_1 \lc  \SI^1[K]\,.
\end{align*}
Let $\widetilde Q$ be a cube which is contained in $CQ^*\setminus Q^*$ and has distance $O(\diam (Q)) $ to $Q^*$, say,  a cube adjacent to $Q^*$ and of same sidelength. 
The above  calculation also yields
\Be\label{ontildeQ}
\int_{\widetilde  Q} |Ta(x)|dx \lc \SI^1[K]\,.\Ee 
We choose such a cube $\widetilde Q$ and estimate
$$\int_{Q^*} |Ta(x)|dx\lc I_Q+II_Q+III_Q$$ where
\begin{align*}
I_Q&=\int_{Q^*} \Big|Ta(x)-\intslash_{Q^*} Ta(y) dy
\Big| \, dx\,,\\
II_Q&=
|Q^*| \Big|\intslash_{Q^*} Ta(y) dy- \intslash_{\widetilde Q} Ta(y)
dy\Big|\,,
\\
III_Q&=|Q^*|\Big|\intslash_{\widetilde Q} Ta(y) dy\Big|\,.
\end{align*}
Clearly
$$|I_Q|
\le |Q^*|\|Ta\|_{BMO}\le
\|T\|_{L^\infty\to BMO} |Q^*|\|a\|_{L^\infty}
\lc\|T\|_{L^\infty\to BMO} .
$$
To estimate $II_Q$ we let
$Q^{**}$ be  a cube containing both $Q^*$ and $\widetilde Q$, and of comparable sidelength. Then
\begin{align*}
&\Big|\intslash_{Q^*} Ta(y) dy- \intslash_{\widetilde Q} Ta(y)
dy\Big|
\\&\le
\intslash_{Q^*} \Big|Ta(y) - \intslash_{Q^{**}} Ta(z) dz\Big|\,dy
+\intslash_{\widetilde Q} \Big|Ta(y) - \intslash_{Q^{**}} Ta(z) dz\Big|\,dy
\\&\lc
\intslash_{Q^{**}} \Big|Ta(y) - \intslash_{Q^{**}} Ta(z) dz\Big|\,dy \lc \|Ta\|_{BMO}
\end{align*}
and thus
$$|II_Q| \lc \|T\|_{L^\infty_0\to BMO} |Q|\|a\|_\infty \lc \|T\|_{L^\infty_0\to BMO} |Q|\|a\|_\infty \lc
\|T\|_{L^\infty_0\to BMO}\,.
$$
Finally,
$$| III_Q|\le |Q^*|\Big|\intslash_{\widetilde Q} Ta(y) dy\Big|\lc \|Ta\|_{L^1(\widetilde Q)} \lc A\,,$$
by \eqref{ontildeQ}, and the proof of \eqref{atomicestimate}  is finished.
\end{proof}

\begin{thm} \label{Journe-equiv}
Let $T: C^\infty_0(\bbR^d)\to \cD'(\bbR^d)$ and assume that 
the Schwartz kernel $K$ is locally integrable 
in $(\bbR^d\times \bbR^d)\setminus\Delta$.
Assume that
$$\SI[K]:= \ann_\av[K]+\SI^1[K]+\SI^\infty[K] <\infty.$$ 

(i) Let $1<q<\infty$. The following statements are equivalent.

\begin{itemize}
\item    $T$ satisfies a Carleson condition.

\item  $T$ maps $L^\infty_0\to BMO$.

\item $T$ satisfies the atomic boundedness condition.

\item $T$ extends to a bounded  operator  $H^1\to L^1$.

\item  $T$ extends to an operator  bounded on $L^q$. 
\end{itemize}

(ii) We have the following equivalences of norms.
\Be\label{j-equivalences}
\begin{aligned}
\|T\|_{\Carl} + \SI[K] & \approx
\|T\|_{L^\infty_0\to BMO} + \SI[K] \approx
\|T\|_\At + \SI[K] 
 \approx\ci{q} \|T\|_{L^q\to L^q}+ \SI[K].
\end{aligned}
\Ee
Moreover,
\Be \label{Hardyequiv}\|T\|_\At \approx\|T\|_{H^1\to L^1} \,.
\Ee
\end{thm}
\begin{proof} The first three equivalences are immediate from a combination of 
Propositions \ref{Jpropi}, \ref{Jpropii} and \ref{Jpropiii}.
Since $\infty$-atoms satisfy $\|a\|_{H^1}\le C$ it is clear that
$$\|T\|_\At\lc \|T\|_{H^1\mapsto L^1}\,.
$$

The converse 
\Be \label{H1atomic} \|T\|_{H^1\to L^1}\, \lc \|T\|_\At 
\Ee
is not obvious (and the inequality without the term $\SI[K]$ might  not hold if we drop our assumption $\SI[K]<\infty$, see \cite{bownik}).
By the Coifman-Latter theorem about the atomic decomposition (see 
\cite[\S III.2]{stein-ha}) we may write
$f=\sum_Q \la_Q a_Q$,
with $\sum_Q|\la_Q| \lc \|f\|_{H^1}$ and $a_Q$ being $\infty$-atoms; here  the convergence of the series is understood in the $L^1$ sense.
We immediately  get
$$\Big\| \sum_Q \la_Q Ta_Q 
\Big\|_1 \le \sum_Q|\la_Q| \|T\|_\At \|a_Q\|_1 \lc
\|T\|_{\At}\|f\|_{H_1}.
$$
However the decomposition $f=\sum_Q \la_Q a_Q$ is not unique and  in order 
to prove that the expression $\sum_Q \la_Q Ta_Q $ can be used as a definition for $Tf$ we need to show the following consistency condition for a sequence of atoms $\{a_\nu\}_{\nu=1}^\infty$,
\Be \label{cons}\sum_Q|c_\nu| <\infty \,,\,   \sum_\nu c_\nu a_\nu=0 \,\implies \, 
\sum_\nu c_\nu Ta_\nu=0 .
\Ee
Fortunately, a version of an approximation (or weak compactness)  argument in \cite[\S7.2]{MC} applies to our situation.
As stated above  the atomic boundedness condition implies the Carleson condition.
Let $\phi\in C^\infty_0$ be supported in a ball of radius $1/2$ such that $\int \phi(x) dx=1$. Set $P_m f= \phi^{(2^m)}*f$.
Let $K_m$ be the distribution kernel for $P_mTP_m$.
Note that we have 
$$|K_m(x,y)| \lc 2^{md} \ann_\av[K] \text{ if } |x-y| \ge  2^{2-m}$$
and
$$|K_m(x,y)| \lc 2^{md}  \|T\|_\Carl \text{ if } |x-y| \le  2^{2-m}.$$
Hence   $K_m\in L^\infty(\bbR^d\times \bbR^d)$ 
and thus $P_mTP_m$ maps $L^1$ to $L^\infty$.
This implies $\sum_\nu c_\nu P_m T P_m a_\nu=P_mTP_m (\sum c_\nu a_\nu)=0.$ Now, since the $P_m$ form an approximation of the identity, it is clear that, for each atom $a_\nu$, we have  $\|P_mTP_m a_\nu- Ta_\nu\|_1\to 0$ as $\nu\to\infty$.
Taking in account that $\sum_\nu| a_\nu|<\infty$, a straightforward limiting argument yields $\sum_\nu c_\nu Ta_\nu=0 .$
Note that the condition  $\SI[K]<\infty$ is used to establish  \eqref{Hardyequiv} only in order to verify the implication \eqref{cons} (via the boundedness of $K_m$); it does not enter in \eqref{Hardyequiv} itself.

We still have to show the equivalence of the first three  conditions in 
\eqref{j-equivalences} 
with the fourth condition.
Assume first that $T$ is $L^q$-bounded. 
Then we have the standard estimates \eqref{hardy}, \eqref{bmothm} and thus the $H^1\to L^1$ operator norms 
and $L^\infty_0\to L^\infty$ operator norms of $T$ are bounded by
$\|T\|_{L^q\to L^q}+\SI[K]$. 
The other direction uses the interpolation result (cf. the remarks below)
$$
\|T\|_{L^q\to L^q} 
\le C_q
\|T\|_{H^1\to L^1}^{1/q}
\|T\|_{L^\infty_0\to BMO}^{1-1/q}
$$
together with the equivalence of the first three conditions in 
\eqref{j-equivalences} and the equivalence \eqref{Hardyequiv}.
\end{proof}

\noi {\it Remarks on interpolation of $H^1$ and $BMO$.}
In the above interpolation one uses the interpolation formulas
$[H^1, BMO]_{\theta,q} =L^{p,q}$,
$[H^1, BMO]_{\theta}
=L^{p}$ for $1-\theta=1/p$, $1<p<\infty$,  or a direct interpolation result for operators in \S3.III of Journ\'e's monograph
\cite{journe}.
One also has $[L^1, BMO]_{\theta,q} =L^{p,q}$,
$[L^1, BMO]_{\theta}
=L^{p}$ for $1-\theta=1/p$, $1<p<\infty$.

The result for complex interpolation can be obtained
from the results
$[H^1, L^{p_1}]_{\vartheta} =L^p$, 
$1/p=1-\vartheta+\vartheta/p_1$, $1<p_1<\infty$,
(or its respective standard counterpart $[L^1, L^{p_1}]_{\vartheta} =L^p$),  together 
with $[L^{p_0}, BMO]_\vartheta= L^p$, $1/p=(1-\theta)/p_0$, $1<p_0<\infty$ which can be found in Fefferman and Stein \cite{fs}, see also the discussion in  Janson and Jones \cite{jajo}.
The stated interpolation formula for $H^1$ and $BMO$ follows then from Wolff's four space reiteration theorem for the complex method \cite{wolff}.
One can also use the results by Fefferman, Rivi\`ere, Sagher \cite{frs} for the real method, and then combine it with Wolff's result \cite{wolff}  for the real method.
From the above remarks we also get
an interpolation inequality for functions
 $g\in L^1\cap BMO$, 
\Be\label{L1BMOconvex} \|g\|_p \le C_p \|g\|_{L^1}^{1/p} \|f\|_{BMO}^{1-1/p}, \quad 1<p<\infty\Ee
which will be useful in the proof of Theorem \ref{ThmJourne2} below.


\subsection{Sums of dilated kernels}\label{sumsofdilated}
We shall now formulate some corollaries for operators of the form \eqref{Ksum} or its relatives.
We use norms combining the various Schur and regularity norms.

For each $j\in \Z$, let $\tau_j:\R^d\times \R^d\rightarrow \bbC$ be a measurable function. Let $0<\eps\le 1$.
Set, for $0<\eps\le 1$, 
\begin{equation*}
\|\tau\|_{\Op_\eps}=
\Sha_\eps^1[\tau]
+\Sha_\eps^\infty[\tau]
+\Zhe_{\eps, \rml}^1[\tau]
+\Zhe_{\eps, \rml}^\infty[\tau]
+\Zhe_{\eps, \rmr}^1[\tau]
+\Zhe_{\eps, \rmr}^\infty[\tau] ,
\end{equation*}
and set
\[\|\tau\|_{\Op_0}:=
\Sha_0^1[\tau]
+\Sha_0^\infty[\tau].
\]
This means  for $\eps>0$ 
\begin{equation}\label{EqnJourneBoundCep}
\begin{split}
\|\tau\|_{\Op_\eps}=
&\sup_{x} \int(1+|x-y|)^{\epsilon} |\tau(x,y)|\: dy + \sup_{y} \int(1+|x-y|)^{\epsilon} |\tau(x,y)|\: dx
\\+&\sup_{\substack{y\\ 0<|h|\leq 1}} |h|^{-\epsilon} \int |\tau(x+h,y)-\tau(x,y)|\: dx 
+\sup_{\substack{x\\ 0<|h|\leq 1}} |h|^{-\epsilon} \int |\tau(x+h,y)-\tau(x,y)|\: dy 
\\+&\sup_{\substack{y\\ 0<|h|\leq 1}} |h|^{-\epsilon} \int |\tau(x,y+h)-\tau(x,y)|\: dx 
+\sup_{\substack{x\\ 0<|h|\leq 1}} |h|^{-\epsilon} \int |\tau(x,y+h)-\tau(x,y)|\: dy 
\end{split}
\end{equation}
and, for $\eps=0$,
\begin{equation}\label{schurassu}
\|\tau\|_{\Op_0}=
\sup_{x} \int |\tau(x,y)|\: dy + \sup_{y} |\tau(x,y)|\: dx\,.
\end{equation}


We shall consider families $\{\tau_j\}$ for which the $\Op_\eps$ norm is uniformly bounded in $j$. 
We let $T_j$ be the operator with kernel $\Dilj\tau_j$, i.e.
\begin{equation}\label{Tjdef}
T_j f(x) = \int 2^{jd} \tau_j(2^j x, 2^j y) f(y) dy\,.
\end{equation}

\begin{thm}\label{ThmJourne1} Suppose that 
$\sup_j\|\tau_j\|_{\Op_\eps}\le \cC_\eps$ for some $\eps\in (0,1)$ and that 
$\sup_j\|\tau_j\|_{\Op_0}\le \cC_0$. Let 
 $T_j$ be the operator with kernel $\Dilj\tau_j$
and suppose that 
$\sum_j T_j$ converges to 
an operator $T:L^\infty_\comp\to L^1_\loc$ in the sense that
for compactly supported $L^\infty$ functions $f$ and $g$
$$\biginn{\sum_{j=-N}^N T_j f}{g} 
\to \inn {Tf}{g}$$
as $N\to \infty$ and assume that there exists $A>0$ such that
for all $x\in \bbR^d$, $t>0$, $N\in \bbN$,
\begin{equation}\label{carlesonforsums}
\Big |\biginn{\sum_{j=-N}^N T_j f}{g}\Big|\le
A t^d \|f\|_{L^\infty} \|g\|_{L^\infty}\quad  \text{ if } \mathrm{supp}(f)\cup\mathrm{supp} (g)\subset 
B^d(x,t).
\end{equation}
Then $T$ extends to an operator bounded on $L^2(\bbR^d)$ 
and \begin{equation*}
\| T\|_{L^2\rightarrow L^2} \leq C_{d,\epsilon} \Big(A+ \cC_0\log\big(1+\frac{\cC_\epsilon}{\cC_0}\big)\Big).
\end{equation*}
\end{thm}
\begin{proof} 
The inequality \eqref{carlesonforsums}  implies $\|\sum_{j=-N}^NT_j\|_{\Carl} \lc A$. This inequality extends to the limit $T$.
 Let 
$K_N$, $K$ be the Schwartz kernels of the operators $\sum_{j=-N}^NT_j$ and $T$ respectively. Then
by Propositions  \ref{ShaZheimplSI} and  \ref{ShaZheimplAnn}, applied to both $\tau_j$ and its adjoint version we have
$\SI[K_N], \, \SI[K]
 \lc \cC_0  \log (2+\cC_\eps/\cC_0)\,.$
The assertion follows now from Theorem \ref{Journe-equiv}.
\end{proof}

\begin{thm}\label{ThmJourne2}
Suppose that 
$\sup_j\|\tau_j\|_{\Op_\eps}\le \cC_\eps$ for some $\eps\in (0,1)$ and that 
$\sup_j\|\tau_j\|_{\Op_0}\le \cC_0$.
 Let 
 $T_j$ be the operator with kernel $\Dilj\tau_j$ and suppose that the sum $T=\sum T_j$ converges in the sense of distributions on $C^\infty_{0,0}$ (test functions with vanishing integrals),
i.e. for every $f\in C^\infty_{0,0}$ and every $g\in C^\infty_0$ we have
\Be\label{convindistr0}
\lim_{N\to\infty} \sum_{j=-N}^N \inn{T_j f}{g} = \inn {Tf}{g}.\Ee
Then the following statements hold.

(i) If 
$\sup_N \|\sum_{j=-N}^N T_j \|_{H^1\to L^1} \le A,$ for some $A<\infty$, 
then we also have 
$$\sup_N \Big\|\sum_{j=-N}^N T_j \Big\|_{L^2\to L^2} \lesssim  A +\cC_0
\log \big(1+\frac {\cC_\eps}{\cC_0}\big).$$ Moreover,  $T$ extends to a bounded operator 
on $L^2$,  $\sum_{j=-N}^N T_j$ converges to $T$  in the weak  operator topology
and
 $\|T\|_{L^2\to L^2} \lesssim  A +\cC_0
\log \big(1+ {\cC_\eps}/{\cC_0}\big).$

(ii) If 
$\sup_N \|\sum_{j=-N}^N T_j \|_{L^2\to L^2} \le B,$ for some $B<\infty$, 
then we also have 
$$\sup_N \Big\|\sum_{j=-N}^N T_j \Big\|_{H^1\to L^1} \lesssim  B +\cC_0
\log \big(1+\frac {\cC_\eps}{\cC_0}\big).$$ Moreover $T$ extends to an operator bounded from $H^1$ to $L^1$, 
 $\sum_{j=-N}^N T_j\to T$ converges in the weak  operator topology (as operators $H^1\to L^1$)  and 
 $\|T\|_{H^1\to L^1} \lesssim  B +
\cC_0\log \big(1+{\cC_\eps}/{\cC_0}\big).$

(iii) The sum $T=\sum_{j\in \Z} T_j$ converges in the strong operator topology as operators $H^1\rightarrow L^1$ if and only if it converges in the strong operator topology as operators $L^2\rightarrow L^2$.  

\end{thm}
\begin{proof}
The assertions on the operators $\sum_{j=-N}^N T_j$ follow immediately from Theorem 
\ref{Journe-equiv}. Note that $C^\infty_{0,0}$ is dense in both $H^1$ and $L^p$, $1<p<\infty$. The uniform bounds for the operator norms of $\sum_{j=-N}^N T_j$ 
and the convergence hypothesis \eqref{convindistr0} imply convergence in the respective 
weak operator topologies.

Now we prove (iii). If 
$T=\sum_{j\in \Z} T_j$ converges 
in the strong operator topology as operators $L^2\rightarrow L^2$ then it is immediate from Proposition \ref{CorWeakTypeInterp}  that 
$T=\sum_{j\in \Z} T_j$ converges 
in the strong operator topology as operators $H^1\rightarrow L^1.$ 

Vice versa assume that 
$T=\sum_{j\in \Z} T_j$ converges 
in the strong operator topology as operators $H^1\rightarrow L^1.$ 
By the interpolation inequality \eqref{L1BMOconvex} we have for any finite set 
$\cJ\in \bbZ$ and any $f\in C^\infty_{0,0}$.
$$
\Big \|\sum_{j\in \cJ} T_j f\Big\|_{2}\le C 
\Big \|\sum_{j\in \cJ} T_j f\Big\|_{1}^{1/2} 
\Big \|\sum_{j\in \cJ} T_j f\Big\|_{BMO}^{1/2} 
\le C 
\Big \|\sum_{j\in \cJ} T_j f\Big\|_{1}^{1/2} 
\Big \|\sum_{j\in \cJ} T_j \Big\|_{L^\infty\to BMO}^{1/2}  \|f\|_\infty
$$
and since 
$\|\sum_{j\in \cJ} T_j \|_{L^\infty\to BMO}$ is bounded independently of $\cJ$ we 
see that $\sum_{j}T_j f$ converges in $L^2$  for any $f\in C^\infty_{0,0}$. Since 
$C^\infty_{0,0}$ is dense in $L^2$ we conclude that $\sum_j T_j$ converges in the strong operator topology as operators $L^2\to L^2$.
\end{proof}

We now formulate a version of Theorem \ref{ThmJourne1} which has a  convergence statement with respect to the strong operator topology.

\begin{thm}\label{ThmJourne3} Suppose that 
$\sup_j\|\tau_j\|_{\Op_\eps}\le \cC_\eps$ for some $\eps\in (0,1)$ and that 
$\sup_j\|\tau_j\|_{\Op_0}\le \cC_0$. Let 
 $T_j$ be the operator with kernel $\Dilj\tau_j$.
 Suppose that 
$\sum_j T_j$ converges to 
an operator $T:L^\infty_\comp\to L^1_\loc$ in the strong sense that
for any  compactly supported $L^\infty$ function $f$  and for any compact set $K$ 
$$\lim_{N\to\infty} \int_{K}  \Big|\sum_{j=-N}^N T_jf(x) -Tf(x) \Big| dx =0.$$ 
Suppose that  there exists $A>0$ such that
for all $x\in \bbR^d$, $t>0$, $N\in \bbN$, 
\begin{equation}\label{carlesonforsumsmod}
\int_{B_d(x,t)} \Big |\sum_{j=-N}^N T_j f(w)\Big|\, dw \le
A t^d \|f\|_{\infty} \quad  \text{ if } \mathrm{supp}(f)\subset 
B^d(x,t).
\end{equation}
Then 
the sum $T=\sum_{j\in \Z}T_j$ converges in the strong operator topology as operators $L^2\rightarrow L^2$ and
and \begin{equation*}
\| T\|_{L^2\rightarrow L^2} \leq C_{d,\epsilon} \Big(A+ \cC_0\log\big(1+\frac{\cC_\epsilon}{\cC_0}\big)\Big).
\end{equation*}
\end{thm}

\begin{proof} If $a$ is an $L^\infty$ atom supported on a cube $Q$ and $Q^*$ is the double cube,
we see that $\sum_{j=-N}^N Ta \to Ta$ in $L^1(Q^*)$.  
Standard arguments using the cancellation of $a$  yield
$$\int_{\bbR^d \setminus Q^*} |T_j a(x)| \lc \begin{cases} 
\Sha^1_\eps[\tau_j]  \,(2^j\diam(Q))^{-\eps} \quad
 \quad &\text{ if } 2^j\diam(Q)\ge 1,
\\
 \Zhe^1_{\eps,\rml}[\tau_j] \,(2^j\diam(Q))^{\eps}\quad
 &\text{ if } 2^j\diam(Q)\le 1.
\end{cases}
$$
Altogether we see that $\sum_{j=-N}^N T_ja \to Ta$ in $L^1$.
By Theorem \ref{ThmJourne1} we also have the uniform bounds 
$\| Ta\|_1 \leq C_{d,\epsilon} \Big(A+ \cC_0\log\big(1+\frac{\cC_\epsilon}{\cC_0}\big)\Big)$ for $L^\infty$ atoms. Now, writing $f\in H^1$  as $f=\sum_\nu c_\nu a_{\nu}$ where the $a_\nu$ are $\infty$-atoms and $\sum_\nu|c_\nu|<\infty$, we easily derive that 
$\sum_{j=-N}^N T_jf \to Tf$ in $L^1$. Thus we see that $\sum_j T_j$ converges in the strong operator topology as operators $H^1\to L^1$  and we have the uniform bound 
$$\Big\|\sum_{j=-N}^N T_j\Big\|_{H^1\to H^1} \lc 
\Big(A+ \cC_0\log\big(1+\frac{\cC_\epsilon}{\cC_0}\big)\Big).$$
We apply parts (i) and (iii) of Theorem \ref{ThmJourne2} to see that 
that $\sum_j T_j$ converges in the strong operator topology as operators $L^2\to L^2$, 
and obtain the asserted bounds on the $L^2\to L^2$ operator norms.
\end{proof}


The following lemma  allows us to apply 
Theorems \ref{ThmJourne1}, \ref{ThmJourne2} 
and \ref{ThmJourne3} 
to sums of the form
$\sum_j P_j S_j P_j$ 
 where $P_j f = f*\dil{\phi}{2^j}$, and $S_j$ is an integral operator with kernel 
 $\Dilj s_j$, with $\sup_j (\Sha_\eps^1[s_j]+\Sha_\eps^\infty[s_j])<
 \infty$.

\begin{lemma}\label{PSP-lemma} Suppose that  $ \Sha_\eps^1[s]+\Sha_\eps^\infty[s]\le C_\eps$ 
 and
 $\Sha^1[s]+\Sha^\infty[s]\le C_0$. Let $\phi\in C^\infty_0$ supported in $\{v:|v|\le 10\}$.
Let $$\widetilde s(x,y)= \iint \phi(x-w)s(w,z)\phi(z-y) \,dw\,dz.$$
Then
$\|\widetilde s\|_{\Op_\eps} \lc C_\eps$ and 
$\|\widetilde s\|_{\Op_0} \lc C_0$.
\end{lemma}

\begin{proof} Left to the reader. \end{proof}

We also have
\begin{lemma}\label{PS-SP-lemma} Let $s\in \Op_\eps$, $0\le \eps\le 1$.
Let $\phi\in C^1$ supported in $\{v:|v|\le 10\}$
and let 
\begin{align*} 
s_1(x,y)&=
  \int \phi(x-w)s(w,y) \,dw,
  \\
  s_2(x,y) &=
  \int s(x,z)\phi(z-y) \,dz.
  \end{align*}
Then
$\| s_1\|_{\Op_\eps} \lc \|s\|_{\Op_\eps}\|\phi\|_{C^1}
$,
$\| s_2\|_{\Op_\eps} \lc \|s\|_{\Op_\eps}\|\phi\|_{C^1}$.
\end{lemma}

\begin{proof} Immediate from the definition. \end{proof}

\section{Almost orthogonality} \label{almost-orth}
We shall repeatedly use a rather standard  almost orthogonality lemma 
which involves the Littlewood-Paley operators $\cQ_k$ introduced in \eqref{cQdef}. 

\begin{lemma}\label{ThmAlmostOrthogonal}
Let  $\cI$ be an index set. 
 Suppose that for each $j\in \Z$, $\nu\in \cI$, $V_j^\nu:L^2\rightarrow L^2$ is a bounded  operator such that for $k_1,k_2\in \Z$,
\begin{equation}\label{Ajk2assu}
\big\|\Qt_{k_1} V_{j+k_1}^\nu \Qt_{j+k_1+k_2}\big\|_{L^2\to L^2}\lesssim A_{j,k_2},
\end{equation}
where $$\sum_{j,k_2} A_{j,k_2}<\infty.$$  Then the sum 
$V^\nu:=\sum_{j\in \Z} V_j^\nu$, converges in the strong operator topology (as operators on $L^2$),
with equiconvergence with respect to $\cI$, and  we have
\begin{equation} \label{Ajk2concl}
\sup_{\nu\in\cI} \LpOpN{2}{V^\nu}\lesssim \sum_{j,k_2} A_{j,k_2}.
\end{equation}

\end{lemma}
\begin{proof}
Recall, from \S\ref{SectionAuxOps}, $\sum_{k} \Qtt_k \Qt_k=\sum_{k} \Qt_k \Qtt_k =I$.  Let $f,g\in L^2(\R^d)$ with
$\|{f}\|_2=\|{g}\|_2=1$.  By \eqref{EqnAuxLittlewoodPaley}, we have
\begin{equation*}
\Big(\sum_k \|\Qtt_k f\|_2^2\Big)^{\frac{1}{2}} = \Big\|\Big( \sum_{k} \q|\Qtt_k f\w|^2 \Big)^{\frac{1}{2} }\Big\|_2\approx 1,\quad \Big(\sum_k \|\Qtt_k^{*} g\|_2^2\Big)^{\frac{1}{2}} = \Big\|\Big( \sum_{k} \q|\Qtt_k^{*} g\w|^2 \Big)^{\frac{1}{2} }\Big\|_2\approx 1.
\end{equation*}
First observe, for integers $J_1\le J_2$,

\begin{equation*}
\begin{split}
&\Big|\Ltip{g}{\sum_{j=J_1}^{J_2}
 V_j^\nu f}\Big| =  \Big|\Ltip{g}{\sum_{k_1,k_2\in \Z} 
\sum_{j=J_1}^{J_2} \Qtt_{k_1} \Qt_{k_1} V_j^\nu \Qt_{k_2} \Qtt_{k_2} f }\Big| 
\\&= \Big|\Ltip{g}{\sum_{j=J_1}^{J_2} \sum_{k_1,k_2\in \Z} \Qtt_{k_1} \Qt_{k_1} V_{j}^\nu \Qt_{k_2} \Qtt_{k_2 } f}\Big|
\\&\leq \sum_{k_1\in \Z} \Big|\Ltip{\Qtt_{k_1}^{*} g}{  \sum_{j=J_1-k_1}^{J_2-k_1} \sum_{k_2\in \Z} \Qt_{k_1} V_{j+k_1}^\nu \Qt_{j+k_1+k_2} \Qtt_{j+k_1+k_2} f }\Big|
\\&\leq \Big( \sum_{k_1\in \Z}\|\Qtt_{k_1}^{*} g\|_2^2   \Big)^{\frac{1}{2}} \Big( \sum_{k_1\in \Z}  \Big\|  \sum_{j=J_1-k_1}^{J_2-k_1}\sum_{k_2\in \Z} \Qt_{k_1} V_{j+k_1}^\nu \Qt_{j+k_1+k_2} \Qtt_{j+k_1+k_2} f \Big\|_2^2  \Big)^{\frac{1}{2}}.
\end{split}
\end{equation*}
Now
\begin{align*}
&\Big( \sum_{k_1\in \Z}  \Big\| \sum_{j=J_1-k_1}^{J_2-k_1}
\sum_{k_2\in \Z} \Qt_{k_1} V_{j+k_1}^\nu\Qt_{j+k_1+k_2} \Qtt_{j+k_1+k_2} f \Big\|_2^2  \Big)^{\frac{1}{2}}
\\&\lesssim \sum_{j\in \bbZ}\sum_{k_2\in \Z} \Big( \sum_{k_1=J_1-j}^{J_2-j} 
 \big\|\Qt_{k_1} V_{j+k_1}^\nu\Qt_{j+k_1+k_2} \Qtt_{j+k_1+k_2} f \big\|_2^2  \Big)^{\frac{1}{2}}
\\&\leq \sum_{j\in \bbZ}\sum_{k_2\in \Z} A_{j,k_2}\Big( \sum_{k_1=J_1-j}^{J_2-j} \| \Qtt_{j+k_1+k_2} f \|_2^2 \Big)^{\frac{1}{2}}.
\end{align*}

We take the sup over $g$ with $\|g\|_2=1$ and obtain from the two previous displays
\Be \label{checkequi}
\Big\|\sum_{j=J_1}^{J_2} V_j^\nu f\Big\|_2 \lc 
\sum_{j\in \bbZ}\sum_{k_2\in \Z} A_{j,k_2}\Big( \sum_{k_1=J_1-j}^{J_2-j} \| \Qtt_{j+k_1+k_2} f \|_2^2 \Big)^{\frac{1}{2}}
\lc 
\sum_{j\in \bbZ}\sum_{k_2\in \Z} A_{j,k_2}\|f\|_2 .
 \Ee
 The first inequality in \eqref{checkequi}  implies  that for fixed $f\in L^2$ 
 the partial sums 
 of $\varSigma_N^\nu  f=\sum_{j=-N}^N V_j^\nu f$ 
 form a Cauchy sequence, more precisely,  for each $\eps>0$ there is $N(\eps,f)\in \bbN$ (independent of $\cI$) such that $\|\varSigma_{N_1} f-\varSigma_{N_2} f\|_2<\eps$ for $N_1, N_2>N(\eps,f)$.
 By  completeness of $L^2$, $\varSigma_N^\nu f$ converge to a limit $\varSigma^\nu\! f$ and 
 $\varSigma^\nu$  defines a linear bounded operator on $L^2$. Thus $\varSigma_N^\nu \to \varSigma^\nu$
  in the strong operator topology, and, by the above, we get equiconvergence with respect to $\cI$.
  \end{proof}

\section{Proof of Theorem \ref{main-parts}: 
Part \ref{ItemBoundRoughest}}
\label{Sectionpartrough}


We are given a family $\vectsig=\{\vsig_j\}$ with $\sup_j \|\vsig_j\|_{\cB_\eps}<\infty$. 
In this and the following sections  we  use the notation $$\Gamma_\eps=
\frac{\sup_j \|\vsig_j\|\ci{\cB_\eps}}
{\sup_j \|\vsig_j\|\ci{L^1} }
$$
 introduced in \eqref{Gammaeps}. Notice that always $\Ga_\eps\ge 1$.

Recall, $$\La^1_{n+1,n+2}(b_1,\ldots, b_{n+2}) = \sum_{j\in \Z} \La[\vsigjj](b_1,\ldots, b_n, (I-P_j)b_{n+1}, (I-P_j) b_{n+2}).$$  
Given $\eps>0$ and  
$\vectsig$,  
it is our goal 
to prove
Part \ref{ItemBoundRoughest} of Theorem \ref{main-parts}, i.e. for $1<p
\le 2$,
the estimate
\Be\label{ThmBoundRoughest}
\q|\La^1_{n+1,n+2}\q(b_1,\ldots, b_{n+2}\w)\w|
\leq C_{d,p,\eps} \q(\sup_{j} \LpN{1}{\vsig_j} \w)
\log^{2} ( 1+n\Ga_\eps)
\Big(\prod_{l=1}^n \| {b_l}\|_\infty\Big)\|b_{n+1}\|_p \|b_{n+2}\|_{p'}.
\Ee

We formulate a stronger result which will also 
be useful in other parts of the paper.
For this, we need some
new notation.  Let $1\leq l_1\ne l_2\leq n+2$ and let 
 $\q\{b_l^j :j\in \Z, \,l\neq l_1,l_2 \w\}\subset L^\infty(\R^d)$ be a bounded subset of $L^\infty(\bbR^d)$.
 Let  $k_1,k_2\in \N$,  
and fix $\zeta_1,\zeta_2\in \sU$.

Define an operator $S_{k_1,k_2,j}^{l_1,l_2}$ (which implicitly depends on $\q\{b_l^j:j\in \Z, l\ne l_1,l_2\w\}$, 
$\zeta_1$, and $\zeta_2$) by the formula
\begin{multline*}
\int g(x) \q(S_{k_1,k_2,j}^{l_1,l_2} f\w)(x)\: dx \\:= 
\La[\vsigjj](b_1^j, \ldots, b_{l_1-1}^j, \Qb{j+k_1}{\zeta_1} f, b_{l_1+1}^j, \ldots, b_{l_2-1}^j, \Qb{j+k_2}{\zeta_2} g, b_{l_2+1}^j,\ldots, b_{n_2}^j).
\end{multline*}

\begin{thm}\label{PropL2BoundFor2Qs} Let $0<\eps<1$ and suppose that $\sup_j \|\vsig_j\|_{\cB_\eps}<\infty$. Then 
$$S^{l_1,l_2}_{k_1,k_2}=\sum_{j\in \bbZ} S^{l_1,l_2}_{k_1,k_2,j}$$  converges in the strong operator topology, as bounded operators on $L^2$.
Moreover there is $c>0$ such that
\begin{equation*}
\LpOpN{2}{S_{k_1,k_2}^{l_1,l_2}}\lesssim \sUN{\zeta_1}\sUN{\zeta_2} 
\sup_j \|\vsig_j\|_{L^1} \min\{ 1, \,n 2^{-(k_1+k_2)\eps'} \Ga_\eps\}
\big(\prod_{l\ne l_1,l_2} \sup_j \|b_j^l\|_\infty\w\big).
\end{equation*}
\end{thm}

\begin{proof} [Proof that Theorem \ref{PropL2BoundFor2Qs} implies inequality 
\eqref{ThmBoundRoughest}]
For this, fix $b_1,\ldots, b_n\in L^\infty\q(\R^d\w)$ with 
\Be\label{bnormalization} \|b_j\|_\infty=1, \quad j=1,\dots,n.
\Ee 
  For $k_1,k_2\in \N$, define operators
$\cV$, $\cV_{k_1}$, and $\cV_{k_1,k_2}$ by the following formulas.
\begin{equation*}
\int g(x) \q( \cV f\w)(x) \: dx := \sum_{j} \La[\vsigjj](b_1,\ldots, b_n, (I-P_j) f, (I-P_j)g),
\end{equation*}
\begin{equation*}
\int g(x) \q( \cV_{k_1} f\w)(x)\: dx:= \sum_{j} \La[\vsigjj](b_1, \ldots, b_n, Q_{j+k_1}f, (I-P_j) g),
\end{equation*}
\begin{equation*}
\int g(x) \q( \cV_{k_1,k_2} f\w)(x)\: dx:= \sum_{j} \La[\vsigjj](b_1, \ldots, b_n, Q_{j+k_1} f, Q_{j+k_2} g).
\end{equation*}
The estimate \eqref{ThmBoundRoughest} is equivalent to
\begin{equation}\label{EqnBoundRoughestToShow}
\LpOpN{p}{\cV} \lesssim  \sup_{j} \LpN{1}{\vsig_j} \log^{2} \q( 1+n\Ga_\eps).
\end{equation}
In light of \eqref{EqnAuxOpId}, we have the following identities,
\begin{equation*}
\cV= \sum_{k_1>0} \cV_{k_1}, \quad \cV_{k_1}=\sum_{k_2>0} \cV_{k_1,k_2}.
\end{equation*}

To see \eqref{EqnBoundRoughestToShow} we first use 
Theorem  \ref{PropL2BoundFor2Qs} to deduce
\begin{equation*}
\LpOpN{2}{\cV_{k_1,k_2}}\lesssim \min\{ \sup_j \sBN{\eps}{\vsig_j} n 2^{-(k_1+k_2)c\eps},\, \sup_j \LpN{1}{\vsig_j} \}.
\end{equation*}
Thus, by Lemma \ref{LemmaBasicSum},
\begin{equation*}
\begin{split}
\LpOpN{2}{\cV_{k_1}}&\lesssim \sum_{k_1>0}  \min\big\{\sup_j \sBN{\eps}{\vsig_j} n 2^{-(k_1+k_2)c \eps}\w, \,\, \sup_j \LpN{1}{\vsig_j}\big\}
\end{split}
\end{equation*}
which implies
\Be\label{LemmaBoundRoughest1}
\LpOpN{2}{\cV_{k_1}}\lesssim 
\sup_j\|\vsig_j\|_{L^1} \min\{ 
n\Ga_\eps 2^{-k_2c\eps},\, \log\q(1+n \Ga_\eps)\}\,.
\Ee

We turn to the proof of   \eqref{EqnBoundRoughestToShow}. 
Define an operator $W_j$ by
 $$\La[\vsigjj](b_1,\ldots, b_n,  b_{n+1}, b_{n+2})=\int W_jb_{n+1}(x) b_{n+2} (x) \,dx.$$
The Schwartz kernel of $W_j$ is $\Dil_{2^j} w_j(x,y)$ where
\Be\label{schwartzkernelofwj}
w_j(x,y)=  \int\vsig_j(\alpha, x-y)
\prod_{i=1}^n b_i(2^{-j}(x-\alpha_i(x-y))\, d\alpha\,.
\Ee
We observe that $\cV_{k_1}=\sum_j (I-P_j)W_jQ_{j+k_1}$.
If we set $S_j=(I-P_j) W_j$ then the Schwartz kernel of $S_j$ is $\Dil_{2^j} s_j$
where $s_j(x,y) = w_j(x,y)-\int \phi(x-x') w_l(x',y)$.
It is easy to see that
$\Sha^1(s_j) \lc \|\vsig\|_{L^1}=:A$ and 
$\Sha_\eps^1(s_j) \lc \|\vsig\|_{\cB_\eps}=:B$. 

We wish to apply  Corollary \ref{CorWeakTypeWithQ}, with $S^{k_1}\equiv \sum S_j Q_{j+k_1}=\cV_{k_1}$.
By Lemma \ref{LemmaBoundRoughest1}, we have
$D_{\eps'}\lesssim
 \sup_{j} \sBN{\eps}{\vsig_j} $ and $D_0\lesssim  \q(\sup_j \LpN{1}{\vsig_j}\w) \log\q(1+n \frac{\sup_j \sBN{\eps}{\vsig_j} }{\sup_j \LpN{1}{\vsig_j} }\w)$.
Plugging this into the conclusion of Corollary \ref{CorWeakTypeWithQ}, \eqref{EqnBoundRoughestToShow} follows, and the proof is complete.
\end{proof}

\subsubsection*{Proof of Theorem \ref{PropL2BoundFor2Qs}}  In light of Theorem \ref{ThmOpResAdjoints}, it suffices to prove Theorem \ref{PropL2BoundFor2Qs}
in the case $l_1=n+1$, $l_2=n+2$.  We may also assume
the normalizations
\Be \label{bjl-and-u-normalization}
\begin{aligned}
&\sup_{j} \|{b_j^l}\|_\infty =1, \quad 1 \leq l\leq n, 
\\
&\sUN{\zeta_1}=1=\sUN{\zeta_2}.
\end{aligned}
\Ee
  With these reductions, our goal is to show
\begin{equation}\label{EqnBoundRoughestToShowProp}
\LpOpN{2}{S_{k_1,k_2}^{n+1,n+2}} \lesssim\max  \big\{\sup_j \sBN{\eps}{\vsig_j} n 2^{-(k_1+k_2)c\eps}, \, 
\sup\|\vsig_j\|_{L^1}\big\}\,.
\end{equation}
To finish the proof we define, for $j\in \Z$, $k_1,k_2\in \N$,  an operator $S_{j,k_1,k_2}\equiv S_{j,k_1,k_2}^{n+1,n+2}$ by
\begin{equation*}
\int g(x)\, S_{j,k_1,k_2} f \q(x\w)\: dx = \La[\vsigjj] (b_1^j, \ldots, b_n^j, \Qb{j+k_1}{\zeta_1} f, \Qb{j+k_2}{\zeta_2} g),
\end{equation*}
so that $S_{k_1,k_2}^{n+1,n+2} = \sum_{j\in \Z} S_{j,k_1,k_2}$.

We claim that there is $c>0$ 
such that for $j,k_1',k_2'\in \Z$, $k_1,k_2\in \N$,
\begin{multline}
\label{LemmaBoundRoughestPieces}
\big\|\Qt_{k_1'} S_{j+k_1', k_1,k_2} \Qt_{j+k_1'+k_2' }\big\|_{L^2\to L^2}
\\ \lesssim \min\big\{\sup_j \sBN{\eps}{\vsig_j} n 2^{-(k_1+k_2)c\eps},\, 2^{-|k_2-k_2'|-|k_1+j|}\sup_j \LpN{1}{\vsig_j}\big\}.
\end{multline}
To see this observe first that
using 
\[\big\|\Qt_{k_1'} {}^t\Qb{j+k_1'+k_1}{\zeta_1}\big\|_{L^2\to L^2}\lesssim 2^{-|k_1+j|}, \qquad
\LpOpN{2}{\Qb{j+k_1'+k_2}{\zeta_2} \Qt_{j+k_1'+k_2'}}\lesssim 2^{-|k_2-k_2'|},\]
it follows from the simple  Lemma \ref{LemmaBasicL2BasicLpEstimate} that
\begin{equation*}
\big\|\Qt_{k_1'} S_{j+k_1', k_1,k_2} \Qt_{j+k_1'+k_2' }\big\|_{L^2\to L^2} \lesssim2^{-|k_2-k_2'|-|k_1+j|} \LpN{1}{\vsig_{j+k_1'}}.
\end{equation*}
Using $\LpOpN{2}{\Qt_{k_1'}}, \LpOpN{2}{\Qt_{j+k_1'+k_2'}}\lesssim 1$, it follows from the main $L^2$-estimate, Theorem \ref{ThmBasicL2Result}, that
\begin{equation*}
\big\|\Qt_{k_1'} S_{j+k_1', k_1,k_2} \Qt_{j+k_1'+k_2' }\big\|_{L^2\to L^2} \lesssim \sBN{\eps}{\vsig_{j+k_1'}} n 2^{-(k_1+k_2)c\eps}
\end{equation*}
for some $c>0$ (independent of $n$). Inequality \eqref{LemmaBoundRoughestPieces}  follows from a combination of the two bounds.

To prove \eqref{EqnBoundRoughestToShowProp} we use  Lemma \ref{ThmAlmostOrthogonal} and inequality \eqref{LemmaBoundRoughestPieces} to conclude 
\begin{equation*}
\begin{split}
\LpOpN{2}{S_{k_1,k_2}^{n+1,n+2}}&\lesssim \sum_{j,k_2'\in \Z} \min\big\{\sup_{j'} \sBN{\eps}{\vsig_{j'}} n 2^{-(k_1+k_2)c \eps} ,\, 2^{-|k_2-k_2'|-|k_1+j|}\sup_{j '}\LpN{1}{\vsig_{j'}}\big\}
\\&\lesssim \min\big\{\sup_j \sBN{\eps}{\vsig_j} n^{1/2} 2^{-(k_1+k_2) c\eps/2},\, \sup_j \LpN{1}{\vsig_j}\big\},
\end{split}
\end{equation*}
where we have used $\LpN{1}{\vsig_j}\leq \sBN{\eps}{\vsig_j}$.  This completes the proof (with $c$ replaced by $c/2$). \qed

\section{Proof of Theorem \ref{main-parts}: Part \ref{ItemBoundT1}}
\label{Sectionpartt1}

This section is devoted to the boundedness of the multilinear forms 
 $\La_{l,n+2}^1$ and 
 $\La_{l,n+1}^1$.
In \S\ref{mainLp10}  we shall formulate and prove a crucial  $L^2$ bound for a useful generalization
 of the form of $\La_{l,n+2}^1$ and then deduce the asserted  estimates for  $\La_{l,n+2}^1$, and  $\La_{l,n+2}^1$. The proof of the main $L^2$ bound will be given in \S\ref{pfofmainL2est10}.

 \subsection{The main $L^2$ estimate}\label{mainLp10}

For $2\leq l\leq n$, fix bounded sets
$\q\{b^j_l: j\in \Z\w\}\subset L^\infty(\R^d)$ with $$\sup_{j} \|b^j_l\|_\infty\le 1, \quad l=2,\dots, n.$$
For $b_1\in L^\infty(\R^d)$, $j\in \Z$ define an operator 
\[W_j[\vsig_j, b_1]\equiv W_j[\vsig_j, b_1, b_2^j,\dots, b_n^j]\] by 
\begin{equation*}
\int g(x) \,W_j[\vsig_j, b_1] f(x) \: dx= \La[\vsigjj](b_1,b_2^j,\ldots, b_n^j,  f, g),
\end{equation*}
and we denotes its transpose by ${}^tW_j[b_1]$:
\begin{equation*}
\int f(x) \,{}^tW_j[\vsig_j, b_1] g(x) \: dx= \La[\vsigjj](b_1,b_2^j,\ldots, b_n^j,  f, g),
\end{equation*}
Define an operator $\cT_N=\cT_N[\vectsig,b_1]$ by
\begin{equation*}
\cT_N=\sum_{j=-N}^N (I-P_j) W_j[\vsig_j, (I-P_j)b_1] P_j.
\end{equation*}
Using $I-P_j=\sum_{k>0} Q_{j+k}$  for the operator on the left 
we  decompose $\cT_N=\sum_{k>0} \cT^k_N$
where 
\[\cT^k_N=\sum_{j=-N_1} ^{N_2}Q_{j+k} W_j[\vsig_j, (I-P_j)b_1] P_j.\]
We now  state our main estimate and give the proof that it implies 
Part \ref{ItemBoundTwoPts} of Theorem \ref{main-parts} in \S\ref{implyBoundTwoPts}
below.

\begin{thm}\label {ThmBoundT1MainRes}
Let $0<\eps\le 1$, and $\sup_j\|\vsig_j\|_{\cB_\eps}<\infty$. Let $\Gamma_\eps $ be as in 
\eqref{Gammaeps}. Then $\cT_N^k$ converges to an operator $\cT^k$, 
and $\cT_N$ converges to an operator $\cT$, in the strong operator topology as operators 
$L^2\to L^2$.
Moreover,
\begin{equation*}
\|\cT^k\|_{L^2\to L^2}\leq C_{d,\eps} 
\|b_1\|_\infty \sup_j\|\vsig_j\|_{L^1} \min \{2^{-\eps_1 k} n \Ga_\eps^2 , 
\log^{3/2}(1+n\Ga_\eps)\}.
\end{equation*}
and
\begin{equation*}
\|\cT\|_{L^2\to L^2}\leq C_{d,\eps} \|b_1\|_\infty \sup_{j} \LpN{1}{\vsig_j} 
\log^{5/2} \!( 1+n\Ga_\eps).
\end{equation*}

\end{thm}

\subsection{Proof of Theorem  \ref{ThmBoundT1MainRes}}\label{pfofmainL2est10}

For fixed $k>0$, in order to bound $\cT^k$ we need to prove that for $f\in L^2$ the limit
$$\sum_{j=-N}^N Q_{j+k} W_j[\vsig_j, (I-P_j)b_1] P_jf$$
exists in $L^2$ as $N\to \infty$ and that the estimate
\Be \label{Qj+kdec}
\Big\| \sum_{j=-N}^N Q_{j+k} W_j[\vsig_j, (I-P_j)b_1] P_j\Big\|_{L^2\to L^2} 
\lc\|b_1\|_\infty  \sup_j\|\vsig_j\|_{L^1} \min \{2^{-\eps_1 k} n 
\Ga_\eps^2 , \log^{3/2}(1+n\Ga_\eps)\}
\Ee holds uniformly in $N$. 
By Proposition  \ref{PropAuxQb},  both statements 
are a consequence of a  square-function estimate, namely, for $k>0$
\begin{multline} 
\label{square-fctest}
\Big(\sum_{j\in \Z} \big\| 
\Qb{j+k}{\zeta}
W_j[\vsig_j, (I-P_j)b_1] P_jf\big\|_{2} ^2\Big)^{1/2}
\\
\lc\|b_1\|_\infty  \|f\|_2 \|u\|_{\sU} \sup_j\|\vsig_j\|_{L^1} \min \{2^{-\eps_1 k} n \Ga_\eps^2 , \log^{3/2}(1+n\Ga_\eps)\}.
\end{multline}
To show \eqref{square-fctest} one establishes the following two inequalities:
\begin{multline}  \label{square-fctest-canc}
\Big(\sum_{j\in \Z} \int\big|
\Qb{j+k}{\zeta}
W_j[\vsig_j, (I-P_j)b_1] P_jf(x) - 
\Qb{j+k}{\zeta}
W_j[\vsig_j, (I-P_j)b_1] 1(x)\cdot P_jf(x) 
\big|^2\Big)^{1/2}\\ \lc 
 \|f\|_2\|b_1\|_\infty  \|u\|_{\sU} \sup_j\|\vsig_j\|_{L^1} \min \{2^{-\eps_1 k} n \Ga_\eps^2 , \log(1+n\Ga_\eps)\}.
\end{multline} 
and 
\begin{multline}  \label{square-fctest-carl}
\Big(\sum_{j\in \Z} \int\big|
\Qb{j+k}{\zeta}
W_j[\vsig_j, (I-P_j)b_1] 1(x)
\,\cdot\, P_jf(x) 
\big|^2\Big)^{1/2} \\\lc 
 \|f\|_2 \|b_1\|_\infty \|u\|_{\sU} \sup_j\|\vsig_j\|_{L^1} \min \{2^{-\eps_1 k} n \Ga_\eps^2 , \log^{3/2}(1+n\Ga_\eps)\}.
\end{multline} 
For 
the proof of \eqref{square-fctest-carl} we need the notion of a {\it Carleson function}.
\begin{defn}\label{carlesonfunction}
We say a measurable function $w:\bbR^d\times \bbZ\rightarrow \bbC$ is a Carleson function
if there is a constant $c$ such that for all $k\in \bbZ$ and balls $B$ of radius $2^{-k}$ ($k\in \Z$),
\begin{equation*}
\Big(\frac{1}{|B|} \int_B \sum_{j=k}^\infty |w(x,j)|^2\: dx\Big)^{\frac{1}{2}}\leq c<\infty.
\end{equation*}
The smallest such $c$ is denoted by $\|w\|_\carl$.
\end{defn}

\begin{rem} $w$ is a Carleson function if the measure 
$d\mu(x,t)=\sum_{j\in \bbZ} |w(x,j)|^2 dx \:d \delta_{2^{-j}}(t)$ is a Carleson measure on   the upper half plane (in the usual sense)  and the norm  $\|w\|_{carl} $ is equivalent with  the square root of the Carleson norm  of $\mu$.
\end{rem}

Carleson measures or Carleson functions can be used to prove
$L^2$-boundedness of nonconvolution operators. 
This  idea  goes back  to Coifman and Meyer \cite[ch. VI] {coi-me-audela} and was crucial in the proof of the David-Journ\'e theorem \cite{david-journe}.
One  uses Carleson functions via the following special case of 
the {\it Carleson embedding theorem}.
A proof can be found e.g. in \cite[\S 6.III]{journe} or
\cite[\S II.2]{stein-ha}.

\begin{named}{Theorem}
Let $w$ be a Carleson function.  Then,
\begin{equation*}
\Big(\sum_{j\in \Z} \int |P_j f (x)|^2 |w(x,j)|^2\: dx\Big)^{\frac{1}{2}} \leq C_d \| w\|_{carl}\|f\|_2.
\end{equation*}
\end{named}
 
Note that
\eqref{square-fctest-carl} is an immediate consequence of this theorem and the following proposition.
\begin{prop} \label{PropT1Carleson}
 The function 
 $$w_k(x,j)=\Qb{j+k}{\zeta}
W_j[\vsig_j, (I-P_j)b_1] 1(x)$$  
defines a Carleson function and there is $C\lc 1$ so that for $0<\eps'\le C^{-1}\eps^2$
we have the estimate
\Be
\|w_k\|_{\carl}
\lc 
 \Qb{j+k}{\zeta}\|b_1\|_\infty \|u\|_{\sU} \sup_j\|\vsig_j\|_{L^1} \min \{2^{-k\eps' } n \Ga_\eps^2 , \log^{3/2}(1+n\Ga_\eps)\}.
 \Ee
 \end{prop}
 
 Our next proposition is a restatement of the other square-function estimate \eqref{square-fctest-canc}, in a slightly more general form.
\begin{prop} \label{PropT1Canc}
Let $0<\eps\le 1$. There exists $C\lc 1$ so that for $0<\eps'\le C^{-1}\eps$
\begin{multline*}  
\Big(\sum_{j\in \Z} \int\big|
\Qb{j+k}{\zeta}
W_j[\vsig_j, b_1^j] P_jf(x) - 
\Qb{j+k}{\zeta}
W_j[\vsig_j, b_1^j] 1(x)\cdot P_jf(x) 
\big|^2\Big)^{1/2}\\ \lc 
 \|f\|_2\sup_j\|b_1^j\|_\infty  \|u\|_{\sU} \sup_j\|\vsig_j\|_{L^1} 
 \min \{2^{-k\eps' } n \Ga_\eps^2 , \log(1+n\Ga_\eps)\}.
\end{multline*} 
 \end{prop}

 We emphasize that the  implicit constants in the above propositions are independent of $n$ and independent of the choices of $b_i^j$ 
 with $\|b_i^j\|_\infty=1$.

\subsubsection{Proof of Proposition \ref{PropT1Carleson}}
 We need to prove for $x_0\in \bbR^d$, $\ell\in \bbZ$,
 \begin{multline}\label{CarlesonQW}
 \Big(\sum_{j\ge -\ell}
 \frac{1}{|B^d(x_0, 2^\ell)|}\int_{B^d(x_0, 2^\ell)} 
 \big|\Qb{j+k}{\zeta} W_j[\vsig_j, (I-P_j) b_1, b_2^j, \dots, b_n^j] 1 (x)
 \big| \,dx\Big)^{1/2} \\ \lc
 \|b_1\|_\infty \|u\|_{\sU} \sup_j\|\vsig_j\|_{L^1}
  \min \{2^{-k\eps' } n \Ga_\eps^2 , \log^{3/2}(1+n\Ga_\eps)\}.
 \end{multline}
 Now 
\begin{multline*} \frac{1}{|B^d(x_0, 2^\ell)|}\int_{B^d(x_0, 2^\ell)} 
 \big|\Qb{j+k}{\zeta} W_j[\vsig_j, (I-P_j) b_1, b_2^j, \dots, b_n^j] f (x)
 \big| \,dx\\= 
 \frac{1}{|B^d(0, 1)|}\int_{B^d(0, 1)} 
 \big|\Qb{j+k}{\zeta} W_j[\vsig_j, (I-P_j) b_1, b_2^j, \dots, b_n^j] f (x_0+2^{\ell }x)
\big | \,dx
 \end{multline*}
 and we have by changes of variables 
\begin{multline} \label{changeofvarQW}
\Qb{j+k}{\zeta} W_j[\vsig_j, (I-P_j) b_1, b_2^j, \dots, b_n^j] f (x_0+2^{\ell }x)\\=
\Qb{j+\ell+k}{\zeta} W_{j+\ell}[\vsig_j, (I-P_{j+\ell}) \tilde b_1, \tilde b_2^j, \dots, \tilde b_n^j]
\tilde  f (x)
\end{multline}
where $\tilde b_1(x)= b_1(x_0+2^\ell x)$, $\tilde b_i^j(x)= b_i^j(x_0+2^\ell x)$,  $\tilde f(x)= f(x_0+2^\ell x)\,.$
 Applying this with $f=1$ we see that it suffices to prove
 \eqref{CarlesonQW} with $x_0=0$, $\ell=0$.

 The somewhat lengthy proof will be given in a series of lemmata, partially relying on the $L^2$ boundedness results in \S \ref{Sectionltwo}. Our first lemma is a restatement of such a result.

\begin{lemma}\label{LemmaBoundT1QbUjL2Bound}
Let $0<\eps<1$. There is  $C\lc 1$ so that for all 
$\eps' \le C^{-1}\eps$ we have for all $k\geq 0$,  and for all $\zeta\in \sU$,
\begin{equation*}
\LpOpN{2}{\Qb{j+k}{\zeta}
 W_j[\vsig_j, b_1]} \lesssim \min\big\{n 2^{-k\eps'} \sBN{\eps}{\vsig_j},\, \LpN{1}{\vsig_j}\big\}\,  
 \|{b_1}\|_\infty \sUN{\zeta}.
\end{equation*}
\end{lemma}
\begin{proof}
For $f,g\in L^2$, we have
\begin{equation*}
\int g(x) \q( \Qb{j+k}{\zeta} W_j[\vsig_j, b_1] f(x)\w)\: dx = \La[\vsigjj]( b_1,b_2^j,\ldots, b_n^j, f, {}^t\Qb{j+k}{\zeta} g).
\end{equation*}
From here, the result follows immediately from Theorem \ref{ThmBasicL2Result}.
\end{proof}

We  now 
give an estimate on 
$\La[\vsigj](b_1,\dots, b_{n+2})$ under the assumptions that 
the supports of $b_1$ and $b_{n+2}$ are separated.

\begin{lemma} Let $0<\eps\le 1$. 
For all $j,k\geq 0$, $\vsig\in \sBtp{\eps}{\R^n\times \R^d}$, $\zeta\in \sU$, $R\geq 5$, $b_1,\ldots, b_{n+1}\in L^\infty(\R^d)$, $b_{n+2}\in L^1(\R^d)$,
with $$\supp{b_1}\subseteq \q\{|v|\geq R\w\}, \quad \supp{b_{n+2}}\subseteq \q\{|v|\leq 1\w\},$$ we have
\begin{equation*}
\big|\La[\vsigjj](b_1,\ldots,b_{n+1}, \Qb{j+k}{\zeta}b_{n+2})\big|\lesssim \sUN{\zeta}\big(\prod_{l=1}^{n+1} \|{b_l}\|_\infty\big) \LpN{1}{b_{n+2}} \min\big\{
 (2^{j}R\w)^{-\eps/4}\sBN{\eps}{\vsig},
\, \LpN{1}{\vsig}\big\}.
\end{equation*}
\end{lemma}
\begin{proof}
Without loss of generality, we take $\LpN{\infty}{b_l}=1$, $1\leq l\leq n+1$, $\LpN{1}{b_{n+2}}=1$, and $\sUN{\zeta}=1$.
The bound
\begin{equation*}
\q|\La[\vsigjj](b_1,\ldots,b_{n+1}, \Qb{j+k}{\zeta}b_{n+2})\w|\lesssim \LpN{1}{\vsig}
\end{equation*}
follows immediately from Lemma \ref{LemmaBasicL2BasicLpEstimate}, so we prove only the estimate
\begin{equation}\label{2^jga-after-normal}
\q|\La[\vsigjj](b_1,\ldots,b_{n+1}, \Qb{j+k}{\zeta}b_{n+2})\w|\lesssim\sBN{\eps}{\vsig} \q(2^jR)^{-\eps/4}.
\end{equation}
We estimate
\begin{equation*}
\begin{split}
&\q| \La[\vsigjj](
b_1,\ldots, b_{n+1}, \Qb{j+k}{\zeta} b_{n+2})\w| 
\\&= \Big|\iiiint \dil{\vsig}{2^j}(\alpha,v) 
\big(\prod_{i=1}^n b_i(x-\alpha_i v)\big)
b_{n+1}(x-v) \dil{\zeta}{2^{j+k}}(x-x') b_{n+2}(x')\: dx\: dx'\: d\alpha\: dv\Big|
\\&\leq \sup_{|x'|\leq 1} \iiint \q| \dil{\vsig}{2^j}(\alpha,v)\w| |b_1(x-\alpha_1 v)| \q|\dil{\zeta}{2^{j+k}}(x-x')\w|\: dx\: d\alpha\: dv.
\end{split}
\end{equation*}
Fix $x'\in \R^d$ with $|x'|\leq 1$.  Then 
\begin{equation*}
\begin{split}
& \iiint \q| \dil{\vsig}{2^j}(\alpha,v)\w| |b_1(x-\alpha_1 v)| \q|\dil{\zeta}{2^{j+k}}(x-x')\w|\: dx\: d\alpha\: dv
\\&\leq \iiint |\vsig(\alpha,v)| |b_1(x-\alpha_1 2^{-j} v)| 2^{d(j+k)} (1+2^{j+k} |x-x'|)^{-d-\frac{1}{2}}\: dx\: d\alpha\: dv
\\& =\sum_{l_1=0}^\infty \sum_{l_2=0}^\infty \: \iiint\limits_{\substack{2^{l_1}\leq 1+|v|\leq 2^{l_1+1} \\ 2^{l_2}\leq 1+2^{j+k} |x-x'|\leq 2^{l_2+1}}}  |\vsig(\alpha,v)| |b_1(x-\alpha_1 2^{-j}v)| 2^{d(j+k)} (1+2^{j+k}|x-x'|)^{-d-\frac{1}{2}}\: dx\: d\alpha\: dv
\\&=\sum_{l_1=0}^\infty\sum_{2^{l_2}\geq R 2^{j+k-2}} + \sum_{l_1=0}^\infty \sum_{2^{l_2}<R 2^{j+k-2}} =:(I)+(II).
\end{split}
\end{equation*}
We begin with $(I)$.  We have, provided $\eps'\leq \eps$,
\begin{equation*}
\begin{split}
(I)&\le \sum_{l_1=0}^\infty \sum_{2^{l_2}\geq R 2^{j+k-2}} 2^{-l_1\eps' - l_2/4}\times
\\&\quad\quad\quad  \iiint\limits_{\substack{2^{l_1}\leq 1+|v|\leq 2^{l_1+1} \\ 2^{l_2}\leq 1+2^{j+k} |x-x'|\leq 2^{l_2+1}}} (1+|v|)^{\eps'} |\vsig(\alpha,v)| |b_1(x-\alpha_1 2^{-j}v)| 
\frac{2^{d(j+k)}}{ (1+2^{j+k}|x-x'|)^{d+\frac{1}{4}}}
\: dx\: d\alpha\: dv
\\&\lesssim \sum_{l_1=0}^\infty  \sum_{2^{l_2}\geq R 2^{j+k-2}} 2^{-l_1\eps' - l_2/4} \sBN{\eps}{\vsig} 
\lesssim (2^{j+k}R)^{-1/4}
\sBN{\eps}{\vsig} 
\lesssim (2^{j}R)^{-1/4}\sBN{\eps}{\vsig}.
\end{split}
\end{equation*}

We now turn to $(II)$.  We have 
\begin{multline*}
(II)=\sum_{l_1=0}^\infty \sum_{2^{l_2}<R 2^{j+k-2}} 2^{-l_1\eps' - l_2/4}\,\,\times
\\ \iiint\limits_{\substack{2^{l_1}\leq 1+|v|\leq 2^{l_1+1} \\ 2^{l_2}\leq 1+2^{j+k} |x-x'|\leq 2^{l_2+1}}} (1+|v|)^{\eps'} |\vsig(\alpha,v)| |b_1(x-\alpha_1 2^{-j}v)|
\frac{2^{d(j+k)}}{ (1+2^{j+k}|x-x'|)^{d+\frac{1}{4}}}
 \: dx\: d\alpha\: dv\,.
\end{multline*}
On the support of the integral, $|x-\alpha_1 2^{-j} v|\geq R$ (by the support of $b_1$).  Since $1+2^{j+k}|x-x'|\leq 2^{l_2+1}$, we have 
$|x-x'|\leq 2^{l_2+1-j-k}$.  Thus, $|x|\leq 2^{l_2+1-j-k}+1\leq \frac{R}{2} +1\leq \frac{R}{2}+\frac{R}{5}\leq \frac{3}{4} R$.
Thus, $|\alpha_1 2^{-j} v|\gtrsim R$ and therefore $|\alpha_1|\gtrsim 2^{j} \frac{R}{|v|}\gtrsim 2^{j-l_1}R$.  Plugging this in, we have for $\eps'= \eps/2$,
\begin{equation*}
\begin{split}
(II)&\lesssim \sum_{l_1=0}^\infty \sum_{2^{l_2}< R 2^{j+k-2}} 2^{-l_1\eps' - l_2/4} (1+2^{j-l_1}R)^{-\frac{\eps'}{2}}
  \iiint\limits_{\substack{2^{l_1}\leq 1+|v|\leq 2^{l_1+1} \\ 2^{l_2}\leq 1+2^{j+k} |x-x'|\leq 2^{l_2+1}}} (1+|v|)^{\eps'} 
\times
\\&\qquad\qquad
(1+|\alpha_1|)^{\frac{\eps'}{2}}|\vsig(\alpha,v)| |b_1(x-\alpha_1 2^{-j}v)| 
\frac{2^{d(j+k)}}{ (1+2^{j+k}|x-x'|)^{d+\frac{1}{4}}}
\: dx\: d\alpha\: dv
\\&\lesssim \sum_{l_1=0}^\infty \sum_{2^{l_2}< R 2^{j+k-2}} 2^{-l_1\eps' - l_2/4} (1+2^{j-l_1}R)^{-\frac{\eps'}{2}}\sBN{\eps}{\vsig}
\lesssim (2^jR)^{-\eps'/2}\,.
\end{split}
\end{equation*}
Combine the estimates for $(I)$ and $(II)$ to obtain  
\eqref{2^jga-after-normal}
and the proof of the lemma is complete.
\end{proof}

\begin{lemma}\label{LemmaBoundT13On22}
Let $0<\eps\le 1$. Then for all 
$j,k\geq 0$, $\zeta\in \sU$,  $R\geq 5$, and $b_1\in L^\infty(\R^d)$ with $\supp{b_1}\subseteq \{|v|\geq R\}$
we have
\begin{equation*}
\Big( \int_{|x|\le 1} \q| \q(\Qb{j+k}{\zeta} W_j[\vsig_j,b_1] 1\w)(x) \w|^2\: dx\Big)^{1/2}\lesssim \sUN{\zeta} \|{b_1} \|_\infty
\min\big\{(2^jR)^{-\eps/ 4}\sBN{\eps}{\vsig_j},\,  \LpN{1}{\vsig_j}\big\}.
\end{equation*}
\end{lemma}
\begin{proof}

Let $B=\{x:|x|\le 1\}$. 
We have, by the previous lemma,
\begin{equation*}
\begin{split}
&\Big(\int_B \q|\q(\Qb{j+k}{\zeta} W_j[\vsig_j,b_1] 1\w)(x)\w|^2\: dx\Big)^{1/2} \leq 
\sup_{\substack{\|b_{n+2}\|_1=1 \\ \supp{b_{n+2}}\subseteq B } } \big|
\La[\vsigjj](b_1, b_2^j,\ldots, b_n^j,1, {}^t\Qb{j+k}{\zeta}b_{n+2} )\big|
\\&\lesssim \sup_{\substack{\|b_{n+2}\|_1=1 \\ \supp{b_{n+2}}\subseteq B } } \sUN{\zeta} 
\|{b_1}\|_\infty\|b_{n+2}\|_1
\min\big\{(2^jR)^{-\eps /4}\sBN{\eps}{\vsig_j},\,  \LpN{1}{\vsig_j}\big\}
\end{split}
\end{equation*}
and the assertion follows.
\end{proof}


For $j,k_1,k_2\geq 0$ and $\zeta\in \sU$, define an operator 
$V_{j,k_1,k_2}\equiv V_{j,k_1,k_2}^{\vsig_j,\zeta}$ by
\begin{align*}
\int f(x) \,V_{j,k_1,k_2}g(x)\: dx &= \int g(x) \,\q(\Qb{j+k_1}{\zeta} W_j [ \vsig_j, Q_{j+k_2} f] 1\w)(x)\: dx \\
&= \La[\vsigjj](Q_{j+k_2} f, b_2^j,\ldots, b_n^j,1, {}^t\Qb{j+k_1}{\zeta}g ).
\end{align*}
%
%



\begin{lemma}\label{CorBoundT1V}
Let $0<\eps\le 1$. There exists $c>0$  
(independent of $n$ and $\eps$) 
such that for $\eps'\le c\eps$, $k_1, k_2\ge 0$,
and for all $f\in L^2(\R^d)$,
\begin{equation*}
\Big(\int\sum_{j\geq 0} \q| {}^tV_{j,k_1,k_2} f (x)\w|^2\: dx\Big)^{1/2} \lesssim \LpN{2}{f} \sUN{\zeta}  \sup_j \|\vsig_j\|\ci{L^1} 
 \min 
\big\{ 1, n 2^{-\eps' (k_1+ k_2)} \Ga_\eps\big\}.
\end{equation*}
\end{lemma}
\begin{proof}
From Theorem \ref{PropL2BoundFor2Qs}
we get the bound
\begin{equation} \label{sumVjk's}
\Big\|\sum_{j\geq 0} V_{j,k_1,k_2}\Big\|_{L^2\to L^2}
\lesssim \sUN{\zeta}
\min\big\{ 1, n 2^{-\eps' (k_1+ k_2)} \Ga_\eps\big\}.
\end{equation}
Let $\delta_j$ be any sequence of $\pm 1$.  Note that $\delta_j V_{j,k_1,k_2}$ is of the same form as $V_{j,k_1,k_2}$ with $\vsig_j$
replaced by $\delta_j \vsig_j$.  Thus, by \eqref{sumVjk's},
\begin{equation*}
\Big\|\sum_{j\geq 0} \delta_j {}^tV_{j,k_1,k_2}f\Big\|_2\lesssim \LpN{2}{f} \sUN{\zeta} \min\big\{ 1, n 2^{-\eps' (k_1+ k_2)} \Ga_\eps\big\},
\end{equation*}
where the implicit constant does not depend on the particular sequence $\delta_j$.
By  Khinchine's inequality 
\begin{equation*}
\Big(\int\sum_{j\geq 0} \q| {}^tV_{j,k_1,k_2}f (x)\w|^2\: dx\Big)^{1/2} \lesssim \sup
\Big\|\sum_{j\geq 0} \delta_j \,{}^tV_{j,k_1,k_2}f\Big\|_2,
\end{equation*}
where the sup is taken over all $\pm1$-sequences $\{\delta_j\}$. The result follows.
\end{proof}

\begin{lemma} Let $0<\eps<1$. There exists $c>0$ (independent of $n$ and $\eps$) 
so that for $\eps'\le c\eps^2$, for all  $ b_1\in L^\infty(\R^d)$, for all $ \zeta\in \sU$, 
\begin{multline*}
\Big(\sum_{j\geq 0} \int_{|x|\le 1} \big| \q( \Qb{j+k_1}{\zeta} W_j[\vsig_j,(I-P_j) b_1] 1 \w)(x) \big|^2\: dx\Big)^{\frac{1}{2}}
\\ \le C(\eps,d)\sUN{\zeta} \|{b_1} \|_\infty
\sup_j\|\vsig_j\|_{L^1} 
\min\{ 2^{-k_1\eps'} n\Gamma_\eps^2, \log^{3/2}(1+n\Gamma_\eps)\}.
\end{multline*}
\end{lemma}
\begin{proof}
Fix $b_1\in L^\infty(\R^d)$ and $\zeta\in \sU$.  We may assume $\LpN{\infty}{b_1}=1$ and $\sUN{\zeta}=1$.
Fix $0<\beta\leq 1$ and $\delta>0$ to be chosen later, see \eqref{betadelta} below.  Given  $k_1,k_2\geq 0$ we decompose $b_1= 
b_{1,\infty}^{k_1,k_2}+b_{1,0}^{k_1,k_2}$ where 
\begin{align*}
b_{1,\infty}^{k_1,k_2}(y) &:= 
\begin{cases} b_1(y) &\text{ if } |y|\ge \max \{ 10, \beta\, 2^{1+\delta(k_1+k_2)}\}
\\
0 &\text{ if } |y|< \max \{ 10, \beta \,2^{1+\delta(k_1+k_2)}\}
\end{cases},
\\
 b_{1,0}^{k_1,k_2}(y) &:= b_1(y)-b_{1,\infty}^{k_1,k_2}(y).
 \end{align*}
We expand $I-P_j=\sum_{k_2}Q_{j+k_2}$ and then have
\begin{equation*}
\Big(\sum_{j\geq 0} \int_B \q| \q(\Qb{j+k_1}{\zeta}W_j[\vsig_j, (I-P_j)b_1] 1  \w)(x)  \w|^2\: dx\Big)^{1/2} \le (I)+(II)
\end{equation*} where
\begin{align*}
(I)&:=\sum_{k_2>0 } \Big(\sum_{j\geq 0} \int_B \q| \q( \Qb{j+k_1}{\zeta} W_j[\vsig_j,Q_{j+k_2} b_{1,\infty}^{k_1,k_2}] 1 \w)(x) \w|^2\:dx\Big)^{1/2}\,,
\\
(II)&:=\sum_{k_2>0 } \Big(\sum_{j\geq 0} \int_B \q| \q( \Qb{j+k_1}{\zeta} W_j[\vsig_j, Q_{j+k_2} b_{1,0}^{k_1,k_2}] 1 \w)(x) \w|^2\:dx\Big)^{1/2}\,.
\end{align*}
We begin by estimating $(I)$.  Because $j,k_2\geq 0$, $$\supp{Q_{j+k_2} b_{1,\infty}^{k_1,k_2}}\subseteq\{y: |y|\ge R_{k_1,k_2}\} \text{ where $R_{k_1,k_2}:=\max \{5,\, \beta 2^{(k_1+k_2)\delta}\},$}
$$
we may apply Lemma \ref{LemmaBoundT13On22} to see
\begin{equation*}
\begin{split}
(I)&=\sum_{k_2>0 } \Big(\sum_{j\geq 0} \int_B \q| \q( \Qb{j+k_1}{\zeta} W_j[\vsig_j, Q_{j+k_2} b_{1,\infty}^{k_1,k_2}] 1 \w)(x) \w|^2\:dx\Big)^{1/2}
\\&\lesssim \sum_{k_2>0}\Big(  \sum_{j\geq 0} \min\big\{(2^j R_{k_1,k_2})^{-2\eps/4} \sBN{\eps}{\vsig_{j}}^2, \sup_{j'} \LpN{1}{\vsig_{j'}}^2\big\}  \Big)^{1/2}
\\&\leq 
\sup_{j'} \LpN{1}{\vsig_{j'}}
\sum_{k_2>0} \Big( \sum_{j\geq 0}\min\{ 1, 
2^{-j\eps/2-(k_1+k_2)\eps\delta/2}  \beta^{-\eps/2} \Gamma_\eps^2
\Big)^{1/2}
\\&\lc \sup_{j} \LpN{1}{\vsig_{j}}
\sum_{k_2>0} \min\{ 1,\, 2^{-(k_1+k_2)\eps\delta/4}\beta^{-\eps/4} \Gamma_\eps\}
\log^{1/2}(1+\beta^{-2\eps/4}\Gamma_\eps^2)
\\&\lc \sup_{j} \LpN{1}{\vsig_{j}}
\min\{ 1,\, 2^{-k_1\eps\delta/4}\beta^{-\eps/4} \Gamma_\eps\}
\log^{3/2}(1+\beta^{-\eps/4}\Gamma_\eps).
\end{split}
\end{equation*}

We now turn to $(II)$.  We have
$\|b_{1,0}^{k_1,k_2}\|_2
\lc R_{k_1,k_2}^{d/2} 
\|b_{1,0}^{k_1,k_2}\|_\infty
\lc R_{k_1,k_2}^{d/2} 
$ and use  Lemma  \ref{CorBoundT1V} to estimate, for some $c_1\in (0,1)$,
\begin{equation}\label{EqnBoundT1PropII}
\begin{split}
(II) &= \sum_{k_2>0 } \Big(\sum_{j\geq 0} \int_B \q| \q( \Qb{j+k_1}{\zeta} W_j[\vsig_j,Q_{j+k_2} b_{1,0}^{k_1,k_2}] 1 \w)(x) \w|^2\:dx\Big)^{\frac{1}{2}}
\\&=\sum_{k_2>0} \Big(\sum_{j\geq 0} \int_B \q|{}^tV_{j,k_1,k_2}
 b_{1,0}^{k_1,k_2} (x)\w|^2 \: dx\Big)^{\frac{1}{2}}
\\&\lc 
 \sup_{j} \LpN{1}{\vsig_j} 
  \sum_{k_2>0}\LpN{2}{b_{1,0}^{k_1,k_2}}
\min\{1,  \,n 2^{-c_1\eps(k_1+ k_2)}  \Gamma_\eps\}
\\&\lesssim  \sup_{j} \LpN{1}{\vsig_j} 
\sum_{k_2>0} \big(1+\beta 2^{k_1\delta+k_2\delta}\big)^{d/2} 
\min\{1, \, n 2^{-c_1\eps( k_1+ k_2)}  \Gamma_\eps\}
\\&=
\sum_{\substack{k_2>0 \\ 1< \beta\,2^{k_1\delta+k_2\delta}}}
+\sum_{\substack{k_2>0 \\ 1\geq \beta \,2^{k_1\delta+k_2\delta}}} 
=:(II_1)+(II_2).
\end{split}
\end{equation}
We take
\begin{equation}\label{betadelta}
\beta= (n\Ga_\eps)^{-1/d},\quad
  \quad \delta=\frac{c_1\eps}{2d}.
\end{equation}
Notice that since 
$ \beta\,2^{k_1\delta+k_2\delta}\ge 1$  in the sum $(II_1)$  we may replace the power $d/2$ by $d$ and get, 
 with the choice \eqref{betadelta},
\begin{align*}
\q(\beta 2^{k_1\delta+k_2\delta}\w)^{d/2} (n 2^{-\eps' k_1-\eps'k_2} \Gamma_\eps)
&\le \beta^d n\Gamma_\eps 2^{(k_1+k_2) (\delta d -c_1\eps)}
\\
&\le 2^{-(k_1+k_2) c_1\eps/2}\,
\end{align*}
and thus
\begin{equation*}
\begin{split}
&(II_1) 
\lesssim \sum_{k_2>0} 2^{-(k_1+k_2)c_1\eps/2}
\sup_j \LpN{1}{\vsig_j}\lesssim 2^{-k_1 c_1\eps /2} \sup_j \LpN{1}{\vsig_j}.
\end{split}
\end{equation*}
Next,
\begin{equation*}
\begin{split}
(II_2)&\lesssim \sup_j \LpN{1}{\vsig_j} \sum_{k_2>0}
\min\{1,\,  n 2^{-(k_1+ k_2)c_1\eps}  \Gamma_\eps\}
\\
&\lesssim \sup_j \LpN{1}{\vsig_j}\times 
\begin{cases} \log (2+2^{-c_1\eps k_1} \Ga_\eps n) &\text{ if } 2^{-c_1 \eps k_1}\Gamma_\eps n\ge 1
\\
2^{-c_1\eps k_1} \Ga_\eps n  &\text{ if } 2^{-c_1\eps k_1}\Gamma_\eps n\le 1
\end{cases}
\\&\lesssim \sup_j \LpN{1}{\vsig_j}\min \{2^{-c_1\eps k_1} n\Ga_\eps,\,
 \log(1+n \Ga_\eps)\}.
\end{split}
\end{equation*}

Finally we use the choice \eqref{betadelta} in
the above estimate for $(I)$ and get
\begin{align*}
(I)&\lc  \sup_{j} \LpN{1}{\vsig_{j}}
\min\{ 1,\, 2^{-k_1\frac{c_1\eps^2}{8d}} n^{\frac{\eps}{4d}} \Gamma_\eps^{1+\frac{\eps}{4d}}
\}
\log^{3/2}(1+ n^{\frac{\eps}{4d}} \Gamma_\eps^{1+\frac{\eps}{4d}})
\\
&\lc \sup_{j} \LpN{1}{\vsig_{j}}\min\{1,\, 2^{-k_1 c\eps^2} n\Gamma_\eps^2\}
\log^{3/2}(1+ n\Gamma_\eps)
\end{align*}
with $c=c_1/8d$.
Combining  this estimate with the  above estimates for $(II_1)$ and $(II_2)$  yields the assertion.
\end{proof}

\begin{proof}[Proof of Proposition \ref{PropT1Carleson}, conclusion.]
The lemma is just a restatement of \eqref{CarlesonQW} 
for $x_0=0$ and $\ell=0$ and by 
\eqref{changeofvarQW} we reduced the proof of  \eqref{CarlesonQW}  to this special case. 
\end{proof}


\subsubsection{Proof of Proposition \ref{PropT1Canc}}
We start with an elementary observation for $f\in L^\infty$.

\begin{lemma}\label{LemmaBoundT1LinfNOfOp}
For all $k\geq 0$, $j\in \Z$, $b_1\in L^\infty(\R^d)$, and $\zeta\in \sU$,
\begin{equation*}
\LpN{\infty}{ \Qb{j+k}{\zeta} W_j[\vsig_j,b_1] f} \lesssim 
\sUN{\zeta} \LpN{1}{\vsig_j} \|b_1\|_\infty \|f\|_\infty.
\end{equation*}
\end{lemma}
\begin{proof}
For $g\in L^1$ with $\|g\|_1=1$  we have, using Lemma \ref{LemmaBasicL2BasicLpEstimate},
\begin{equation*}
\begin{split}
&\Big|\int g(x) \q(\Qb{j+k}{\zeta} W_j[b_1] 1  \w)(x)\: dx\Big| =\big| \La[\vsigjj](b_1, b_2^j,\ldots, b_n^j, 
f, {}^t\Qb{j+k}{\zeta} g)\big|
\\&\lesssim \|b_1\|_\infty \|f\|_\infty \|{}^t\Qb{j+k}{\zeta}g\|_1 \LpN{1}{\vsig_j}\lesssim \|{b_1}\|_\infty \sUN{\zeta} \LpN{1}{\vsig_j},
\end{split}
\end{equation*}
completing the proof.
\end{proof}

\begin{lemma}\label{LemmaBoundT1QbUj1Pj}
There is $c\in (0,1)$ (independent of $n$ and $\eps$)  so that for $\eps'\le c\eps^2$, 
and all $k\geq 0$, $j\in \Z$, $\zeta\in \sU$, $b_1\in L^\infty(\R^d)$, $f\in L^2(\bbR^d)$ we have
\begin{equation*}
\Big(\int\big | \Qb{j+k}{\zeta} W_j[\vsig_j, b_1]1(x) P_jf(x)\big|^2dx\Big)^{1/2}
\lc  \|f\|_2  \sUN{\zeta} \LpN{\infty}{b_1}
\min\{\LpN{1}{\vsig_j},\,  n2^{-k\eps'} \|\vsig_j\|_{\cB_\eps}\}.
\end{equation*}
\end{lemma}
\begin{proof}
We may normalize and assume $\|b_1\|_\infty=1$. We may assume, by  scale invariance of the result, that $j=0$ (see \eqref{changeofvarQW}).  The assertion follows then from  the inequality
\Be\label{normalizedcase}
\Big(\int\big | \Qb{k}{\zeta} W_0[\vsig, b_1]1(x) P_0f(x)\big|^2dx\Big)^{1/2}
\lc \|u\|_{\sU} \|b_1\|_\infty \|f\|_2  \min\{ \|\vsig\|_{L^1}, n2^{-k\eps'} \|\vsig\|_{\cB_\eps}
\}.
\Ee
Because the convolution kernel of $P_0$ is supported in $B^d(0,1)$, it suffices to show \eqref{normalizedcase} for functions 
 supported in a ball $B$ of radius $1$. We may assume 
 (by translating the functions $b_i$) that $B$ is centered at the origin.
 Let $B^*$ be the ball of double radius.

  Now $\|P_0 f\|_\infty \lc\|f\|_2$  for $f$ supported in $B$,
  and  therefore it suffices to show
 \begin{equation}\label{L2Bnorm}
\|\Qb{k}{\zeta} W_0[\vsig, b_1]1\|_{L^2(B^*)} \lesssim \sUN{\zeta}
\min\{ n 2^{-k\eps'} \sBN{\eps}{\vsig}, \, \LpN{1}{\vsig_j} \}.
\end{equation}
To show \eqref{L2Bnorm} we split 
$1= \bbone_{\Omega_{k\delta}}+\bbone_{\Omega_{k\delta}^\complement}$ 
where  $\Omega_{k\delta}=\{x: |x|\le 5\cdot 2^{k\delta}\}$, with a choice of  $\delta\ll \eps$ to be determined.

It follows from Lemma \ref{LemmaBoundT1QbUjL2Bound}
(or directly from  Theorem \ref{ThmBasicL2Result}) that for some $c>0$ (independent of $n$)
\begin{align}\notag
\|\Qb{k}{\zeta} W_0[\vsig, b_1]\bbone_{\Omega_{k\delta}}\|_{L^2(B^*)}& \lc
\|\bbone_{\Omega_{k\delta}}\|_2
 \|u\|_{\sU} 
  \min\{ n2^{-kc\eps} \|\vsig\|_{\cB_\eps}, \|\vsig\|_{L^1}\}
\\
\label{Omkdeltaest} 
& \lc
 \|u\|_{\sU} 
  \min\{ n2^{-k(c\eps-d\delta)} \|\vsig\|_{\cB_\eps}, \|\vsig\|_{L^1}\}
\end{align} and thus we want to choose $\delta\le c\eps(2d)^{-1}$.

Next we estimate the $L^2(B^*)$ norm of $\Qb{k}{\zeta} W_0[\vsig, b_1]\bbone_{\Omega_{k\delta}^\complement}$.
  Let $\tilde \vsig(\alpha,v)=\vsig(1-\alpha_1,\cdots, 1-\alpha_n, v)$ so that
$\sBN{\eps}{\tilde \vsig}\lesssim \sBN{\eps}{\vsig_j}$ and $\LpN{1}{\tilde \vsig}=\LpN{1}{\vsig_j}$.  We have, for $\|g\|_{L^2(B^*)}=1$, 
\begin{equation*}
\begin{split}
&\Big|\int g(x)\q(\Qb{k}{\zeta} W_0[\vsig, b_1] \bbone_{\Omega_{\delta k}^\complement}
\w)(x)\: dx\Big|
\\&= \big|\La[\vsig](b_1, b_2^0,\ldots, b_n^0, 
\bbone_{\Omega_{k\delta }^\complement},{}^t \Qb{k}{\zeta}g)\big|
=
\big|\La[\tilde \vsig](b_1, b_2^0,\ldots, b_n^0,{}^t \Qb{k}{\zeta}g,
\bbone_{\Omega_{k\delta }^\complement} )\big|
\\&=\q\Big|\iiiint \tilde \vsig(\alpha,v) b_1(x-\alpha_1 v) 
\big(\prod_{i=2}^n b_i^0(x-\alpha_i v)\big) 
\bbone_{\Omega_{\delta k}^\complement} (x)\dil{\zeta}{2^k}(y-x+v) g(y)\: dx\: dy\: dv\: d\alpha\Big|
\end{split}
\end{equation*} and this is estimated by
\begin{equation*}\begin{split}
& \iiiint\limits_{\substack{|x|>5\cdot 2^{\delta k}}} |\tilde\vsig(\alpha,v)| \q|\dil{\zeta}{2^k}(y-x+v) g(y)\w|\: dx\: dv\:d\alpha\: dy
\\&\leq \sum_{2^{l_1}\geq 2\cdot 2^{k\delta}} \sum_{l_2=0}^\infty \iiiint\limits_{\substack{2^{l_1}\leq |x|\leq 2^{l_1+1} \\ 2^{l_2}\leq 1+|v|\leq 2^{l_2+1}}} |\tilde\vsig(\alpha,v)| \q| \dil{\zeta}{2^k}(y-x+v) g(y)\w|\: dx\: dv\: d\alpha\: dy
\\&=  \sum_{2^{l_1}\geq 2\cdot 2^{k\delta}} \sum_{l_2=(l_1-3)\vee 0}^{\infty} + \sum_{2^{l_1}\geq 2\cdot 2^{k\delta}} \sum_{l_2= 0}^{l_1-3}=:(I)+(II).
\end{split}
\end{equation*}
We estimate $(I)\lc$
\begin{equation*}
\begin{split}
&\sum_{2^{l_1}\geq 2\cdot 2^{k\delta}} \sum_{l_2=(l_1-3)\vee 0}^{\infty} 2^{-\eps l_2} \iiiint\limits_{\substack{2^{l_1}\leq |x|\leq 2^{l_1+1} \\ 2^{l_2}\leq 1+|v|\leq 2^{l_2+1}}}(1+|v|)^\eps |\tilde \vsig(\alpha,v)| \q| \dil{\zeta}{2^k}(y-x+v) g(y)\w|\: dx\: dv\: d\alpha\: dy
\\&\lesssim  \sum_{2^{l_1}\geq 2\cdot 2^{k\delta}} \sum_{l_2=(l_1-3)\vee 0}^{\infty} 2^{-\eps l_2} \sBN{\eps}{\tilde \vsig} \sUN{\zeta} \|g\|_1
\lesssim  \sBN{\eps}{\vsig} \sUN{\zeta} \|{g}\|_1 2^{-k\delta \eps} \lesssim  \sBN{\eps}{\vsig} \sUN{\zeta} 2^{-k\delta \eps},
\end{split}
\end{equation*}
where the last inequality uses the support of $g$ to see 
$\|g\|_1\lc \|g\|_2=1$.

For $(II)$, we use the fact that $l_2\leq l_1-3$ to see that on the support of the integral, since $|y|\leq 1$ (due to the support of $g$), we have $|y-x+v|\approx 2^{l_1}$.
Thus, we have
\begin{equation*}
\begin{split}
&(II)\lesssim \sum_{2^{l_1}\geq 2\cdot 2^{k\delta}} \sum_{l_2= 0}^{l_1-3} \sUN{\zeta} \iiiint\limits_{\substack{2^{l_1}\leq |x|\leq 2^{l_1+1} \\ 2^{l_2}\leq 1+|v|\leq 2^{l_2+1}}}
|\tilde\vsig(\alpha,v)| \frac{2^{kd}}{(1+2^k |x-v-y|\w)^{d+\frac{1}{2}}} |g(y)|\: dx\: dv\: dy\: d\alpha 
\\&\lesssim \sum_{2^{l_1}\geq 2\cdot 2^{k\delta}} \sum_{l_2= 0}^{l_1-3}2^{(-k-l_1)/4} \sUN{\zeta} \iiiint\limits_{\substack{2^{l_1}\leq |x|\leq 2^{l_1+1} \\ 2^{l_2}\leq 1+|v|\leq 2^{l_2+1}}}
|\tilde\vsig(\alpha,v)| \frac{2^{kd}}{(1+2^k |x-v-y|\w)^{d+\frac{1}{4}}} |g(y)|\: dx\: dv\: dy\: d\alpha
\\&\lesssim \sum_{2^{l_1}\geq 2\cdot 2^{k\delta}} \sum_{l_2= 0}^{l_1-3}2^{(-k-l_1)/4} \sUN{\zeta} \LpN{1}{\tilde \vsig} \|g\|_1
\lesssim \sUN{\zeta} \LpN{1}{\vsig} \|g\|_1 2^{-k/4} \lesssim 2^{-k/4} \sUN{\zeta} \LpN{1}{\vsig}.
\end{split}
\end{equation*}

Finally, we have, 
by Lemma \ref{LemmaBoundT1LinfNOfOp} applied to $f= \bbone_{\Omega_{k\delta}^\complement}$, 
\begin{equation*}
\Big|\int g(x)\q(\Qb{k}{\zeta} W_0[b_1] \bbone_{\Omega_{k\delta}^\complement}
(x)\: dx\Big| 
\lesssim \LpN{1}{\vsig}\sUN{\zeta},
\end{equation*}
where the last inequality uses the support of $g$ again to see 
$\|g\|_1\lc\|g\|_2=1$.
If we take $\delta= c\eps/(4d)$ then a combination of the estimates  for (I) and (II), and \eqref{Omkdeltaest} , yields 
\eqref{L2Bnorm}
for $\eps'\le c\eps^2/(4d)$. This completes the proof.
\end{proof}

In what follows we find it convenient to occasionally use the notation
\Be \label{multiplicate} \mathrm{Mult}\{g\}f = fg\Ee for the operator of pointwise multiplication with $g$.

\begin{lemma}\label{LemmaBoundT1BoundRkj}
Let $0<\eps\le 1/2$. Then there is $c>0$ (independent of $n,\eps$) such that 
for $\eps'\le c\eps^2$, 
for all $k\geq 0$, $j,l\in \Z$, $\vsig_j\in \cB_\eps$, $\zeta\in \sU$, $b_1\in L^\infty(\R^d)$,
\begin{multline*}
\big\|\Qb{j+k}{\zeta} W_j[\vsig_j, b_1] P_j\Qt_{j+l} - \mathrm{Mult}\{ \Qb{j+k}{\zeta} W_j[\vsig_j,b_1] 1 \} P_j\w\Qt_{j+l} \big\|_{L^2\to L^2}
\\\lesssim \begin{cases}
\sUN{\zeta} \|{b_1}\|_\infty
 \min\{ n\|\vsig_j\|_{\cB_\eps} 2^{-k\eps'},\,\, 2^{-l}\|\vsig\|_{L^1} \} &\text{ if } l\ge 0,
\\
\sUN{\zeta}  \|{b_1}\|_\infty \min\{ n\|\vsig_j\|_{\cB_\eps} 2^{l\eps/4} 2^{-k\eps'},\,\, 
 \|\vsig\|_{L^1} \}  &\text{ if } l\le 0.
 \end{cases}
\end{multline*}
\end{lemma}
\begin{proof}
We may assume $\sUN{\vsig}=1$ and $\LpN{\infty}{b_1}=1$.
We have \Be\label{PjQj+l}
\LpOpN{2}{P_j \Qt_{j+l}}\lesssim \min\{ 2^{-l}, 1\}\,.\Ee
Now, by Lemma \ref{LemmaBoundT1QbUjL2Bound},
\Be\label{lemma105cite}
\LpOpN{2}{\Qb{j+k}{\zeta} W_j[b_1]}\lesssim \LpN{1}{\vsig_j}
\Ee 
and, by Lemma \ref{LemmaBoundT1LinfNOfOp} and \eqref{PjQj+l},
\Be \label{lemma1010cite}
\|\mathrm{Mult}\{\Qb{j+k}{\zeta} W_j[\vsig_j,b_1]1 \} P_j \Qt_{j+l}\|_{L^2\to L^2}
\lc \min\{1, 2^{-l}\} \|\vsig\|_{L^1};
\Ee
moreover, by Lemma \ref{LemmaBoundT1QbUj1Pj},
\Be\label{lemma1011cite}
\|\mathrm{Mult}\{\Qb{j+k}{\zeta} W_j[\vsig_j,b_1] 1\} P_j \Qt_{j+l}\|_{L^2\to L^2}
\lc n 2^{-k\eps'} \|\vsig\|_{\cB_\eps}.
\Ee
A combination of
 \eqref{lemma105cite},  \eqref{lemma1010cite}, and  \eqref{lemma1010cite} immediately gives the assertion for $l\ge 0$, and also the second estimate for $l<0$. It remains to show that
 \begin{multline}\label{maincasel<0}
 \LpOpN{2}{\q(\Qb{j+k}{\zeta} W_j[\vsig_j, b_1] P_j\Qt_{j+l} - \mathrm{Mult}\{ \Qb{j+k}{\zeta} W_j[\vsig_j,b_1] 1 \} P_j\w\Qt_{j+l} }
\\ \lesssim 
n\|\vsig_j\|_{\cB_\eps} \max \{2^{l\eps/2}, 2^{l/4}\}  \,\text{ if } l\le 0;
 \end{multline}
 indeed the assertion follows by taking a geometric mean of the bounds in
 \eqref{lemma1011cite} and \eqref{maincasel<0}.
 
 By scale invariance (see \eqref{changeofvarQW}) it suffices to show \eqref{maincasel<0} for $j=0$, i.e. 
  \Be\label{maincasel<0j=0}
  \big\|
  (R_1-R_2)\Qt_l \big\|_{L^2\to L^2} \lesssim 
n\|\vsig_j\|_{\cB_\eps} \max \{2^{l\eps/2}, 2^{l/4}\}  \,\text{ if } l\le 0;
\Ee
for  $R_1=\Qb{k}{\zeta} W_0[\vsig, b_1] P_0$ and $R_2=\mathrm{Mult}\{ \Qb{k}{\zeta} W_0[\vsig,b_1] 1 \} P_0$.
Let $\rho_1$, $\rho_2$, $\rho$ be the  Schwartz kernels of $R_1$, $R_2$, $R_1-R_2$, and let $\sigma_{-l}$ be the Schwartz kernel of $\Qt_l$. We wish to apply Lemma \ref{LemmaBoundT1RonS} (note the notation  $l=-\ell$ in that lemma).
It is immediate that $\sigma_\ell$ satisfies assumptions  \eqref{EqnBoundSell}
and \eqref{EqnBoundnablaxSell} with $B_1, B_\infty, \widetilde B_\infty\lc 1$.
The function $\rho$ satisfies the crucial cancellation condition \eqref{rhocancel} since \begin{equation*}
\q(\Qb{k}{\zeta} W_0[\vsig,b_1] P_0 - \mathrm{Mult}\{ \Qb{k}{\zeta} W_0[\vsig,b_1] 1 \w\} P_0\w)1 =0\,.
\end{equation*}
It remains to check the size conditions \eqref{EqnBoundR}.
We have
$$|\rho_1(x,y)| \leq \iiint \q| \dil{\zeta}{2^k}(x-x') \w| |\vsig(\alpha, x'-y')| |\phi(y'-y)|\: dx' \:d\alpha\: dy'$$
and thus clearly
$$\sup_y \int |\rho_1(x,y)| dx \le \|\zeta\|_{1} \|\vsig\|_{L^1} \|\phi\|_1 \lc 1$$
since $\|\zeta\|_1\le \|\zeta\|_{\sU}$.
Also for some $M>d+1$,
\begin{align*}
&\int|\rho_1(x,y)|  (1+|x-y|)^\eps  dy
\\&\lc\int (1+|x-y|)^\eps \iiint \Big|\frac{2^{kd}} {(1+2^k |x-x'|)^{d+\frac 12}} 
\frac{|\vsig (\alpha, x'-y')|} {(1+|y'-y|)^M} 
 \:dx'\:d\alpha \:dy' \: dy
\\& \lc
\iiint |\vsig(\alpha,x'-y')|(1+|x'-y'|)^\eps
\om(x,x',y') \,d\alpha\:dy' \:dx'\,,
\end{align*} where 
$$\om(x,x',y') =\frac{2^{kd}} {(1+2^k |x-x'|)^{d+\frac 12}} 
\int \frac{1}{(1+|y'-y|)^M} 
\frac{(1+|x-y|)^\eps}{(1+|x'-y'|)^\eps} \:dy\,.$$
We have $$\sup_{x'}\int |\vsig(\alpha,x'-y')|(1+|x'-y'|)^\eps\: d\alpha \:dy'
\le \|\vsig\|_{\cB_\eps}$$
and thus it suffices to show that
\Be\label{omegaxybd} \sup_{x,y} \int \omega(x,x',y) dx' \lc 1.\Ee
Now by the triangle inequality 
$(1+|x-y|)^\eps\le (1+|x-x'|)^\eps(1+|x'-y'|)^\eps(1+|y'-y|)^\eps$
and hence 
\begin{align*}\int \omega(x,x',y) dx' &\le \int \frac{2^{kd}(1+|x-x'|)^\eps}
 {(1+2^k |x-x'|)^{d+\frac 12}} \int 
\frac{1}{(1+|y'-y|)^{M-\eps}}   dy \: dx' 
\\&\lc
\int \frac{2^{kd}(1+|x-x'|)^\eps}
 {(1+2^k |x-x'|)^{d+\frac 12}} \:dx'
 \end{align*}
 and \eqref{omegaxybd} follows easily,  provided that $\eps<1/2$.
 Thus condition \eqref{EqnBoundR} is satisfied for  $\rho_1$. 
 By Lemma \ref{LemmaBoundT1LinfNOfOp}
 it is immediate that condition \eqref{EqnBoundR} is satisfied for  $\rho_2$ as well. Thus we have verified the assumptions of Lemma
 \ref{LemmaBoundT1RonS}  and 
 \eqref{maincasel<0j=0} follows. This completes the proof of the lemma.
 \end{proof} 
 
 \begin{proof}[Proof of Proposition \ref{PropT1Canc}, conclusion]
 We may assume $\sUN{\zeta}=1$, $\|f\|_2=1$, and $\sup_j \|b_1^j\|_\infty=1$.
For $k\geq 0$, define $$R_{k,j}:=\Qb{j+k}{\zeta}W_j[\vsig_j,b_1^j] P_j- \mathrm{Mult}\{ \Qb{j+k}{\zeta} W_j[\vsig_j,b_1^j]1 \} P_j.$$ 
The proof is complete if we can show, for $k\ge 0$, 
\Be \label{sqfctforRkj}
\Big(\sum_j\|R_{j,k} f\|_2^2\Big)^{1/2} 
\lc \sup_j\|\vsig_j\|_{L^1}
 \min \big\{ 2^{-\eps_1 k}n\Gamma_\eps, \,\log (1+n\Ga_\eps)\big\}.
 \Ee
 Lemma \ref{LemmaBoundT1BoundRkj}
implies
\begin{equation*}
\|R_{k,j} \Qt_{j+l}\|_{L^2\to L^2}\lesssim 
\begin{cases}\sup_{j}\|\vsig\|_{L^1} \min \{ n\Ga_\eps 2^{-k\eps'}, 2^{-l}\}, &\text{ if }l\ge 0,
\\
\sup_{j}\|\vsig\|_{L^1} \min \{ n\Ga_\eps  2^{l\eps/4}2^{-k\eps'},\, 1\} &\text{ if }l<0.
\end{cases}
\end{equation*}
Now
\begin{align*}
&\Big(\sum_j\|R_{k,j} f\|_2^2\Big)^{1/2} = \Big(\sum_{j}\Big\|
 R_{k,j} \sum_{l\in \Z} \Qt_{j+l} \Qtt_{j+l} f\Big\|_2^2\Big)^{\frac{1}{2}} 
 \\
 &\lc\sum_{l\in \bbZ} 
  \Big(\sum_{j}\Big\|
 R_{k,j} \Qt_{j+l} \Qtt_{j+l} f\Big\|_2^2\Big)^{\frac{1}{2}} 
 \lc \sum_{l\in \bbZ} \sup_{j'} \|R_{k,j'} \Qt_{j'+l}\|_{L^2\to L^2}
 \Big(\sum_j \|\Qtt_{j+l} f\|_2^2\Big)^{1/2}
 \\
 &\lc \sup_j\|\vsig_j\|_{L^1} \Big[
 \sum_{l\ge 0} \min \{n \Ga_\eps 2^{-k\eps'}, 2^{-l}\} + \sum_{l<0} \min 
 \{n\Ga_\eps 2^{l\eps/4} 2^{-k\eps'}, 1\}\Big]
 \\
 &\lc \sup_j\|\vsig_j\|_{L^1}
 \min \big\{ 2^{-\eps_1 k}n\Gamma_\eps, \,\log (1+n\Ga_\eps)\big\}
 \end{align*}
for some sufficiently small  $\eps_1>0$,  and the proof is complete.
 \end{proof}

\subsection{Proof that Theorem \ref{ThmBoundT1MainRes} implies 
Part 
\ref{ItemBoundT1} 
of Theorem \ref{main-parts}} \label{implyBoundTwoPts}
Let $1<p\le 2$. The asserted result follows from 
\begin{multline} \label{Lal-n+1-Lp}
\Big|\sum_{j\in \Z} \La[\vsigjj](b_1^j,\ldots, b_{l-1}^j, (I-P_j) b_{l}, b_{l+1}^j,\ldots, b_{n}^j, \q(I-P_j\w) b_{n+1}, P_j b_{n+2}) \Big| \\
\lc \sup_j\|\vsig_j\|_{L^1} 
\log^{5/2} ( 1+n\Ga_\eps)
\big(\prod_{\substack{i=1,\dots, n\\i\neq l}}\sup_j 
\|b_i^j\|_\infty \big)
\|b_l\|_\infty \|b_{n+1} \|_p \|b_{n+2} \|_{p'} \end{multline}
and 
\begin{multline} \label{Lal-n+2-Lp}
\Big|\sum_{j\in \Z} \La[\vsigjj](b_1^j,\ldots, b_{l-1}^j, (I-P_j) b_{l},  b_{l+1}^j,\ldots,  b_{n}^j, P_j b_{n+1}, (I-P_j) b_{n+2}) \Big|\\
\lc \sup_j\|\vsig_j\|_{L^1} \log^{5/2} (1+n\Ga_\eps)
\big(\prod_{\substack{i=1,\dots, n\\i\neq l}} \sup_j \|b_i^j\|_\infty \big)
\|b_l\|_\infty \|b_{n+1} \|_{p} \|b_{n+2} \|_{p'}. \end{multline}
Once \eqref{Lal-n+1-Lp} and 
\eqref{Lal-n+2-Lp} are established we use them 
for the choices
$b_i^j=b_i,$ if $i<l$, 
$b_i^j=P_j b_i,$ if $l<i\le n$. Now it is crucial
that $\|P_j\|_{L^\infty\to L^\infty}\le 1$ 
(here $\phi_j\ge 0$, and $\int\phi_j=1$ are used). Hence the two inequalities for 
$\La^1_{l,n+1}$ and $\La^1_{l,n+2}$ claimed in Theorem \ref{main-parts} are an immediate consequence of \eqref{Lal-n+1-Lp} and 
\eqref{Lal-n+2-Lp}.

In order to  establish 
 \eqref{Lal-n+1-Lp} and 
\eqref{Lal-n+2-Lp} we may assume without loss of generality that $l=1$.
This is because we can permute the first $n$ entries of the multilinear form
and replace $\vsig_j$ by $\ell_\vp\vsig_j$ as in \eqref{firstnvar}.
We may also assume that  $$\|b_1\|_\infty \le 1, \qquad \|b_i^j\|_\infty=1, \,\,2\le i\le n.$$

Now, in what follows let 
$$\tilde \vsig_j(\alpha,v) = \vsig(1-\alpha_1, \dots, 1-\alpha_n,v)$$ 
(as in \eqref{transpostriv}). 
To prove 
\eqref{Lal-n+1-Lp} for $l=1$  we observe
\begin{align*}
\sum_{j} \La[\vsigjj]((I-P_j)b_1,b_2^j,\dots,  b_{n}^j, \q(I-P_j\w) b_{n+1}, P_j b_{n+2}) = \int b_{n+2}(x)  \,{}^t\cT b_{n+1} (x)dx
\end{align*} where
$${}^t\cT = 
\sum_j P_j {}^tW_j[\vsig_j,(I-P_j)b_1](I-P_j) 
=\sum_j P_j W_j[\tilde \vsig_j,(I-P_j)b_1] 
(I-P_j). $$

Now we expand $I-P_j=\sum_{k>0}Q_{j+k}$ and we get 
${}^t\cT= \sum_{k>0} {}^t\cT^k$  where 
$${}^t\cT^k
=\sum_j S_j Q_{j+k}, \text{ with }  S_j= P_j W_j[\tilde \vsig_j,(I-P_j)b_1] .$$
The Schwartz  kernel of $S_j$ is equal to $\Dil_{2^j}s_j$ where
$$s_j(x,y)= \int\phi(x-x') \sigma_j(x',y) dy$$ with
\Be \label{schwartzofsigmaj}
\sigma_j(x,y)= \int\vsig_j(\alpha, x-y)
(I-P_0)b_1(2^{-j}(x-\alpha_i(x-y))
\prod_{i=2}^n b_i(2^{-j}(x-\alpha_i(x-y))\, d\alpha .
\Ee
We wish to apply
Corollary \ref{CorWeakTypeWithQ}.
It is easy to check that
$$\Sha^1[s_j] \lc \sup_j\|\vsig_j\|_{L^1}=:A, \qquad \Sha_\eps^1[s_j] \lc \sup_j\|\vsig_j\|_{\cB_\eps}=:B.$$ 
Now $\|\sum_j S_jQ_{j+k}\|_{L^2\to L^2}
= \|\cT^k\|_{L^2\to L^2} $
and by Theorem  \ref{ThmBoundT1MainRes}
\begin{align*}
 &\Big\|\sum_j S_jQ_{j+k}\Big\|_{L^2\to L^2} \lc 
\sup_j\|\vsig_j\|_{L^1}   \log^{3/2}(1+n\Ga_\eps) := D_{0}\,,
\\
&2^{\eps_1 k} \|\sum_j S_jQ_{j+k}\|_{L^2\to L^2} \lc 
\sup_j\|\vsig_j\|_{L^1}  n \Ga_\eps^2
:= D_{\eps_1}\,.
\end{align*}
Now   we easily obtain from Corollary \ref{CorWeakTypeWithQ}  that 
$$ \Big\|\sum_{k>0}\sum_j S_jQ_{j+k}\Big\|_{L^p\to L^p} \lc C_p 
\sup_j\|\vsig_j\|_{L^1}   \log^{5/2}(1+n\Ga_\eps) $$
and 
\eqref{Lal-n+1-Lp} is proved.

Finally we turn to \eqref{Lal-n+2-Lp}, for $l=1$. The case $p=2$ 
 follows immediately from 
\eqref{Lal-n+1-Lp}, by duality replacing $\vsig_j$ with $\tilde \vsig_j$.
For $p< 2$ we observe that 
\[
\sum_{j\in \Z} \La[\vsigjj]((I-P_j)b_1, b_2^j ,\ldots,  b_{n}^j, P_j b_{n+1}, (I-P_j) b_{n+2}) =
\int b_{n+2}(x) \, S_jP_j b_{n+1}(x) 
dx
\] with $S_j =(I-P_j) W_j[\vsig_j, b_1]$.
The Schwartz kernel of $S_j$ is equal to $\Dil_{2^j} s_j$ where 
$$s_j(x,y)= \sigma_j(x,y)-\int \phi(x-x')\sigma_j(x',y) dy $$
with $\sigma_j$ as in \eqref{schwartzofsigmaj}. Then 
 $s_j$  satisfies $\Sha^1[s_j]\lc \|\vsig_j\|_{L^1}$ and
 $\Sha_\eps^1[s_j]\lc \|\vsig_j\|_{\cB_\eps}$ 
and  \eqref{Lal-n+2-Lp} for $p<2$ follows immediately from the case $p=2$ and Corollary \ref{ThmWeakTypeWithP}.

\section{Proof of Theorem \ref{main-parts}: Part \ref{ItemBoundTwoPts}}
\label{Sectionpart2pts}

Let $n\ge 2$ and  $1\leq l_1<l_2\leq n$.  In this section, we consider the multilinear functional
\begin{equation}\label{partIIIop}
\begin{split}
&\La_{l_1,l_2}^1(b_1,\ldots, b_{n+2}):=
\\&\sum_{j\in \Z} \La[\vsigjj](b_1,\ldots, b_{l_1-1}, (I-P_j)b_{l_1}, P_j b_{l_1+1},\ldots, P_j b_{l_2-1}, (I-P_j) b_{l_2},P_j b_{l_2+1},\ldots, P_j b_{n+2}),
\end{split}
\end{equation}
where, for some fixed $\eps>0$, $\vectsig=\{\vsig_j:j\in \Z\}\subset \sBtp{\eps}{\R^n\times \R^d}$ is a bounded set.  The goal of this section is to prove, for $p\in (1,2]$, $b_1,\ldots, b_n\in L^\infty(\bbR^d)$, $b_{n+1}\in L^{p}(\bbR^d)$, $b_{n+2}\in L^{p'}(\bbR^d)$,
 the inequality 
\begin{equation}\label{partIIIineq}
\begin{split}
&\big| \La_{l_1,l_2}^1(b_1,\ldots, b_{n+2})\big| 
\leq C_{d,p,\eps} \,\sup_j \LpN{1}{\vsig_j}\, \log^3 (1+n \Ga_\eps)
\big(\prod_{l=1}^n \|b_l\|_\infty\big)\|b_{n+1}\|_p \|b_{n+2}\|_{p'},
\end{split} 
\end{equation}
together with convergence of the sum \eqref{partIIIop} in the operator topology of multilinear functionals. Moreover the operator sum  $T^1_{l_1,l_2}$ associated to 
$\La_{l_1,l_2}^1$ converges in the strong operator topology.

It will be convenient to
 prove a slightly more general theorem.  Let $\{ b_l^j: 3\leq l\leq n, j\in \Z \}\subset L^\infty(\R^d)$ be
a bounded set, with $\sup_{j\in\Z} \LpN{\infty}{b_l^j}=1$, for $3\leq l\leq n$.  For $b_1,b_2\in L^\infty(\R^d)$ define an operator
$S_j[b_1,b_2]$ by
\begin{equation*}
\int g(x) \q(S_j[b_1,b_2]f\w)(x)\:dx :=\La[\vsigjj]((I-P_j) b_1, (I-P_j) b_2, b_3^j,\ldots, b_n^j, f, g).
\end{equation*}
\begin{thm}\label{PropBound2PtsRealThm}
With the above assumptions, for $1<p\leq 2$,
the  sums  $\sum_{j=-\infty}^\infty P_j S_j[b_1,b_2] P_j$ converge to
$S[b_1,b_2]$, in the strong operator topology 
as operators $L^p\to L^p$, and  $S[b_1,b_2]$ satisfies 
the estimate
\begin{equation}\label{SNb1n2op}
\LpOpN{p}{S[b_1,b_2]} \leq C_{d,p,\eps} \sup_j \LpN{1}{\vsig_j}\log^3 
(1+n\Ga_\eps) \|b_1\|_\infty \|b_2\|_\infty.
\end{equation}
\end{thm}

\begin{proof}[Proof of \eqref{partIIIineq}  given Theorem \ref{PropBound2PtsRealThm}]
Using Theorem   \ref{ThmOpResAdjoints} we see that 
Theorem \ref{PropBound2PtsRealThm} also implies the inequality 
\begin{multline*}
\Big|\sum_j\La[\vsigjj](b_1^j,\ldots, b_{l_1-1}^j, (I-P_j)b_{l_1}, b_{l_1+1}^j,\ldots,  b_{l_2-1}^j, (I-P_j) b_{l_2}^j,b_{l_2+1}^j,\ldots,b_{n}^j, b_{n+1}, b_{n+2})\Big|
\\ \lc 
 \sup_j \LpN{1}{\vsig_j}\log^3 
(1+n\Ga_\eps) \|b_{l_1} \|_\infty \|b_{l_2}\|_\infty \big(\prod_{\substack{1\le i\le n\\i\neq l_1, l_2}} \|b_i^j\|_\infty\big) \|b_{n+1}\|_p\|b_{n+2}\|_{p'}.
\end{multline*}
Since  $\|P_j b_l\|_q\leq \|b_l\|_q$   we may replace 
$b_i$ by $P_j b_i$ for   $l_1+1\le i\le l_2-1$, $i\ge l_2+1$, and 
if we use also 
$P_j={}^t\!P_j$ then 
\eqref{partIIIineq} follows.
\end{proof}

The rest of this section is devoted to the proof of Theorem \ref{PropBound2PtsRealThm}.  Thus, we consider
sequences $b_l^j\in L^\infty(\R^d)$ fixed ($3\leq l\leq n$) with $\sup_{j} \LpN{\infty}{b_l^j}=1$.
The $L^2$ estimates in \S\ref{Sectionpartrough} will be crucial. We restate them as
\begin{prop}\label{PropBound2PtsBasicL2}
There is $C\lc 1$ such that for $\eps'\le \eps/C$, and for all 
collections  $$\{ b_{n+1}^j : j\in \Z\}, \{b_{n+2}^j:j\in \Z\}\subset L^\infty(\R^d),
\text{ with } \sup_{j} \|b_{n+1} ^j\|_\infty=1, \quad
 \sup_{j} \|b_{n+2} ^j\|_\infty=1, \quad
$$ 
we have for $f,g\in L^2(\R^d)$ and $k_1,k_2\in \N$,
\begin{multline*}
\Big|\sum_{j\in \Z} \La[\vsigjj](Q_{j+k_1} f, Q_{j+k_2} g, b_3^j,\ldots, b_{n+2}^j)\Big|\\  \lesssim \|f\|_2\|g\|_2\min\big\{2^{-\eps' k_1-\eps' k_2} n \sup_j \sBN{\eps}{\vsig_j},\, \sup_j \LpN{1}{\vsig_j}\big\}.
\end{multline*}
Let  $\cT_{k_1,k_2}$ be defined  by
\Be\label{Tk1k2j} \La[\vsigjj](Q_{j+k_1} f, Q_{j+k_2} g, b_3^j,\ldots, b_{n+2}^j)=
\int g(x)\,\cT_{k_1,k_2,j} f(x)\,dx.
\Ee 
Then 
$\sum_j \cT_{k_1,k_2,j}$ and  $\sum_j {}^t\cT_{k_1,k_2,j}$
 converge in the strong operator topology as operators $L^2\to L^2$, with equiconvergence with respect to $b_3^j,\ldots, b_{n+2}^j$.
\end{prop}
\begin{proof}
This follows from Theorem \ref{PropL2BoundFor2Qs}.
\end{proof}



\begin{prop}\label{PropBound2PtSumNeg}
Let $\{ b_1^j, b_2^j : j\leq -1\}\subset L^\infty(\R^d)$ be a bounded set with $\sup_{j\leq -1} \LpN{\infty}{b_l^j}=1$, $l=1,2$, and let 
$b_{n+1},b_{n+2}$ be $L^\infty$  functions supported in $\{y:|y|\le 1\}$.
\begin{equation*}
\sum_{j=-\infty}^{-1}\big| \La[\vsigjj](b_1^j,\ldots, b_{n}^j, P_j b_{n+1}, P_j b_{n+2})\big| \lesssim \LpN{\infty}{b_{n+1}} \LpN{\infty}{b_{n+2}} \sup_{j} \LpN{1}{\vsig_j}.
\end{equation*}
\end{prop}
\begin{proof}
We may assume $\LpN{\infty}{b_{n+1}}=\LpN{\infty}{b_{n+2}}=1$.  
Then by Lemma \ref{LemmaBasicL2BasicLpEstimate}
\begin{align*}
&\big|\La[\vsigjj](b_1^j,\ldots, b_{n}^j, P_j b_{n+1}, P_j b_{n+2})\big| 
\lesssim   \sup_j \LpN{1}{\vsig_j} \|P_j b_{n+1}\|_2  \|P_j b_{n+2}\|_2
\\
&\lesssim   \sup_j \LpN{1}{\vsig_j}2^{jd}  \| b_{n+1}\|_1  \|P_j b_{n+2}\|_1
\lesssim   \sup_j \LpN{1}{\vsig_j}2^{jd}  
\end{align*}
where we have used $\|P_j\|_{L^1\to L^2}\lc 2^{jd/2}$ and then the support assumption on $b_{n+1}$, $b_{n+2}$. Now sum over $j\le -1$ and the proof is complete.
\end{proof}

\begin{lemma} Let $0<\eps\le 1$.  For all $R\ge 5$, all 
$j\geq 0$, $b_{n+1}, b_{n+2}\in L^\infty$ supported in $\{x:|x|\le 4\}$, $b_1,b_2\in L^\infty(\R^d)$ with $\supp{b_1}\subseteq \{v: |v|\geq R\}$, we have
\begin{equation*}
\q|\La[\vsigjj](b_1,b_2,b_3^j,\ldots, b_n^j, b_{n+1},b_{n+2})\w|\lesssim \min\big\{ 
(2^j R)^{-\eps/2} 
 \sBN{\eps}{\vsig_j},
 \, \LpN{1}{\vsig_j} \big\} 
\prod_{l\in \{1,2,n+1,n+2\}} \|b_l\|_\infty.
\end{equation*}
\end{lemma}
\begin{proof}
We may assume $\LpN{\infty}{b_l}=1$, $l=1,2,n+1,n+2$.  The bound
\begin{equation}\label{EqnBound2PtsSimpleBoundGamma}
\q|\La[\vsigjj](b_1,b_2,b_3^j,\ldots, b_n^j, b_{n+1},b_{n+2})\w|\lesssim \LpN{1}{\vsig_j}
\end{equation}
follows immediately from Lemma \ref{LemmaBasicL2BasicLpEstimate} and the assumptions on the supports of $b_{n+1}$ and $b_{n+2}$.

In order to establish the bound  $(2^j R)^{-\eps/2} 
 \sBN{\eps}{\vsig_j}$ we estimate, using the assumption on $\text{supp}(b_1)$,
\begin{equation*}
\begin{split}
&\big|\La[\vsigjj](b_1,b_2,b_3^j,\ldots, b_n^j, b_{n+1},b_{n+2})\big| 
\\&= \Big|\iiint \dil{\vsig_j}{2^j}(\alpha,v) b_1(x-\alpha_1 v) b_2(x-\alpha_2 v) 
\big(\prod_{i=3}^n
b_i^j(x-\alpha_3 v)\big) b_{n+1}(x-v) b_{n+2}(x)\: dx\: d\alpha\: dv\Big|
\\&\leq \int_{|x|\leq 4} \int_{|v|\leq 8}\int_{|\alpha_1|\ge \frac{R-|x|}{|v|}} \q|\dil{\vsig_j}{2^j}(\alpha,v)\w| |b_1(x-\alpha_1 v)| \: d\alpha\: dv\:dx
\\
\\&\lc \int_{|w|\leq 2^{j+3}}\int_{|\alpha_1|\ge \frac{R-|4|}{2^{-j} |w|}} \q|\vsig_j(\alpha,w)\w| \,
 d\alpha\: dw\,;
 \end{split}\end{equation*}
 here we have used $R\ge 5$.
Let $m\le j+3$. Clearly 
\Be\label{smallmB1}
\int_{2^{m-1}\le |w|\leq 2^{m}}\int_{|\alpha_1|
\ge \frac{R-|4|}{2^{-j} |w|}} \q|\vsig_j(\alpha,w)\w| 
\,d\alpha\,dw\,
\lc (2^{j-m} R)^{-\eps} \|\vsig_j\|_{\cB_{\eps,1}}
\lc 2^{m\eps}
(2^{j} R)^{-\eps} \|\vsig_j\|_{\cB_{\eps}}.
\Ee
Also 
\Be\label{largemB4}
\int_{2^{m-1}\le |w|\leq 2^{m}}\int_{|\alpha_1|\ge \frac{R-|4|}{2^{-j} |w|}} \q|\vsig_j(\alpha,w)\w| 
\,d\alpha\,dw\,\lc 2^{-m\eps} \|\vsig_j\|_{\cB_{\eps,4}} \lc 2^{-m\eps} \|\vsig_j\|_{\cB_{\eps}}.
\Ee
We use \eqref{smallmB1} for $2^m< (2^j R)^{1/2}$ and 
\eqref{largemB4} for  $2^m\ge (2^j R)^{1/2}$, and sum. The assertion follows.
\end{proof}

\begin{lemma}\label{LemmaBound2ptsForBigsupp}
For $l=1,2,n+1,n+2$, let $\{ b_l^{j,k_1,k_2} :j,k_1,k_2\in\N\} \subset L^\infty(\R^d)$ be bounded sets with $\sup_{j,k_1,k_2} \LpN{\infty}{b_l^{j,k_1,k_2}} =1$.
Let $\beta>0$, $\delta>0$ and assume
\begin{equation}\label{EqnBound2Ptsa1jksupp}
\supp{b_1^{j,k_1,k_2}}\subseteq \big\{ v:|v|\geq \max\{5, \beta2^{k_1 \delta +k_2\delta}\}\big\}, \quad \forall j,k_1,k_2\in \N
\end{equation}
and for $l=n+1,n+2$,
$$\supp{b_l^{j,k_1,k_2}} \subseteq \{ v:|v|\leq 4\}, \quad \forall j,k_1,k_2\in \N.$$
Then
\begin{equation*}
\sum_{j,k_1,k_2\in \N} |\La[\vsigjj](b_1^{j,k_1,k_2}, b_2^{j,k_1,k_2}, b_3^j,\ldots, b_n^j, b_{n+1}^{j,k_1,k_2}, b_{n+2}^{j,k_1,k_2})|\lesssim \sup_j \LpN{1}{\vsig_j} \log^3(1+\beta^{-1}\Ga_\eps).
\end{equation*}
Here the implicit constant depends on $\delta$, but not on $\beta$.
The same result holds if instead of \eqref{EqnBound2Ptsa1jksupp} we have
\begin{equation}\label{EqnBound2Ptsa2jksupp}
\supp{b_2^{j,k_1,k_2}}\subseteq\big \{ |v|\geq\max\{ 5, \beta 2^{k_1 \delta +k_2\delta}\}\big\}, \quad \forall j,k_1,k_2\in \N.
\end{equation}
\end{lemma}
\begin{proof}
Because our definitions are symmetric in $b_1$ and $b_2$, the result with \eqref{EqnBound2Ptsa2jksupp} in place of \eqref{EqnBound2Ptsa1jksupp} follows
from the result with \eqref{EqnBound2Ptsa1jksupp}.  Thus, we may focus only on the proof with the assumption \eqref{EqnBound2Ptsa1jksupp}.
Applying the previous lemma, we have
\begin{equation*}
\begin{split}
&\sum_{j,k_1,k_2\in \N} \big|\La[\vsigjj](b_1^{j,k_1,k_2}, b_2^{j,k_1,k_2}, b_3^j,\ldots, b_n^j, b_{n+1}^{j,k_1,k_2}, b_{n+2}^{j,k_1,k_2})\big|
\\&\lesssim
\sum_{j,k_1,k_2\in \N} \min\big\{ 2^{-j\eps/2} \big(\max\{5 ,\beta 2^{k_1\delta+k_2\delta}\} \big)^{-\eps/2} \sup_j \sBN{\eps}{\vsig_j} , \, \sup_j \LpN{1}{\vsig_j}\big\}
\\&\lesssim \sup_j \LpN{1}{\vsig_j} \log^3 (1+\beta^{-1} \Ga_\eps),
\end{split}
\end{equation*}
completing the proof.
\end{proof}

\begin{prop}\label{PropBound2PtSumPos}
Let $b_1,b_2, b_{n+1}, b_{n+2} \in L^\infty(\R^d)$.
Let 
$\fS_j $ be defined by
 \Be \La[\vsigjj]((I-P_j) b_1, (I-P_j) b_2,b_3^j,\ldots, b_n^j, P_j b_{n+1}, P_j b_{n+2})
 =\int b_{n+2}(x)\, \fS_j b_{n+1}(x) \,dx.
 \Ee
 Consider $\fS_j$ as a bounded operator mapping $L^\infty$  functions supported in $B^d(0,1)$ to $L^1(B^d(0,1))$.
Then the sum $\sum\fS_j$ converges in the strong operator topology as 
bounded operators $L^\infty(B^d(0,1))$ to $L^1(B^d(0,1))$ 
and we have for 
$ \supp{b_{n+1}}, \supp{b_{n+2}}\subseteq\{y:|y|\le 1\}$,
\begin{multline*}
\Big|\sum_{j\in \bbZ} \La[\vsigjj]((I-P_j) b_1, (I-P_j) b_2,b_3^j,\ldots, b_n^j, P_j b_{n+1}, P_j b_{n+2})\Big|
\\ \lesssim
\sup_j \LpN{1}{\vsig_j} \log^3\q(1+n\Ga_\eps\w)
\prod_{l\in \{1,2,n+1,n+2\}} \|b_l\|_\infty.
\end{multline*}
\end{prop}
\begin{proof}
We may assume $\LpN{\infty}{b_l}=1$, $l=1,2,n+1,n+2$. 

By Proposition  \ref{PropBound2PtSumNeg} the required estimate holds for the sum over
 negative $j$ and thus we only bound
 \begin{multline}\label{positivejLinfty}
\Big|\sum_{j\ge 0} \La[\vsigjj]((I-P_j) b_1, (I-P_j) b_2,b_3^j,\ldots, b_n^j, P_j b_{n+1}, P_j b_{n+2})\Big|
\\ \lesssim
\sup_j \LpN{1}{\vsig_j} \log^3\q(1+n\Ga_\eps\w).
\end{multline}

 Let $0<\beta\leq 1$, $0<\delta<1$ be constants, to be chosen later (see \eqref{betadeltasec}).
Implicit constants below are allowed to depend on $\delta$, but do not depend on $\beta$.
For $l=1,2$ and $k_1,k_2>0$ define
\begin{equation*}
b_{l,\infty}^{k_1,k_2}(v) :=
\begin{cases} b_l(v) &\text{ if }  |v| > \max\{10, \beta\cdot 2^{k_1\delta+k_2\delta+1}\} \\
0 &\text{ otherwise }
\end{cases} 
\end{equation*} 
and
\[b_{l,0}^{k_1,k_2}(v) := b_l(v)- b_{l,\infty}^{k_1,k_2}(v).\]
We have, by \eqref{EqnAuxOpId} and Remark \ref{RmkAuxOpSublety},
\begin{equation*}
\begin{split}
&\Big|\sum_{j\geq 0} \La[\vsigjj]{(I-P_j) b_1, (I-P_j) b_2,b_3^j,\ldots, b_n^j, P_j b_{n+1}, P_j b_{n+2}}\Big|
\\&=\Big|\sum_{k_1,k_2>0}\sum_{j\geq 0} \La[\vsigjj](Q_{j+k_1} b_1, Q_{j+k_2} b_2,b_3^j,\ldots, b_n^j, P_j b_{n+1}, P_j b_{n+2})\Big|\, \le (I)+(II)+(III)
\end{split}
\end{equation*} 
where
\begin{align*}
\\&(I):=
 \sum_{k_1,k_2>0}
 \sum_{j\geq 0}\big|
  \La[\vsigjj](Q_{j+k_1} b_{1,\infty}^{k_1,k_2}, Q_{j+k_2} b_2,b_3^j,\ldots, b_n^j, P_j b_{n+1}, P_j b_{n+2})\big|\,,
\\&(II):= \sum_{k_1,k_2>0}
 \sum_{j\geq 0}\big| \La[\vsigjj](Q_{j+k_1} b_{1,0}^{k_1,k_2}, Q_{j+k_2} b_{2,\infty}^{k_1,k_2},b_3^j,\ldots, b_n^j, P_j b_{n+1}, P_j b_{n+2})\big|\,,
\\&(III):= 
\sum_{k_1,k_2>0}
\Big|\sum_{j\geq 0} \La[\vsigjj](Q_{j+k_1} b_{1,0}^{k_1,k_2}, Q_{j+k_2} b_{2,0}^{k_1,k_2},b_3^j,\ldots, b_n^j, P_j b_{n+1}, P_j b_{n+2})\Big|\,.
\end{align*}
Because $j,k_1,k_2\geq 0$, and by the supports of the functions in question, we have
\[
\supp{Q_{j+k_1} b_{1,\infty}^{k_1,k_2}}, \,\supp{Q_{j+k_2} b_{2,\infty}^{k_1,k_2}}\subseteq\big\{v: |v| > \max\{5, \beta\cdot 2^{k_1\delta +k_2\delta}\}\big\}, 
\] and 
\[ \supp{P_j b_{n+1}}, \supp{P_j b_{n+2}}\subseteq \{ v: |v| \le 4 \}\,.
\]
Lemma \ref{LemmaBound2ptsForBigsupp} applies to show
\begin{equation}\label{EqnBound2PtsIandIIBound}
|(I)|+ |(II)| \lesssim \sup_j \LpN{1}{\vsig_j} \log^3\q(1+\beta^{-1} \Ga_\eps).
\end{equation}

We now apply  the $L^2$ result in Proposition \ref{PropBound2PtsBasicL2}.
Let $\cT_{k_1,k_2,j}$ be as in \eqref{Tk1k2j}. 
Then 
$\sum_{j\ge 0} \cT_{k_1,k_2,j}$ converges in the strong operator topology as operators $L^2\to L^2$, with
equiconvergence with respect to bounded choices of $b_{n+1}, b_{n+2}\in L^\infty (B^d(0,1))$, moreover the operator norms involve some exponential deacy in $k_1$, $k_2$. 
If we apply this to
$b_{1,0}^{k_1,k_2}, b_{2,0}^{k_1,k_2}$, we may replace the $L^2$ norms with $L^\infty$-norms.
Hence if we define operators 
$\fS_{k_1,k_2,j}$ by 
 $$
\int b_{n+2}(x)\, \fS_{k_1,k_2,j} b_{n+1}(x) \,dx
=
\La[\vsigjj](Q_{j+k_1} b_{1}, Q_{j+k_2} b_{2} ,b_3^j,\ldots, b_n^j, P_j b_{n+1}, P_j b_{n+2})
$$
we  see that 
$\int \sum_j b_{n+2}(x) \fS_{k_1,k_2,j} b_{n+1}(x) dx$ converges with equiconvergence in the choice of 
$b_{n+2}$ with $\|b_{n+2}\|_\infty \le 1$ and $\text{supp}(b_{n+2})\subset B^d(0,1)$. Thus we get convergence of $\sum_{j=0}^\infty  \fS_{k_1,k_2,j}$ in the strong operator topology as operators 
$L^\infty(B^d(0,1))\to L^1(B^d(0,1))$.
For the quantitative estimates
we apply  the $L^2$ result in Proposition \ref{PropBound2PtsBasicL2} 
 and use the supports of $b_{1,0}^{k_1,k_2}, b_{2,0}^{k_1,k_2}$ to get 
for $\eps'<c\eps^2$
\begin{equation}\label{EqnBound2PtsPropBoundsecond}
\begin{split}
(III) &
\lesssim \sum_{k_1,k_2>0} 
\max\big\{2^{-\eps' k_1-\eps' k_2} n \sup_j \sBN{\eps}{\vsig_j},\, \sup_j \LpN{1}{\vsig_j}\big\}\|b_{1,0}^{k_1,k_2}\|_2 \|b_{2,0}^{k_1,k_2}\|_2
\\&\lesssim  \sum_{k_1,k_2>0} 
\max\big\{2^{-\eps' k_1-\eps' k_2} n \sup_j \sBN{\eps}{\vsig_j},\, \sup_j \LpN{1}{\vsig_j}\big\}
\q(\max\{5 ,\, \beta\cdot 2^{k_1\delta+k_2\delta}\}\w)^{2d}.
\end{split}
\end{equation}
Set
\begin{equation}\label{betadeltasec}
\delta=\frac{\eps'}{4d}, \quad \beta = (n\Ga_\eps)^{-\frac{1}{2d}}.
\end{equation}
Note that
\begin{equation*}
\q(\beta\cdot 2^{k_1\delta +k_2\delta}\w)^{2d} \q(2^{-\eps' k_1-\eps' k_2} n \sup_j \sBN{\eps}{\vsig_j}\w) = 2^{-\eps' k_1/2 - \eps' k_2/2} \sup_j \LpN{1}{\vsig_j}.
\end{equation*}
Using this in  \eqref{EqnBound2PtsPropBoundsecond}, we obtain
\begin{equation*}
\begin{split}
(III)&\lc \sup \|\vsig_j\|_{L^1}\sum_{k_1,k_2>0} \max\{ 2^{-\eps' k_1-\eps' k_2} n\Ga_\eps,\, 1\} (1+\beta  \cdot 2^{k_1\delta+k_2\delta})^{2d}
\\&\lesssim  \q\sup_j \LpN{1}{\vsig_j}\w\log^2(1+n\Ga_\eps).
\end{split}
\end{equation*}
Plugging the choice of $\beta$ into \eqref{EqnBound2PtsIandIIBound} completes the proof
of \eqref{positivejLinfty}.

Finally, we reexamine the proof to get the asserted convergence in the strong operator topology.
This is immediate for the sums corresponding to the terms $(I)$, $(II)$ in view of the decay estimates in the proof of Lemma  \ref{LemmaBound2ptsForBigsupp}.
For (III) we easily get the assertion from the above statements about convergence of 
$\sum_{j\ge 0}\fS_{k_1,k_2,j}$ and the exponential decay estimates in $k_1, k _2$.
\end{proof}

\subsubsection*{Proof of Theorem \ref{PropBound2PtsRealThm}, conclusion.}
We shall apply Theorem \ref{ThmJourne3}. 
We need to verify 
that for every ball $B^d(x_0,r)$, $b_{n+1}\in L^\infty (B^d(x_0,r))$, $\|b_{n+1}\|_\infty=1$, 
\[
 \int_{B_d(x_0,r)}
\Big| \sum_{|j|>N} P_j S_j[b_1, b_2] P_j b_{n+1}(x)\Big|dx \to 0
\]
as $N\to \infty$ and
\begin{multline}\label{PjSjPjNest}
 \sup_N r^{-d} \int_{B_d(x_0,r)}
\Big| \sum_{|j|\le N} P_j S_j[b_1, b_2] P_j b_{n+1}(x)\Big|\, dx \\ \lc
\sup_j \LpN{1}{\vsig_j}\log^3 
(1+n\Ga_\eps) \|b_1\|_\infty \|b_2\|_\infty \,.
\end{multline}
For $x_0=0$ and $r=1$ these statements follow from 
Proposition \ref{PropBound2PtSumPos}.
We argue by rescaling to obtain  the same statement for other balls.
Let $\ell$ be such that $2^{\ell-1}\le r\le 2^\ell$.
Let $\widetilde b_i(x)= b_i(x_0+2^\ell x)$, $i=1,2, n+1, n+2$ and 
$\widetilde b_i^j(x)= b_i^{j-\ell}(x_0+2^\ell x)$, $3\le i\le n$.
Then by  changes of variables
\begin{multline*}
\int b_{n+2}(x) \, S_j[b_1,b_2] b_{n+1}(x) \: dx
\\=
2^{\ell d} \La[ \vsig^{(2^{j+\ell})}]\big( (I-P_{j+\ell}) \widetilde b_1,
(I-P_{j+\ell}) \widetilde b_2, \widetilde b_3^{j+\ell}, \dots , \widetilde b_n^{j+\ell},
\widetilde b_{n+1}, \widetilde b_{n+2}\big).
\end{multline*}
We use
the fact that the functions 
$\widetilde b_{n+1}$, $\widetilde b_{n+2}$ are supported in the unit ball centered at the origin. Then  the result follows immediately from the statement for $x_0=0$, $r=1$.

In order to verify the  $\Op_\eps$-assumptions 
in Theorem \ref{ThmJourne3} we use Lemma
 \ref{PSP-lemma}  with $C_0\lesssim \sup_j \LpN{1}{\vsig_j}$ and $C_\eps\lesssim \sup_j \sBN{\eps}{\vsig_j}$. Now Theorem \ref{ThmJourne3} yields  
 \begin{equation*}
\LpOpN{2}{S[b_1,b_2]} \lesssim \LpN{\infty}{b_1} \LpN{\infty}{b_2} \q(\sup_j \LpN{1}{\vsig_j}\w) 
\log^3(1+n \Ga_\eps).
\end{equation*} 
 Finally we combine this inequality with 
 Corollary  \ref{ThmWeakTypeWithP},   with  the choices 
$A\lesssim \sup_j \LpN{1}{\vsig_j}$ and $B\lesssim \sup_j \sBN{\eps}{\vsig_j}$.
This yields the asserted $L^p$ bound. \qed

\section{Proof of Theorem \ref{main-parts}: Part \ref{ItemBound1IminusPt}}
\label{Sectionpart1iminus}

Let $1\leq l\leq n+2$.  In this section, we consider the multilinear form
\begin{equation*}
\La_l^2(b_1,\ldots, b_{n+2}):=
\sum_{j\in \cJ} \La[\vsigjj](P_j b_1,\ldots, P_j b_{l-1}, (I-P_j) b_l, P_j b_{l+1},\ldots, P_j b_{n+2}),
\end{equation*}
where $\cJ\subset\bbZ$ is a finite set, and, given  some fixed $\epsilon>0$, $\vec\vsig=\{\vsig_j : j\in \Z\}\subset \sBtp{\epsilon}{\R^n\times \R^d}$ is a bounded set with $\int \vsig_j(\alpha,v) \: dv=0$, $\forall \alpha,j$.  Our task is to  show that for 
$p\in (1,2]$,
\begin{equation}\label{Lambda-l-2est}
\q|\La^2_l (b_1,\ldots, b_{n+2})\w| \leq C_{d,p,\epsilon} n \sup_j \LpN{1}{\vsig_j}
\log^3\q(1+n\Ga_\eps)
\big[\prod_{i=1}^n \|b_i\|_\infty\big]  \|b_{n+1}\|_p \|b_{n+2}\|_{p'}
\end{equation}
where the implicit constant is independent of $\cJ$. Moreover
we wish to show that the sum defining the operator $T^2_l$ associated to $\La_l^2$ via \eqref{operatorassoc}  converges in the strong  operator topology as operators bounded 
on $L^p$.  The heart of the proof 
lies in the next theorem   which we shall prove first. Let $\Gamma_\eps\equiv \Gamma_\eps(\vec\vsig)$ be as in \eqref{Gammaeps}.

\begin{thm}\label{PropPart1IMinus}
Let $b_1,\ldots, b_n\in L^\infty(\R^d)$, $b_{n+1},b_{n+2}\in L^2(\R^d)$.   Then,
$$\lim_{N\to\infty}\sum_{j=-N}^N \La[\vsigjj](P_j b_1,\ldots, P_j b_{l-1}, (I-P_j) b_l, P_j b_{l+1},\ldots, P_j b_{n+2}) =\La_{n+2}^2(b_1,\ldots, b_{n+2})$$
and $\La_{n+2}^2$ satisfies
\begin{equation*}
\q|\La_{n+2}^2(b_1,\ldots, b_{n+2})
\w|\leq C_{d,\epsilon} n \q(\sup_j \LpN{1}{\vsig_j}\w) \log^3(1+n\Gamma_\eps) \big[\prod_{m=1}^n \LpN{\infty}{b_m}\w\big] \LpN{2}{b_{n+1}} \LpN{2}{b_{n+2}}.
\end{equation*}
Moreover the sums defining the operator $T^2_{n+2} $ associated with 
 $\La^2_{n+2} $ converge in the strong operator topology as operators $L^2\to L^2$.
 \end{thm}
The full proof of \eqref{Lambda-l-2est} will be given in \S\ref{sect12concl} below.

\subsection{Outline of the proof of Theorem \ref{PropPart1IMinus}}\label{partfouroutline}
We give an outline of the steps and refer to  \S\ref{Opepssect} for some technical details.

We first describe the basic decomposition of 
$\La_{n+2}^2(b_1,\ldots, b_{n+2})$ which is derived from a decomposition 
of $\La[\vsigjj](P_jb_1,\dots, P_j b_{n+1}, (I-P_j)b_{n+2}$, for fixed $j$. Write 
\begin{equation*} \begin{split}
& \La[\vsigjj](P_j b_1,\ldots, P_j b_{n+1}, (I-P_j) b_{n+2})
\\
&=\lim_{M\to\infty}\Big(
 \La[\vsigjj](P_{j+M}P_j b_1,\ldots, P_{j+M}P_j b_{n+1}, (I-P_j) b_{n+2})
\\&\qquad\qquad\qquad \qquad-
 \La[\vsigjj](P_{j-M}P_j b_1,\ldots, P_{j-M}P_j b_{n+1}, (I-P_j) b_{n+2})\Big)
\\
&=\lim_{M\to\infty}\sum_{m=-M+1}^M\Big(
 \La[\vsigjj](P_{j+m}P_j b_1,\ldots, P_{j+m}P_j b_{n+1}, (I-P_j) b_{n+2})
\\&\qquad\qquad\qquad \qquad
-
 \La[\vsigjj](P_{j+m-1}P_j b_1,\ldots, P_{j+m-1}P_j b_{n+1}, (I-P_j) b_{n+2})\Big)
\end{split}
\end{equation*}
and use the multilinearity 
to obtain the decomposition
\Be\label{Q-expansion}
\begin{aligned}
 \La[\vsigjj](P_j b_1,\ldots, P_j b_{n+1}, &(I-P_j) b_{n+2})\,
\\
=\sum_{l=1}^{n+1} \sum_{m=-\infty}^{\infty} 
\La[\vsigjj]\big(&P_{j+m-1}P_jb_1,\dots, 
P_{j+m-1}P_jb_{l-1} , Q_{j+m} P_j b_l,
\\
&P_{j+m}P_j b_{l+1},\dots, P_{j+m}P_jb_{n+1}, (I-P_j)b_{n+2}\big).
\end{aligned}
\Ee
The terms for $l=1,\dots,n$ are handled in a similar fashion,
in fact the estimates can be reduced to the case $l=1$ by using Theorem \ref{ThmOpResAdjoints}, permuting the first and the ${l}^{\text {th}}$ entry, and accordingly changing the family $\{\vsig_ j\}$.

Now let  
\Be \label{Xsub-def} X_{k}^i \in \{ P_{k}, P_{k-1}\}.\Ee
Then we need to show
\begin{multline} \label{Q-n+1pos-XP-estimate}
\Big|\sum_{j=-N}^N \sum_{m=-\infty}^\infty
\La[\vsigjj] (X_{j+m}^1P_j b_1, X_{j+m}^2P_j b_2, \dots, Q_{j+m} P_j b_{n+1}, (I-P_j)b_{n+2})\Big|
\\
\lc \sup_j\|\vsig_j\|_{L^1} \log^2(1+n\Ga_\eps) \big(\prod_{i=1}^n\|b_i\|_\infty\big) \|b_{n+1}\|_2 \|b_{n+2}\|_2
\end{multline}
and
\begin{multline} \label{Q-firstpos-XP-estimate}
\Big|\sum_{j=-N}^N \sum_{m=-\infty}^\infty
\La[\vsigjj] (Q_{j+m}P_j b_1, X_{j+m}^2P_j b_2, \dots, X_{j+m}^{n+1} P_j b_{n+1}, (I-P_j)b_{n+2})\Big|
\\
\lc \sup_j\|\vsig_j\|_{L^1} \log^2(1+n\Ga_\eps) \big(\prod_{i=1}^n\|b_i\|_\infty\big) \|b_{n+1}\|_2 \|b_{n+2}\|_2
\end{multline}
with implicit constants uniform in $N$; moreover we need to show the existence of the limits as $N\to \infty$,
for the corresponding operator sums in the strong operator topology.  By another application of Theorem 
 \ref{ThmOpResAdjoints}
(this time permuting the entries $(1, n+1)$), with the corresponding change of the family $\{\vsig_j\}$),
we see that \eqref{Q-n+1pos-XP-estimate} can be deduced from
\begin{multline} \label{Q-1-n+1-XP-estimate}
\Big|\sum_{j=-N}^N \sum_{m=-\infty}^\infty
\La[\vsigjj] (Q_{j+m}P_j b_1,  X_{j+m}^2P_j b_2, \dots, X^{n+1}_{j+m}P_j b_{n+1}, (I-P_j)b_{n+2})\Big|
\\
\lc \sup_j\|\vsig\|_{L^1} \log^2(1+n\Ga_\eps) \big(\prod_{i=2}^{n+1}\|b_i\|_\infty\big) \|b_{1}\|_2 \|b_{n+2}\|_2.
\end{multline}
It remains to prove \eqref{Q-firstpos-XP-estimate}, \eqref{Q-1-n+1-XP-estimate}.
We shall also decompose further using $(I-P_j)b_{n+2}= \sum_{m_2\in \bbN} Q_{j+m_2} b_{n+2}$. This  leads to the following  definition.
\begin{defn}
Let $m, m_1\in \bbZ$, $m_2>0$. 

For $b_{n+1}\in L^\infty(\bbR^d)$ the  operators 
$S^{m_1,m_2}_j[b_{n+1}]$ are defined by
\begin{multline}
\label{Sjmmdef}
\int g(x) S_{j}^{m_1,m_2} [b_{n+1}] f(x)\: dx \\
:=\La[\vsigjj](Q_{j+m_1}P_j g, X_{j+m_1}^2 P_j b_2, \cdots, 
X_{j+m_1}^n P_j b_n, 
X_{j+m_1}^{n+1} P_j b_{n+1}, Q_{j+m_2} f).
\end{multline}

For $b_1\in L^\infty(\bbR^d)$ 
 the operators 
 $T^{m_1,m_2}_j[b_1]$ are defined by 
\begin{multline}
 \label{Tjmmdef}\int g(x)T_{j}^{m_1,m_2} [b_{1}] f(x)\: dx \\
 :=\La[\vsigjj](Q_{j+m_1}P_j b_1, X_{j+m_1}^2 P_j b_2, \cdots,X_{j+m_1}^n P_j b_n,  X_{j+m_1}^{n+1} P_j g, Q_{j+m_2} f). 
\end{multline}
\end{defn}

We formulate an auxiliary result.
It gives bounds in the $\Op(\eps)$-classes
defined in \eqref{EqnJourneBoundCep}
for suitable normalizing dilates of the operators
$S^{m_1,m_2}_j[b_{n+1}]$, $T^{m_1,m_2}_j[b_{1}]$. We use the same notation for these operators and their Schwartz kernels.
\begin{prop}\label{OpestST}
Let
\Be \label{sigmammdef}
\sigma^{m_1,m_2}_j = \begin{cases}
\Dil_{2^{-j}} ( S^{m_1,m_2}_j[b_{n+1}]) &\text{ if } m_1\ge 0,
\\
\Dil_{2^{-j-m_1}} ( S^{m_1,m_2}_j[b_{n+1}]) &\text{ if } m_1< 0,
\end{cases}\Ee 
and 
\Be \label{taummdef}
\tau^{m_1,m_2}_j = \begin{cases}
\Dil_{2^{-j}} ( T^{m_1,m_2}_j[b_{1}]) &\text{ if } m_1\ge 0,
\\
\Dil_{2^{-j-m_1}} ( T^{m_1,m_2}_j[b_{1}]) &\text{ if } m_1< 0.
\end{cases} 
\Ee There  exists $\eps'>c(\eps)$ (independent of $n$)  such that,  for $m_2>0$,
\Be \label{sigmammest}\begin{aligned}
&\big\| \sigma^{m_1,m_2}_j
\big\|_{\Op_\eps} \lc 2^{-\eps' (|m_1|+m_2)} n^2 \|\vsig_j\|_{\cB_\eps} \|b_{n+1}\|_\infty,
\\
&\big\| \sigma^{m_1,m_2}_j
\big\|_{\Op_0} \lc \|\vsig_j\|_{L^1} \|b_{n+1}\|_\infty,
\end{aligned}
\Ee
and
\Be\label{taummest}
\begin{aligned}
&\big\| \tau^{m_1,m_2}_j
\big\|_{\Op_\eps} \lc 2^{-\eps' (|m_1|+m_2)} n^2 \|\vsig_j\|_{\cB_\eps} \|b_{1}\|_\infty,
\\
&\big\| \tau^{m_1,m_2}_j
\big\|_{\Op_0} \lc 
\|\vsig_j\|_{L^1} \|b_{1}\|_\infty\,.
\end{aligned}
\Ee
\end{prop} 
The proof will be given in \S\ref{Opepssect} below. Note that we have the trivial estimate
$\|\cdot\|_{\Op_0} \le  \|\cdot\|_{\Op_\eps}$, and therefore the $\Op_0$ bounds stated in Proposition 
\ref{OpestST} will only be used for  $2^{\eps(|m_1|+ m_2)}\lc n^2\Gamma_\eps$.



The estimates   \eqref{Q-firstpos-XP-estimate}, \eqref{Q-1-n+1-XP-estimate}
and the asserted existence  of the  limits follow easily from the following  Proposition.

\begin{prop}\label{mainpropforSmm}
Let $b_2, \dots, b_n\in L^\infty(\bbR^d)$, with $\|b_i\|_\infty\le 1$, $i=2,\dots, n$.
Let $\vec \vsig=\{\vsig_j\}$ be  a bounded family in $\cB_\eps$, $\cJ\subset \bbZ^d$ with $
\#\cJ<\infty$  and let $m_1\in \bbZ$, $m_2\in \bbN$.

Then there exist $\eps'>0$ so that the  following estimates hold, uniformly in $\cJ$.

(i) If $b_{n+1}\in L^\infty(\bbR^d)$, 
\Be \label{SNmmbound}\Big\|\sum_{j\in \cJ}  S_j^{m_1,m_2}[b_{n+1}]\Big\|_{L^2\to L^2}
\lc \min\big\{ 2^{-\eps' (|m_1|+ m_2)} n^2 \sup_j \|\vsig_j\|_{\cB_\eps},\, \sup_j\|\vsig_j\|_{L^1}\big\}
\|b_{n+1}\|_\infty.
\Ee

(ii) We have $\lim_{N\to \infty} \sum_{j=-N}^N   S_j^{m_1,m_2}[b_{n+1}]
= S^{m_1,m_2}[b_{n+1}]$ in the strong operator topology (as operators $L^2\to L^2$) and the bound \eqref{SNmmbound} remains true for the limit $S^{m_1,m_2}$.

(iii) We have $\sum_{m_1\in \bbZ}\sum_{m_2>0} S^{m_1,m_2}[b_{n+1}]\to S[b_{n+1}]$ with absolute convergence in $\cL(L^2, L^2)$. 
Also $\sum_{j=-N}^N S_j[b_{n+1}] $ converges to  an operator $S[b_{n+1}]$ in the strong operator topology as operators $L^2\to L^2$ and 
$$ \|S[b_{n+1}]\|_{L^2\to L^2} \lc \sup_j \|\vsig_j\|_{L^1} \log^2  (1+n \Gamma_\eps)\, \|b_{n+1}\|_\infty.$$

(iv) In (ii), (iii) the convergence in the strong operator topology is equicontinuous with respect to
$\{b_{n+1}: \|b_{n+1}\|_\infty \le 1\}$.
\end{prop}
\begin{proof} [Proof of Proposition \ref{mainpropforSmm}, given Proposition \ref{OpestST}]
For the proof of (i) we apply the almost orthogonality Lemma \ref{ThmAlmostOrthogonal}.
To this end we need to derive
the estimate
\begin{multline}\label{Ajk2}
\big\|\cQ_{k_1} S^{m_1,m_2}_{j+k_1}[b_{n+1}] \cQ_{j+k_1+k_2}\big\|_{L^2\to L^2}  \\ \lc A_{j,k_2}^{m_1,m_2}:=
\min\|b_{n+1}\|_\infty \big\{2^{-\eps_1(|m_1|+m_2)} n^2 \sup_j \|\vsig_j\|_{\cB_\eps}, \,2^{-|j+m_1| -|m_2+k_2|} \sup_j\|\vsig_j\|_{L^1}\big\}
\end{multline}
for some $\eps_1>0$.
To see this we note that the bound
\[
\big\| S^{m_1,m_2}_{j+k_1}[b_{n+1}] \big\|_{L^2\to L^2} \\ \lc
\min\|b_{n+1}\|_\infty \big\{2^{-\eps_1(|m_1|+m_2)} n^2 \sup_j \|\vsig_j\|_{\cB_\eps}\big\}
\]
(and hence the corresponding estimate for 
$\cQ_{k_1} S^{m_1,m_2}_{j+k_1}[b_{n+1}] \cQ_{j+k_1+k_2}$) follows from Proposition \ref{OpestST}. The bound 
\[\big\|\cQ_{k_1} S^{m_1,m_2}_{j+k_1}[b_{n+1}] \cQ_{j+k_1+k_2}\big\|_{L^2\to L^2} \lc 2^{-|j+m_1| -|m_2+k_2|} \sup_j\|\vsig_j\|_{L^1}\]
follows from the fact that $\LpOpN{2}{\Qt_{k} Q_l}, \LpOpN{2}{Q_l \Qt_k} \lesssim 2^{-|k-l|}$, the definition of $S_{j+k}^{m_1,m_2}$, and 
Lemma \ref{LemmaBasicL2BasicLpEstimate}.

We now observe that for   $A_{j,k_2}^{m_1,m_2}$  as in \eqref{Ajk2} we have
$$\sum_{j,k_2} A_{j,k_2}^{m_1,m_2}\lc \|b_{n+1}\|_\infty \min\big\{  \sup_j\|\vsig_j\|_{L^1},\,
2^{-\eps_1(|m_1|+m_2)} (|m_1|+m_2)^2 n^2 \sup_j \|\vsig_j\|_{\cB_\eps}\big\}.
$$
By an application of Lemma \ref{ThmAlmostOrthogonal} this yields 
\eqref{SNmmbound} and the convergence result in (ii), with equiconvergence with respect to 
$b_{n+1}$ in the unit ball of $L^\infty(\bbR^d)$. Summing in $m_1, m_2$ yields (iii).
\end{proof}

\begin{prop}\label{mainpropforTmm}
Let $b_2, \dots, b_n\in L^\infty(\bbR^d)$, with $\|b_i\|_\infty\le 1$, $i=2,\dots, n$.
Let $\vec \vsig=\{\vsig_j\}$ be  a bounded family in $\cB_\eps$, $\cJ\subset \bbZ^d$ with $
\#\cJ<\infty$  and let $m_1\in \bbZ$, $m_2\in \bbN$.

(i) If $b_{1}\in L^\infty(\bbR^d)$, 
\begin{multline} \label{TNmmbound}\Big\|\sum_{j\in \cJ}  T_j^{m_1,m_2}[b_{1}]\Big\|_{L^2\to L^2}
\\
\lc \min\big\{ 2^{-\eps' |m_1|-\eps' m_2} n^2 \sup_j \|\vsig_j\|_{\cB_\eps}, \,\sup_j\|\vsig_j\|_{L^1}
\log(1+n^2 \Ga_\eps)
\big\}
\|b_{1}\|_\infty.
\end{multline}

(ii) We have $\lim_{N\to \infty} \sum_{j=-N}^N   T_j^{m_1,m_2}[b_{1}]= T^{m_1,m_2}[b_{1}]$ in the strong  operator topology (as operators $L^2\to L^2$) and the bound \eqref{TNmmbound} remains true for the limit $T^{m_1,m_2}$.

(iii) We have $\sum_{m_1\in \bbZ}\sum_{m_2>0} T^{m_1,m_2}[b_{1}]\to T[b_{1}]$ with absolute convergence in $\cL(L^2, L^2)$. 
Moreover  $\sum_{j=-N}^N T_j[b_{1}] $ converges to  an operator $T[b_{1}]$ in the strong  
operator topology as operators $L^2\to L^2$ and 
$$ \|T[b_{1}]\|_{L^2\to L^2} \lc \sup_j \|\vsig_j\|_{L^1} \log^3  (1+n \Gamma_\eps)\, \|b_{1}\|_\infty.$$
\end{prop}

\begin{proof}
Use  Propositions \ref{mainpropforSmm} and \ref{OpestST}, together with Theorem 
\ref{ThmJourne2} to deduce that 
$S^{m_1, m_2}[b_{n+1}]= \sum_jS^{m_1, m_2}_j[b_{n+1}]$ converges in the strong operator topology as operators $H^1\to L^1$, with uniformity in $b_{n+1}$, $\|b_{n+1}\|_\infty\le 1$, and we get the estimate
\[ \big\|S^{m_1, m_2}[b_{n+1}]\big\|_{H^1\to L^1} \lc 
\sup\|\vsig_j\|_{L^1}\min \big \{  \log(1+n^2\Gamma_\eps),\,
2^{-\eps'(|m_1|+m_2)} n^2 \Gamma_\eps\big\} \|b_{n+1}\|_\infty
\]
Now for $b_1\in L^\infty$, $b_{n+1}\in L^\infty$ we have  by \eqref{Sjmmdef}, \eqref{Tjmmdef}
\[
\int b_1(x) \,S_j^{m_1, m_2} [b_{n+1}] f(x)\,dx=
\int b_{n+1}(x) \,T_j^{m_1, m_2} [b_{1}] f(x)\,dx\,.
\]
The uniformity with respect to $b_{n+1}$ in the strong operator convergence of 
$\sum_jS^{m_1, m_2}_j[b_{n+1}]$  now implies that  
$T^{m_1, m_2}[b_{1}]= \sum_jT^{m_1, m_2}_j[b_{1}]$ 
converges in the strong operator topology as operators $H^1\to L^1$ and we have the estimate
\[ \big\|T^{m_1, m_2}[b_{1}]\big\|_{H^1\to L^1} \lc \|b_1\|_\infty
\sup\|\vsig_j\|_{L^1}\min \big \{  \log(1+n^2\Gamma_\eps),\,
2^{-\eps'(|m_1|+m_2)} n^2 \Gamma_\eps\big\}\,.
\]
From Theorem 
\ref{ThmJourne2} we then get 
\[ \big\|T^{m_1, m_2}[b_{1}]\big\|_{L^2\to L^2} \lc \|b_1\|_\infty
\sup\|\vsig_j\|_{L^1}\min \big \{  \log(1+n^2\Gamma_\eps),\,
2^{-\eps'(|m_1|+m_2)} n^2 \Gamma_\eps\big\}\,
\]
which is (ii).  
Statement (iii) follows after summing in $m_1, m_2$.
\end{proof}

\subsection{${\Op_\eps}$-bounds and the proof of Proposition \ref{OpestST}}
\label{Opepssect}
\begin{lemma}\label{LemmaPart1MinusSecondOpInt1}
Let $\eps>0$, $\phi_0\in C^1$, supported in $\{y:|y|\le 10\}$, $\vsig\in \sBtp{\epsilon}{\R^n\times \R^d}$. For $\ell\ge 0$  define
\begin{equation*}
F_\ell(x,y):=\iiint\limits_{|x-v-y|\leq 100} \q| \dil{\vsig}{2^\ell}(\alpha,v)\w| |\phi_0(y-\alpha_1v -y')-\phi_0(y-y')|\: dv\: d\alpha\: dy'.
\end{equation*}
Then,
\begin{equation*}
\sup_x \int (1+|x-y|)^{\eps/2} |F_\ell(x,y)|\: dy + \sup_y\int(1+|x-y|)^{\eps/2} |F_\ell(x,y)|\: dx\lesssim 2^{-\ell\eps/2} \CoN{\phi_0}\sBN{\eps}{\vsig}. 
\end{equation*}
\end{lemma}
\begin{proof}
We may assume $\CoN{\phi}=1$.
We estimate, for each $y$,
\begin{equation*}
\begin{split}
&\int (1+|x-y|)^{\eps/2} |F_\ell(x,y)|\: dx
\\&= \iiiint\limits_{|x-v-y|\leq 100} (1+|x-y|)^{\eps/2} \q|\dil{\vsig}{2^\ell}(\alpha,v)\w| |\phi_0(y-\alpha_1v-y')-\phi_0(y-y')|\: dv\: d\alpha\: dy'\: dx
\\&\lesssim \iiint (1+|v|)^{\eps/2} \q|\dil{\vsig}{2^\ell}(\alpha,v)\w| |\phi_0(y-\alpha_1 v-y')-\phi_0(y-y')|\: dv\: d\alpha\: dy'
\\&\lesssim \iint (1+|v|)^{\eps/2} \q|\dil{\vsig}{2^\ell}(\alpha,v)| \min\{1, |\alpha_1 v|^{\eps/2}\}\: dv\: d\alpha
\\
&\lesssim \iint |v|^{\eps/2} \q|\dil{\vsig}{2^\ell}(\alpha,v)| \: dv\: d\alpha
+\iint |\alpha_1v|^{\eps/2} \q|\dil{\vsig}{2^\ell}(\alpha,v)| \: dv\: d\alpha\,.
\end{split}
\end{equation*}
Now 
\begin{equation*}
\iint |v|^{\eps/2} \q|\dil{\vsig}{2^\ell}(\alpha,v)\w|\: d\alpha\: dv = 2^{-\ell\eps/2} \iint |v|^{\eps/2} |\vsig(\alpha,v)|\: d\alpha\: dv\lesssim 2^{-\ell\eps/2} \|\vsig\|_{\cB_{\eps/2}}
\lesssim 2^{-\ell\eps/2} \sBN{\eps}{\vsig},
\end{equation*}
and 
\begin{align*}
&\iint |\alpha_1 v|^{\eps/2} \q|\dil{\vsig}{2^\ell}(\alpha,v)\w|\: d\alpha\: dv = 2^{-\ell\eps/2} \iint |\alpha_1 v|^{\eps/2} |\vsig(\alpha,v)|\: d\alpha\: dv
\\
&\le 2^{-\ell\eps/2} \iint (|\alpha_1|+| v|)^{\eps} |\vsig(\alpha,v)|\: d\alpha\: dv
\lesssim 2^{-\ell\eps/2} \sBN{\epsilon}{\vsig}.
\end{align*}
This completes the proof that $\sup_y \int (1+|x-y|)^{\eps/2} |F_\ell(x,y)|\: dx
\lesssim 2^{-\ell\eps/2}\sBN{\epsilon}{\vsig}.$

Next we estimate  for $x\in \bbR^d$,
\begin{equation*}
\begin{split}
&\int(1+|x-y|)^{\eps/2} |F_\ell(x,y)|\: dy
\\& =  \iiiint\limits_{|x-v-y|\leq 100} (1+|x-y|)^{\eps/2} \q|\dil{\vsig}{2^\ell}(\alpha,v)\w| |\phi_0(y-\alpha_1v-y')-\phi_0(y-y')|\: dv\: d\alpha\: dy'\: dy
\\&\lesssim \iiiint\limits_{|x-v-y|\leq 100} (1+|v|)^{\eps/2} \q|\dil{\vsig}{2^\ell}(\alpha,v)|
\min\{1,|\alpha_1 v|^{\eps/2}\}
 \bbone_{\{ |y-\alpha_1v-y'|\leq 10 \text{ or } |y-y'|\leq 10\} }\: dv\: d\alpha\: dy'\: dy
\\&\lesssim \iiint\limits_{|x-v-y|\leq 100} (1+|v|)^{\eps/2} \q|\dil{\vsig}{2^\ell}(\alpha,v)|
\min\{1,|\alpha_1 v|^{\eps/2}\} \: dv\: d\alpha\: dx
\\&\lesssim \iint  (1+|v|)^{\eps/2} \q|\dil{\vsig}{2^\ell}(\alpha,v)|
\min\{1,|\alpha_1 v|^{\eps/2}\} \: dv\: d\alpha
\end{split}
\end{equation*}
and above the last quantity has  already been shown to be $\lc 
2^{-\ell \eps/2} \sBN{\epsilon}{\vsig}$. 
 This completes the proof of the lemma.
\end{proof}

\begin{lemma}\label{LemmaPart1MinusSecondOpInt2}
Let $\epsilon>0$.  For $\phi\in C^1$, supported in $\{y:|y|\le 10\}$, $\vsig\in \sBtp{\eps}{\R^d\times \R^n}$,
$j\geq 0$, 
let 
\[g_j(x,y)= \int \q|\dil{\vsig}{2^j}(\alpha,v)\w| |\phi(x-v-y)-\phi(x-y)|\: d\alpha\: dv.\]
Then \[\sup_x \int g_j(x,y)\,dy+\sup_y \int g_j(x,y)\,dx
 \lesssim 2^{-\eps j} \sBN{\eps}{\vsig}\CoN{\phi}\,.\]
\end{lemma}
\begin{proof}
We may assume $\CoN{\phi}=1$.  For any $x$, we have
\begin{equation*}
\begin{split}
&\iiint (1+|x-y|)^{\eps} \q|\dil{\vsig}{2^j}(\alpha,v)\w| |\phi(x-v-y)-\phi(x-y)|\: d\alpha\: dv\: dy
\\&\lesssim\iiint (1+|x-y|)^{\eps} \q|\dil{\vsig}{2^j}(\alpha,v)|\min\{1, |v|^{\eps}\} 
\chi_{\{ |x-v-y|\leq 10\text{ or }|x-y|\leq 10 \}}\: d\alpha\: dv\: dy
\\&\lesssim\iint (1+|v|)^{\eps} \q|\dil{\vsig}{2^j}(\alpha,v)|\min\{1, |v|^{\eps}\}  \: d\alpha\: dv
\\&\lesssim\iint |v|^{\eps} \q|\dil{\vsig}{2^j}(\alpha,v)| \: d\alpha\: dv
\lesssim 2^{-j\epsilon}\sBN{\eps}{\vsig},
\end{split}
\end{equation*}
where the last inequality has already been used  in the proof of Lemma \ref{LemmaPart1MinusSecondOpInt1}.  By symmetry we also get the corresponding second inequality with the roles of $x$ and $y$ reversed.
\end{proof}

\begin{lemma}\label{PropPart1IMinusOpProp1}
For  $\epsilon>0$ there is $\eps'>0$ such that the following holds. Let $\phi_1,\ldots, \phi_{n+1}\in C^2$ supported in $\{y:|y|\le 10\}$ and 
such that for all but at most two $l$, $\phi_l\geq 0$ and $\int \phi_l=1$.  For $k\in \Z$ set $Y_k^l f=f*\dil{\phi_l}{2^k}$.
For $b_1,\ldots, b_n\in L^\infty(\R^d)$, $\vsig\in \sBtp{\epsilon}{\R^n\times \R^d}$ with 
\Be\label{vsigcanc}\int \vsig(\alpha,v)\: dv=0,\Ee and  define a kernel 
$K_{j,k}\equiv K_{j,k}[b_1,\dots, b_n]$
by
\begin{equation*}
\int g(x) \int K_{j,k}(x,y)f(y) \,dy\, dx
= \La[\vsigj](Y_k^1 b_1,\ldots, Y_k^n b_n, Y_k^{n+1}g, f).
\end{equation*}
Then, for $j\ge k$, 
\begin{align*}
\|\Dil_{2^{-k} }K_{j,k}\|_{\Op_{\eps'}}&\lesssim 2^{-\epsilon' (j-k)} n \sBN{\epsilon}{\vsig} \prod_{i=1}^n
 \|{b_i}\|_\infty,
\\
\|\Dil_{2^{-k}} K_{j,k}\|_{\Op_{0}}&\lesssim \|\vsig\|_{L^1} \prod_{i=1}^n  \|{b_i}\|_\infty.
\end{align*}
Here, the implicit constants may depend on $\max\limits_{i_1,i_2,i_3,i_4\in \{1,\ldots, n+1\}} \CtN{\phi_{i_1}}\CtN{\phi_{i_2}} \CtN{\phi_{i_3}}\CtN{\phi_{l_4}}$.
\end{lemma}

\begin{proof}
The bound for the $\Op_0$ norm is  immediate so we focus only on the bound for the $\Op_\eps$-norms.
Note that by scaling (see  Lemma \ref{scalinglemma})
$$\La[\vsigj](Y_k^1 b_1,\ldots, Y_k^n b_n, Y_k^{n+1}g, f)=2^{-kd}
\La[\vsig^{(2^{j-k})}] (Y_0 b^k_1, \dots, Y_0 b^k_n, Y_0 g^k, f^k)$$
where $b^k_i= b_i(2^{-k}\cdot)$,
$f^k= f(2^{-k}\cdot)$, $g^k= g(2^{-k}\cdot)$. This leads to
$$ K_{j,k}[b_1,\dots, b_n](x,y)  = 2^{kd} K_{ j-k,0}[ b^k_1,\dots, b^k_n](2^k x, 2^k y).$$ 

Now $\|b_i^k\|_\infty=\|b_i\|_\infty$, $i=1,\dots,n$,  and hence 
after replacing the functions $b_i$ by $b_i^k$, $i=1,\dots,n$, it suffices to check the case $k=0$. That is, we need to prove, for $\ell\ge 0$,
\Be\label{rescaledk=0}
\|K_{\ell,0}[b_1,\dots,b_n]\|_{\Op_\eps} \lc 2^{-\eps'\ell} n \|\vsig\|_{\cB_\eps} \prod_{i=1}^n \|b_i\|_\infty.
\Ee

In what follows we may assume
$\LpN{\infty}{b_i}=1$, $i=1,\ldots, n$. 
We will prove, under the assumption that all but at most {\it three} of the $\phi_i$ satisfy $\phi_i\geq 0$, $\int \phi_i=1$ we have
\begin{equation}\label{EqnPart1IMinusSecondOpIntToShow}
\sup_x \int (1+|x-y|)^{\eps'} |K_{\ell,0}(x,y)|\: dy + \sup_y \int (1+|x-y|)^{\eps'} |K_{\ell,0}(x,y)|\: dx\lesssim 2^{-\eps'\ell} n\sBN{\eps}{\vsig},
\end{equation}
where the implicit constant is allowed to depend on the $C^1$ norms of up to three of $\phi_i$ (instead of the $C^2$ norms).

First we see why \eqref{EqnPart1IMinusSecondOpIntToShow} yields the result.
The explicit formula for the kernel is
\Be \label{Kell0}
K_{\ell,0}(x,y)= 
\int \phi_{n+1}(y-v-x) \int \vsig^{(2^\ell)}(\alpha, v) \prod_{i=1}^n Y_0^ib_i(y-\alpha_i v )\,d\alpha\, dv.
\Ee It implies 
that
$\partial_{x_m} K_{\ell,0}(x,y) $ is a term of the form covered by \eqref{EqnPart1IMinusSecondOpIntToShow}  (with  $\phi_{n+1}$ replaced by $-\partial_{x_m} \phi_{n+1}$). Moreover, $\partial_{y_m} K_{\ell,0}(x,y) $ is a sum of  $n+1$ terms of the form covered by \eqref{EqnPart1IMinusSecondOpIntToShow}, indeed differentiating 
\eqref{Kell0} yields 
(setting $b_{n+1}:=g$)
\begin{multline*}
\int b_{n+1}(x) \int \partial_{y_m} K_{\ell,0}(x,y) f(y) \,dy \, dx
\\
= \sum_{i=1}^{n+1} \La [\vsig^{(2^\ell)}] (Y_0^1 b_1,\ldots, Y_0^{i-1} b_{i-1} , \partial_{{x}_m} Y_0^l b_i, Y_0^{i+1} b_{i+1},\ldots, Y_0^{n+1} b_{n+1}, f).
\end{multline*}
Thus, $\partial_{x_m} K_{\ell,0}(x,y) $ is a sum of  $n+1$ terms of the form covered by \eqref{EqnPart1IMinusSecondOpIntToShow}.
From these remarks, it follows, given
 \eqref{EqnPart1IMinusSecondOpIntToShow}, 
 that the expressions
\begin{align*}
&\sup_{\substack{y\\ 0<|h|\leq 1}} |h|^{-1} \int |K_{\ell,0}(x,y+h)-K_{\ell,0}(x,y)|\: dx,
\\&\sup_{\substack{x\\ 0<|h|\leq 1}} |h|^{-1} \int |K_{\ell,0}(x,y+h)-K_{\ell,0}(x,y)|\: dy,
 \\&\sup_{\substack{y\\ 0<|h|\leq 1}} |h|^{-1} \int |K_{\ell,0}(x+h,y)-K_{\ell,0}(x,y)|\: dx,\quad
 \\&\sup_{\substack{x\\ 0<|h|\leq 1}} |h|^{-1} \int |K_{\ell,0}(x+h,y)-K_{\ell,0}(x,y)|\: dy
 \end{align*}
 are all bounded by a constant times $2^{-\ell\eps'} n \sBN{\epsilon}{\vsig}.$

It remains to prove \eqref{EqnPart1IMinusSecondOpIntToShow}. We first compute,
with $\tilde \vsig(\alpha,v)= \vsig(1-\alpha_1, \dots, 1-\alpha_n,v)$,
\begin{align*}
&\La[\vsig^{(2^\ell)}](Y_0^1b_1,\dots, Y_0^n, Y_0^{n+1} g, f)=
\La[\vsig^{(2^\ell)}](Y_0^1b_1,\dots, Y_0^n, f, Y_0^{n+1} g)
\\&= \iiint \tilde \vsig^{(2^\ell)}
(\alpha, w-y) f(y) 
\int\phi_{n+1} (w-x) g(x) dx\,  \prod_{i=1}^n Y_0^ib_i (w(1-\alpha_i)+\alpha_i y)\,
d\alpha\,dw \,dy
\\&=
 \iint g(x) f(y) \iint \tilde  \vsig^{(2^\ell)}(\alpha,v) \phi_{n+1}(y+v-x) \prod_{i=1}^n 
 Y_0^ib_i(y+(1-\alpha_i)v) \,dv \, d\alpha \, dx\, dy
  \end{align*}
  and changing variable in $\alpha$ again we get
\begin{align*}   K_{\ell,0}(x,y)&= 
  \iint \vsig^{(2^\ell)}(\alpha,v) \phi_{n+1}(y+v-x) \prod_{i=1}^n Y_0^ib_i(y+\alpha_i v) \,dv \, d\alpha
  \\
&=    \iint
     \vsig^{(2^\ell)}(\alpha,v) 
     \Big[\phi_{n+1}(y+v-x) \prod_{i=1}^n Y_0^ib_i(y+\alpha_i v) 
     -
     \phi_{n+1}(y-x) \prod_{i=1}^n Y_0^ib_i(y) \Big]\,
     dv \, d\alpha;
     \end{align*}
here we have used the cancellation condition \eqref{vsigcanc}.
Now 
\[ |K_{\ell,0}(x,y)| \le I(x,y)+\sum_{i=1}^n II_i(x,y)\]
where
\begin{align*}  I(x,y)&= \iint |
 \vsig^{(2^\ell)}(\alpha,v)| |\phi_{n+1}(y+v-x) -\phi_{n+1}(y-x)| \, dv \, d\alpha\,,
 \\
 II_i (x,y)&= \iint
  |\vsig^{(2^\ell)}(\alpha,v)| |\phi_{n+1}(y-x)| \int| \phi_i(y+\alpha_i v-w) -
  \phi_i(y-w)| dw\, dv\, d\alpha\,.
  \end{align*}
  Now apply 
  Lemma \ref{LemmaPart1MinusSecondOpInt1} to the expessions $II_i$ and Lemma \ref{LemmaPart1MinusSecondOpInt2} to $I$, and \eqref{EqnPart1IMinusSecondOpIntToShow}  follows. This completes the proof.
  \end{proof}

\begin{proof} [Proof of Proposition \ref{OpestST}, conclusion]
We  focus on the estimates for $S^{m_1,m_2}_j[b_{n+1}]$ as the estimates for 
$T^{m_1,m_2}_j[b_1]$  are analogous (switch the roles of $b_1$ and $b_{n+1}$).
We may assume $\|b_{n+1}\|_\infty=1$.

In what follows we identify operators with their  Schwartz kernels.
For an operator $R$ we denote by $\partial_{x_\mu}R$ the operator with Schwartz kernel
$\partial_{x_\mu} R(x,y)$.

We use Lemma  \ref{LemmaAuxDecompPsi} to write
$Q_{j+m_2} = \sum_{\mu=1}^d 2^{-(j+m_2)} \partial_{x_\mu} R_{j+m_2}^\mu$, where $R_{j+m_2}^\mu  = f*\dil{\phit_\mu}{2^{j+m_2}}$, and $\phit_l\in C^\infty_0$ supported in $\{x:|x|\le 2\}$.
Now
\begin{align*}
&\La[\vsigjj](Q_{j+m_1}P_j b_1, X^2_{j+m_1 }P_j b_2, \dots, X^{n+1}_{j+m_1} P_j b_{n+1}, Q_{j+m_2} f)
\notag
\\
&=2^{-(j+m_2)}\sum_{\mu=1}^d\iint \vsigjj(\alpha, v) 
\int \partial_{x_\mu} R^\mu_{j+m_2} f(x)  X^{n+1}_{j+m_1}P_j b_{n+1}(x-v) 
\,\times
\\
&\qquad\qquad\qquad\qquad\qquad\qquad\qquad
Q_{j+m_1}P_j b_1(x-\alpha_1 v)
\prod_{i=2}^nQ_{j+m_1}P_j b_i(x-\alpha_i v)\, dx\,  \,dv\, d\alpha\,.
\end{align*} 
Integrating by parts we see that this expression equals
\begin{multline}\label{intbypartsform}
- 2^{-(j+m_2)} \sum_{\mu=1}^d 
\biggl(
\La[\vsigjj](\partial_{x_\mu}Q_{j+m_1}P_j b_1, X^2_{j+m_1 }P_j b_2, 
\dots, X^{n+1}_{j+m_1} P_j b_{n+1}, 
R_{j+m_2} ^\mu f)
\\  +
 \sum_{\nu=2}^{n+1}
\La[\vsigjj](Q_{j+m_1}P_j b_1, X^2_{j+m_1 }P_j b_2, \dots,
\partial_{x_\mu} X^\nu_{j+m_1}P_j b_\nu
\dots, X^{n+1}_{j+m_1} P_j b_{n+1}, 
R_{j+m_2} ^\mu f)\biggr) .
\end{multline}
We distinguish the cases $m_1\le 0$ and $m_1\ge 0$.

For $m_1\le 0$ we write \eqref{intbypartsform} as
\begin{align*}
&\La[\vsigjj](Q_{j+m_1}P_j b_1, X^2_{j+m_1 }P_j b_2, \dots, X^{n+1}_{j+m_1} P_j b_{n+1}, Q_{j+m_2} f)
\\
&= -\, 2^{-m_2+m_1}
\sum_{\mu=1}^d\sum_{\nu=1}^{n+1}
\Lambda[\vsigjj]( Y^{1,\mu,\nu}_{j+m_1,j} b_1,\dots,
Y^{n+1,\mu,\nu}_{j+m_1,j} b_{n+1}, R^{\mu}_{j+m_2}f)
\end{align*} 
where,  for $m_1\le 0$, the operators $Y^{i,\mu,\nu}_{j+m_1,j}$ are given by 
\[
Y^{1,\mu,\nu}_{j+m_1,j}= 
\begin{cases}
2^{-j-m_1}\partial_{x_\mu} (Q_{j+m_1}P_j) &\text{ if }\nu=1,
\\
Q_{j+m_1} P_j &\text{ if } \nu\in \{2, \dots,n+1\}
\end{cases}
\]
if $i=1$, and by
\[
Y^{i,\mu,\nu}_{j+m_1,j}= 
\begin{cases}
2^{-j-m_1}\partial_{x_\mu} (P_{j+m_1}P_j) &\text{ if }\nu=i,
\\
P_{j+m_1} P_j &\text{ if } \nu\in \{1, \dots,n+1\}\setminus\{i\}\,
\end{cases}
\]
if  $2\le i \le n+1$.

Hence for $m_1\le 0$
\begin{align*}
&\La[\vsigjj](Q_{j+m_1}P_j b_1, X^2_{j+m_1 }P_j b_2, \dots, X^{n+1}_{j+m_1} P_j b_{n+1}, Q_{j+m_2} f)
\\
&= 2^{-m_2+m_1} \sum_{\mu=1}^d\sum_{\nu=1}^n \int b_{n+1}(x) K^{\mu,\nu}_{j+m_1, j}(x,y)R^\mu_{j+m_2} f(y) dy
\end{align*}
and by Lemma \ref{PropPart1IMinusOpProp1} 
$$\|\Dil_{2^{-j-m_1}} K^{\mu,\nu}_{j+m_1,j}\|_{Op_{\eps'} } \lc \|\vsig_j\|_{\cB_\eps}$$
for some $\eps'\le \eps$.
This, together with  Lemma \ref{PS-SP-lemma}, implies the asserted bound \eqref{sigmammest}, for $m_1\le 0$.

We now consider the case $m_1> 0$. Now use the cancellation and support properties of $Q_{j+m_1}$ to write 
\[Q_{j+m_1} P_j = 2^{-m_1}Z_{j,m_1}\] where $Z_{j,m_1}=f*\upsilon_{j,m}^{(2^j)}$ and
$\{\upsilon_{j,m}: j\in \bbZ, m_1\in \bbN\}$ is a bounded family of $C^\infty_c$ functions supported in $\{y:|y|\le 2\}$.

We now write  \eqref{intbypartsform} as
\begin{align*}
&\La[\vsigjj](Q_{j+m_1}P_j b_1, X^2_{j+m_1 }P_j b_2, \dots, X^{n+1}_{j+m_1} P_j b_{n+1}, Q_{j+m_2} f)
\\
&= -\, 2^{-m_2-m_1}
\sum_{\mu=1}^d\sum_{\nu=1}^{n+1}
\Lambda[\vsigjj]( Y^{1,\mu,\nu}_{j+m_1,j} b_1,\dots,
Y^{n+1,\mu,\nu}_{j+m_1,j} b_{n+1}, R^{\mu}_{j+m_2}f)
\end{align*} 
where (now for $m_1>0$)
\[
Y^{1,\mu,\nu}_{j+m_1,j}= 
\begin{cases}
2^{-j}\partial_{x_\mu} Z_{j,m_1} &\text{ if }\nu=1,
\\
Z_{j,m_1} &\text{ if } \nu\in \{2, \dots,n+1\},
\end{cases}
\]
and for $2\le i \le n+1$
\[
Y^{i,\mu,\nu}_{j+m_1,j}= 
\begin{cases}
2^{-j}\partial_{x_\mu} (P_{j+m_1}P_j) &\text{ if }\nu=i,
\\
P_{j+m_1} P_j &\text{ if } \nu\in \{1, \dots,n+1\}\setminus\{i\}\,.
\end{cases}
\]
We see, using Lemma \ref{PropPart1IMinusOpProp1}, that for $m_1>0$ 
\begin{align*}
&\La[\vsigjj](Q_{j+m_1}P_j b_1, X^2_{j+m_1 }P_j b_2, \dots, X^{n+1}_{j+m_1} P_j b_{n+1}, Q_{j+m_2} f)
\\
&= 2^{-m_2-m_1} \sum_{\mu=1}^d\sum_{\nu=1}^n \int b_{n+1}(x) K^{\mu,\nu, m_1}_{ j}(x,y)R^\mu_{j+m_2} f(y) \,dy
\end{align*}
with \[ \big\|\Dil_{2^{-j}} K^{\mu,\nu, m_1}_j\big\|_{\Op_\eps}\lc \|\vsig_j\|_{\cB_\eps}\,.\]
Using also
Lemma \ref{PS-SP-lemma} we obtain the asserted bound \eqref{sigmammest}, for $m_1> 0$.
\end{proof}

\subsection{Proof of  the bound \eqref{Lambda-l-2est}, concluded}\label{sect12concl}
The following proposition will conclude the proof of part 
\ref{ItemBound1IminusPt}
 in Theorem \ref{main-parts}. 
\begin{prop}\label{PropPart1IMinusRed}
Let $1\leq l_1\ne l_2 \leq n+2$.  Then, for $p\in (1,2]$ and $p'=p/(p-1)$ 
\begin{multline*}
\q|\sum_{j\in \Z} \La[\vsigjj](P_j b_1,\ldots, P_j b_{n+1}, (I-P_j) b_{n+2} )\w|
\\ \leq C_{d,p,\epsilon} n \q(\sup_j \LpN{1}{\vsig_j}\w) \log^3\q(1+n\Ga_\eps\w) \big(\prod_{l\ne l_1,l_2} \|{b_l}\|_\infty\big) \|b_{l_1}\|_p \|b_{l_2}\|_{p'}\,.
\end{multline*}
\end{prop}
\begin{proof}
By symmetry of the roles of $b_1,\ldots, b_{n+1}$, via Theorem \ref{ThmOpResAdjoints}, it suffices to prove
the result for three cases: $(l_1,l_2)=(n+1,n+2)$, $(l_1,l_2)=(n+2,n+1)$, and $(l_1,l_2)=(1,n+1)$.

We begin with the case $(l_1,l_2)=(n+1,n+2)$. 
For this we define an operator
$S_{1,j}\equiv S_{1,j}[b_1,\ldots, b_n]$ by
\begin{equation*}
\int g(x) \q(S_{1,j}[b_1,\ldots, b_n] f\w)(x)\: dx := \La[\vsigjj](P_j b_1,\ldots, P_j b_{n+1}, (I-P_j) b_{n+2}).
\end{equation*}
It is straightforward  to verify
the inequalities
\begin{align*}
\|\Dil_{2^{-j} } S_{1,j} \|_{\Op_\eps}&\lc
 n(\sup_{j\in \Z} 
 \|\vsig_j\|_{\cB_\eps} \prod_{i=1}^n \|b_i\|_\infty,
 \\
 \|\Dil_{2^{-j} } S_{1,j} \|_{\Op_0}&\lc
   \q(\sup_{j\in \Z} \LpN{1}{\vsig_j}\w) \prod_{i=1}^n \|b_i\|_\infty;
\end{align*}
here $ \eps\le 1$ and the $\Op_\eps$, $\Op_0$ norms are as in \eqref{EqnJourneBoundCep}, \eqref{schurassu}.

Theorem \ref{PropPart1IMinus} shows
\begin{equation*}
\Big\|\sum_{j\in \Z}S_{1,j}[b_1,\ldots, b_n]\Big\|_{L^2\to L^2}
 \lesssim n (\sup_j \LpN{1}{\vsig_j})\log^3\q(1+n\Ga_\eps)  \prod_{i=1}^n \|b_i\|_\infty.
\end{equation*}
with convergence in the strong operator topology.
By Proposition \ref{CorWeakTypeInterp} we get, for  $1<p\leq 2$,
\begin{equation*}
\Big\|\sum_{j\in \Z} S_{1,j}[b_1,\ldots, b_n]\Big\|_{L^p\to L^p} \leq C_{d,p,\epsilon} n \q(\sup_j \LpN{1}{\vsig_j}\w) \log^3\q(1+n\Ga_\eps) \prod_{i=1}^n \|b_i\|_\infty,
\end{equation*}
and
\begin{equation*}
\Big\|\sum_{j\in \Z} {}^t\!S_{1,j}[b_1,\ldots, b_n]\Big\|_{L^p\to L^p} \leq C_{d,p,\epsilon} n \q(\sup_j \LpN{1}{\vsig_j}\w) \log^3\q(1+n\Ga_\eps) \prod_{i=1}^n \|b_i\|_\infty,
\end{equation*}
which are equivalent to the statement of the proposition in the cases $(l_1,l_2)=(n+1,n+2)$ and $(l_1,l_2)=(n+2,n+1)$, respectively. 
The convergence is in the sense of the strong operator topology (as operators bounded on $L^p$).

We now turn to the case $(l_1, l_2)=(1, n+1)$. If we apply  Theorem \ref{ThmJourne2} 
to $\sum {}^t\!S_{1,j}$ we also get an $H^1\to L^1$ bound 
\begin{equation*}
\Big\|\sum_{j\in \Z} {}^t\!S_{1,j}[b_1,\ldots, b_n]\Big\|_{H^1\rightarrow L^1} \lesssim n \q(\sup_j \LpN{1}{\vsig_j}\w) \log^3\q(1+n\Ga_\eps) \prod_{i=1}^n \LpN{\infty}{b_i}.
\end{equation*}
This means that  for $b_1,\ldots, b_n\in L^\infty(\R^d)$, $b_{n+2}\in L^\infty(\R^d)$, $b_{n+1}\in H^1(\R^d)$, we have
\begin{multline}\label{EqnPart1IMinusRedBound}
\Big|\sum_{j\in \Z} \La[\vsigjj](P_j b_1,\ldots, P_j b_{n+1}, (I-P_j) b_{n+2})\Big|
\\ \lesssim n\q(\sup_j \LpN{1}{\vsig_j}\w) \log^3\q(1+n\Ga_\eps ) \big(\prod_{i=1}^n \|b_i\|_\infty\big)\q\|b_{n+1}\w\|_{H^1} \|b_{n+1}\|_\infty.
\end{multline}
For $j\in \Z$, define an operator $S_{2,j}[b_2,\ldots, b_n,b_{n+2}]$ by
\begin{equation*}
\int g(x) \q(S_{2,j}[b_2,\ldots, b_n, b_{n+2}] 
f\w)(x)\: dx:= \La[\vsigjj](g, P_j b_2, \ldots, P_j b_n, f, (I-P_j)b_{n+2}).
\end{equation*}
Since $\tr P_j=P_j$ the case  $(l_1,l_2)=(1,n+1)$ is equivalent to the inequality
\Be \label{1-n+1-case}
\Big\|\sum_{j\in \bbZ} P_j S_{2,j}[b_2,\ldots, b_n, b_{n+2}] P_j\Big\|_{L^p\to L^p}
\lc n \sup_j\|\vsig\|_{L^1} (1+n\Ga_\eps) \prod_{l\in \{2,\dots, n, n+2\} }\|b_l\|_\infty.
\Ee
To show \eqref{1-n+1-case} we first observe that 
by Theorem \ref{ThmOpResAdjoints}, there is a $c>0$ (independent of $n$) such that
for $\eps'<c\eps$ there are $\vsigt_j\in \cB_{\eps'}(\R^n\times \R^d)$ with
$\|\vsigt_j\|_{\cB_{\eps'}}\lesssim n \|\vsig\|_{\cB_\eps}$ and  $\LpN{1}{\vsigt_j}=\LpN{1}{\vsig_j}$ such that
\begin{equation*}
\int b_1(x) \q(S_{2,j}[b_2,\ldots, b_n, b_{n+2}] b_{n+1}\w)(x)\: dx = \La[\vsigt_j^{(2^j)}] 
(P_j b_2,\ldots, P_j b_n, (I-P_j) b_{n+2}, b_1, b_{n+1}).
\end{equation*}
If we apply \eqref{EqnPart1IMinusRedBound}  with the family $\{\tilde \vsig_j\}$ in place of $\{\vsig_j\}$ and $\eps'$ in place of $\eps$) we 
get
\begin{align*}
&\Big|\sum_{j\in \Z} \La[\tilde \vsig_j^{(2^j)}](P_j b_1,\ldots, P_j b_{n+1}, (I-P_j) b_{n+2})\Big|
\\ &\quad\lesssim n\q(\sup_j \LpN{1}{\vsig_j}\w) \log^3
\big(1+n \frac{\sup_j\|\tilde \vsig_j\|_{\cB_\eps}}{\sup_j\|\tilde \vsig_j\|_{L^1}}\big) \big(\prod_{i=1}^n \|b_i\|_\infty\big)\q\|b_{n+1}\w\|_{H^1} \|b_{n+1}\|_\infty
\\ &\quad\lesssim n\q(\sup_j \LpN{1}{\vsig_j}\w) \log^3(1+n\Ga_\eps)
\big(\prod_{i=1}^n \|b_i\|_\infty\big)\q\|b_{n+1}\w\|_{H^1} \|b_{n+1}\|_\infty
\end{align*} which (in view of
${}^t\!P_j=P_j$) can be rephrased as 
\begin{equation*}
\Big\| \sum_{j} P_j S_{2,j}[b_2,\ldots, b_n, b_{n+2}]P_j \Big\|_{H^1\rightarrow L^1} \lesssim  n\q(\sup_j \LpN{1}{\vsig_j}\w) \log^3\q(1+n\Ga_\eps)\prod_{l\in \{2,\ldots, n, n+2\}} \|b_l\|_\infty.
\end{equation*}
We wish to apply Lemma \ref{PSP-lemma} to the kernels 
 $\sigma_j= \Dil_{2^{-j}} S_{2,j}$. Observe that the Schur integrability norms for these
 kernels satisfy the uniform estimates 
\begin{equation*}
\Sha^1_\eps[\sigma_j]+\Sha^\infty_\eps[\sigma_j]
\lesssim  \sBN{\epsilon'}{\vsigt_j} \prod_{l\in \{2,\ldots, n, n+2\}} \LpN{\infty}{b_l}\lesssim n \sup_j \sBN{\epsilon}{\vsig_j}\prod_{l\in \{2,\ldots, n, n+2\}} \|b_l\|_\infty,
\end{equation*}
and 
\begin{equation*}
\Sha^1_\eps[\sigma_j]+\Sha^\infty_\eps[\sigma_j]
\lesssim  \LpN{1}{\vsigt_j} \prod_{l\in \{2,\ldots, n, n+2\}} \|b_l\|_\infty\le \sup_j \LpN{1}{\vsig_j} 
\prod_{l\in \{2,\ldots, n, n+2\}}\|b_l\|_\infty.
\end{equation*}

Now Theorem \ref{ThmJourne2} 
in conjunction with 
Lemma \ref{PSP-lemma} applies to show
\begin{equation*}
\Big\| \sum_{j} P_j S_{2,j}[b_2,\ldots, b_n, b_{n+2}]P_j \Big\|_{L^2\rightarrow L^2} \lesssim  n\q(\sup_j \LpN{1}{\vsig_j}\w) \log^3(1+n\Ga_\eps) \prod_{l\in \{2,\ldots, n, n+2\}} \|{b_l}\|_\infty,
\end{equation*}
with convergence in the strong operator topology.
Finally \eqref{1-n+1-case} follows by interpolation (see Corollary \ref{ThmWeakTypeWithP}).
This completes the proof.
\end{proof}

\section{Proof of Theorem \ref{main-parts}: Part \ref{itemBoundAllPts}}
\label{Sectionpartallpts}

In this section, we consider the multilinear form
\begin{equation*}
\La^3(b_1,\ldots, b_{n+2}):=\sum_{j} \La[\vsigjj](P_j b_1,\ldots, P_j b_{n+2}),
\end{equation*}
where the summation is a priori extended over a finite subset of $\bbZ$, and where,  for some fixed $\epsilon>0$, $\{\vsig_j : j\in \Z\}\subset \sBtp{\epsilon}{\R^n\times \R^d}$ is a bounded set with $\int \vsig_j(\alpha,v)\: dv=0,$ for  all $j$ and almost every $\alpha$.
To prove part \ref{itemBoundAllPts} of Theorem \ref{main-parts} we need to establish for $1<p\le 2$
the inequality
\begin{equation}\label{partfiveest}
\q|\La^3(b_1,\ldots, b_{n+2})\w|\leq C_{d,p,\epsilon} n^2 \q(\sup_{j} \LpN{1}{\vsig_j}\w) \log^3 \q(1+n\Ga_\eps) \big(\prod_{i=1}^n\|b_i\|_\infty \big) \|b_{n+1}\|_{p'} \|b_{n+2}\|_p.
\end{equation}


As in the previous section the  heart of the proof lies in the case $p=2$ which we state as a theorem.
\begin{thm}\label{ThmPartAllPtsMainProp}
Let $b_1,\ldots, b_n\in L^\infty(\R^d)$ and $b_{n+1},b_{n+2}\in L^{2}(\R^d)$.  Then,
$$\lim_{N\to\infty}
\sum_{j=-N}^N \La[\vsigjj](P_j b_1,\ldots, P_j b_{n+2})= \La^3(b_1,\ldots, b_{n+2})
$$ and $\La^3$ satisfies 
\begin{equation*}
\q|\La^3(b_1,\ldots, b_{n+2})
\w|\leq C_{d,\epsilon} n^2 \sup_{j} \LpN{1}{\vsig_j} \log^3 \q(1+n\Ga_\eps)
\big( \prod_{i=1}^n \|b_i\|_\infty \big)\|b_{n+1}\|_2 \|b_{n+1}\|_2.
\end{equation*} 
The sum defining the operator $T^3[n_1,\dots, b_n]$ associated to $\La^3$ converges in the strong operator topology as bounded operators $L^2\to L^2$.
\end{thm}

\begin{proof}[Proof of \eqref{partfiveest} 
given Theorem \ref{ThmPartAllPtsMainProp}]
We may assume $\LpN{\infty}{b_l}=1$, $l=1,\ldots, n$.
For $j\in \Z$ define the operator $T_j$ by 
\begin{equation*}
\int g(x)  \,{T_j}f(x)\: dx := \La[\vsigjj](P_j b_1,\ldots, P_j b_n, P_j g, {f}).
\end{equation*}
 Theorem \ref{ThmPartAllPtsMainProp} is equivalent to
\begin{equation*}
\Big\|\sum_{j\in \Z} T_jP_j \Big\|_{L^2\to L^2}
 \lesssim n^2 \sup_{j} \LpN{1}{\vsig_j} \log^3 \q(1+n\Ga_\eps)
 \prod_{i=1}^n \|b_i\|_p
\end{equation*}
Corollary  \ref{ThmWeakTypeWithP} applies since 
$\sup_j \Sha_\eps^1[ \Dil_{2^{-j} }T_j] \lesssim \sup_j \|\vsig_j\|_{\cB_\eps} $,
$\sup_j \Sha_0^1[ \Dil_{2^{-j} }T_j] \lesssim \sup_j \|\vsig_j\|_{L^1} $. This completes the proof.
\end{proof}
We now turn to the proof of Theorem \ref{ThmPartAllPtsMainProp}. The argument is analogous to the
arguments in the previous section and therefore we shall be brief.

\subsection{Basic decompositions}
We argue as in \S \ref{partfouroutline}
and decompose
\begin{equation*} \begin{split}
& \La[\vsigjj](P_j b_1,\ldots, P_jb_{n+2})
\\
&=\lim_{M\to\infty}\Big(
 \La[\vsigjj](P_{j+M}P_j b_1,\ldots, P_{j+M}P_j b_{n+1}, P_{j+M}P_j b_{n+2})
\\&\qquad\qquad\qquad \qquad-
 \La[\vsigjj](P_{j-M}P_j b_1,\ldots, P_{j-M}P_j b_{n+1}, P_{j-M}P_j b_{n+2})\Big)
\\
&=\lim_{M\to\infty}\sum_{m=-M+1}^M\Big(
 \La[\vsigjj](P_{j+m}P_j b_1,\ldots, P_{j+m}P_j b_{n+2})
\\&\qquad\qquad\qquad \qquad
-
 \La[\vsigjj](P_{j+m-1}P_j b_1,\ldots, P_{j+m-1}P_j b_{n+1}, P_{j+m-1}P_j b_{n+2})\Big)
\end{split}
\end{equation*}
and thus 
\begin{multline*}\label{Q-expansion}
 \La[\vsigjj](P_j b_1,\ldots, P_j b_{n+2})\,=
\\
\sum_{l=1}^{n+1} \sum_{m=-\infty}^{\infty} 
\La[\vsigjj](P_{j+m-1}P_jb_1,\dots, 
P_{j+m-1}P_jb_{l-1} , Q_{j+m} P_j b_l,
P_{j+m}P_j b_{l+1},\dots, P_{j+m}P_jb_{n+2}).
\end{multline*}
We repeat  the same procedure to each term 
and write,  for fixed $m\in \bbZ$
\[ \La[\vsigjj](P_{j+m-1}P_jb_1,\dots, 
P_{j+m-1}P_jb_{l-1} , Q_{j+m} P_j b_l,
P_{j+m}P_j b_{l+1},\dots, P_{j+m}P_jb_{n+2})\]
as the limit (as $M\to \infty$) of the differences
\begin{align*}
\La[\vsigjj]\big(&P_{j+M}P_{j+m-1}P_jb_1,\dots, 
P_{j+M}P_{j+m-1}P_jb_{l-1} ,
\\
&\quad P_{j+M}Q_{j+m} P_j b_l,
P_{j+M}P_{j+m}P_j b_{l+1},\dots, P_{j+M}P_{j+m}P_jb_{n+2}\big)
\\-
\La[\vsigjj](&P_{j-M}P_{j+m-1}P_jb_1,\dots, 
P_{j-M}P_{j+m-1}P_jb_{l-1} , 
\\& \quad P_{j-M}Q_{j+m} P_j b_l,
P_{j-M}P_{j+m}P_j b_{l+1},\dots, P_{j-M}P_{j+m}P_jb_{n+2})\,.
\end{align*}
We continue as above, writing each difference as a
collapsing
sum, and than expanding each summand using the multilinearity of the functionals.
The limit of the expressions in the last display becomes
$$
\La[\vsigjj](P_j b_1,\ldots, P_jb_{n+2})
=\sum_{\substack {(l_1,l_2)\\ 1\le l_1\neq l_2\le n+2}}
\sum_{(m_1,m_2)\in \bbZ^2} \la_{j,l_1,l_2}^{m_1,m_2}(b_1,\dots,b_{n+2})
$$
where, for $l_1<l_2$,  
\begin{align*} 
\la_{j,l_1,l_2}^{m_1,m_2}(&b_1,\dots,b_{n+2}):=
\\ \La[\vsigjj]\big(&P_{j+m_2-1} P_{j+m_1-1} P_j b_1,\ldots, P_{j+m_2-1}P_{j+m_1-1} P_j b_{l_1-1}, 
\\&
P_{j+m_2-1}Q_{j+m_1} P_j b_{l_1}, 
P_{j+m_2-1} P_{j+m_1} P_j b_{l_1+1}, \ldots, P_{j+m_2-1}P_{j+m_1} P_j b_{l_2-1}, 
\\ & 
Q_{j+m_2} P_{j+m_1} P_j b_{l_2}, P_{j+m_2} P_{j+m_1} P_j b_{l_2+1},\ldots, 
P_{j+m_2} P_{j+m_1} P_j b_{n+2}\big ).
\end{align*}
For $l_1>l_2$ there is an obvious modification.

There are $(n+2)(n+1)=O(n^2)$ terms in the sum $\sum_{\substack{1\leq l_1\ne l_2\leq n+2}}$. 
It is therefore our task to show that
\Be \label{fixedl1l2termest}
\Big|\sum_{m_1,m_2} \sum_j\la_{j,l_1,l_2}^{m_1,m_2}(b_1,\dots,b_{n+2})\Big| \lc
\sup_j \LpN{1}{\vsig_j}\log^3(1+n \Ga_\eps)\big(\prod_{i=1}^n\|b_i\|_\infty \big)
\|b_{n+1}\|_2\|b_{n+2}\|_2;
\Ee
then summing the  $O(n^2)$ terms will complete the proof.


\subsection{Proof of the bound  \eqref{fixedl1l2termest}}
For $k\in \Z$, $1\leq l\leq n+2$, let $$X_{k}^{1,l},X_{k}^{2,l}\in \{P_k, P_{k-1}\}.$$
For $1\leq l_1,l_2\leq n+2$, $j,k_1,k_2\in \Z$, define the operator
\begin{multline*}
\int b_{l_1}(x) \, T_{j,l_1,l_2}^{m_1,m_2}
 b_{l_2}(x)\: dx\\
:= \La[\vsigjj](X_{j+m_1}^{1,1} X_{j+m_2}^{2,1} P_j b_1,\ldots, X_{j+m_1}^{1,n} X_{j+m_2}^{2,n} P_j b_{n}, X_{j+m_1}^{1,n+1} Q_{j+m_2} P_j  b_{n+1}, Q_{j+m_1} P_j b_{n+2}),
\end{multline*}
where
we have suppressed the dependance of $T_{j,l_1,l_2}^{m_1,m_2}$
 on $b_l$, $l\ne l_1,l_2$.

\begin{lemma}\label{LemmaPartAllPtsFinalOpN} Let $\rho_{j, m_1,m_2}= \min\{2^j, 2^{j+m_1}, 
2^{j+m_2}\}$.  There is a $c>0$ (independent of $n$ so that
for  $\eps'>c\eps$ 
\begin{equation}\label{EqnPartAllPtsToShowOpN}
\big\| \Dil_{\rho_{m_1,m_2, j}^{-1}} T_{j,l_1,l_2}^{m_1,m_2}
\big\|_{\Op_{\eps'}} \lesssim 
\min\{2^{-\eps'|m_1|}, 2^{-\eps'|m_2|}\}
n^2\sBN{\epsilon}{\vsig_j}
\prod_{l\ne l_1,l_2} \LpN{\infty}{b_l},
\end{equation} and 
\begin{equation*}
\|\Dil_{\rho_{m_1,m_2, j}^{-1}} T_{j,l_1,l_2}^{m_1,m_2}\|_{\Op_{0}}  
\lesssim \LpN{1}{\vsig_j} \prod_{l\ne l_1,l_2} \|b_l\|_\infty\,.
\end{equation*}
\end{lemma}
\begin{proof}
The bound for
$\|\Dil_{\rho_{m_1,m_2, j}^{-1}} T_{j,l_1,l_2}^{m_1,m_2}\|_{\Op_{0}}  $,
and,  equivalently, for
$\| T_{j,l_1,l_2}^{m_1,m_2}\|_{\Op_{0}}  $  is immediate, so 
 we focus only on the bound for 
$ \| \Dil_{\rho_{m_1,m_2, j}^{-1}} T_{j,l_1,l_2}^{m_1,m_2}
\|_{\Op_{\eps'}} $. 
Fix $l_1,l_2$.  We may assume $\LpN{\infty}{b_l}=1$, $l\ne l_1,l_2$.
We distinguish the cases (i) $m_1, m_2\ge 0$,
 (ii) $m_1\le \min \{0, m_2\}$,
(iii) $m_2 \le \min\{0, m_1\}$.

\smallskip

{\it (i)  The case $m_1, m_2\ge 0$.} Now $\rho_{j,m_1,m_2}=2^j$. 
One uses that, for $m\ge 0$,  $Q_{j+m} P_j = 2^{-m} X_{m,j}$, where $X_{m,j} f = f*\dil{\phi_{m,j}}{2^j}$ and $\{\phi_{m,j}:m\ge 0\}$ is a bounded subset of $C^\infty$ functions supported in $\{|y|\le 2\}$.
 Then the bound
 \begin{equation*}
\big\|\Dil_{2^{-j}} T_{j,l_1,l_2}^{m_1,m_2}\big\|_{\Op_{\eps_1}} \lesssim 2^{-m_1-m_2} \sBN{\epsilon}{\vsig_j}
\end{equation*}
follows quickly.  \eqref{EqnPartAllPtsToShowOpN} follows in this case.

\smallskip

{\it  (ii) The case   $m_1\le \min \{0, m_2\}$,} that is,  $\rho_{j,m_1,m_2}=2^{j+m_1}$. 
Lemma  \ref{PropPart1IMinusOpProp1} (combined with Theorem \ref{ThmOpResAdjoints}) shows that we have
\begin{equation*}
\big\|\Dil_{2^{-j-m_1}} T_{j,l_1,l_2}^{m_1,m_2}\big\|_{\Op_{\eps_2}}
\lesssim 2^{-\epsilon_2 m_1} n^2 \sBN{\epsilon}{\vsig_j}.
\end{equation*}
Using that  $X^{1,n+1}_{j+m_1} Q_{j+m_2} = 2^{-(m_2-m_1)} X_{j,m_1,m_2} f$, where
$X_{j,m_1,m_2} f = f*\dil{\phi_{j,m_1,m_2}}{2^{j+m_1}}$ and $\{\phi_{j,m_1,m_2}: m_2\geq m_1\}\subset \Czip{B^d(2)}$ is a bounded set, the bound
\begin{equation*}
\big\|\Dil_{2^{-j-m_1}} 
T_{j,l_1,l_2}^{m_1,m_2}\big\|_{\Op_{\eps_3}}
 \lesssim 2^{-(m_2-m_1)} \sBN{\epsilon}{\vsig_j}
\end{equation*}
follows easily.  Combining these two estimates, \eqref{EqnPartAllPtsToShowOpN} follows.

\smallskip

{\it  (iii) The case $m_2\le \min\{0, m_1\}$}, that is 
 $\rho_{j,m_1,m_2}=2^{j+m_2}$. 
Now we use an integration by parts argument as in the proof of Proposition \ref{OpestST}
to  obtain
\begin{equation*}
\big\|\Dil_{2^{-j-m_2}} T_{j,l_1,l_2}^{m_1,m_2}\big \|_{\Op_{\eps_4}}
\lesssim 2^{-(m_1-m_2)} \sBN{\epsilon}{\vsig_j}.
\end{equation*}
Using Lemma  \ref{PropPart1IMinusOpProp1} (combined with Theorem \ref{ThmOpResAdjoints}), as above, we have
\begin{equation*}
\big\|\Dil_{2^{-j-m_2}} T_{j,l_1,l_2}^{m_1,m_2}\big \|_{\Op_{\eps_5}}
\lesssim 2^{-\epsilon' m_1} n^2 \sBN{\epsilon}{\vsig_j}.
\end{equation*}
Combining these two estimates 
yields \eqref{EqnPartAllPtsToShowOpN} in this last case and the proof is complete.
\end{proof}

\begin{prop}\label{PropPartAllPtsL2Sumn1n2}
For each $m_1, m_2$, 
$\sum_{j\in \Z} T_{j,n+1,n+2}^{m_1,m_2}$ converges in the strong operator topology as operators $L^2\to L^2$ (with equiconvergence with respect to the $\{(b_1,\dots, b_n): \|\|b_i\|_\infty\le C\}$) and the estimates
\begin{equation} \label{Tm1m2L2est}
\Big\|\sum_{j\in \bbZ}
 T_{j,n+1,n+2}^{m_1,m_2}
\Big\|_{L^2\to L^2}\lesssim  \min\big\{ 2^{-\eps'(|m_1|+|m_2|)} n^M \sup_j\|\vsig_j\|_{\cB_\eps},\,
\sup_j\|\vsig_j\|_{L^1}\big\}
\prod_{i=1}^n\|b_i\|_\infty,
\end{equation}
for suitable $M\lc 1$, and
\begin{equation}\label{sumTm1m2L2est}
\sum_{m_1=-\infty}^\infty\sum_{m_2=-\infty}^\infty
\Big\|\sum_{j\in \bbZ}
 T_{j,n+1,n+2}^{m_1,m_2}
\Big\|_{L^2\to L^2}\lesssim  
\sup_j \LpN{1}{\vsig_j}\log^2(1+n \Ga_\eps)\prod_{i=1}^n\|b_i\|_\infty.
\end{equation}
hold.
\end{prop}
\begin{proof}
With Lemma \ref{LemmaPartAllPtsFinalOpN} in hand, 
\eqref{Tm1m2L2est}
is based on  almost orthogonality (Lemma  \ref{ThmAlmostOrthogonal}) and 
 follows just as in the proof of Proposition \ref{mainpropforSmm}. \eqref{sumTm1m2L2est} follows after summing in $m_1$, $m_2$.
\end{proof}
We combine the above results with several applications of Theorem \ref{ThmJourne2} to prove our last proposition.
\begin{prop}\label{CorPartAllPtsL2Allls}
For $1\leq l_1,l_2\leq n+2$,
\begin{equation*}
\Big\|\sum_{j,m_1,m_2} T_{j,l_1,l_2}^{m_1,m_2}\Big\|_{L^2\to L^2}
\lesssim  
\sup_j \LpN{1}{\vsig_j}\log^3\q(1+n \Ga_\eps)
\prod_{i=1}^n \|b_i\|_\infty.
\end{equation*}
The sum converges in the strong operator topology, with equiconvergence with respect to
$\{(b_1,\dots, b_n): \|b_i\|_\infty \le C\}$.
\end{prop}
\begin{proof}
For $r\in \Z$ define
\begin{equation*}
S_{r,l_1,l_2} := \sum_{\substack{j,m_1,m_2:\\ \min\{j,j+m_1, j+m_2\}=r }}
 T_{j,l_1,l_2}^{m_1,m_2}.
\end{equation*}
Note that $\sum_{r\in \Z} S_{r,l_1,l_2} = \sum_{j,m_1,m_2\in \Z} T_{j,l_1,l_2}^{m_1,m_2}$,
and Lemma \ref{LemmaPartAllPtsFinalOpN} shows
\begin{equation}\label{EqnPartAllPtsOpNT1}
\big\|\Dil_{2^{-r}} S_{r,l_1,l_2}\big\|_{\Op_\eps'} 
\lesssim n^M \sup_j \sBN{\epsilon}{\vsig_j}
\prod_{l\ne l_1,l_2} \|b_l\|_\infty,
\end{equation}
and
\begin{equation}\label{EqnPartAllPtsOpNT2}
\big\|\Dil_{2^{-r}} S_{r,l_1,l_2}\big\|_{\Op_0} 
\lesssim  
\log^2(1+n\Ga_\eps) \sup_j \LpN{1}{\vsig_j}
\prod_{l\ne l_1,l_2} \|b_l\|_\infty.
\end{equation}
By Proposition \ref{PropPartAllPtsL2Sumn1n2},
\begin{equation*}
\Big\|\sum_{r\in \Z} S_{r,n+1,n+2}\Big\|_{L^2\to L^2} \lesssim 
\sup_j \LpN{1}{\vsig_j}\log^2\q(1+n \Ga_\eps)
\prod_{i=1}^n \|b_i\|_\infty
\end{equation*}
and using \eqref{EqnPartAllPtsOpNT1},  \eqref{EqnPartAllPtsOpNT2}, Theorem \ref{ThmJourne2}
shows
\begin{equation*}
\Big\| \sum_{r\in \Z} S_{r,n+1,n+2}\Big\|_{H^1\rightarrow L^1} \lesssim
\q(\sup_j \LpN{1}{\vsig_j}\w)\log^3\q(1+n \Ga_\eps)
\prod_{i=1}^n \|b_i\|_\infty.
 \end{equation*}
 Here we have convergence in the strong operator topology (as operators $H^1\to L^1$),
 with equicontinuity with respect to $b_1,\dots, b_n$ in bounded subsets of $L^\infty(\bbR^d)$. 
Using the definition of $S_{r,l_1,l_2}$, this is equivalent to
\begin{equation*}
\Big\| \sum_{r\in \Z} S_{r,l_2,n+2}\Big\|_{H^1\rightarrow L^1} \lesssim 
\sup_j \LpN{1}{\vsig_j}\log^3\q(1+n \Ga_\eps)
\prod_{l\ne l_2,n+2} \|b_l\|_\infty,
\end{equation*}
with convergence in the strong operator topology (as operators $H^1\to L^1$)
 with equicontinuity with respect to $b_l$, $l\notin \{l_2,n+2\}$, 
 in bounded subsets of $L^\infty(\bbR^d)$.  This argument will now be used repeatedly. 
Using this $L^1\to L^1$ result together with  \eqref{EqnPartAllPtsOpNT1} and \eqref{EqnPartAllPtsOpNT2}, Theorem \ref{ThmJourne2}
shows
\begin{equation*}\Big\|
\sum_{r\in \Z} S_{r,l_2,n+2} \Big\|_{L^2\to L^2} \lesssim 
\sup_j \LpN{1}{\vsig_j}\log^3\q(1+n \Ga_\eps)
\prod_{l\ne l_2,n+2} \|b_l\|_\infty.
\end{equation*}
Taking transposes, this shows
\begin{equation*}\Big\|
\sum_{r\in \Z} S_{r,n+2,l_2} \Big\|_{L^2\to L^2} \lesssim 
\sup_j \LpN{1}{\vsig_j}\log^3\q(1+n \Ga_\eps)
\prod_{l\ne l_2,n+2} \|b_l\|_\infty.
\end{equation*}
Using this,  \eqref{EqnPartAllPtsOpNT1} and \eqref{EqnPartAllPtsOpNT2}, Theorem \ref{ThmJourne2}
shows
\begin{equation*}\Big\|
\sum_{r\in \Z} S_{r,n+2,l_2} \Big\|_{H^1\to L^1} \lesssim 
\sup_j \LpN{1}{\vsig_j}\log^3\q(1+n \Ga_\eps)
\prod_{l\ne l_2, n+2} \|b_l\|_\infty.
\end{equation*}
Using the definition of $S_{r,l_1,l_2}$, this is equivalent to
\begin{equation*}\Big\|
\sum_{r\in \Z} S_{r,l_1,l_2} \Big\|_{H^1\to L^1} \lesssim 
\sup_j \LpN{1}{\vsig_j}\log^3\q(1+n \Ga_\eps)
\prod_{l\ne l_1,l_2} \|b_l\|_\infty.
\end{equation*}
Finally, using this again with \eqref{EqnPartAllPtsOpNT1} and \eqref{EqnPartAllPtsOpNT2}, one last application of Theorem \ref{ThmJourne2}
completes the proof of the proposition.
\end{proof}

\section{Interpolation}\label{interpolsect}

We  use complex interpolation to show that  the  $L^{p_1}\times\dots\times L^{p_{n+2}}$ estimates in Theorem 
\ref{ThmOpResBoundGen} follow  from  the special case in Theorem
\ref{ThmOpResBound}, together with Theorem \ref{ThmOpResAdjoints}.

Let  $K= \sum_j \vsig_j^{(2^j)}$ be  as in the assumption of Theorem \ref{ThmOpResBoundGen}  with $\sup\|\vsig_j\|_{\cB_\eps}<\infty$. Define for a permutation $\vp$ of $\{1,\dots, n+2\}$
\[\La^\vp[K](b_1,\dots, b_{n+2})= 
\La[K] (b_{\vp\q(1\w)},\ldots, b_{\vp\q(n+2\w)}) \]
so that $\La^\vp[K]= \La[K^\vp]$ with 
\[K^\vp = \sum_j (\ell_\vp \vsig_j)^{(2^j)} \]
where $\ell_\vp$ is as in Theorem  \ref{ThmOpResAdjoints}.
There is $\eps'>c(\eps)$, $B\ge 1$, both independent of $n$,  such that for all permutations
$\|\ell_\vp\sig\|_{\cB_{\eps'}} \le B n^2\|\vsig\|_{\cB_\eps}$ and
$\|\ell_\vp\sig\|_{L^1}=\|\vsig\|_{L^1}$. 
As a consequence we get for any pair $l_1,l_2 \in \{1,\dots, n+2\}$, $l_1\neq l_2$ the estimate
\begin{align}
&\q|\La[K]\q(b_1,\ldots, b_{n+2}\w)\w|
\notag
\\&\leq C_{\eps',d,\delta}n^2  \q \sup_{j\in \Z} \sBzN{\vsig_j}\w\, \log^3\!\q\big(2 + n\frac{Bn^2 \sup_{j\in \Z} \sBN{\epsilon}{\vsig_j}} {\sup_{j\in \Z} \sBzN{\vsig_j}}\w\big) 
\Big(\prod_{l\notin \{l_1,l_2\} }\|b_l\|_\infty\Big) \|b_{l_1}\|_p\|b_{l_2}\|_{p'}
\notag
\\&\le A
\Big(\prod_{l\notin \{l_1,l_2\} }\|b_l\|_\infty\Big) \|b_{l_1}\|_p\|b_{l_2}\|_{p'}
\label{Amultlinearbd}
\end{align}
where $1+\delta\le p\le 2$ and 
\[A:= 
 3^3BC_{\epsilon',d,\delta}n^2  \q \sup_{j\in \Z} \sBzN{\vsig_j}\w\, \log^3\!\q\big(2 + n\frac{ \sup_{j\in \Z} \sBN{\epsilon}{\vsig_j}} {\sup_{j\in \Z} \sBzN{\vsig_j}}\w\big) .\]
 Let
 $\cR$ be the set of points 
 $(p_1^{-1}, \dots, p_{n+2}^{-1}) \in [0,1]^{n+2}$ for which
 the inequality
 \Be \label{generalpi}
 |\La[K]\q(b_1,\ldots, b_{n+2}\w)\w|
 \le A \prod_{i=1}^{n+2} \|b_i\|_{p_i}
 \Ee 
 holds for all $(b_1,\dots, b_{n+2})\in
 L^{p_1}(\bbR^d)\times \dots \times L^{p_{n+2}}(\bbR^d)$.
 
We note  that if   $P_0=(p_{1,0}^{-1}, \dots, p_{n+2,0}^{-1} )$ 
and
$P_1=(p_{1,1}^{-1}, \dots, p_{n+2,1}^{-1} )$  both belong to $\cR$ then, by complex interpolation for multilinear functionals, we also  have for $0\le \vartheta\le 1$
\[
| \La[K]\q(b_1,\ldots, b_{n+2}\w)\w|
 \le A \prod_{i=1}^{n+2} \|b_i\|_{[L^{p_{i,0}}, L^{p_{i,1}}]_\vartheta}
 \]
 where $[\cdot, \cdot]_\theta$ denotes Calder\'on's complex interpolation method, see 
Theorem 4.4.1 in \cite{bergh-lofstrom}.
By Theorem  5.1.1 in \cite{bergh-lofstrom} (a version of the Riesz-Thorin theorem) we have the identification of the 
complex interpolation norm with the standard $L^p$ norm:
\[ \|f\|_{[L^{p_{i,0}}, L^{p_{i,1}}]_\vartheta}  = \|f\|_{L^p}, \quad   p^{-1}= (1-\vth) p_{i,0}^{-1} +
\vth p_{i,1}^{-1}.\]
We conclude that the set $\cR$ is convex. 
Denote by $e_i$, $i=1,\dots, n+2, $ the standard basis in $\bbR^{n+2}$. 
By \eqref{Amultlinearbd},  $\cR$ contains all points in $\bbR^{n+2}$ of the form 
$$P_{i,j}(\delta)
= \frac{\delta}{1+\delta} e_i+\frac{1}{1+\delta}  e_j, \quad i\neq j.$$
Let
$$\fP_\delta=\Big\{x\in \bbR^{n+2}:  \sum_{i=1}^{n+2} x_i=1, \quad 0\le x_j \le (\delta+1)^{-1}, \, j=1,\dots, n+2\Big\}\,.$$
$\fP_\delta$ is a compact convex subset of $\bbR^{n+2}$,  of dimension $n+1$. It is easy to see that 
$\{P_{i,j}(\delta):i\neq j\}$ is the set of  the extreme points of $\fP_\delta$. 
By Minkowski's theorem (see e.g. Theorem 2.1.9 in \cite{hoermander}) every point in 
$\fP_\delta$ is a convex combination of (at most $n+2$ of) the extreme points $P_{i,j}(\delta)$.
Thus we can conclude $$ \fP_\delta\subset \cR,$$  
 and we have verified 
\eqref{generalpi} for all $(n+2)$-tuples of  exponents $p_i$, with $\sum_{i=1}^{n+2}p_i^{-1}=1$ and 
$1+\delta\le p_i\le \infty$. This completes the proof.
\qed


\bibliographystyle{amsalpha}


\begin{thebibliography}{10}

\bibitem{bergh-lofstrom} J. Bergh, J. L\"ofstr\"om, \emph{Interpolation spaces.}
Springer-Verlag, Berlin, Heidelberg, New York, 1976.


\bibitem{bianchini} S. Bianchini, {On Bressan's conjecture on mixing properties of vector fields.} Self-similar solutions of nonlinear PDE, 13--31, Banach Center Publ., 74, Polish Acad. Sci., Warsaw, 2006.

\bibitem{bownik}
M. Bownik, Boundedness of operators on Hardy spaces via atomic decompositions. 
Proc. Amer. Math. Soc. 133 (2005), no. 12, 3535--3542.

\bibitem{bressan}  A. Bressan,
{A lemma and a conjecture on the cost of rearrangements.}
Rend. Sem. Mat. Univ. Padova 110 (2003), 97--102.


\bibitem {calderon}
    A.P. Calder\'on,
{Commutators of singular integrals.} Proc. Nat. Acad. Sci. U.S.A.,
53 (1965), 1092--1099.

\bibitem {cpcalderon}
    C.P. Calder\'on, {On commutators of singular integrals.} Studia Math., 53 (1975), 139--174.


\bibitem{cj} M. Christ, J.L. Journ\'e,
{Polynomial growth estimates for multilinear singular integral operators.}
Acta Math. 159 (1987), no. 1-2, 51--80.

\bibitem{cdm}
R.R. Coifman, G.  David, Y. Meyer,  La solution des conjecture de Calder\'on. Adv. in Math. 48 (1983), no. 2, 144--148. 

\bibitem{coi-me}
    R. Coifman and Y. Meyer, {On commutators of singular integral and bilinear singular integrals.}
Trans. Amer. Math. Soc., 212 (1975), 315--331.

\bibitem{coi-mc-me} R. Coifman, A. McIntosh, Y. Meyer,   L'int\'egrale de Cauchy d\'efinit un op\'erateur born\'e sur $L^2$ pour les courbes lipschitziennes.  Ann. of Math. (2) 116 (1982), no. 2, 361--387.


\bibitem{coi-me-audela}    R. Coifman and Y. Meyer, \emph{Au del\`a des 
op\'erateurs pseudo-diff\'erentiels,} Ast\'erisque, 57 (1978).

\bibitem{Crippa-deLellis} G. Crippa, C.  de Lellis, \emph{Estimates and regularity results for the DiPerna-Lions flow,}
 J. Reine Angew. Math. 616 (2008), 15--46.

\bibitem{david-journe} G. David, J.-L.Journ\'e, A boundedness criterion for generalized Calder\'on-Zygmund operators. Ann. of Math. (2) 120 (1984), 371-397.


\bibitem{dgy} X.T. Duong, L. Grafakos, and L. Yan, {Multilinear operators with non-smooth kernels and commutators of singular integrals.}
 Trans. Amer. Math. Soc. 362 (2010), no. 4, 2089--2113.

\bibitem{frs}
C. Fefferman, N.M.  Rivi\`ere, Y. Sagher, {
Interpolation between $H^p$ spaces: the real method,}
Trans. Amer. Math. Soc. 191 (1974), 75--81.

\bibitem{fs}
C. Fefferman and E. M. Stein, {$H^p$ spaces of several variables,} Acta Math. 129 (1972), 137--194.

\bibitem{robfeff} R. Fefferman, 
On an operator arising in the Calder\'on-Zygmund method of rotations and
the Bramble-Hilbert lemma. Proc. Nat. Acad. Sci. U.S.A., 80 (1983), 3877--3878.

\bibitem{grafakos-czech} L. Grafakos, Multilinear harmonic analysis.
In: \emph{Nonlinear analysis, function spaces and applications} Vol. 9, 63--104, Acad. Sci. Czech Repub. Inst. Math., Prague, 2011. 

\bibitem{grafakos-honzik} L. Grafakos,  and P. Honz\'ik,  A weak type estimate for commutators. Int. Math. Res. Not. IMRN 2012, no. 20, 4785--4796. 


\bibitem{grli}    L. Grafakos and X. Li, {Uniform bounds for the bilinear Hilbert transforms, I.} Ann. of Math. (2), 159 (2004), 889--933.

\bibitem{grafakos-torres}
   L. Grafakos and R.H. Torres, {Multilinear Calder\'on-Zygmund theory.}
 Adv. in Math., 165 (2002), 124--164.

\bibitem{hsss} M. Had\v zi\'c, A. Seeger, C. Smart, B. Street, in preparation.

\bibitem{han-hofmann}
Y.-S. Han, S. Hofmann, $T1$ theorems for Besov and Triebel-Lizorkin spaces,
Trans. Amer. Math. Soc. 337, no.2 (1993), 839--853.

\bibitem{hoermander} L. H\"ormander, \emph{Notions of Convexity}, Birkh\"auser, Progress in mathematics, vol. 127, 1994.



\bibitem{hofmann} S. Hofmann, {Boundedness criteria for rough singular integrals,} Proc. London Math. Soc. (3) 70 (1995), no. 2, 386--410.

\bibitem{hofmann-off-diagonal} S. Hofmann, 
An off-diagonal T1 theorem and applications. 
With an appendix "The Mary Weiss lemma'' by Loukas Grafakos and the author. 
J. Funct. Anal. 160 (1998), no. 2, 581--622. 

\bibitem{jajo} S. Janson, P.W.  Jones, {Interpolation between $H^p$
spaces: the complex method.}
J. Funct. Anal. 48 (1982), no. 1, 58--80.



\bibitem{journe} J.-L.Journ\'e,
\emph{Calder\'on-{Z}ygmund operators, pseudodifferential operators
              and the {C}auchy integral of {C}alder\'on},
   Lecture Notes in Mathematics 994, 
   Springer-Verlag, Berlin, 1983.








\bibitem{MSV} S. Meda, P.  Sj\"ogren, M.  Vallarino, 
On the $H^1-L^1$ boundedness of operators.
Proc. Amer. Math. Soc. 136 (2008), no. 8, 2921--2931. 


\bibitem{MC} Y.  Meyer, R.  Coifman, \emph{Wavelets. Calder\'on-Zygmund and multilinear operators.} Translated from the 1990 and 1991 French originals by David Salinger. Cambridge Studies in Advanced Mathematics, 48. Cambridge University Press, Cambridge, 1997.



\bibitem{muscalu2} C. Muscalu, {
Calder\'on commutators and the Cauchy integral on Lipschitz curves revisited, II. Cauchy integral and generalizations.}
 Rev. Mat. Iberoam. 30 (2014), no. 3, 1089--1122. 

 \bibitem{muscalu-schlag2} C. Muscalu, W.  Schlag,
 \emph{Classical and multilinear harmonic analysis.} Vol. II. Cambridge Studies in Advanced Mathematics, 138. Cambridge University Press, Cambridge, 2013.

\bibitem{NRS} A. Nagel, F.  Ricci, E.M.  Stein, Singular integrals with flag kernels and analysis on quadratic CR manifolds. J. Funct. Anal. 181 (2001), no. 1, 29--118. 

\bibitem{see} A. Seeger, 
A weak type bound for a singular integral. 
Rev. Mat. Iberoam. 30 (2014), no. 3, 961--978. 




\bibitem{stein-si} E. M. Stein, \emph{Singular integrals and differentiability properties of functions.} Princeton Mathematical Series, No. 30,  Princeton University Press, Princeton, N.J. 1970.

\bibitem{stein-ha} E.M. Stein, \emph{Harmonic analysis: real-variable methods, orthogonality, and oscillatory integrals.} With the assistance of Timothy S. Murphy. Princeton Mathematical Series, 43. Monographs in Harmonic Analysis, III. Princeton University Press, Princeton, NJ, 1993. 

\bibitem{th} C. Thiele, {A uniform estimate.} Ann. of Math. (2) 156 (2002), no. 2, 519--563.

\bibitem{treves} F. Tr\`eves, \emph{Topological vector spaces, Distributions and Kernels.} Unabridged republication of the 1967 original. Dover Publications, Inc., Mineola, NY, 2006.

\bibitem{triebel} H. Triebel, \emph{Theory of function spaces.} Birkh\"auser Verlag,
Basel 1983.


\bibitem{wolff}
T. H. Wolff,
{A note on interpolation spaces.}  Harmonic analysis
(Minneapolis, Minn., 1981), pp. 199--204,
Lecture Notes in Math., 908, Springer, Berlin-New York, 1982.


\end{thebibliography}




\end{document}